\documentclass[10pt]{amsart}
\usepackage{amsmath}
\usepackage{amssymb}
\usepackage{enumerate}
\usepackage{amsbsy}
\usepackage{amsfonts}
\usepackage{color}
\usepackage{upgreek}
\usepackage{slashed}
\usepackage{mathrsfs}
\usepackage[notcite,notref]{showkeys}
\newtheorem{thm}{Theorem}[section]
\newtheorem{lem}[thm]{Lemma}
\newtheorem{prop}[thm]{Proposition}
\newtheorem{cor}[thm]{Corollary}
\newtheorem{defn}[thm]{Definition}
\newtheorem{rem}[thm]{Remark}

\numberwithin{equation}{section}

\begin{document}

	\title[Geometric analysis of Nonlinear Dirac equations]{A spinor-adapted geometric approach to nonlinear Dirac systems on the Schwarzschild black hole background }

    \author[S. Hong]{Seokchang Hong}
    \address{Universit\"at M\"unster, Mathematisches Institut,
Einsteinstrasse 62 48149 M\"unster, Germany}
    \email{seokchangh.11@uni-muenster.de}

	
	\begin{abstract}

Building on our previous work \cite{hong}, we extend the spinor-adapted geometric method for nonlinear Dirac systems from small perturbations of Minkowski spacetime to the Schwarzschild background. As a linear foundation, we establish integrated local energy decay estimates and an \(r^p\)-weighted energy hierarchy for spinor fields, adapting the physical-space approach of Dafermos and Rodnianski to the Dirac setting.

The main focus of the paper is nonlinear. We consider a semilinear tensorial wave--Dirac model motivated by the wave-gauge reduction of the Einstein--Dirac system and prove global small-data stability. A principal difficulty is the apparent loss of derivatives arising from the coupling between the first-order Dirac equation and the second-order tensorial wave equation. We overcome this difficulty through a spinor-adapted energy hierarchy that exploits the complementary roles of the Dirac-current energy and the spinorial wave energy. In particular, the nonlinear estimates are closed without requiring the top-order Dirac-current estimate.

The nonlinear analysis demonstrates that the spinorial structures developed in \cite{hong} remain essential on the black-hole background. The combination of Dirac-current energy, spinorial wave energy, Kosmann Lie derivatives, and the Clifford-algebraic null structure enables us to close the nonlinear estimates without imposing an independent null condition on the tensorial wave equation. Thus, the nonlinear closure relies fundamentally on structures intrinsic to the Dirac equation.

A second main result is the global small-data stability of the Maxwell--Dirac system in the Lorenz gauge on the Schwarzschild background, under a vanishing-charge assumption. This extends the nonlinear stability theory to a genuine gauge-coupled Dirac system and demonstrates the applicability of the spinor-adapted physical-space approach beyond the tensorial wave--Dirac model.


	\end{abstract}

		\maketitle

\tableofcontents

\section{Introduction}\label{sec:intro}
We consider the tensorial wave-Dirac system
\begin{align}\label{eq-tensor-dirac}
    \begin{aligned}
        i\gamma^\mu\nabla_\mu\psi
        &= \gamma^\nu\psi\,\nabla_\nu \mathrm{Tr} h,\\
        \Box_g h_{\mu\nu}
        &= \mathscr T_{\mu\nu}[\psi],
    \end{aligned}
\end{align}
on the exterior region of the Schwarzschild spacetime $(\mathcal M,g)$. Here $\psi$ is a section of the spinor bundle over $\mathcal M$, $h$ is a symmetric $2$-tensor, $\mathscr T_{\mu\nu}[\psi]$ denotes the stress-energy tensor associated with the Dirac field, and $\Box_g$
is the wave operator acting on each component $h_{\mu\nu}$ of the tensor field $h$. We work with a foliation $\{\Sigma_\tau\}_{\tau\ge0}$
adapted to the black-hole geometry and asymptotically
outgoing null in the far region.

Our first main result is the global small-data stability
of \eqref{eq-tensor-dirac}:
\begin{thm}[Informal version]\label{main-thm-informal}
For sufficiently regular and sufficiently decaying
initial data of sufficiently small size, the system
\eqref{eq-tensor-dirac} admits a global solution in
the exterior Schwarzschild region. Moreover, the
solution obeys quantitative energy and pointwise
decay estimates.
\end{thm}
The system \eqref{eq-tensor-dirac} is inspired by the wave-gauge formulation of the Einstein-Dirac equations and is intended to isolate the nonlinear interactions between a first-order spinorial equation and a second-order tensorial wave equation. In particular, the derivative coupling in the Dirac equation gives rise to a potential loss of derivatives at the highest order. We overcome this difficulty by constructing a spinor-adapted hierarchy of complementary energy estimates, combining the first-order Dirac-current energy with the second-order spinorial wave energy. A crucial feature of this hierarchy is that the top-order Dirac-current estimate is not required for the nonlinear closure. Together with the Clifford-algebraic structure of the Dirac equation, this allows us to close the nonlinear estimates
without imposing a separate null condition on the
tensorial wave equation.

We emphasise that \eqref{eq-tensor-dirac} is not
intended to reproduce the full perturbative Einstein-Dirac system around Schwarzschild. In particular, the tensorial wave equation does not incorporate the curvature-dependent and gauge-dependent linear terms governing gravitational perturbations of the Schwarzschild metric. Rather, our objective is to develop a spinor-adapted geometric method for nonlinear Dirac couplings
on a fixed black-hole background, independently of the additional difficulties associated with the linear and quasilinear dynamics of the Einstein equations.

To demonstrate that the underlying method is not restricted to the model system \eqref{eq-tensor-dirac}, we also consider the Maxwell--Dirac system in the Lorenz gauge with vanishing total charge,
\begin{align}
\begin{aligned}
i\gamma^\mu(\nabla_\mu-iA_\mu)\psi &=0,\\
\Box_g A_\mu - R_\mu{}^\nu A_\nu
&= \langle\psi,\gamma_\mu\psi\rangle,\\
\nabla^\mu A_\mu&=0.
\end{aligned}
\end{align}
Since the Schwarzschild spacetime is Ricci-flat,
the curvature term in the equation for $A$ vanishes.
Our second main result is the following:
\begin{thm}[Informal version]
\label{main-thm-maxwell-dirac-informal}
For sufficiently regular and sufficiently decaying
initial data of sufficiently small size, satisfying
the Lorenz constraint and having vanishing total
charge, the Maxwell--Dirac system on the exterior
Schwarzschild spacetime admits a global solution.
Moreover, the solution satisfies quantitative
energy and pointwise decay estimates.
\end{thm}

\subsection{Motivation}
The starting point of the present work is the
spinor-adapted geometric method developed in our previous work \cite{hong}. There, we observed that the nonlinear analysis of Dirac systems cannot be based solely on the wave-energy estimates obtained by squaring the Dirac equation. Although the squared equation provides the second-order estimates needed for decay, it no longer directly retains some of the first-order spinorial structure that is essential in treating the nonlinear interactions.
For this reason, the Dirac-current energy must be used in conjunction with the spinorial wave energy. The complementary use of these two energies is an essential ingredient of
our nonlinear analysis.

Two further structures play a fundamental role. The Kosmann Lie derivative provides the appropriate commutation theory for spinor fields, incorporating the spinorial action of spacetime symmetries, while the Clifford algebra gives rise to a null
structure in the nonlinear spinorial interactions. These ingredients---the simultaneous use of Dirac-current and spinorial wave energies, spinorial commutation, and the Clifford-algebraic null structure---form the basic analytic approach that we carry over from \cite{hong}.

In \cite{hong}, however, the spinorial structure interacted favorably with an additional structure in the tensorial part of the system. The tensor field considered there was a closed $2$-form and therefore possessed the Maxwell-type null structure. The nonlinear closure consequently resulted from an interplay between the null structure of the closed $2$-form and the Clifford-algebraic structure of the Dirac field.

The situation considered in the present paper is substantially different. The tensor field $h$ in \eqref{eq-tensor-dirac} is a general symmetric $2$-tensor and carries no analogous closed-form structure. In particular, we do not assume a separate null condition for the tensorial wave equation. This leads to a natural question:
to what extent can the spinorial structure itself compensate for the absence of a favorable null structure in the tensor sector?

A further difficulty arises from the mismatch in differential order between the Dirac and tensorial wave equations. The derivative coupling in \eqref{eq-tensor-dirac} leads to an apparent loss of derivatives in the top-order Dirac-current estimate. Rather than introducing a renormalised spinorial unknown, we avoid this difficulty through a carefully arranged energy hierarchy. Although the Dirac-current energy is used extensively at both lower and higher derivative orders, its top-order estimate is not needed at any stage of the bootstrap argument. The apparent derivative loss in the spinorial wave estimates is resolved differently: the Clifford algebra converts the dangerous second derivatives of the tensorial factor into its wave operator, which can then be replaced using the tensorial wave equation. Thus, the two derivative-loss mechanisms are avoided without increasing the regularity requirements or modifying the original spinorial unknown.

This leads to one of the main structural observations of the paper: the combination of a suitable energy hierarchy and the intrinsic Clifford-algebraic structure of the Dirac equation is sufficient to close the nonlinear system without relying on a Maxwell-type null structure in the tensorial part. In this sense, the present work separates the spinorial mechanism from the additional tensorial structure used in \cite{hong} and demonstrates that the former can serve as the principal mechanism for nonlinear closure.

The Maxwell-Dirac system provides a further application of the spinor-adapted method in a genuine gauge-coupled model. Under the Lorenz gauge condition, the combination of Dirac-current and spinorial wave energies, together with the Clifford-algebraic structure and the gauge constraint, leads to a global
small-data stability result. This demonstrates that the method is not restricted to the tensorial model \eqref{eq-tensor-dirac}, but also applies to a physically distinguished
nonlinear Dirac system.

\subsection{Previous and related results}

The analysis of wave equations on black-hole backgrounds has been developed extensively over the past several decades. Of particular relevance to the present work is the physical-space approach of Dafermos and Rodnianski, including the red-shift estimate near the event horizon, integrated local energy decay in the presence of trapping, and the \(r^p\)-weighted hierarchy leading to quantitative decay \cite{DR,DRred,DRlecture}. These estimates provide the basic geometric ingredients for our second-order analysis of the Dirac field. The principal difference is that, after squaring the Dirac equation, the resulting wave equation acts on the spinor bundle rather than on scalar fields, and its energy identity contains additional terms arising from the spinorial pairing and connection.

Dirac equations on curved spacetimes have attracted considerable interest in the mathematical literature. For spectral analysis of the Dirac operator, we refer the reader to \cite{bar}. See also \cite{artzcaccia,cacciasu1,cacciasu,caccia1,caccia2} for work on dispersive estimates for the Dirac equation. The characteristic Cauchy problem for Dirac fields on curved backgrounds was studied by H\"afner and Nicolas \cite{hafnernicolas}. More recently, small-data global well-posedness and scattering for cubic Dirac equations on asymptotically flat backgrounds were established by Herr and Hong \cite{herrh}.

The linear theory of Dirac fields on black-hole backgrounds is also well established. Nicolas developed a scattering theory for linear Dirac fields on spherically symmetric black holes \cite{nicolas}, and H\"afner and Nicolas studied scattering of massless Dirac fields on Kerr \cite{hanicolas}. More recently, Ma and Zhang \cite{mazhang} obtained sharp decay estimates for massless Dirac fields on Schwarzschild, including integrated local energy decay and Price-law asymptotics, by reducing the spin-\(\frac12\) components to a symmetric hyperbolic wave system. Thus, the purpose of the linear part of this paper is not to establish decay of Dirac fields on Schwarzschild for the first time. Rather, we develop a covariant physical-space formulation directly on the spinor bundle, adapted to the commuted and inhomogeneous estimates needed for the nonlinear problems considered below.

Nonlinear Dirac equations on curved or black-hole backgrounds have also been considered previously. Bachelot-Motet studied stationary nonlinear Dirac fields in Schwarzschild settings, including nonexistence results in the vacuum exterior and numerical stationary states outside a star \cite{bachelotmotet}. Paturel later proved the existence of infinitely many stationary solutions outside a massive ball by a variational method \cite{paturel}. These stationary problems differ from the global Cauchy and decay problem considered here. Bachelot \cite{bachelot} established global low-regularity solutions for several nonlinear relativistic field models, including the Dirac-Klein-Gordon system on Schwarzschild-type manifolds, by means of an abstract theory for semilinear hyperbolic systems with nonlocal interactions. The objective of the present work is different: we seek quantitative decay estimates that retain the horizon and trapping geometry and can be incorporated into a weighted nonlinear energy hierarchy. In this sense, our analysis is closer in spirit to the physical-space stability theory developed for nonlinear wave equations on black-hole exteriors.

The tensorial model studied here is motivated by the Einstein--Dirac system. Zhao and Wu proved local existence for its characteristic initial-value problem \cite{zhaowu}. For the massless Einstein--Dirac equations near Minkowski spacetime, Chen \cite{chen} proved global nonlinear stability in generalized wave coordinates. An important feature of that work is the simultaneous use of the first-order Dirac equation and a second-order wave equation for the spinor field, together with a resolution of a potential derivative-loss problem. More recently, LeFloch, Ma, and Zhang \cite{leflochmazhang} established global nonlinear stability of Minkowski spacetime for self-gravitating massive Dirac fields using a hyperboloidal--Euclidean approach and a gauge-invariant Clifford-algebraic formalism. These works address the genuinely quasilinear Einstein--Dirac system in a perturbative Minkowski setting. By contrast, the present paper fixes the Schwarzschild geometry and isolates a semilinear tensor--Dirac coupling that retains the first-/second-order derivative mismatch arising from the spinor--metric interaction. This allows us to separate the spinorial nonlinear mechanism from the genuinely quasilinear difficulties of the Einstein equations. The linear stability of Schwarzschild under gravitational perturbations is a distinct problem; see Dafermos, Holzegel, and Rodnianski \cite{dahorod1}. For the broader black-hole stability problem in ultimately Schwarzschildean spacetimes, see Holzegel \cite{holzegel10}.

The present work also continues the spinor-adapted geometric approach introduced in \cite{hong}. Related vector-field methods for nonlinear Dirac models in Minkowski spacetime were developed by Li and Zang \cite{lizhang}. In the setting of \cite{hong}, nonlinear closure relied on the simultaneous use of the Dirac-current and wave energies, spinorial commutation through the Kosmann Lie derivative, and the Clifford-algebraic null structure. The tensorial field considered there was a closed \(2\)-form and therefore possessed an additional Maxwell-type null structure. The present paper removes this favorable feature: the coupled field is now a general symmetric \(2\)-tensor, and no independent null condition is imposed on its wave equation. The role of classical null conditions in nonlinear wave equations is well known \cite{klai}; here, the Clifford-algebraic structure of the Dirac equation remains sufficiently strong to drive nonlinear closure without an additional tensorial null structure.

Finally, there is an extensive literature on the Maxwell--Dirac system. Ginoux and Müller \cite{ginoux} proved global solvability for the massless Dirac--Maxwell system on a class of asymptotically flat backgrounds using conformal methods and a symmetric-hyperbolic reduction. On Minkowski spacetime, Chadam and Glassey studied the Maxwell--Dirac equations with zero magnetic field in two space dimensions \cite{chagla1}, and Georgiev \cite{geogiev} established global solutions in the Lorenz gauge. Psarelli \cite{psarelli1} proved global existence in four-dimensional Minkowski spacetime using hyperboloidal foliations and the squared Dirac equation. D'Ancona, Foschi, and Selberg identified the null structure of the Maxwell--Dirac system and proved near-optimal local well-posedness \cite{dafosel}. Global well-posedness for the $(1+4)$-dimensional system was later established in the Coulomb gauge by Gavrus and Oh \cite{gavrusoh}, while Lee obtained scattering results in the Lorenz gauge \cite{klee}. Cho and Lee \cite{choklee} studied the global dynamics and modified scattering of the massive system in the Lorenz gauge, including an explicit nonlinear phase correction. Herr, Ifrim, and Spitz \cite{herrifrimspitz} proved small-data global well-posedness and modified scattering by working directly with the Dirac equation. For the source-free Maxwell system on Schwarzschild, Pasqualotto established decay estimates using the spin-\(\pm1\) Teukolsky equations \cite{pasqualotto}. Our Maxwell--Dirac result is different: it concerns the coupled system on the Schwarzschild exterior and uses physical-space estimates that accommodate the event horizon, photon-sphere trapping, and the \(r^p\)-weighted hierarchy. It shows that the spinor-adapted mechanism developed for the tensorial model also applies to a genuinely gauge-coupled nonlinear system.

\subsection{Linear analysis of the Dirac equation on the Schwarzschild black-hole background}

We first consider the linear massless Dirac equation on the Schwarzschild spacetime,
\begin{align}
    \mathscr D\psi=0,
\end{align}
where
\[
    \mathscr D:=i\gamma^\mu\nabla_\mu^{\mathbf S},
\]
and $\gamma^\mu$ denotes Clifford multiplication, and $\nabla^{\mathbf S}$ is the spin connection induced by the Levi-Civita connection of the Schwarzschild metric. As in our previous work \cite{hong}, we employ two complementary energy structures: the first-order energy associated with the Dirac current and the second-order energy associated with the squared Dirac equation.

\subsubsection{The Dirac-current energy}

The Dirac current associated with the massless Dirac equation is
\begin{align}
    J^\mu[\psi]
    :=\langle\psi,\gamma^\mu\psi\rangle.
\end{align}
For solutions of the homogeneous Dirac equation, it satisfies
\begin{align}
    \nabla_\mu J^\mu[\psi]=0.
\end{align}
Moreover, $J[\psi]$ is future-directed causal. Consequently, its flux through the future event horizon is non-negative. At this first-order level, the energy identity itself does not exhibit the photon-sphere degeneracy which appears in integrated local energy decay estimates.

We emphasise, however, that this does not provide an integrated decay estimate free of trapping. The Dirac-current identity is primarily a first-order conservation law. In our argument, quantitative decay is obtained by passing to the second-order equation satisfied by the spinor field.

\subsubsection{The spinorial wave energy}

To derive quantitative decay estimates, we square the Dirac equation. By the Lorentzian Lichnerowicz formula, under our sign conventions, the squared Dirac equation takes the form
\begin{align}
    \Box_g^{\mathbf S}\psi
    :=
    \nabla^{\mathbf S\mu}\nabla_\mu^{\mathbf S}\psi
    -\frac14\mathcal R_g\psi
    =0,
\end{align}
where $\mathcal R_g$ denotes the scalar curvature. Since the Schwarzschild spacetime is Ricci-flat, $\mathcal R_g=0$, and therefore
\begin{align}
    \Box_g^{\mathbf S}\psi=0.
\end{align}
It is at this second-order level that the photon-sphere trapping becomes visible, and we derive integrated local energy decay estimates for the spinorial wave equation.

To retain the spinorial structure while exploiting the physical-space vector-field method, we introduce the spinorial wave energy-momentum tensor associated with the squared Dirac equation. Schematically, it is given by
\begin{align}
    T_{\mu\nu}^{(n)}[\psi]
    :=
    \Re\langle\nabla_\mu^{\mathbf S}\psi,
    \gamma(n)\nabla_\nu^{\mathbf S}\psi\rangle
    -\frac12 g_{\mu\nu}
    \Re\langle\nabla^{\mathbf S\lambda}\psi,
    \gamma(n)\nabla_\lambda^{\mathbf S}\psi\rangle,
\end{align}
where \(n\) is a smooth future-directed timelike reference vector field chosen as follows. Near the event horizon, \(n\) agrees with the future-directed red-shift vector field \(\mathcal N\), whereas away from the horizon it agrees with the stationary Killing vector field \(T\). The insertion of \(\gamma(n)\) converts the natural spinorial pairing into a positive-definite Hermitian pairing and hence yields the coercive quadratic form required for the wave-energy estimates.

There is, however, an important difference from the scalar wave equation. Although the spin connection is compatible with Clifford multiplication, one has
\begin{align}
    \nabla_\alpha^{\mathbf S}\gamma(T)
    =
    \gamma(\nabla_\alpha T),
\end{align}
which does not vanish in general. Consequently, the spinorial wave energy-momentum tensor is not divergence-free even when $\psi$ solves the homogeneous spinorial wave equation. Its divergence contains additional lower-order terms which have no analogue in the scalar wave energy identity. To control these terms in the integrated local energy decay estimate, we supplement the usual modified Morawetz current with an additional spinor-adapted Morawetz correction.

\begin{rem}
Linear decay for massless Dirac fields on Schwarzschild is already well understood. In particular, Ma and Zhang \cite{mazhang} obtained a positive energy estimate, integrated local energy decay, and sharp Price-law asymptotics for the spin-$\frac12$ components through a symmetric hyperbolic wave system. Our purpose at the linear level is different. We formulate the estimates directly on the spinor bundle and in a physical-space form adapted to the commuted and inhomogeneous estimates required by the nonlinear problem.
\end{rem}

In what follows, we assume that the spacetime region under consideration is foliated by a family of hypersurfaces $\{\Sigma_\tau\}_{\tau\geq0}$. For some fixed $R>3M$, each $\Sigma_\tau$ is strictly spacelike in the region $\{r\leq R\}$ and outgoing null in the far region $\{r>R\}$. We denote by $D^\tau_0$ the spacetime region between $\Sigma_0$ and $\Sigma_\tau$. We further decompose
\begin{align}
    D^\tau_{0}
    := \left( D^\tau_0\cap\{r>R\} \right) \cup \left( D^\tau_0 \cap \{r\le R\} \right),
\end{align}
and refer to these as the far and interior regions, respectively.

In the far region, we employ the $r^p$-weighted hierarchy introduced by Dafermos and Rodnianski \cite{DR}. For $R$ chosen sufficiently large with $R>3M$, the Schwarzschild metric is a sufficiently small asymptotically flat perturbation of the Minkowski metric, and the $r^p$ argument proceeds essentially as in our previous work \cite{hong}. Thus, the principal new feature in passing from the asymptotically flat setting of \cite{hong} to the Schwarzschild background arises in the interior region.

In the interior region, one encounters both the photon sphere $\{r=3M\}$ and the event horizon $\mathcal H^+=\{r=2M\}$. Near the event horizon, we use the red-shift vector field to recover a non-degenerate energy estimate. The integrated local energy decay estimate obtained without derivative loss degenerates at the photon sphere. Consequently, when restricted away from a fixed neighbourhood of $\{r=3M\}$, it yields a non-degenerate local energy estimate without loss of derivatives. To obtain non-degenerate control through the trapping region itself, we instead use an integrated local energy decay estimate with a loss of derivatives.
\subsection{Application to the nonlinear tensorial wave--Dirac system}

The main result of this paper is the global small-data stability of the nonlinear tensorial wave--Dirac system
\begin{align}\label{eq-tensor-dirac-intro}
    \begin{aligned}
        \mathscr D\psi
        &= \gamma^\nu\psi\,\nabla_\nu \mathrm{Tr} h,\\
        \mathcal L_g h_{\mu\nu}
        &= \mathscr T_{\mu\nu}[\psi],
    \end{aligned}
\end{align}
where $\mathscr T_{\mu\nu}[\psi]$ is the Dirac stress--energy tensor and $\mathcal L_g$ is the geometric wave operator acting on symmetric $2$-tensors.

\subsubsection{Motivation}

The model \eqref{eq-tensor-dirac-intro} is motivated by the wave-gauge formulation of the Einstein--Dirac system. At the schematic level, the spin connection produces interactions of the form
\[
    (\nabla h)\psi
\]
in the Dirac equation, whereas the metric perturbation satisfies a second-order wave equation sourced by the Dirac stress--energy tensor. The model \eqref{eq-tensor-dirac-intro} isolates this first-/second-order derivative mismatch while suppressing the genuinely quasilinear metric interactions. A more precise derivation and comparison with the Einstein--Dirac system are given in Section~\ref{sec:tensor-dirac}.

\subsubsection{Main nonlinear difficulties and structural mechanisms}

We now describe the principal difficulties in the nonlinear problem and the structures that allow us to overcome them.

\begin{enumerate}

\item \emph{A spinor-adapted energy hierarchy.}

A central feature of our approach is the simultaneous use of two distinct energy structures associated with the Dirac equation: the spinorial wave energy and the Dirac current energy. Although both are essential to the nonlinear analysis, they measure fundamentally different quantities and play complementary roles in the closure of the energy estimates.

The spinorial wave energy provides control of derivatives of the spinor field and, through the $r^p$-weighted hierarchy, yields quantitative decay estimates. However, the corresponding energy fluxes through null hypersurfaces are naturally adapted to tangential derivatives. In particular, on an outgoing null hypersurface, the principal wave energy flux controls the outgoing null derivative $L\psi$ and angular derivatives, but does not directly control the transversal derivative $\underline L\psi$.

The Dirac current energy provides a different form of control. By the causal structure of the Dirac current, its flux through an outgoing null hypersurface controls the $L^2$ norm of the corresponding null spinorial component $\psi_+$, while the flux through an ingoing null hypersurface controls the complementary component $\psi_-$. Unlike the spinorial wave energy, the Dirac current energy is intrinsically an $L^2$-based estimate for the spinor field itself. When applied to commuted spinors, the null flux is determined by their spinorial projections rather than by the differential directions appearing in the commutators.

These two energy structures are therefore complementary rather than interchangeable. In controlling the nonlinear interactions, we exploit the Dirac current energy whenever the required spinorial component can be estimated more effectively through its null flux. One could instead attempt to recover such estimates from the spinorial wave energy, for instance by applying Hardy-type inequalities. However, this may introduce unnecessary derivative requirements or additional radial weights, leading to weaker bounds that are not suitable for the nonlinear closure. Retaining the Dirac current energy as an independent component of the hierarchy allows us to avoid these losses.

Importantly, the Dirac current energy is used at both lower and higher derivative orders throughout the nonlinear estimates for the Dirac and tensorial fields. Nevertheless, the hierarchy is arranged so that the Dirac current estimates are required only up to order $N-1$. In particular, the top-order
Dirac current energy is never invoked in the bootstrap argument. This is crucial
for avoiding the apparent loss of derivatives associated with the first-order Dirac equation.

Thus, the spinor-adapted energy hierarchy consists not merely in combining two energy estimates, but in exploiting their distinct geometric properties and assigning them different roles at different derivative levels. This complementary structure is a key ingredient in closing the nonlinear system.

\item \emph{Derivative mismatch and avoidance of top-order derivative loss.}

The mismatch between the first-order Dirac equation and the second-order tensorial wave equation gives rise to a potential loss of derivatives. Indeed, a direct estimate of the Dirac current energy at the highest derivative order would require one additional derivative of the tensorial field, which is not available within the energy hierarchy.

Our approach avoids this difficulty through the complementary use of the Dirac current energy and the spinorial wave energy. Although the Dirac current energy plays an essential role at both lower and higher derivative orders in controlling the nonlinear interactions, the top-order Dirac current energy is not required at any stage of the bootstrap argument. Consequently, the potentially derivative-losing top-order Dirac current estimate can be avoided altogether.

For the spinorial wave energy estimates, a similar apparent derivative loss arises when the Dirac equation is squared, producing terms involving second derivatives of $H=\mathrm{Tr}h$. However, the Clifford algebraic structure, together with the symmetry of the Hessian of $H$, yields
\[
\gamma^\lambda\gamma^\nu\nabla_\lambda\nabla_\nu H
    =\Box_g H.
\]
The tensorial wave equation then allows us to replace this term by the corresponding nonlinear source, thereby avoiding derivative loss in the spinorial wave energy estimates.

We remark that the derivative loss in the top-order Dirac current estimate could alternatively be removed by a suitable phase renormalisation associated with the gradient coupling. However, such a renormalisation is unnecessary in the present argument, since the top-order Dirac current energy is never invoked in the bootstrap closure.

\item \emph{Absence of a tensorial null structure and Clifford-algebraic closure.}

A second difficulty is the absence of a Maxwell-type null structure in the tensorial sector. In our previous work \cite{hong}, the tensor field was a closed $2$-form, and the associated differential identities provided a favorable null structure. In particular, the tensorial and spinorial null structures reinforced one another in the nonlinear estimates.

The situation here is different: $h$ is a general symmetric $2$-tensor, and no analogous closed-form structure or independent null condition is imposed. Nevertheless, the Clifford-algebraic structure of the Dirac equation is sufficiently strong to control the dangerous nonlinear interactions.

This becomes particularly transparent after squaring the Dirac equation. Schematically, applying $\mathscr D$ to the nonlinear term produces contributions of the form
\begin{align}
    \gamma^\mu\gamma^\nu
    (\nabla_\mu\psi)(\nabla_\nu h)
    +
    \gamma^\mu\gamma^\nu
    \psi\,\nabla_\mu\nabla_\nu h.
\end{align}
When the additional derivative falls on the tensorial factor, the symmetry of the Hessian together with the Clifford relation
\begin{align}
    \gamma^\mu\gamma^\nu+\gamma^\nu\gamma^\mu
    =2g^{\mu\nu}
\end{align}
converts the principal second-order contribution into the wave operator acting on the tensorial quantity. The tensor equation may then be used to replace this term by the Dirac stress--energy source, rather than leaving an uncontrolled second derivative of $h$.

On the other hand, when the differentiated Dirac operator falls on the spinorial factor, the Clifford decomposition prevents the worst null derivatives from occurring simultaneously on both interacting spinorial factors. In this sense, the dangerous bad--bad interaction is excluded algebraically: whenever one factor carries a bad null derivative, the complementary factor necessarily carries a better derivative or component.

Thus, unlike in \cite{hong}, the nonlinear closure does not rely on a separate Maxwell-type null structure in the tensorial field. The decisive structural mechanism is instead provided by the Clifford algebra of the Dirac equation.

\item \emph{Tensor component hierarchy.}

The weighted estimates for the tensor field are not uniform across all null components. Components containing favorable null directions admit stronger $r^p$-weighted estimates, whereas the worst components involving the incoming null direction cannot, in general, be controlled with the full top-order weighted hierarchy. In particular, the strongest weighted estimate for these components has to be taken at a lower derivative level.

This loss does not obstruct the Dirac estimates. The coupling in the Dirac equation occurs through $\mathrm{Tr} h$, and hence only a restricted collection of tensorial components enters the dangerous spinorial nonlinearities. In a null frame, the pure worst component does not appear in the trace. Consequently, the components for which the top-order weighted tensor estimate fails are precisely not those that drive the critical Dirac interaction.

Moreover, the tensor field is generated entirely by the Dirac stress--energy tensor. The decay of the tensorial components is therefore ultimately governed by the decay of the spinor field. This hierarchy allows us to close the nonlinear argument without imposing any additional null condition on the symmetric $2$-tensor itself.
\end{enumerate}




\subsubsection{Global stability for the Maxwell-Dirac system}

The preceding argument is not restricted to the tensorial model. As a genuine gauge-coupled application, we consider the massless Maxwell-Dirac system on the Schwarzschild exterior under the Lorenz gauge condition. In terms of the gauge potential \(A\), the system takes the schematic form
\begin{align}
    \begin{aligned}
        i\gamma^\mu(\nabla_\mu-iA_\mu)\psi &=0,\\
        \Box_g A_\nu-R_\nu{}^\mu A_\mu
        &=J_\nu[\psi],\\
        \nabla^\mu A_\mu&=0.
    \end{aligned}
\end{align}
Since the Schwarzschild spacetime is Ricci-flat, the curvature term in the equation for \(A\) vanishes.

The nonlinear argument follows essentially the same spinor-adapted mechanism developed for the tensorial wave-Dirac system. In particular, the Dirac-current estimates, the spinorial wave estimates, the commutation theory, and the Clifford-algebraic structure continue to control the nonlinear interactions. This shows that the mechanism identified above persists in a genuine gauge theory and is not an artifact of the toy model.

There is, however, an important quantitative difference. The direct coupling between the gauge potential and the spinor produces a borderline interaction at the highest weighted order. Consequently, our argument yields a mild growth of the top-order weighted energy rather than the uniform bound obtained for the tensorial model. This growth is nevertheless sufficiently weak to close the global small-data stability argument. The details of the Maxwell-Dirac estimates are given in Appendix~\ref{appendix:maxwell-dirac}.



\subsection{Extensions and outlook}

\subsubsection{Einstein-Dirac system}

A natural long-term objective is to extend the present analysis to the full Einstein-Dirac system. The tensorial model considered here is inspired by the wave-gauge formulation of the Einstein--Dirac equations and is designed to isolate a particular nonlinear difficulty arising from the coupling between a first-order Dirac equation and a second-order tensorial wave equation. In the present work, the semilinear derivative mismatch is handled by a hierarchy of complementary energy estimates; a phase renormalisation provides an alternative treatment of the top-order Dirac-current estimate, but is not needed for the nonlinear closure. Moreover, the nonlinear estimates can be closed without imposing an independent null condition on the tensorial wave equation, relying instead on the Dirac-current energy and the intrinsic Clifford-algebraic structure of the spinor field.

However, the present model does not capture the full structure of gravitational perturbations around Schwarzschild. In particular, the linearised Einstein equations contain curvature-dependent couplings and gauge-dependent terms which are absent from the tensorial wave equation considered here. These terms cannot, in general, be treated as perturbative errors in standard wave energy estimates. Furthermore, the full Einstein-Dirac system is genuinely quasilinear, since the Dirac operator depends on the evolving metric through the Clifford representation, the spin connection, and the choice of orthonormal frame. Such a phase renormalisation cannot, by itself, be expected to resolve the full quasilinear spinor-metric interaction.

A natural intermediate step towards the Einstein-Dirac system is to replace the tensorial wave equation by a system based on the spin-$\pm2$ Teukolsky equations, coupled to the Dirac field. Such a formulation would retain the appropriate linear gravitational dynamics around Schwarzschild and provide a natural starting point for extensions to rotating Kerr backgrounds, where the Teukolsky framework plays a central role in the analysis of gravitational perturbations; see Teukolsky's foundational work \cite{teukolsky}. The resulting coupling is expected to introduce additional difficulties, particularly a potential loss of derivatives arising from differentiated Dirac stress-energy terms in the Teukolsky equations, as well as the reconstruction of the metric and spin connection from curvature variables. An important question is whether the Dirac equation, the Bianchi identities, and the underlying Clifford-algebraic structure can be exploited to overcome this derivative loss and establish a closed hierarchy of nonlinear energy estimates.

We regard the present work as a first step towards such a framework, isolating the spinorial nonlinear mechanisms that may be relevant to more geometrically faithful models of the Einstein--Dirac interaction on black-hole backgrounds.

\subsubsection{From Schwarzschild to Kerr}

A further direction is to extend the spinor-adapted geometric analysis developed here from Schwarzschild to the Kerr family, beginning with the slowly rotating regime. Already at the linear level, the loss of spherical symmetry prevents the direct use of the rotational Killing commutators available on Schwarzschild, while the trapping geometry becomes substantially more complicated. Kerr nevertheless possesses a hidden symmetry generated by its Killing--Yano tensor. As shown by Carter and McLenaghan \cite{cartermcl}, this tensor gives rise to a first-order symmetry operator that commutes with the Dirac operator. It is therefore natural to investigate whether this spinorial symmetry operator can supplement the remaining Killing symmetries and provide an appropriate higher-order commutation structure. A central challenge is to combine these symmetries with spinor-adapted physical-space energy estimates and integrated local energy decay, particularly in the presence of Kerr trapping. For related linear estimates on slowly rotating Kerr backgrounds and for the Teukolsky equations, see \cite{daferrod,dahorod}; quasilinear wave estimates in this setting are developed in \cite{DHRT}.

At the nonlinear level, a natural extension is not merely to replace the Schwarzschild background by Kerr in the tensorial wave--Dirac model. As discussed above, a more geometrically faithful intermediate problem would be to couple the Dirac field to the spin-$\pm2$ Teukolsky equations governing gravitational perturbations. The existing linear decay theory for the Teukolsky equations on Kerr provides a starting point, while the nonlinear coupling introduces additional challenges involving differentiated Dirac stress-energy terms, potential derivative loss, and the reconstruction of metric and spin-connection perturbations from curvature variables. Combining such a gravitational framework with the spinor-adapted methods developed here would provide a natural route towards the analysis of Einstein--Dirac interactions on rotating black-hole backgrounds.

\subsection{Organisation}
The remainder of the paper is organised as follows. We conclude this section with the notation used throughout. Section~\ref{sec:prelim-geometry} reviews the basic differential-geometric preliminaries, while Section~\ref{sec:prelim-dirac} recalls the geometric formalism for the Dirac equation. In Section~\ref{sec:lin-dirac}, we develop the linear analysis on Schwarzschild, with particular emphasis on establishing integrated local energy decay estimates by a physical-space approach suitable for nonlinear applications.

Section~\ref{sec:tensor-dirac} contains the main nonlinear analysis and states the global stability theorem for the tensorial wave--Dirac system. Sections~\ref{sec:bdd-weight-dirac} and~\ref{sec:bdd-tensor} establish weighted energy estimates for the spinor and tensor fields, respectively. In Section~\ref{sec:imp-weight}, we derive improved weighted energy estimates at lower derivative orders, which we use in Section~\ref{sec:decay} to obtain improved decay estimates for both fields. We close the nonlinear estimates in the interior region in Section~\ref{sec:interior}, and establish the Dirac-current energy estimates in Section~\ref{sec:dirac-current}.

Finally, Appendix~\ref{appendix:maxwell-dirac} establishes the global small-data stability of the Maxwell--Dirac system in the Lorenz gauge.

\subsection{Notations}

We assume that $\mathcal M$ is foliated by hypersurfaces $\{\Sigma_{\tau}\}_{\tau\ge0}$. For $r\le R$, where $R>3M$ is fixed and sufficiently large, $\Sigma_\tau$ is the spacelike slice $\{t=\tau\}$. For $r\ge R$, it is the outgoing null hypersurface $\{u=\tau-R\}\cap\{v\ge\tau+R\}$. Here $u$ and $v$ are null coordinates defined by
\begin{align}
    u= t- r^*, \ v = t+r^*.
\end{align}

We consider the domain $D^{\tau}_{\tau_0}$ bounded by two hypersurfaces $\Sigma_\tau$ and $\Sigma_{\tau_0}$. 

We denote by $\mathcal D^{\tau}_{\tau_0}$ the far region of the domain $D^{\tau}_{\tau_0}$, i.e., $\mathcal D^{\tau}_{\tau_0}= D^{\tau}_{\tau_0}\cap \{r \ge R\} $.

In the far region $r\ge R$, we define the Dirac-current energies of the spinor field by
\begin{align}
    \mathscr E^D[ \psi ]^2(\tau) := \int_{\substack{u=\tau-R \\ v\ge \tau+R } } \langle \psi, \gamma(L)\psi\rangle \,d\sigma, \\
    \mathscr F^D[\psi]^2(v,\tau_0,\tau) := \int_{\underline{\mathcal C}_v\cap \mathcal D^{\,\tau}_{\tau_0} } \langle \psi, \gamma(\underline L)\psi\rangle \,d\sigma.
\end{align}
We also define the spinorial wave energies by
\begin{align}
    \mathcal E^D[\psi]^2(\tau) := \int_{\Sigma_\tau\cap \{r\ge R\} } \langle \nabla_L\psi,\gamma(T)\nabla_L\psi\rangle+\langle\slashed\nabla\psi,\gamma(T)\slashed\nabla\psi\rangle+ \frac{1}{r^2}\langle\psi,\gamma(T)\psi\rangle  \,d\sigma, \\
    \mathcal F^D[ \psi ]^2(v,\tau_0,\tau) := \int_{\underline{\mathcal C}_v\cap \mathcal D^{\,\tau}_{\tau_0}}  \langle\nabla_{\underline L}\psi,\gamma(T)\nabla_{\underline L}\psi\rangle+ \langle\slashed\nabla\psi,\gamma(T)\slashed\nabla\psi\rangle+  \frac{1}{r^2}\langle\psi,\gamma(T)\psi\rangle \,d\sigma.
\end{align}

In the interior region, we define the energy ${}_{\rm int}\mathcal E^D[\psi]$ by
\begin{align}
    \begin{aligned}
        {}_{\rm int}\mathcal E^D_{}[\psi]^2(\tau) := \int_{\Sigma_{\tau}\cap\{r\le R\} } \mathcal J^{\mathcal N} [ \psi ]\,d\sigma_{\Sigma_\tau},
    \end{aligned}
\end{align}
where $\mathcal N$ is the vector field which is future-directed time-like and is equal to $T$ away from an open neighborhood of the event horizon $\mathcal H^+$.
We also define the interior Dirac-current energy by
\begin{align}
\begin{aligned}
	{}_{\rm int}\mathscr E^D_{}[\psi]^2(\tau) := \int_{\Sigma_{\tau}\cap\{r\le R\} } \langle \psi, \gamma(\mathcal N)\psi\rangle\,d\sigma_{\Sigma_\tau}.
\end{aligned}	
\end{align}

In the interior estimates below, we do not need to distinguish between the wave and Dirac-current energies, since the null-component decomposition into $+$ and $-$ spinors is not needed there.

In particular, we define the energy on the interior away from the trapping region:
\begin{align}
    \begin{aligned}
        {}_{\rm int}\mathcal E^D_{}[\psi]^2(\tau; |r-3M|>1 ) := \int_{\Sigma_{\tau}\cap\{r\le R\} \cap \{ |r-3M|>1 \} } \mathcal J^T [ \psi ]\,d\sigma_{\Sigma_\tau}.
    \end{aligned}
\end{align}
Equivalently, we may define
\begin{align}
    \begin{aligned}
        {}_{\rm int}\mathcal E^D_{}[\psi]^2(\tau; |r-3M|>1 ) := \int_{\Sigma_{\tau}\cap\{r\le R\}} \chi \mathcal J^T [ \psi ]\,d\sigma_{\Sigma_\tau}.
    \end{aligned}
\end{align}
where we introduce a smooth bump function $\chi\in C^\infty_0(\mathbb R^+)$ such that $0\le\chi\le1$ and it degenerates near the trapping region $\{r=3M\}$. Roughly, one has
\begin{align}
    \chi(r) \approx \left( 1-\frac{3M}{r} \right)^2
\end{align}
near the trapping region.

We also invoke the following bulk terms:
\begin{align}
    \begin{aligned}
        {}_{(-1-\delta)}{\mathcal E^D}[\psi]^2 (\tau) := \int_{\Sigma_\tau}\frac1{r^{1+\delta}} (  \langle \nabla_L\psi,\gamma(T)\nabla_L\psi\rangle+\langle\slashed\nabla\psi,\gamma(T)\slashed\nabla\psi\rangle+\langle\nabla_{\underline L}\psi,\gamma(T)\nabla_{\underline L}\psi\rangle  )\,d\sigma, \\
        {}_{(-3-\delta)}{\mathcal E^D}[\psi]^2(\tau) := \int_{\Sigma_\tau}\frac1{r^{3+\delta}}\langle\psi,\gamma(T)\psi\rangle\,d\sigma.
    \end{aligned}
\end{align}
We define the weighted energy derived from the $r^p$-method \cite{DR}: for $1\le p<2$,
\begin{align}
\begin{aligned}
    {}^{(p)}\mathcal E^D[\psi]^2(\tau) := \int_{\Sigma_\tau} r^p \langle \nabla_L\psi,\gamma(T)\nabla_L\psi\rangle\,d\sigma, \\
    {}^{(p-1)}\mathcal E^D[\psi]^2(\tau) := \int_{\Sigma_\tau} r^{p-1} ( \langle \nabla_L\psi,\gamma(T)\nabla_L\psi\rangle+ \langle\slashed\nabla\psi,\gamma(T)\slashed\nabla\psi\rangle)\,d\sigma, \\
    {}^{(p)}\mathcal F^D[\psi]^2(v,\tau_0,\tau) := \int_{\underline{\mathcal C}_v \cap \mathcal D^{\tau}_{\tau_0} } r^p \langle \slashed\nabla\psi,\gamma(T)\slashed\nabla\psi\rangle\,d\sigma.
    \end{aligned}
\end{align}
We let $0\le k\le N$. We define
\begin{align}
    \mathscr E^D_{\le k}[ \psi ] := \sum_{|I|\le k}\mathscr E^D[ \mathscr L_Z^I\psi ],
\end{align}
and we also define the analogue for $\mathscr F^D_{\le k},\mathcal E^D_{\le k}$, and so on.

For $0\le k\le N$, we define the total energy of the spinor field by
\begin{align}
    \begin{aligned}
        \mathcal X^D_{\le k}[\psi]^2(\tau) & := \mathscr E^D_{\le k}[\psi]^2(\tau) + \sup_{v:\tau\le\tau(v)}\mathscr F^D_{\le k}[\psi]^2(v,\tau_i,\tau_{i+1}) \\
        & \qquad + \mathcal E^D_{\le k-1}[\psi]^2(\tau)+\sup_{v:\tau\le \tau(v)}\mathcal F^D_{\le k-1}[\psi]^2(v,\tau_i,\tau_{i+1})+ {}_{\rm int} \mathcal E^D_{\le k-1}[\psi]^2(\tau).
    \end{aligned}
\end{align}
For consistency, we define
\begin{align}
    \begin{aligned}
        \mathcal E^D_{-1}[\psi]^2(\tau) := \int_{\Sigma_\tau} \frac1{r^2}|\psi|^2_T\,d\sigma , \ \mathcal F^D_{-1}[\psi]^2(v,\tau_i,\tau_{i+1} ) := \int_{\underline{\mathcal C}_v \cap \mathcal D^{\tau_{i+1}}_{\tau_i} } \frac1{r^2}|\psi|^2_T\,d\sigma.
    \end{aligned}
\end{align}

For scalar or tensorial wave equations, we define
\begin{align}
    \begin{aligned}
     \mathcal E^W_{}[\phi]^2(\tau) := \int_{\Sigma_\tau\cap \{r\ge R\} } |\nabla_L\phi|^2+|\slashed\nabla \phi|^2 + \frac1{r^2}|\phi|^2 \,d\sigma , \\
     \mathcal F^W[\phi]^2(v,\tau_0,\tau) := \int_{\underline{\mathcal C}_v \cap \mathcal D^{\tau}_{\tau_0} } |\nabla_{\underline L}\phi|^2+|\slashed\nabla\phi|^2+ \frac1{r^2}|\phi|^2 \,d\sigma .
    \end{aligned}
\end{align}
We define the corresponding interior energy by
\begin{align}
    \begin{aligned}
        {}_{\rm int}\mathcal E^W[\phi]^2(\tau) := \int_{\Sigma_\tau \cap \{r\le R\} } \mathcal J^T[\phi] \,d\sigma_{\Sigma_\tau}, \\
        {}_{\rm int}\mathcal E^W[\phi]^2(\tau; |r-3M|>1 ) := \int_{\Sigma_\tau \cap \{r\le R\} \cap \{ |r-3M|>1 \} } \mathcal J^T[\phi] \,d\sigma_{\Sigma_\tau}.
    \end{aligned}
\end{align}
We define the weighted energies ${}^{(p)}\mathcal E^W[\phi]$ and ${}^{(p)}\mathcal F^W[\phi]$ analogously. For $0\le k\le N-1$, we set
\begin{align}
    \mathcal E^W_{\le k}[\phi] := \sum_{|I|\le k} \mathcal E^W [ \mathcal L_Z^{I} \phi ] ,
\end{align}
and define $\mathcal F^W_{\le k}$ analogously. Finally, for $0\le k\le N-1$, we define the total energy for scalar and tensorial waves by
\begin{align}
    \begin{aligned}
        \mathcal X^W_{\le k}[\phi]^2(\tau) := \mathcal E^W_{\le k}[\phi]^2(\tau) + \sup_{v:\tau\le \tau(v)} \mathcal F^W_{\le k}[\phi]^2(v,\tau_i,\tau_{i+1})+ {}_{\rm int}\mathcal E^W_{\le k}[\phi]^2(\tau).
    \end{aligned}
\end{align}

\section{Preliminaries on Geometry}\label{sec:prelim-geometry}

\subsection{Geometry of the background spacetime}\label{subsec:background-geometry}
As noted in Section~\ref{sec:intro}, we consider a Lorentzian manifold $(\mathcal M,g)$ equipped with the Schwarzschild metric.

More precisely, we take the underlying differential structure of $\mathcal M$ to be the manifold with boundary $\mathbb R^4 \setminus \{(x^1)^2+(x^2)^2+(x^3)^2<r_0^2\}$. The coordinate vector fields $(\frac{\partial}{\partial x^0},\frac{\partial}{\partial x^1},\frac{\partial}{\partial x^2},\frac{\partial}{\partial x^3})$ are then globally regular. We define
\begin{align*}
   r = \sqrt{ (x^1)^2+ (x^2)^2+(x^3)^2}.
\end{align*}
We also define the spherical coordinates in the usual way:
\begin{align}
    (x^1,x^2,x^3) := ( r\sin\theta\cos\phi, r\sin\theta\sin\phi, r\cos\theta),
\end{align}
where $0\le\theta\le\pi$ and $0\le\phi\le2\pi$. We also put $t=x^0$.

Then the Schwarzschild metric $g$ is given by
\begin{align}
    g= -\left( 1-\frac{2M}{r} \right)dt^2+ \left( 1-\frac{2M}{r} \right)^{-1}dr^2+ r^2 ( d \theta^2+ \sin^2\theta d\phi^2).
\end{align}
We set $r_0:=(2-\delta_1)M$, where $0<\delta_1\ll M$ is sufficiently small.

We denote the event horizon by $\mathcal H^+=\{r=2M\}$, with future-directed null generator $\partial_{x^0}$.

We define $r^* = r+2M\log(r-2M)$ and introduce the null coordinates $u=t-r^*$ and $v=t+r^*$ for $r>2M$. We also set
$t^* =t+ 2M\log(r-2M)  $.
Let $\phi_\tau$ denote the one-parameter group of diffeomorphisms generated by $\partial_{x^0}$. We define the initial hypersurface $\Sigma_0$ by
\begin{align}
    \Sigma_{0} := (\{ t^* = 0\}\cap \{r\le R\}  ) \cup ( \{v=R\} \cap \{r\ge R \}  ),
\end{align}
and $\Sigma_\tau = (\phi_{\tau})_* (\Sigma_0)$. 
 We denote by $J^+(\Sigma_{\tau=0})$ the causal future of $\Sigma_{\tau=0}$ in $(\mathcal M,g)$ and assume that
\begin{align}
    J^+(\Sigma_{\tau=0}) = \bigcup_{\tau\ge0}\Sigma_{\tau},
\end{align}
so that for any $\tau\ge0$, $\Sigma_\tau\cap \{r\le R\}$ is the time-slice and $\Sigma_\tau\cap \{r\ge R\}$ is an outgoing null hypersurface.

For $r\ge R$, define the null vector fields $L=\partial_v$ and $\underline L=\partial_u$ in the coordinates $(u,v,\theta,\phi)$. We extend them to globally defined, linearly independent, translation-invariant vector fields in the span of $\partial_t$ and $\partial_{r^*}$.

Following \cite{DRred}, choose a $\phi_\tau$-invariant vector field $Y$ that is future-directed and null on $\mathcal H^+$, transversal to $\mathcal H^+$, and orthogonal to $\mathcal H^+\cap\Sigma_0$, with $g(\partial_{r^*},Y)=-2$ at $r=2M$. Assume that $Y$ is supported in $r\le r_1+(r_2-r_1)/2$ and that $\partial_t$, $\Omega$, and $Y$ span the tangent space for $r\le r_1+(r_2-r_1)/4$, where $r_1=(2+\delta_1)M$ and $\chi=1$ for $r\le r_2$.

\subsection{Basic differential geometric conventions}
We present the elementary geometric conventions used throughout the paper. Let $\nabla$ be the Levi-Civita connection of $g$. Thus $\nabla$ is torsion-free and metric-compatible:
\begin{align}
    \nabla_XY-\nabla_YX=[X,Y], \qquad \nabla g=0,
\end{align}
for all vector fields $X,Y$ on $\mathcal M$. If $T$ is a tensor field, its covariant derivative is denoted by $\nabla T$, and repeated Greek indices are contracted using the spacetime metric $g$.

\begin{defn}[Curvature convention]
The Riemann curvature tensor is defined by
\begin{align}
    R(X,Y)Z=\nabla_X\nabla_YZ-\nabla_Y\nabla_XZ-\nabla_{[X,Y]}Z.
\end{align}
In coordinates, we write
\begin{align}
    (\nabla_\mu\nabla_\nu-\nabla_\nu\nabla_\mu)V^\alpha
    ={R^\alpha}_{\beta\mu\nu}V^\beta.
\end{align}
The Ricci tensor and scalar curvature are given by
\begin{align}
    \mathrm{Ric}_{\mu\nu}={R^\alpha}_{\mu\alpha\nu}, \quad
    \mathcal R=g^{\mu\nu}\mathrm{Ric}_{\mu\nu}.
\end{align}
We also use the wave operator
\begin{align}
    \Box_g\phi=g^{\mu\nu}\nabla_\mu\nabla_\nu\phi
\end{align}
for scalar functions $\phi$.
\end{defn}

\begin{prop}[Ricci identity]
For any $(k,l)$-tensor field $T$, the commutator of covariant derivatives is determined by the curvature tensor. In particular,
\begin{align}
    [\nabla_\mu,\nabla_\nu]X^\alpha={R^\alpha}_{\beta\mu\nu}X^\beta,
    \qquad
    [\nabla_\mu,\nabla_\nu]\omega_\alpha=-{R^\beta}_{\alpha\mu\nu}\omega_\beta.
\end{align}
More generally, each contravariant index contributes a term with sign $+$ and each covariant index contributes a term with sign $-$.
\end{prop}

\begin{prop}[Commutation formula for the wave operator]
Let $X$ be a vector field and let $\phi$ be a scalar function. Then
\begin{align}\label{comm-wave-vectorfield}
    [\Box_g,X]\phi
    =2\nabla^\mu X^\nu\nabla_\mu\nabla_\nu\phi+(\Box_gX^\nu)\nabla_\nu\phi.
\end{align}
Equivalently, using the deformation tensor ${}^{(X)}\pi$, one has
\begin{align}
    [\Box_g,X]\phi
    = {}^{(X)}\pi_{\mu\nu}\nabla^\mu \nabla^\nu\phi + \nabla_\mu \pi^{\mu\nu} \nabla_\nu\phi -\frac12 \nabla^\mu ( \mathrm{Tr}\pi)\nabla_\mu\phi. 
\end{align}
\end{prop}
We refer the reader to \cite{alinhac} for the proof.
\begin{defn}[Volume forms and divergence]
Let $dV_g$ be the spacetime volume form induced by $g$. For a vector field $P$, its divergence is
\begin{align}
    \mathrm{div}\,P=\nabla_\mu P^\mu.
\end{align}
If $\mathcal D\subset \mathcal M$ is a spacetime domain with piecewise smooth boundary, we denote by $n$ the unit normal to the boundary $\partial\mathcal D$ of $\mathcal D$.
\end{defn}

\begin{prop}[Divergence theorem]
Let $P$ be a smooth vector field on a spacetime domain $\mathcal D$. Then
\begin{align}
    \int_{\mathcal D}\nabla_\mu P^\mu\,dV_g
    =\int_{\partial\mathcal D} P^\mu n_\mu\,d\sigma_{\partial\mathcal D},
\end{align}
with the usual interpretation of the boundary measure and normal on null hypersurfaces.
\end{prop}
We apply this identity to energy currents associated with wave and Dirac-type equations.
\begin{defn}[Deformation tensor]
For a vector field $X$, its deformation tensor is
\begin{align}
    {}^{(X)}\pi_{\mu\nu}=\nabla_\mu X_\nu+\nabla_\nu X_\mu.
\end{align}
It measures the failure of $X$ to be Killing and appears in the divergence identity for vector-field multipliers.
\end{defn}

\begin{defn}[Energy-momentum tensor for waves]
For a scalar field $\phi$, we define the energy-momentum tensor by
\begin{align}
    T_{\mu\nu}[\phi]
    =\nabla_\mu\phi\nabla_\nu\phi-\frac12 g_{\mu\nu}\nabla^\alpha\phi\nabla_\alpha\phi.
\end{align}
Given a vector field $X$, the associated current is
\begin{align}
   \mathcal J^X_\mu[\phi]=T_{\mu\nu}[\phi]X^\nu.
\end{align}
\end{defn}

\begin{prop}[Multiplier identity for waves]
Let $\phi$ be a scalar field and let $X$ be a vector field. Then
\begin{align}
    \nabla^\mu\mathcal J^X_\mu[\phi]
    = (\Box_g\phi)(X\phi)+\frac12 T^{\mu\nu}[\phi]{}^{(X)}\pi_{\mu\nu}.
\end{align}
In particular, if $\Box_g\phi=F$, then the bulk term generated by the multiplier $X$ is
\begin{align}
    F X\phi+\frac12 T^{\mu\nu}[\phi]{}^{(X)}\pi_{\mu\nu}.
\end{align}
\end{prop}
This identity is the basic starting point for the vector-field method and for the $r^p$-weighted estimates.
\subsection{Double null foliation}\label{subsec:double-null-foliation}
Let $u$ and $v$ be optical functions satisfying $g^{-1}(du,du)=0$ and $g^{-1}(dv,dv)=0$. Then the level sets
\begin{align*}
	\mathcal C_u = \{u=const\}, \ \underline{\mathcal C}_{v} = \{v=const\},
\end{align*}
are outgoing and ingoing null hypersurfaces, respectively.
We define the normalised null pair $(L,\underline L)$ by
\begin{align*}
	L=-2\Omega^2 \nabla u, \ \underline L = -2\Omega^2 \nabla v,
\end{align*}
where $2\Omega^{-2}=-g^{-1}(du,dv)$ and $\Omega>0$. Thus,
\begin{align*}
	g(L,L)=g(\underline L,\underline L)=0, \ g(L,\underline L)=-2.
\end{align*}
We set $S_{u,v}=\mathcal C_u\cap \underline{\mathcal C}_{v}$, a two-surface, and choose orthonormal tangent vector fields $e_A$, $A=1,2$, on $S_{u,v}$. The associated null frame is $\{e_A,L,\underline L\}$.

\begin{defn}[Null frame decomposition]
Relative to the null frame $\{e_A,L,\underline L\}$, we first recall the following notation. For any vector field $X$, we denote by $X^\flat$ its metric dual $1$-form,
\begin{align}
    X^\flat(Y)=g(X,Y)
\end{align}
for all vector fields $Y$. In particular, $L^\flat$ and $\underline L^\flat$ denote the metric duals of the null generators $L$ and $\underline L$. The metric and inverse metric decompose as
\begin{align}
    g=-\frac12\left(L^\flat\otimes \underline L^\flat+\underline L^\flat\otimes L^\flat\right)+\sum_{A=1}^2 e^A\otimes e^A,
\end{align}
and
\begin{align}
    g^{-1}=-\frac12\left(L\otimes \underline L+\underline L\otimes L\right)+\sum_{A=1}^2 e_A\otimes e_A.
\end{align}
Consequently, for a scalar function $\phi$,
\begin{align}
    |\nabla\phi|^2=-L\phi\,\underline L\phi+\sum_{A=1}^2 |e_A\phi|^2.
\end{align}
For a $1$-form $\xi$, we write its null components as
\begin{align}
    \xi_L=\xi(L), \qquad \xi_{\underline L}=\xi(\underline L), \qquad \xi_A=\xi(e_A).
\end{align}
Similar notation will be used for higher-rank tensors by evaluating them on the null frame.
\end{defn}

\begin{defn}[Projected connection on $S_{u,v}$]
Let $\Pi$ be the orthogonal projection onto the tangent bundle of $S_{u,v}$. For an $S_{u,v}$-tangent tensor field $\theta$, we define
\begin{align}
    \slashed\nabla_X\theta:=\Pi(\nabla_X\theta), \qquad X\in TS_{u,v}.
\end{align}
In components, for an $S_{u,v}$-tangent $1$-form $\theta$, we have
\begin{align}
    \slashed\nabla_A\theta_B=e_A(\theta_B)-\theta_C\langle \nabla_{e_A}e_B,e_C\rangle.
\end{align}
We also set
\begin{align}
    \slashed{\mathrm{div}}\,\theta=\slashed\nabla^A\theta_A,
    \qquad
    \slashed{\mathrm{curl}}\,\theta=\epsilon^{AB}\slashed\nabla_A\theta_B,
    \qquad
    \slashed\Delta f=\slashed\nabla^A\slashed\nabla_A f.
\end{align}
Here indices $A,B$ are raised and lowered by the induced metric on $S_{u,v}$.
\end{defn}

\begin{defn}[Hodge operators on $S_{u,v}$]
Let $\epsilon_{AB}$ be the volume form of $(S_{u,v},\slashed g)$, normalised by $\epsilon_{12}=1$ with respect to an oriented orthonormal frame. For an $S_{u,v}$-tangent $1$-form $\theta$, its Hodge dual is
\begin{align}
    (*\theta)_A={\epsilon_A}^{B}\theta_B.
\end{align}
For an $S_{u,v}$-tangent $2$-form $F$, we identify $F$ with the scalar $*F$ defined by
\begin{align}
    *F=\frac12\epsilon^{AB}F_{AB}.
\end{align}
If $\theta$ and $\zeta$ are $S_{u,v}$-tangent $1$-forms, we use the notation
\begin{align}
    \theta\wedge\zeta=\epsilon^{AB}\theta_A\zeta_B.
\end{align}
These conventions fix the signs in the angular divergence and curl identities used below.
\end{defn}

\begin{defn}[Flux measures on null hypersurfaces]
The induced measure on $S_{u,v}$ is denoted by $d\sigma_{S_{u,v}}$. With the normalisation $g(L,\underline L)=-2$, the natural flux measures on the null hypersurfaces are written as
\begin{align}
    d\sigma_{\mathcal C_u}=dv\,d\sigma_{S_{u,v}},
    \qquad
    d\sigma_{\underline{\mathcal C}_v}=du\,d\sigma_{S_{u,v}}.
\end{align}
In the asymptotically flat region, one has $d\sigma_{S_{u,v}}\simeq r^2d\omega$, and hence
\begin{align}
    d\sigma_{\mathcal C_u}\simeq r^2dv\,d\omega,
    \qquad
    d\sigma_{\underline{\mathcal C}_v}\simeq r^2du\,d\omega,
\end{align}
where $d\omega$ denotes the standard measure on the unit sphere. These equivalences are used only at the level of uniform constants depending on the asymptotic flatness assumptions.
\end{defn}

In addition, we fix the gauge $\Omega=1$ on $\mathcal C_u$ and $\underline{\mathcal C}_{v}$ so that $\nabla_LL=\nabla_{\underline L}\underline L=0$. This choice is standard in the null frame formalism and simplifies the commutator computations without loss of generality.

Now we introduce the connection coefficients:
\begin{align}\label{conn-coeff}
    \begin{aligned}
        \chi_{AB} = \langle \nabla_{e_A}L,e_B\rangle, \ \underline\chi_{AB} = \langle \nabla_{e_A}\underline L,e_B\rangle, \\
        2\eta_A = \langle \nabla_{\underline L}L, e_A\rangle, \ 2\underline\eta_A = \langle \nabla_L\underline L,e_A\rangle, 2\xi_A =\langle\nabla_LL,e_A\rangle, \ 2\underline\xi_A = \langle\nabla_{\underline L}\underline L,e_A\rangle, \\
        4\omega = \langle\nabla_LL,\underline L\rangle, \ 4\underline\omega = \langle \nabla_{\underline L}\underline L,L\rangle.
        \end{aligned}
\end{align}
We note that $\xi=\underline\xi=\omega=\underline\omega=0$ under the aforementioned gauge choice.
The tensors $\chi$ and $\underline\chi$ are the null second fundamental forms of the spheres $S_{u,v}$ along the outgoing and incoming null directions, respectively. We decompose them into trace and trace-free parts by
\begin{align}
    \chi_{AB}=\hat\chi_{AB}+\frac12\mathrm{Tr}\chi\,\slashed g_{AB},
    \qquad
    \underline\chi_{AB}=\hat{\underline\chi}_{AB}+\frac12\mathrm{Tr}\underline\chi\,\slashed g_{AB},
\end{align}
where
\begin{align}
    \mathrm{Tr}\chi=\slashed g^{AB}\chi_{AB}, \qquad
    \mathrm{Tr}\underline\chi=\slashed g^{AB}\underline\chi_{AB}.
\end{align}
The trace-free parts $\hat\chi$ and $\hat{\underline\chi}$ measure the shear of the null congruences, while the traces measure their null expansions.
\begin{prop}\label{formula-conn-coeff}
    The connection $\nabla$ satisfies the following formulas:
    \begin{align}
        \begin{aligned}
            \nabla_{e_A} L = \chi_{AB}e_B-\eta_AL, \ \nabla_{\underline L}L = 2\eta_Ae_A, \\
            \nabla_{e_A}\underline L = \underline{\chi}_{AB}e_B+\eta_A\underline L, \ \nabla_L\underline L = 2\underline\eta_Ae_A, \\
            \nabla_Le_A = \slashed{\nabla}_Le_A+\underline\eta_AL, \ \nabla_{\underline L}e_A = \slashed{\nabla}_{\underline L}e_A+\eta_A\underline L, \\
            \nabla_{e_B}e_A = \slashed{\nabla}_{e_B}e_A+\frac12\chi_{AB}L+\frac12\underline\chi_{AB}\underline L,
        \end{aligned}
    \end{align}
    where $\slashed\nabla$ is the induced connection on the $2$-sphere, i.e., $\slashed{\nabla}_XY = \sum_{A=1,2} \langle \nabla_XY,e_A\rangle$.
\end{prop}
We refer the reader to \cite{alinhac} for the proof.
\subsection{Lie derivatives}
We first recall that the Lie derivative is defined by $\mathcal L_Xf=Xf$, and $\mathcal L_XY=[X,Y]$, for any scalar $f$ and vector fields $X,Y$ and $[X,Y]=\nabla_XY-\nabla_YX$. For any $n$-tensor $T$, the Lie derivative is extended via the identity:
\begin{align*}
	X \left( T(Y_1,\cdots,Y_n) \right) = (\mathcal L_XT)(Y_1,\cdots,Y_n)+\sum_{i=1}^nT(Y_1,\cdots, [X,Y_i],\cdots,Y_n).
\end{align*}
In particular, for $2$-tensor $F$, we have
\begin{align}
(\mathcal L_ZF)_{\mu\nu} = Z^\lambda\nabla_\lambda F_{\mu\nu}+\nabla_\mu Z^\lambda F_{\lambda\nu}+\nabla_\nu Z^\lambda F_{\mu\lambda}.
\end{align}
\subsection{Auxiliary estimates}
We use the double-null foliations $\mathcal C_u$ and $\underline{\mathcal C}_{v}$. We define the $L^p$ norms of functions on these hypersurfaces by
\begin{align*}
	\|f\|_{L^p(\mathcal C_u)}^p := \int_{\mathcal C_u}|f|^p\,d\sigma_{\mathcal C_u} , \  \|f\|_{L^p(\underline{\mathcal C}_{v})}^p := \int_{\underline{\mathcal C}_{v}}|f|^p\,d\sigma_{\underline{\mathcal C}_{v}}.
\end{align*}
\begin{prop}[Sobolev inequalities on $S_{u,v}$]
Let $f$ be a sufficiently regular scalar function on $S_{u,v}$. Then
\begin{align}
    \|f\|_{L^4(S_{u,v})}
    \lesssim r^{-1/2}\|f\|_{L^2(S_{u,v})}^{1/2}
    \left(\|f\|_{L^2(S_{u,v})}+r\|\slashed\nabla f\|_{L^2(S_{u,v})}\right)^{1/2},
\end{align}
and
\begin{align}
    \|f\|_{L^\infty(S_{u,v})}
    \lesssim r^{-1}\sum_{k\le 2} r^k\|\slashed\nabla^k f\|_{L^2(S_{u,v})}.
\end{align}
Equivalently, in the asymptotically flat region, the last estimate may be written schematically as
\begin{align}
    \|f\|_{L^\infty(S_{u,v})}
    \lesssim r^{-1}\sum_{|I|\le 2}\|\mathcal L_{\Omega}^I f\|_{L^2(S_{u,v})},
\end{align}
where $\Omega$ denotes the rotational vector fields. The same estimates hold componentwise for $S_{u,v}$-tangent tensor fields.
\end{prop}
\begin{prop}[Hardy inequality on the null foliations]\label{Hardy}
	For any sufficiently regular scalar function $\phi$, we have
	\begin{align*}
		\|r^{-1}\phi\|^2_{L^2(\mathcal C_u)} \lesssim \int_{\mathcal C_u}|\mathcal L_{L}\phi|^2\,d\sigma_{\mathcal C_u} +\frac1{r^2}\sum_{i<j}\int_{\mathcal C_u} |\mathcal L_{\Omega_{ij}}\phi|^2\,d\sigma_{\mathcal C_u},
	\end{align*}
	and
	\begin{align*}
	\|r^{-1}\phi\|^2_{L^2(\underline{\mathcal C}_v)} \lesssim \int_{\underline{\mathcal C}_v}|\mathcal L_{\underline L}\phi|^2\,d\sigma_{\underline{\mathcal C}_v} +\frac1{r^2}\sum_{i<j}\int_{\underline{\mathcal C}_v} |\mathcal L_{\Omega_{ij}}\phi|^2\,d\sigma_{\underline{\mathcal C}_v},
	\end{align*}
    where $\mathcal L$ is the Lie derivative.
\end{prop}

Consider a $1$-form $\xi$ satisfying
\begin{align}
    \slashed{\mathrm{div}} \xi = f, \quad \slashed{\mathrm{curl}} \xi = *f,
\end{align}
where $f,f_*$ are given scalar functions on a compact Riemannian manifold $(S,\slashed{g})$ and the operators $\slashed{\mathrm{div}}$ and $\slashed{\mathrm{curl}}$ are defined as
\begin{align}
    \slashed{\mathrm{div}} \xi = \slashed{\nabla}^{e_A} \xi_A, \quad \slashed{\mathrm{curl}} \xi = \epsilon^{AB}\slashed{\nabla}_{e_A}\xi_B.
\end{align}
\begin{prop}[Proposition 2.2.1 of \cite{christoklai}]
    On an arbitrary compact Riemannian manifold $(S,\slashed{g})$, with the Gauss curvature $K$ of $S$, the following estimate holds:
    \begin{align}
        \int_{S} |\slashed\nabla\xi|^2+K|\xi|^2\,d\sigma = \int_{S}|f|^2+|*f|^2\,d\sigma.
    \end{align}
\end{prop}
We shall use a standard Morawetz estimate.
\begin{prop}[Morawetz estimates]\label{Morawetz}
    Let an inhomogeneous (scalar) wave equation $\Box_g\phi=F$ be given on the Schwarzschild black hole background. Then on the far region $ \{r\ge R\}$ with $R>3M$, we have
    \begin{align}
        \int_{\mathcal D^{\tau_{i+1}}_{\tau_i} } \frac1{\langle r\rangle^{1+\delta}}\left( |\nabla_L\phi|^2+|\nabla_{\underline L}\phi|^2+|\slashed\nabla\phi|^2 \right)\,dV \lesssim \int_{\partial \mathcal D^{\tau_{i+1}}_{\tau_i}}\mathcal J^T[\phi]\,d\sigma + \left|\int_{\mathcal \mathcal D^{\tau_{i+1}}_{\tau_i}} F ( \partial_r\phi+\frac1r\phi) \,dV \right| ,
    \end{align}
    where the boundary terms $\partial \mathcal D^{\tau_{i+1}}_{\tau_i}$ consist of the outgoing null $\Sigma_\tau$ and the infinity null $\mathcal I^{\tau_{i+1}-R}_{\tau_i-R}$ and the time-like surface $\{r=R\}$. 
\end{prop}
We note that
no trapping degeneracy occurs in the far region, so the standard Morawetz argument applies directly.

The standard energy inequality and the $r^p$-method also yield decay for the flux on ingoing null hypersurfaces.

\begin{prop}[Energy inequality]
On the domain $\mathcal D^{\tau_2}_{\tau_1}$, for the linear wave equation $\Box_g\phi = 0$, we have
\begin{align}
   \sup_{v:\tau\le\tau(v)}  \int_{\underline{\mathcal C}_v\cap \mathcal D^{\tau_{i+1}}_{\tau_i} } |\nabla_{\underline L}\phi|^2\,d\sigma + \int_{\substack{ u=\tau_{i+1}-R \\ v\ge \tau_{i+1}+R }} |\slashed\nabla\phi|^2\,d\sigma   \lesssim \int_{\substack{ u=\tau_i-R \\ v\ge \tau_i+R } }  |\slashed\nabla\phi|^2\,d\sigma + \mathcal B^R_{}[\phi]^2(\tau_i,\tau_{i+1}) +{}^{(\delta)}\mathcal E^W[\phi]^2(\tau_i),
\end{align}
where $\mathcal B^R_{}[\phi]^2(\tau_i,\tau_{i+1}) $ is the energy flux of $\phi$ on the time-like surface $\{r=R\}$ with $\tau_i\le \tau\le \tau_{i+1}$.
\end{prop}

\begin{proof}
    The above inequality immediately follows after the contraction of the energy-momentum tensor $T_{\mu\nu}[\phi]$ with the null vector $\underline L$ and an application of the divergence theorem on the far-region $\mathcal D^{\tau_2}_{\tau_1}$. We recall that $T_{\underline L\underline L}[\phi] = |\nabla_{\underline L}\phi|^2 $ and $T_{L\underline L}[\phi]\approx |\slashed\nabla\phi|^2 $. However, we also have to deal with the bulk term due to the deformation tensor ${}^{(\underline L)}\pi$, which requires us to control the following integrals:
    \begin{align}
        \int_{\mathcal D^{\tau_{i+1}}_{\tau_i}} \frac1{r^2}|\slashed\nabla\phi|^2\,dV + \int_{\mathcal D^{\tau_{i+1}}_{\tau_i}} \frac1{r}|\nabla_L\phi| |\nabla_{\underline L}\phi| \,dV .
    \end{align}
    The first integral is controlled directly. For the second, we use
    \begin{align*}
        \int_{\mathcal D^{\tau_{i+1}}_{\tau_i}} \frac1{r}|\nabla_L\phi| |\nabla_{\underline L}\phi| \,dV & \le C_\epsilon \int_{\mathcal D^{\tau_{i+1}}_{\tau_i}} r^{-1+\epsilon} |\nabla_L\phi|^2\,dV + \epsilon\int_{\mathcal D^{\tau_{i+1}}_{\tau_i}} r^{-1-\epsilon} |\nabla_{\underline L}\phi|^2\,dV \\
        & \le C_\epsilon \langle R\rangle^{-\delta+\epsilon} \int_{\mathcal D^{\tau_{i+1}}_{\tau_i}} r^{-1+\delta} |\nabla_L\phi|^2\,dV + \epsilon\int_{\mathcal D^{\tau_{i+1}}_{\tau_i}} r^{-1-\epsilon} |\nabla_{\underline L}\phi|^2\,dV
    \end{align*}
    The standard Morawetz estimate on the far region $\mathcal D^{\tau_{i+1}}_{\tau_i}$ controls the second integral on the right-hand side by the boundary energies $\mathcal E^W[\phi]^2(\tau_{i+1})$ and $\mathcal E^W[\phi]^2(\tau_i)$, the flux to null infinity $\mathcal F^W[\phi](v=\infty,\tau_i,\tau_{i+1})$, and the flux through the timelike surface $\{r=R\}$. We choose $\epsilon>0$ sufficiently small to absorb $\mathcal E^W[\phi]^2(\tau_{i+1})$ and $\mathcal F^W[\phi]^2(v=\infty)$ into the left-hand side of the energy inequality. The $r^p$-method controls the first integral by ${}^{(\delta)}\mathcal E^W[\phi]^2(\tau_i)$.
\end{proof}
Combining the above energy inequality with the $r^p$-weighted energy inequality, we obtain
\begin{align}
    \begin{aligned}
       \sup_{v:\tau\le\tau(v)} {}^{(0)}\mathcal F^W[\phi]^2(v,\tau_{i},\tau_{i+1}) \lesssim  \mathcal B^R_{}[\phi]^2(\tau_i,\tau_{i+1}) +{}^{(\delta)}\mathcal E^W[\phi]^2(\tau_i).
    \end{aligned}
\end{align}
The boundary term along the time-like surface can be controlled by the energy on the interior domain. However, in the averaging argument, we do not need to cover the whole interior domain. Indeed, it is enough to consider the interior domain away from the trapping region, for example, $\{R-1\le r \le R\}$. In conclusion, we obtain the following:
\begin{cor}
    We have
    \begin{align}
        \begin{aligned}
             \sup_{v:\tau\le\tau(v)} {}^{(0)}\mathcal F^W[\phi]^2(v,\tau_{i},\tau_{i+1}) \lesssim  {}_{\rm int} \mathcal E^W_{}[\phi]^2(\tau_i; |r-3M|>1) + {}^{(\delta)}\mathcal E^W[\phi]^2(\tau_i).
        \end{aligned}
    \end{align}
\end{cor}
We conclude this section with integrated local energy decay estimates on Schwarzschild. For background, see \cite{DRred,DRlecture}.
\begin{prop}[Integrated local energy decay estimates for scalar wave on the interior]
	Let $\Box_g\phi=F$ be given on the Schwarzschild black hole background. Then we have
	\begin{align}
	\begin{aligned}
		\int_{ D } |\nabla_{\partial_{r^*}}\phi|^2+ \chi ( |\nabla_T\phi|^2+ |\slashed\nabla\phi|^2 )+\frac1{r^2}|\phi|^2\,dV & \lesssim \int_{\partial D}\mathcal J^T [\phi] \,d\sigma_{\partial_D} + \left| \int_{D} F ( f(r^*)\nabla_{\partial_{r^*}}\phi + q\phi)\,dV \right|
	\end{aligned}	
	\end{align}
 where a smooth bump function $0\le\chi\le1$ degenerate near the trapping region $\{r=3M\}$ and the radial function $f$ satisfies 
    \begin{align*}
        f(3M) = 0, \ f'(3M) >0, \ f(r)(r-3M)\ge0.
    \end{align*}
    In particular, we have
    \begin{align}
        \begin{aligned}
         &   \int_{D} |\nabla_{\partial_{r^*}}\phi|^2 + |\nabla_T\phi|^2+|\slashed\nabla\phi|^2 +\frac1{r^2}|\phi|^2 \,dV \\
         & \lesssim \sum_{Z\in\{T,\Omega\} }\int_{\partial D} \mathcal J^T[\mathcal L_Z^{\le1} \phi]\,d\sigma_{\partial D}  + \sum_{Z\in\{T,\Omega\} } \left| \int_{D}  \mathcal L_Z^{\le1} F (\nabla_{\partial_{r^*}}\mathcal L_Z^{\le1}\phi+ q\mathcal L_Z^{\le1}\phi)\,dV \right|,
        \end{aligned}
    \end{align}
    and
    \begin{align}
        \begin{aligned}
           & \int_{D\cap \{|r-3M|>1 \} } |\nabla_{\partial_{r^*}}\phi|^2 +  |\nabla_T\phi|^2+|\slashed\nabla\phi|^2 +\frac1{r^2}|\phi|^2 \,dV \\
           & \lesssim \int_{\partial D} \mathcal J^T[\phi]\,d\sigma_{\partial D}  + \left| \int_{D\cap \{|r-3M|>1\} }   F  ( \nabla_{\partial_{r^*}}\phi+ q\phi)\,dV \right|.
        \end{aligned}
    \end{align}

\end{prop}

\section{Preliminaries on the geometric analysis of the Dirac equation}\label{sec:prelim-dirac}

In this section, we collect the basic geometric structures for the Dirac equation that will be used throughout the paper. These structures are intrinsic to the spinorial formulation and do not depend on the particular geometry of the background spacetime. In particular, the spinor-adapted Lie derivative, the null decomposition of spinors, and the causal properties of the Dirac current are formulated at this general level.

The underlying spin-geometric formalism, including the construction of the spinor bundle, the spin connection, and the geometric Dirac operator, was developed in detail in our previous work \cite{hong}; see also the standard text \cite{lawson}. We therefore do not repeat the general formalism here, but recall only those features that enter directly into the subsequent analysis on Schwarzschild spacetime. We refer the reader to \cite{hong} for further details and to \cite{leflochmazhang,treude} for related treatments of the geometric Dirac equation.

In what follows, we consider the linear massless Dirac equation
\begin{align*}
    \mathscr D\psi = 0,
\end{align*}
where $\mathscr D=i\gamma^\mu\nabla_\mu^{\mathbf S}$ and $\nabla_{}^{\mathbf S}$ is the spinorial covariant derivative.

\subsection{Spinor-adapted Lie derivative}
\begin{defn}[Kosmann Lie derivative]
Let $X$ be a smooth vector field on $(\mathcal M,g)$.
The Kosmann Lie derivative, introduced by Kosmann \cite{kosmann}, of a spinor field $\psi\in\Gamma(\mathbf S)$
along $X$ is defined by
\begin{align*}
	\mathscr L_X^{}\psi
	:= \nabla_X^{\mathbf S}\psi
	- \frac14 \left(\nabla_\rho X_\sigma \right)
	\gamma^\rho\gamma^\sigma \psi,
\end{align*}
where $\nabla$ denotes the Levi-Civita connection
and $\gamma^\mu$ denotes the raised Clifford action with respect to a local orthonormal frame, as in the definition of the Dirac operator.
\end{defn}
It is well known that the Lie derivative is not linear in $X$. For the Kosmann Lie derivative, we have
\begin{align}\label{kosmann-fX}
\mathscr L_{fX}\psi = f\mathscr L_X\psi -\frac14\left( (\partial_\rho f)X_\sigma -(\partial_\sigma f)X_\rho \right)\gamma^\rho \gamma^\sigma\psi.
\end{align}

\begin{prop}[Commutator identity with Clifford action]
For any vector field $X$ and any tangent vector $v$,
the Kosmann Lie derivative $\mathscr L$ satisfies
\begin{align*}
	\mathscr L_X (\gamma(v))
	= \gamma(\mathcal L_X v),
\end{align*}
where $\mathcal L_X$ is the ordinary Lie derivative acting on vector fields.
\end{prop}
From now on, we allow the abuse of the notation and simply denote the spinorial connection by $\nabla_\mu:=\nabla_{e_\mu}$.
\begin{prop}
   Let $Z$ be a vector field. We have
    \begin{align}\label{comm-dirac}
    \begin{aligned}
          [\gamma^\mu\nabla_\mu, \mathscr L_Z] \psi &=  \frac12 \gamma^\rho {\pi_{\rho}}^{\! \mu} \nabla_\mu \psi + \frac12 \gamma^\mu Z^\nu \gamma^\rho\gamma^\sigma R_{\mu\nu\rho\sigma}\psi -\frac18 \gamma^\mu\gamma^\rho\gamma^\sigma \left( \nabla_\mu \pi_{\rho\sigma}+\nabla_{\rho}\pi_{\mu\sigma}-\nabla_\sigma\pi_{\mu\rho} \right) \psi,
    \end{aligned}
\end{align}
where $\pi:={}^{(Z)}\pi$ is the deformation tensor of the vector field $Z$.
\end{prop}
In particular, the commutator vanishes when the vector field $Z$ is Killing.


We also need the commutator of the Kosmann Lie derivatives:
\begin{lem}\label{kosmann-commutator}
    We have the following identity
    \begin{align}
        [\mathscr L_X,\mathscr L_Y]\psi & = \frac14 \gamma^\rho\gamma^\sigma X^\mu Y^\nu R_{\rho\sigma\mu\nu}\psi + E_{ X,Y }\psi,
    \end{align}
    where $\mathcal E$ is not harmful error terms:
    \begin{align*}
         E_{X,Y} & = \frac14 Y^\lambda\nabla_\lambda\left( (\nabla_\rho X_\sigma-\nabla_\sigma X_\rho)\gamma^\rho\gamma^\sigma \right)-\frac14 X^\lambda\nabla_\lambda \left( (\nabla_\mu Y_\nu-\nabla_\nu Y_\mu)\gamma^\mu\gamma^\nu \right) \\
        & \qquad +\frac1{16}(\nabla_\rho X_\sigma-\nabla_\sigma X_\rho)(\nabla_\mu Y_\nu-\nabla_\nu Y_\mu)[\gamma^\rho\gamma^\sigma,\gamma^\mu\gamma^\nu].
    \end{align*}
\end{lem}

\subsection{Null decomposition for the Dirac equation}
We deal with the double null foliation introduced in Section~\ref{sec:prelim-geometry}.
In particular, we use optical functions $(u,\underline u)$ with corresponding outgoing/ingoing null hypersurfaces
$\mathcal C_u$ and $\underline{\mathcal C}_{v}$,
and we denote by $(L,\underline L)$ the associated normalised null pair.
We also write $S_{u,v}:=\mathcal C_u\cap\underline{\mathcal C}_{v}$ for the $2$-spheres and fix an orthonormal frame $\{e_A\}_{A=1,2}$ tangent to $S_{u,v}$, so that $\{e_A,L,\underline L\}$ is a null frame.

Now we define the projection operators:
\begin{align*}
	\Pi_+ = \frac14 \gamma(\underline L)\gamma(L), \ \Pi_- =\frac14 \gamma(L)\gamma(\underline L).
\end{align*}
\begin{prop}
The operators $\Pi_\pm$ satisfy the following properties:
\begin{align*}
    \Pi_++\Pi_- = I, \ \Pi_\pm \Pi_\mp = 0.
\end{align*}
\end{prop}

In what follows, we decompose
$$
\psi = \psi_++\psi_-,
$$
where
$$\psi_+= \Pi_+\psi , \quad \psi_-= \Pi_-\psi.
$$
Now we put
\begin{align*}
	T = \frac12 (L+\underline L), 
\end{align*}
Then we have
\begin{align*}
	g(T,T)=-1,
\end{align*}
which implies that
\begin{align*}
	\gamma(T)^2 = I,
\end{align*}

\begin{prop}\label{prop-null-decomp-dirac}
    The Dirac equation $(i\gamma^\mu\nabla_\mu-m)\psi=F$ has the null decomposition of the form:
    \begin{align}
    \nabla_L\psi_+ = -\frac12\gamma(T)\gamma(e_A)\nabla_{e_A}\psi_--\frac i2 m\gamma(T)\psi_--\frac i2 \gamma(T)\Pi_- F+\frac1{2r}\psi +O(r^{-2})\psi, \\
    \nabla_{\underline L}\psi_- = -\frac12\gamma(T)\gamma(e_A)\nabla_{e_A}\psi_+-\frac i2 m\gamma(T)\psi_+-\frac i2 \gamma(T)\Pi_+ F -\frac1{2r}\psi +O(r^{-2})\psi.
\end{align}
\end{prop}
\subsection{Dirac current and energy identity}
\begin{defn}[Dirac current]
Let $\psi\in\Gamma(\mathbf S)$ be a spinor field.
The Dirac current associated to $\psi$ is the vector field $J^\mu(\psi)$ defined by
\begin{align*}
	J^\mu(\psi) := \langle \psi, \gamma^\mu \psi\rangle_D.
\end{align*}
where $\langle\cdot,\cdot\rangle_D$ denotes the Dirac pairing on $\mathbf S$.
\end{defn}
In a local orthonormal frame $\{e_\mu\}$ with $e_0$ future-directed timelike, the components of the Dirac current read
\[
J^\mu = \psi^\dagger\gamma(e_0)\gamma^\mu\psi.
\]
\begin{prop}[Conservation of Dirac current]
Let $\psi$ be a smooth solution to the free Dirac equation
\[
\mathscr D\psi := (i\gamma^\mu\nabla_\mu-m)\psi = 0.
\]
Then the associated Dirac current is divergence-free:
\begin{align*}
	\nabla_\mu J^\mu(\psi) = 0.
\end{align*}
\end{prop}
\begin{prop}[Variation of the Dirac current]
For a spinor field $\psi$ and a vector field $v$, define
\[
J(\psi)(v):=\langle \psi,\gamma(v)\psi\rangle.
\]
Then for any vector field $X$,
\[
\mathcal L_X \big(J(\psi)(v)\big)
= \langle \mathscr L_X\psi,\gamma(v)\psi\rangle
+ \langle \psi,\gamma(v)\mathscr L_X\psi\rangle
+ \langle \psi,\gamma(\mathcal L_X v)\psi\rangle,
\]
where $\mathscr L_X$ is the Kosmann Lie derivative acting on spinors, and $\mathcal L_X$ is the usual Lie derivative acting on vectors and tensors.
\end{prop}
A fundamental property of the Dirac current is its causality. More precisely, for any spinor field $\psi$, $J^\mu(\psi)$ is a future-directed causal vector field; that is,
\begin{align*}
	g_{\mu\nu}J^\mu(\psi) J^\nu(\psi) \le 0,
\qquad J(\psi)(n) \ge 0 \,\text{ for any future-directed timelike } n.
\end{align*}
In consequence, for any future-directed timelike normal vector $n$, we have
\begin{align*}
	J(\psi)(n) := J^\mu(\psi)n_\mu \ge0.
\end{align*}
Due to the above positivity, it is reasonable to define the norm of the spinor $\psi$ associated to the vector $n$ by
\begin{align*}
	|\psi|_n = \sqrt{J(\psi)(n)}.
\end{align*}
\begin{prop}[Lemma 5.2 of \cite{leflochmazhang}]
	Let $n$ be a future-directed timelike unit vector field. For any spinor fields $\psi,\phi$, we have
	\begin{align*}
		|\langle\psi,\gamma(n)\phi\rangle | \lesssim |\psi|_n|\phi|_n, 
	\end{align*}
	and for any space-like vector $X$, we have
\begin{align*}
	|J(\psi)(X)| \le |X| J(\psi)(n).
\end{align*}
\end{prop}
We refer the reader to Lemmas 5.2 and 5.3 of \cite{leflochmazhang}.

\subsection{Spinorial wave equation}

By squaring the Dirac operator, one obtains a wave-type equation for the spinor field. More precisely, the following identity holds.

\begin{prop}
    Let $\psi$ be a solution to the linear Dirac equation
    $\mathscr D\psi=0$. Then
    \begin{align}
        \Box_g^{\mathbf S}\psi
        =\frac14\mathcal R_g\psi,
    \end{align}
    where $\mathcal R_g$ denotes the scalar curvature of
    the background spacetime. In particular, on a Ricci-flat
    spacetime, the spinor field satisfies
    \begin{align}
        \Box_g^{\mathbf S}\psi=0.
    \end{align}
\end{prop}

This identity is a consequence of the
Schrödinger--Lichnerowicz formula. We refer the reader to
\cite{lawson,alcu,chen,parker} for further details.

\begin{defn}[Spinorial wave energy-momentum tensor]
    Let $n$ be a future-directed timelike vector field.
    We define the energy-momentum tensor associated with
    the spinorial wave equation by
    \begin{align}
        T_{\mu\nu}^{(n)}[\psi]
        := \Re\langle \nabla_\mu\psi,
        \gamma(n)\nabla_\nu\psi\rangle
        +\frac12 g_{\mu\nu}
       \Re \langle \nabla^\lambda\psi,
        \gamma(n)\nabla_\lambda\psi\rangle,
    \end{align}
    where $\langle\cdot,\cdot\rangle$ denotes the Dirac
    pairing and $\nabla:=\nabla^{\mathbf S}$ is the
    spinorial covariant derivative.
\end{defn}

We emphasise that, although the Dirac pairing is
compatible with the spin connection, the modified pairing
involving $\gamma(n)$ is not generally preserved under
covariant differentiation. Indeed, the compatibility of
Clifford multiplication with the spin connection gives
\begin{align}
    \nabla_\mu\bigl(\gamma(n)\bigr)
    =\gamma(\nabla_\mu n),
\end{align}
which need not vanish on a curved background.
Consequently, the spinorial wave energy-momentum tensor
is not, in general, divergence-free, even for solutions
of the homogeneous spinorial wave equation.

The resulting spinorial error terms constitute an
additional difficulty in establishing integrated local
energy decay estimates. Their treatment will be a key
ingredient in the proof of
Proposition \ref{prop-iled-dirac-schwarz}.


\section{Linear analysis of the Dirac equation on the Schwarzschild spacetime}\label{sec:lin-dirac}
The purpose of this section is to study the Dirac equation on the Schwarzschild black-hole background.
We consider the $(1+3)$-dimensional Lorentzian manifold $(\mathcal M,g)$ where the metric $g$ is given by
\begin{align}
    g = -\left( 1-\frac{2M}r \right)dt^2+\left(1-\frac{2M}r \right)^{-1}dr^2 +r^2 \left( d\theta^2+ \sin^2\theta  d\phi^2 \right).
\end{align}
We denote the event horizon by $\mathcal H^+ =  \{r=2M\} $. 

Let $Y$ be a $\phi_{\tau}$-invariant vector field that is future-directed and null on $\mathcal H^+$, transversal to $\mathcal H^+$, and orthogonal to $\mathcal H^+\cap\Sigma_0$. Assume that $Y$ is supported in $r\le r_1+\frac{r_2-r_1}{2}$ and that $T,Y,\Omega$ span the tangent space for $r\le r_1+\frac{r_2-r_1}{4}$. Here we choose $2M<r_1<r_2<3M$ with $r_1=(2+\delta_1)M$ for some small $0<\delta_1\ll M$, and let $\chi=1$ for $r\le r_2$. Thus, $Y$ vanishes away from $\mathcal H^+$.

We define the red-shift vector field by $\mathcal N=T+cY$, so that $\mathcal N$ is timelike near $\mathcal H^+$.

We put
\begin{align*}
    \mu := 1-\frac{2M}{r}.
\end{align*}
Given $R>3M$, we define a Regge-Wheeler coordinate 
\begin{align}
    r^* := r+2M \log(r-2M)-R,
\end{align}
and the null coordinates:
\begin{align}
    u = t-r^* , \quad v = t+r^*.
\end{align}
Then the null frame is given by
\begin{align}
    L = \partial_t+\partial_{r^*}= \partial_t+\mu\partial_r, \quad \underline L = \partial_t-\partial_{r^*}= \partial_t-\mu\partial_r
\end{align}
We note that
\begin{align}
    g(L,\underline L) = -2\mu.
\end{align}

\subsection{Some computations}
\begin{prop}
    We have
    \begin{align}
        \eta_A = \frac12\langle \nabla_{\underline L}L,e_A\rangle =0, \quad \underline{\eta}_A = \frac12\langle \nabla_L\underline L,e_A\rangle = 0. 
    \end{align}
    The deformation tensors are given by
    \begin{align}
        & {}^{(L)}\pi_{L\underline L} = -\frac{2M\mu}{r^2}, \quad {}^{(L)}\pi_{AB} = \frac{\mu}{r}\slashed{g}_{AB}, \\
       & {}^{(\underline L)}\pi_{L\underline L} = \frac{2M\mu}{r^2}, \quad {}^{(\underline L)}\pi_{AB} = -\frac{\mu}{r}\slashed{g}_{AB},
    \end{align}
    and all the remaining components for ${}^{(L)}\pi_{\mu\nu}$ and ${}^{(\underline L)}\pi_{\mu\nu}$  are zero, for $\mu,\nu\in\{e_A,L,\underline L\}$.
\end{prop}
\begin{prop}
    We have
    \begin{align}
        [ \gamma^\mu\nabla_\mu , \mathscr L_L ] & = \frac{M}{2r^2}\gamma^L\nabla_L + \frac{M}{2r^2}\gamma^{\underline L}\nabla_{\underline L} -  \frac{\mu}{2r} \gamma^{e_A}\nabla_{e_A}  \\
        [ \gamma^\mu\nabla_\mu , \mathscr L_{\underline L} ] & = -\left( \frac{M}{2r^2}\gamma^L\nabla_L + \frac{M}{2r^2}\gamma^{\underline L}\nabla_{\underline L} \right)+  \frac{\mu}{2r} \gamma^{e_A}\nabla_{e_A}  ,
    \end{align}
    where we omit lower-order terms. Furthermore, we have
    \begin{align}
        [ \gamma^\mu\nabla_\mu, \mathscr L_{\Omega} ] = 0,
    \end{align}
    where $\Omega$ is rotation fields.
\end{prop}
\begin{proof}
Since the Schwarzschild metric is spherically symmetric, commutators with the rotational vector fields vanish.
	The commutators with $\mathscr L_Z$, for $Z\in\{L,\underline L\}$, follow directly from \eqref{comm-dirac}. However, one must account for the fact that $\nabla_{e_\nu}\gamma^{e_\mu}$ does not generally vanish. For example, $\nabla_{L}\gamma^{e_A}=\gamma(\nabla_{L}e_A)=\chi_{AB}\gamma^{e_B}$. We first compute the commutator $[\gamma^\mu\nabla_\mu,\mathscr L_L]$:
	\begin{align*}
		[ \gamma^\mu\nabla_\mu , \mathscr L_{L} ]  & = \frac12 \gamma^\rho {}^{(L)}{\pi_{\rho}}^{\mu}\nabla_\mu \psi - (\nabla_L \gamma^\mu)\nabla_\mu\psi ,
	\end{align*}
	where we omit the lower-order terms. Then
	\begin{align*}
		\nabla_L \gamma^\mu \nabla_\mu\psi & = \gamma(\nabla_L L) \nabla_L\psi + \gamma(\nabla_L \underline L)\nabla_{\underline L}\psi + \gamma(\nabla_L e_A)\nabla_{e_A}\psi \\
		& = \chi_{AB} \gamma^{e_B} \nabla_{e_A}\psi \\
		& = \frac{\mu}{r}\slashed{g}_{AB}. 
	\end{align*}
	On the other hand, we see that
	\begin{align*}
		\frac12 \gamma^\rho {}^{(L)}{\pi_{\rho}}^{\mu}\nabla_\mu \psi  & = \frac12 \gamma^L {\pi_{L}}^{L}\nabla_L\psi+\frac12 \gamma^{\underline L}{\pi_{\underline L}}^{\underline L}\nabla_{\underline L}\psi +\frac12 \gamma^{e_A}{\pi_{A}}^{B}\nabla_{e_B}\psi \\
		& = -\gamma^{L}\pi_{L\underline L}\nabla_L\psi - \gamma^{\underline L}\pi_{\underline LL}\nabla_{\underline L}\psi +\sum_{A,B=1,2}\frac12 \gamma^{e_A}{\pi_{AB}}\nabla_{e_B}\psi.
	\end{align*}
	Combining the above computations, we obtain the commutator identity up to lower-order terms.
\end{proof}


\subsection{Integrated local energy decay estimates for the spinor fields}
Given the linear massless Dirac equation
\begin{align*}
    i\gamma^\mu\nabla_\mu^{\mathbf S}\psi = 0,
\end{align*}
we obtain the following wave-type equation:
\begin{align}
    g^{\mu\nu}\nabla_\mu^{\mathbf S}\nabla_\nu^{\mathbf S}\psi = \frac14\mathcal R_g\psi,
\end{align}
where $\mathcal R_g$ is the scalar curvature. We note that the scalar curvature vanishes for the Schwarzschild metric. Thus we write
\begin{align}
    \Box_g^{\mathbf S} \psi = 0,
\end{align}
where we put $\Box_g^{\mathbf S}=  g^{\mu\nu}\nabla_\mu^{\mathbf S}\nabla_\nu^{\mathbf S}$.

The purpose of this subsection is to prove the integrated local energy decay estimates for the spinor fields on the Schwarzschild black hole background:
\begin{prop}\label{prop-iled-dirac-schwarz}
    Given the massless Dirac equation $\mathscr D\psi=F$ on the Schwarzschild black hole spacetime, we have
    \begin{align}
        \begin{aligned}
            \int_{D} |\nabla_{\partial_{r^*}}\psi|^2_T + \chi \left( |\nabla_T\psi|^2_T+|\slashed\nabla\psi|^2_T \right)+\frac1{r^2}|\psi|^2_T \,dV & \lesssim  \int_{\partial D} \mathcal J^T[\psi]\,d\sigma_{\partial D} +\left| \int_{D} \langle \mathscr D F, \gamma^T ( f(r^*)\nabla_{\partial_{r^*}}\psi+ q\psi)\rangle\,dV \right|,
        \end{aligned}
    \end{align}
    where a smooth bump function $0\le\chi\le1$ degenerate near the trapping region $\{r=3M\}$ and the radial function $f$ satisfies 
    \begin{align*}
        f(3M) = 0, \ f'(3M) >0, \ f(r)(r-3M)\ge0.
    \end{align*}
    In particular, we have
    \begin{align}
        \begin{aligned}
         &   \int_{D} |\nabla_{\partial_{r^*}}\psi|^2_T + |\nabla_T\psi|^2_T+|\slashed\nabla\psi|^2_T +\frac1{r^2}|\psi|^2_T \,dV \\
         & \lesssim \sum_{Z\in\{T,\Omega\} }\int_{\partial D} \mathcal J^T[\mathscr L_Z^{\le1} \psi]\,d\sigma_{\partial D}  + \sum_{Z\in\{T,\Omega\} } \left| \int_{D} \langle \mathscr D \mathscr L_Z^{\le1} F, \gamma^T ( f(r^*)\nabla_{\partial_{r^*}}\mathscr L_Z^{\le1}\psi+ q\mathscr L_Z^{\le1}\psi)\rangle\,dV \right|,
        \end{aligned}
    \end{align}
    and
    \begin{align}
        \begin{aligned}
           & \int_{D\cap \{|r-3M|>1 \} } |\nabla_{\partial_{r^*}}\psi|^2_T +  |\nabla_T\psi|^2_T+|\slashed\nabla\psi|^2_T +\frac1{r^2}|\psi|^2_T \,dV \\
           & \lesssim \int_{\partial D} \mathcal J^T[\psi]\,d\sigma_{\partial D}  + \left| \int_{D\cap \{|r-3M|>1\} } \langle \mathscr D F, \gamma^T ( f(r^*)\nabla_{\partial_{r^*}}\psi+ q\psi)\rangle\,dV \right|.
        \end{aligned}
    \end{align}
\end{prop}
\begin{proof}
To prove this,
we introduce the energy-momentum tensor for the spinor field as a solution to the above wave equation:
\begin{align}
    T_{\mu\nu}[\psi] := \Re \langle \nabla_\mu\psi, \gamma(T)\nabla_\nu\psi\rangle -\frac12  g_{\mu\nu}  \Re\langle \nabla^\lambda\psi, \gamma(T)\nabla_\lambda\psi\rangle.
\end{align}
The Dirac current $J^\mu [\psi] = \langle \psi,\gamma^\mu\psi\rangle $ is a future-directed causal vector field. Thus $J^T[\psi]$ is positive-definite, where $T$ is the future-directed unit normal vector field to the time slice $\{t=constant\}$. We may put $T= \mu^{-\frac12}\partial_t$.

We first compute the divergence of $T_{\mu\nu}[\psi]$:
\begin{align}
    \begin{aligned}
        \nabla^\mu T_{\mu\nu}[\psi] & = \Re \langle \nabla^\mu\nabla_\mu \psi, \gamma^T \nabla_\nu\rangle + \Re\langle \nabla^\mu\psi, \gamma^T [\nabla_\mu,\nabla_\nu]\psi\rangle + \langle \nabla_\mu\psi, \gamma(\nabla_\mu T)\nabla_\nu\psi\rangle -\frac12 \langle \nabla^\lambda\psi, \gamma(\nabla_\nu T)\nabla_\lambda\psi\rangle \\
        & = \Re\langle \nabla^\mu\psi, \gamma^T \mathscr R_{\mu\nu} \psi\rangle + \langle \nabla_\mu\psi, \gamma(\nabla_\mu T)\nabla_\nu\psi\rangle -\frac12 \langle \nabla^\lambda\psi, \gamma(\nabla_\nu T)\nabla_\lambda\psi\rangle,
    \end{aligned}
\end{align}
where we put
\begin{align}
    [ \nabla_\mu^{\mathbf S},\nabla_\nu^{\mathbf S} ] \psi = \frac14 R_{\mu\nu\rho\sigma}\gamma^\rho\gamma^\sigma \psi = \mathscr R_{\mu\nu}\psi.
\end{align}
From now on, we put
\begin{align}
    \nabla^\mu T_{\mu\nu}[\psi] := E^{\nabla T}_\nu [\psi]+ E^{\rm curv}_{\nu}[\psi],
\end{align}
where 
\begin{align}
    E^{\nabla T}_\nu [\psi] := \langle \nabla_\mu\psi, \gamma(\nabla_\mu T)\nabla_\nu\psi\rangle -\frac12 \langle \nabla^\lambda\psi, \gamma(\nabla_\nu T)\nabla_\lambda\psi\rangle,
\end{align}
and
\begin{align}
    E^{\rm curv}_\mu [\psi] :=  \Re\langle \nabla^\mu\psi, \gamma^T \mathscr R_{\mu\nu} \psi\rangle.
\end{align}
Near the trapping region $\{r=3M\}$, we choose the orthonormal frame as follows:
\begin{align}
    e_0 := T, \ e_1 = \mu^{-\frac12}\partial_{r^*}, 
\end{align}
and $e_A$, $A=2,3$ are tangential vector fields on the sphere. A straightforward computation yields
\begin{align}
    \nabla_{e_0}T = \frac{M}{r^2\mu^{\frac12}}e_1, \ \nabla_{e_1}T = \nabla_{e_A}T=0, \quad A=2,3. 
\end{align}

Now we choose the multiplier
\begin{align}
    X = f(r^*)\partial_{r^*}
\end{align}
and using the contraction we have
\begin{align}
    J_{\mu}^X[\psi] := T_{\mu\nu}[\psi] X^\nu.
\end{align}
Then we have
\begin{align}
    \nabla^\mu J_{\mu}^X[\psi] = \frac12 T_{\mu\nu}[\psi] {}^{(X)}\pi^{\mu\nu} +E^{\nabla T}_X[\psi]  +E^{\rm curv}_{X}[\psi],
\end{align}
where 
\begin{align*}
    E^{\nabla T}_X[\psi] := E^{\nabla T}_{\nu}[\psi] X^\nu , \ E^{\rm curv}_X[\psi] := E^{\rm curv}_{\nu}[\psi] X^\nu.
\end{align*}
Using the computation
\begin{align}
    {}^{(X)}\pi_{00} = -2\frac{Mf}{r^2}, \ {}^{(X)}\pi_{11} = 2\left( f' +\frac{Mf}{r^2} \right), \ {}^{(X)}\pi_{AB} = 2 \frac{\mu f}{r}\delta_{AB},
\end{align}
where $f' = \partial_{r^*}f$. Then the Morawetz bulk term $K^X[\psi] := \frac12 T_{\mu\nu}[\psi] {}^{(X)}\pi^{\mu\nu} $ is given by
\begin{align}
    \begin{aligned}
        K^X[\psi] = \left( \frac12 f' +\frac{\mu f}{r} \right) |\nabla_{e_0}\psi|_T^2+\left( \frac12 f' +\frac{\mu f}{r} \right) |\nabla_{e_1}\psi|_T^2 + \left( -\frac12 f' -\frac{Mf}{r} \right) |\slashed\nabla_{}\psi|_T^2,
    \end{aligned}
\end{align}
where $|\psi|^2_T = \langle \psi,\gamma^T\psi\rangle $.
The derivative coefficients above are not all positive. We therefore introduce a modified Morawetz correction.

\subsubsection{Spin-adapted modified Morawetz correction}
Motivated by the scalar wave case, one might introduce the following correction:
\begin{align}
    q \Re\langle \psi, \gamma^T \nabla_\mu \psi\rangle - \frac12 (\nabla_\mu q)\Re\langle \psi, \gamma^T\psi\rangle.
\end{align}
However, since $\nabla_\mu T \neq0$ in general, we encounter the term
\begin{align}
    q \Re\langle \psi, \gamma(\nabla^\mu T) \nabla_\mu \psi\rangle,
\end{align}
which is a first-order error without the trapping-degeneracy. To remedy this bad term, we introduce the following spin-adapted modified Morawetz correction:
\begin{align}
    P^{q,T}_\mu [\psi] := q \Re\langle \psi, \gamma^T \nabla_\mu\psi\rangle -\frac12 (\nabla_\mu q) \Re\langle \psi,\gamma^T\psi\rangle -\frac12 q  \Re\langle \psi, \gamma(\nabla_\mu T)\psi\rangle 
\end{align}
By applying the divergence, we obtain
\begin{align}
    \begin{aligned}
        \nabla^\mu P_{\mu}^{q,T}[\psi] &=q\Re \langle \psi, \gamma^T \nabla^\mu\nabla_\mu\psi\rangle  + q \Re \langle \nabla^\mu \psi, \gamma^T \nabla_\mu\psi\rangle -\frac12 (\nabla^\mu \nabla_\mu q) \Re\langle \psi,\gamma^T\psi\rangle \\
        & \qquad- (\nabla^\mu q) \Re \langle \psi, \gamma(\nabla_\mu T)\psi\rangle -\frac12 \Re \langle \psi, \gamma(\nabla^\mu \nabla_\mu T)\psi\rangle \\
        & = q \Re \langle \nabla^\mu \psi, \gamma^T \nabla_\mu\psi\rangle -\frac12 (\nabla^\mu \nabla_\mu q) \Re\langle \psi,\gamma^T\psi\rangle - (\nabla^\mu q) \Re \langle \psi, \gamma(\nabla_\mu T)\psi\rangle -\frac12 q\Re \langle \psi, \gamma(\nabla^\mu \nabla_\mu T)\psi\rangle.
    \end{aligned}
\end{align}
Note that the aforementioned bad term does not appear in the above computation.

Now we put
\begin{align}
    J_{\mu}^{X,q,T}[\psi] := T_{\mu\nu}[\psi] X^\nu + P^{q,T}_{\mu}[\psi].
\end{align}
Then its divergence is given by
\begin{align}
    \nabla^\mu J_{\mu}^{X,q,T} [\psi] = K^{X,f,q}[\psi] + E^{\nabla T}_X[\psi] + E^{\rm curv}_X [\psi]+ E^{q,T}_{}[\psi] ,
\end{align}
where 
\begin{align}
    K^{X,f,q}[\psi] = C_0 |\nabla_{e_0}\psi|^2_T + C_1|\nabla_{e_1}\psi|^2_T + C_{\Omega} |\slashed\nabla\psi|^2_T +V^{f,q,T}|\psi|^2_T, 
\end{align}
and
\begin{align}
    \begin{aligned}
        C_0 = \frac12 f' +\frac{\mu f}{r}-q, \  C_1 = \frac12 f' -\frac{\mu f}{r}+q, \  C_{\Omega} = -\frac12 f' -\frac{Mf}{r^2}+q, \\
        V^{f,q,T}= -\frac12 \Box_g q,
    \end{aligned}
\end{align}
and
\begin{align}
    E^{q,T}[\psi] := -(\nabla^\mu q) \Re\langle \psi, \gamma(\nabla_\mu T)\psi\rangle - \frac12 q \Re \langle \psi, \gamma(\Box_g T)\psi\rangle  
\end{align}
In particular, for radial $q$, i.e., $q=q(r^*)$, we have
\begin{align}
    \Box_g q = \frac1{\mu r^2}\partial_{r^*}(rq'),
\end{align}
where $q' = \partial_{r^*}q $.

Now we choose a suitable $q$. We first introduce
\begin{align}
    q_0 := \frac12 f' +\frac{\mu f}{r}.
\end{align}
Then $C_0=0$ and $C_1=f'$, and
\begin{align*}
    C_{\Omega} = \frac{ \mu f}{r}-\frac{Mf}{r^2} = \frac{f}{r}\left( 1-\frac{2M}{r} \right) - \frac{Mf}{r^2} = \frac{f}{r}-\frac{3Mf}{r^2} = \frac{f}{r} \left( 1-\frac{3M}{r} \right) .
\end{align*}
Now we choose $f$ so that
\begin{align}
    f(3M)=0, \ f'(3M) >0, \ f(r)(r-3M)\ge0.
\end{align}
Then we have
\begin{align}
    f(r)\frac{r-3M}{r^2} \approx \frac{(r-3M)^2}{r^3}.
\end{align}
To obtain a similar degeneracy for $C_0$, we introduce
\begin{align}
    h := \epsilon \widetilde\chi(r) f(r) \frac{r-3M}{r^2}, 
\end{align}
for a sufficiently small $0<\epsilon\ll1$ and $\widetilde\chi$ is a smooth bump function satisfying $0\le\widetilde\chi\le1$ and is equal to $1$ near the trapping region. Then we put $q= q_0-h$ so that
\begin{align}
C_0 = h, \ C_1 = f'-h, \ C_{\Omega} = f\frac{r-3M}{r^2}-h.
\end{align}
It follows that the principal derivative bulk term satisfies the following lower bound:
\begin{align}
    K^{X,f,q}[\psi] \gtrsim |\nabla_{e_1}\psi|^2_T + \frac{(r-3M)^2}{r^3}\left( |\nabla_{e_0}\psi|^2_T+|\slashed\nabla\psi|^2_T \right).
\end{align}
It remains to estimate the error terms $E^{\nabla T}_X$ and $E^{\rm curv}_X$. Since $X$ is parallel to $e_1$ and $\nabla_{e_1}T=0$,
\begin{align*}
    E^{\nabla T}_X[\psi] &= \Re\langle \nabla^\mu \psi, \gamma(\nabla_\mu T) \nabla_X\psi\rangle  \\
    & =- f(r)\frac{M}{r^2} \Re \langle \nabla_{e_0}\psi, \gamma(e_1)  \nabla_{e_1}\psi\rangle .
\end{align*}
Then we see that
\begin{align}
    | E^{\nabla T}_X[\psi]| \lesssim \epsilon C_1 |\nabla_{e_1}\psi|^2_T + C_{\epsilon} h |\nabla_{e_0}\psi|^2_T,
\end{align}
where $h$ is given as above. Here the smallness of $h$ is assured by the factor $\frac{M^2}{r^4}\approx M^{-2}\ll1$ near the trapping region $\{r=3M\}$. 

Now we turn our attention to the error from the spin-curvature $E^{\rm curv}_X$. Since $X$ is parallel to $e_1$, we have
\begin{align}
    E^{\rm curv}_X[\psi] = f \mu^{\frac12} \Re \langle \nabla^\mu \psi, \gamma^T \mathscr R_{\mu1}\psi\rangle.
\end{align}
By the definition of $\mathscr R_{\mu\nu}$, $\mathscr R_{11}\equiv0$, and therefore
\begin{align}
    \begin{aligned}
        |E^{\rm curv}_X[\psi]| & \lesssim |f| \frac{M}{r^3} |\psi|_T ( |\nabla_{e_0}\psi|_T+|\slashed\nabla\psi|) \\
        & \lesssim \epsilon f^2 |\nabla_{e_0}\psi|^2_T+|\slashed\nabla\psi|^2)+C_\epsilon\frac{M^2}{r^6}|\psi|^2_T.
    \end{aligned}
\end{align}
Note that $C_\epsilon\dfrac{M^2}{r^6}\le \dfrac\epsilon{r^2}$ near $r=3M$.

Finally we consider the error term $E^{q,T}$. Since $q$ is radial, $\nabla q$ is parallel to $e_1$ and hence $\nabla^\mu q \Re\langle \psi, \gamma(\nabla_\mu T)\psi\rangle=0$. On the other hand, an easy computation gives
\begin{align}
    \Box_g T = - \frac{M^2}{r^4\mu}T,
\end{align}
which yields
\begin{align}
    E^{q,T}[\psi] = \frac12 q \frac{M^2}{r^4\mu} |\psi|^2_T.
\end{align}
In conclusion, the zeroth order term is given by
\begin{align}
    V^{f,q,T}_{\rm spin}[\psi] = -\frac12 \Box_g q |\psi|^2_T +\frac12 q \frac{M^2}{r^4\mu}|\psi|^2_T.
\end{align}
Then the remaining task is to show that
\begin{align}
    V^{f,q,T}_{\rm spin}[\psi] \ge \frac{c}{r^2}|\psi|^2_T, \quad \textrm{near } r=3M,
\end{align}
for some $c>0$. To do this, we choose
\begin{align}
    f(r) = M\left(1- \frac{9M^2}{r^2} \right).
\end{align}
By construction, $f(3M)=0$, and
\begin{align*}
    f'= \frac{18M^3}{r^3} >0,
\end{align*}
and
\begin{align*}
    f(r)(r-3M) = \frac{M}{r^2}(r-3M)^2(r+3M)\ge0.
\end{align*}
Substituting this expression into $q_0=\frac12f'+\frac{\mu f}{r}$ yields $q_0=\dfrac{M(r-2M)}{r^2}$.

Combining these computations, we conclude that
\begin{align}
    \nabla^\mu J_\mu^{X,q,T}[\psi] \gtrsim |\nabla_{r^*}\psi|^2_T+ \left( 1-\frac{3M}{r}\right)^2( |\nabla_T\psi|_T^2+|\slashed\nabla\psi|^2_T) +\frac1{r^2}|\psi|^2_T,
\end{align}
on a compact region containing the photon sphere $\{r=3M\}$. Note that all the error terms are absorbed into the above bulk terms.

So far we have considered the Morawetz estimates away from the event horizon. Indeed, $\mu^{-\frac12}$ becomes singular near the event horizon. To remedy this singularity near $\mathcal H^+$, we introduce the ingoing Eddington-Finkelstein coordinate $v=t+r^*$ to write
\begin{align*}
g= - \mu dv^2+ 2dv dr + r^2d\omega^2,
\end{align*}
so that the metric is now regular at $r=2M$. Moreover, we choose a future-directed timelike vector field $\mathcal N$ near $\mathcal H^+$ so that the energy current $\mathcal J^{\mathcal N}[\psi]$ is non-degenerate at $r=2M$ and the bulk term has a good sign: $T_{\mu\nu}[\psi] {}^{(\mathcal N)} \pi^{\mu\nu}\gtrsim |\nabla\psi|^2$. We note that the divergence of the spinorial wave stress tensor introduces an additional error term as the above argument. However, this is not harmful, since we are away from the trapping region, and near $r=2M$, we have
\begin{align}
|\nabla \mathcal N| \lesssim \frac1{r^2} \approx M^{-2},
\end{align}
and hence the error terms can be absorbed into the bulk term.
\end{proof}
\subsection{$r^p$-method for the linear Dirac equation}
By applying the Dirac operator $\gamma^\nu\nabla_\nu$ once again, we obtain
\begin{align}
    g^{\mu\nu}\nabla_\mu\nabla_\nu\psi = -\frac14 \mathcal R\psi,
\end{align}
where $\mathcal R$ is the scalar curvature. We recall that $\mathcal R$ vanishes on the Schwarzschild background.

We now derive an $r^p$-weighted energy inequality. Set $\Psi=r\psi$ and rewrite the linear Dirac equation in wave form:
\begin{align}
    -\frac1{2(1-\frac{2M}r)}(\nabla_{L}\nabla_{\underline L}+\nabla_{\underline L}\nabla_L)\Psi +\slashed\Delta\Psi -\frac{2M}{r^3}\Psi = 0.
\end{align}
The rest of the argument follows the approach in our previous work. Multiply the equation above by $r^p\nabla_L\Psi^\dagger\gamma^T$, and multiply its adjoint by $r^p\gamma^T\nabla_L\Psi$. Taking the real part of their sum and integrating by parts over $\mathcal D^{\tau_2}_{\tau_1}$, bounded by $\Sigma_{\tau_1}$, $\Sigma_{\tau_2}$, and the timelike hypersurface $r=R$, yields
\begin{align}\label{p-we-schwarz}
    \begin{aligned}
     &   \int_{ \substack{u=\tau_2-R \\ v\ge \tau_2+R}  } r^p |\nabla_L\Psi|_T^2\,\sin\theta d\theta d\phi dv + \int_{\mathcal D^{\tau_2}_{\tau_1}} r^{p-1} \left( p |\nabla_L\Psi|_T^2+(2-p)|\slashed\nabla\Psi|_T^2 \right)\,\sin\theta d\theta d\phi du dv + \int_{\mathcal I^{\tau_2-R}_{\tau_1-R}} r^p|\slashed\nabla \Psi|^2_T\,\sin\theta d\theta d\phi du \\
     & \le C\left(  \int_{ \substack{u=\tau_1-R \\ v\ge \tau_1+R}  } r^p |\nabla_L\Psi|_T^2\,\sin\theta d\theta d\phi dv + \int_{\tau_1}^{\tau_2} r^p ( |\slashed\nabla \Psi|^2_T+|\nabla_L\Psi|^2_T \bigg|_{r=R}\,d\tau \right). 
    \end{aligned}
\end{align}
The boundary integral along the timelike surface $\{r=R\}$ can be controlled by averaging over $\{R-1\le r\le R\}$. Choosing $R$ sufficiently large that $R-1>3M$ ensures that this argument incurs no derivative loss. Thus,
\begin{align}\label{p-we-schwarz1}
    \begin{aligned}
     &   \int_{ \substack{u=\tau_2-R \\ v\ge \tau_2+R}  } r^p |\nabla_L\Psi|_T^2\,\sin\theta d\theta d\phi dv + \int_{\mathcal D^{\tau_2}_{\tau_1}} r^{p-1} \left( p |\nabla_L\Psi|_T^2+(2-p)|\slashed\nabla\Psi|_T^2 \right)\,\sin\theta d\theta d\phi du dv + \int_{\mathcal I^{\tau_2-R}_{\tau_1-R}} r^p|\slashed\nabla \Psi|^2_T\,\sin\theta d\theta d\phi du \\
     & \le C\left(  \int_{ \substack{u=\tau_1-R \\ v\ge \tau_1+R}  } r^p |\nabla_L\Psi|_T^2\,\sin\theta d\theta d\phi dv + \int_{\Sigma_{\tau_1}} \mathcal J^N[\psi]\,d\sigma_{\Sigma_\tau}  \right). 
    \end{aligned}
\end{align}
Following the argument of Dafermos and Rodnianski \cite{DR}, we obtain
\begin{align}
    \int_{\Sigma_\tau} \mathcal J^N [ \psi ]\, d\sigma_{\Sigma_\tau} \le C \tau^{-2} \int_{\Sigma_{\tau_1}} r^2 \mathcal J^N [ \mathscr L_T^{\le2} \psi ]\,d\sigma_{\Sigma_\tau}.
\end{align}
The derivative loss arises from the non-degenerate integrated local energy decay estimate near the trapping region $\{r=3M\}$. 

\begin{prop}
    Let $0\le p\le 2$. Given the linear (massless) Dirac equation $i\gamma^\mu\nabla_\mu\psi=0$ on the Schwarzschild black hole background, we have for $Z\in \{L,\underline L\}$,
    \begin{align}
        \begin{aligned}
            \int_{\mathcal D^{\tau_2}_{\tau_1}} r^p | \langle \gamma^\lambda\nabla_\lambda( [ \gamma^\mu\nabla_\mu, \mathscr L_Z ]\psi),\gamma^T \nabla_L\mathscr L_Z\psi\rangle |\,dV  \lesssim \sum_{Z\in \{L,\underline L,\Omega\} } 
               M \langle R\rangle^{-1}\, {}^{(p)} \mathcal E^D[ \mathscr L_Z\psi ]^2(\tau_1),
        \end{aligned}
    \end{align}
    provided that $R>3M$ is sufficiently large number such that $\frac{M}{R}\ll1$.
    In particular, we have
     \begin{align}
        \begin{aligned}
            \int_{\mathcal D^{\tau_2}_{\tau_1}} r^p | \langle \gamma^\lambda\nabla_\lambda( [ \gamma^\mu\nabla_\mu, \mathscr L_\Omega ]\psi),\gamma^T \nabla_L\mathscr L_Z\psi\rangle |\,dV   = 0.
        \end{aligned}
    \end{align}
\end{prop}
\begin{proof}
We have restricted the discussion to the far region, $\{r\ge R\}$, where the null vector fields $\widetilde L$ and $\widetilde{\underline L}$ agree with $L$ and $\underline L$, respectively.
    We are only concerned with $Z=L$. From the Dirac equation $i\gamma^\mu\nabla_\mu\psi=0$, we have
    \begin{align}
        [ \gamma^\mu\nabla_\mu, \mathscr L_L ] \psi = - \frac{M}{2r^2}\gamma^{e_A}\nabla_{e_A}\psi+\frac{\mu}{r}\slashed{g}_{AB}\gamma^{e_B}\nabla_{e_A}\psi.
    \end{align}
Thus we have
\begin{align*}
     \int_{\mathcal D^{\tau_2}_{\tau_1}} r^p | \langle \gamma^\lambda\nabla_\lambda( [ \gamma^\mu\nabla_\mu, \mathscr L_L ]\psi),\gamma^T \nabla_L\mathscr L_Z\psi\rangle |\,dV & \lesssim  M\int_{\mathcal D^{\tau_2}_{\tau_1}} r^{p-2} |  \gamma^\lambda\nabla_\lambda\slashed\nabla\psi|  | \nabla_L\mathscr L_L\psi |\,dV \\
     & \qquad + \int_{\mathcal D^{\tau_2}_{\tau_1}} r^{p-1} |  \gamma^\lambda\nabla_\lambda\slashed\nabla\psi|  | \nabla_L\mathscr L_L\psi |\,dV.
\end{align*}
If $\lambda=\underline L$, we use the Dirac equation once again to get
\begin{align*}
    \int_{\mathcal D^{\tau_2}_{\tau_1}} r^{p-2} |  \gamma^{\underline L}\nabla_{\underline L}\slashed\nabla\psi|  | \nabla_L\mathscr L_L\psi |\,dV & \lesssim \int_{\mathcal D^{\tau_2}_{\tau_1}} r^{p-2} |  \nabla_{ L}\slashed\nabla\psi|  | \nabla_L\mathscr L_L\psi |\,dV + \int_{\mathcal D^{\tau_2}_{\tau_1}} r^{p-2} |  \slashed\nabla^2\psi|  | \nabla_L\mathscr L_L\psi |\,dV \\
    & \lesssim \int_{\mathcal D^{\tau_2}_{\tau_1}}  r^{p-2}|\nabla_L\mathscr L_L\psi|^2\,dV + \int_{\mathcal D^{\tau_2}_{\tau_1}} r^{p-4} |\nabla_L\mathscr L_{\Omega}\psi|^2\,dV + \int_{\mathcal D^{\tau_2}_{\tau_1}} r^{p-4} |\slashed\nabla\mathscr L_{\Omega}\psi|^2\,dV .
\end{align*}
Then it is enough to control the first integral of the last inequality. If $p=1+\delta>1$, then 
\begin{align*}
    \int_{\mathcal D^{\tau_2}_{\tau_1}}  r^{-1+\delta}|\nabla_L\mathscr L_L\psi|^2\,dV & \lesssim  \int_{\substack{u=\tau_1-R \\ v\ge \tau_1+R } } r^\delta |\nabla_L\mathscr L_L\psi|^2\,d\sigma + \int_{\mathcal D^{\tau_2}_{\tau_1}} r^{\delta} | \langle \gamma^\lambda\nabla_\lambda( [ \gamma^\mu\nabla_\mu, \mathscr L_L ]\psi),\gamma^T \nabla_L\mathscr L_Z\psi\rangle |\,dV \\
    & \lesssim \langle R\rangle^{-1}\,{}^{(p)} \mathcal E^D [ \mathscr L_L\psi ]^2(\tau_1) + \langle R\rangle^{\delta-p} \int_{\mathcal D^{\tau_2}_{\tau_1}} r^p | \langle \gamma^\lambda\nabla_\lambda( [ \gamma^\mu\nabla_\mu, \mathscr L_L ]\psi),\gamma^T \nabla_L\mathscr L_Z\psi\rangle |\,dV,
\end{align*}
where we note that the spacetime integral can be absorbed into the left-hand side due to the factor $\langle R\rangle^{\delta-p}=\langle R\rangle^{-1}$. On the other hand, if $p=\delta\le 1$, then 
\begin{align*}
     \int_{\mathcal D^{\tau_2}_{\tau_1}}  r^{-2+\delta}|\nabla_L\mathscr L_L\psi|^2\,dV & \lesssim \langle R\rangle^{-1}  \int_{\mathcal D^{\tau_2}_{\tau_1}}  r^{-1+\delta}|\nabla_L\mathscr L_L\psi|^2\,dV \\
     & \lesssim \langle R\rangle^{-1} \left(  {}^{(p)}\mathcal E^D[ \mathscr L_L\psi ]^2(\tau_1)  + \int_{\mathcal D^{\tau_2}_{\tau_1}} r^{\delta} | \langle \gamma^\lambda\nabla_\lambda( [ \gamma^\mu\nabla_\mu, \mathscr L_L ]\psi),\gamma^T \nabla_L\mathscr L_Z\psi\rangle |\,dV \right), 
\end{align*}
where the spacetime integral is absorbed into the left-hand side due to the factor $\langle R\rangle^{-1}$. It remains to estimate the spacetime integral with weight $r^{p-1}$. As in the preceding argument, it suffices to consider $\lambda=\underline L$; applying the Dirac equation again gives
\begin{align*}
	\int_{\mathcal D^{\tau_2}_{\tau_1}} r^{p-1} |  \gamma^\lambda\nabla_\lambda\slashed\nabla\psi|  | \nabla_L\mathscr L_L\psi |\,dV & \lesssim \int_{\mathcal D^{\tau_2}_{\tau_1}} r^{p-1} |  \nabla_L\slashed\nabla\psi|  | \nabla_L\mathscr L_L\psi |\,dV+ \int_{\mathcal D^{\tau_2}_{\tau_1}} r^{p-1} |  \slashed\nabla^2\psi|  | \nabla_L\mathscr L_L\psi |\,dV \\
	& \lesssim \int_{\mathcal D^{\tau_2}_{\tau_1}} r^{p-2} |  \nabla_L\mathscr L_{\Omega}\psi|  | \nabla_L\mathscr L_L\psi |\,dV+ \int_{\mathcal D^{\tau_2}_{\tau_1}} r^{p-2} |  \slashed\nabla\mathscr L_{\Omega}\psi|  | \nabla_L\mathscr L_L\psi |\,dV \\
	& \lesssim \int_{\mathcal D^{\tau_2}_{\tau_1}} r^{p-2} |  \nabla_L\mathscr L_{\Omega}\psi|^2\,dV + \int_{\mathcal D^{\tau_2}_{\tau_1}} r^{p-2}   | \nabla_L\mathscr L_L\psi |^2\,dV+ \int_{\mathcal D^{\tau_2}_{\tau_1}} r^{p-2} |  \slashed\nabla\mathscr L_{\Omega}\psi|^2 \,dV, 
\end{align*}
The remaining case follows by the same argument, so we omit the details.
\end{proof}

\begin{prop}[Morawetz estimates for the spinor fields on the far region]\label{morawetz-dirac}
    Given the Dirac equation $\mathscr D\psi = F$, we have
    \begin{align}
        \begin{aligned}
            \int_{\mathcal D^{\tau_{i+1}}_{\tau_i}}\frac1{\langle r\rangle^{1+\delta}} \left( |\nabla_L\psi|_T^2+|\nabla_{\underline L}\psi|_T^2+|\slashed\nabla\psi|_T^2    \right)\,dV & \lesssim \int_{\partial \mathcal D^{\tau_{i+1}}_{\tau_i}}\mathcal J^T[ \psi ]\,d\sigma + \left|  \int_{\mathcal D^{\tau_{i+1}}_{\tau_i}} \langle F , \gamma^T \nabla \psi \rangle\,dV \right|,
        \end{aligned}
    \end{align}
    for some small $\delta>0$.
    In particular, we have
    \begin{align}
        \int_{ \mathcal D^{\tau_{i+1}}_{\tau_i}} \frac1{\langle r\rangle^{3+\delta}} \langle \psi,\gamma^T\psi\rangle\,dV \lesssim \int_{\partial\mathcal D^{\tau_{i+1}}_{\tau_i}}\mathcal J^T[ \psi ]\,d\sigma +  \left|  \int_{\mathcal D^{\tau_{i+1}}_{\tau_i}} \langle F , \gamma^T \nabla \psi \rangle\,dV \right|.
    \end{align}
\end{prop}

As in the preceding scalar-wave analysis, we obtain the following decay estimate for the spinor energy flux on ingoing null hypersurfaces:
\begin{cor}
    We have
    \begin{align}
        \begin{aligned}
            \sup_{v:\tau\le\tau(v)} {}^{(0)}\mathcal F^D[ \psi ]^2(v,\tau_i,\tau_{i+1}) \lesssim {}_{\rm int}\mathcal E^D_{} [\psi]^2(\tau_i; |r-3M|>1)+ {}^{(\delta)}\mathcal E^D_{}[\psi]^2(\tau_i).
        \end{aligned}
    \end{align}
    In general, for $0\le k \le N-1$, we have
    \begin{align}
        \begin{aligned}
            \sup_{v:\tau\le\tau(v)} {}^{(0)}\mathcal F^D_{\le k}[ \psi ]^2(v,\tau_i,\tau_{i+1}) \lesssim {}_{\rm int}\mathcal E^D_{\le k} [\psi]^2(\tau_i; |r-3M|>1)+ {}^{(\delta)}\mathcal E^D_{\le k}[\psi]^2(\tau_i).
        \end{aligned}
    \end{align}
\end{cor}
We apply the divergence identity for the spinorial wave stress tensor with $n=T$. The error term involving $\nabla T$ is controlled by the integrated local energy decay estimate in the far region $r\ge R$, where its $r^{-2}$ weight is integrable.

\subsection{Energy in terms of the Dirac current}
\begin{prop}
 We consider the linear (massless) Dirac equation $i\gamma^\mu\nabla_\mu\psi=0$ on the Schwarzschild black hole background. 
We choose the set of the vector fields $\mathcal Z$ consisting of the Killing vectors $T$, $\Omega$, and null vectors $L$, $\underline L$, and the transversal field $Y$ near the horizon $\mathcal H^+$ and $\partial_{r^*}$ away from $\mathcal H^+$. Then we have
 \begin{align}
 \begin{aligned}
& \sum_{Z\in \mathcal Z } \int_{\Sigma_\tau}\langle \mathscr L_Z\psi, \gamma(L)\mathscr L_Z\psi\rangle\,d\sigma + \int_{\mathcal I^{\tau-R}_{\tau_0-R} } \langle \mathscr L_Z\psi, \gamma(\underline L)\mathscr L_Z\psi\rangle\,d\sigma+ \int_{\mathcal H^+[\tau_0,\tau]}\langle \mathscr L_Z\psi, \gamma(L)\mathscr L_Z\psi\rangle\,d\sigma  \\
 &\lesssim  \sum_{Z\in \mathcal Z }\mathscr E^D[ \mathscr L_Z\psi ]^2(\tau_0)+ {}_{\rm int}\mathscr E^D[\mathscr L_Z\psi ]^2(\tau_0) + {}^{(\delta)}\mathcal E^D[\psi ]^2(\tau_0)+ \sup_{v:\tau\le\tau(v)}\mathcal F^D[\psi]^2(v,\tau_0,\tau).
 \end{aligned}
 \end{align}
\end{prop}
\begin{proof}
    Using the divergence theorem for the Dirac current, we have
    \begin{align}
        \begin{aligned}
            & \int_{\Sigma_\tau}\langle \mathscr L_Z\psi, \gamma(L)\mathscr L_Z\psi\rangle\,d\sigma + \int_{\mathcal I^{\tau-R}_{\tau_0-R} } \langle \mathscr L_Z\psi, \gamma(\underline L)\mathscr L_Z\psi\rangle\,d\sigma \\
            & \qquad = \int_{\Sigma_{\tau_0}}\langle \mathscr L_Z\psi, \gamma(L)\mathscr L_Z\psi\rangle\,d\sigma + \int_{D^{\tau}_{\tau_0}} \langle [ \gamma^\mu\nabla_\mu,\mathscr L_Z ] \psi, \mathscr L_Z\psi\rangle\,dV.
        \end{aligned}
    \end{align}
    The spacetime integral of the commutator terms vanishes if $Z=T,\Omega$ since they are Killing. We first consider the case $Z={L}$. Then we decompose the spacetime integral into the far region and the interior region:
    \begin{align*}
       \int_{D^{\tau}_{\tau_0}} \langle [ \gamma^\mu\nabla_\mu,\mathscr L_{{L}} ] \psi, \mathscr L_{{L}}\psi\rangle\,dV   =  \int_{D^{\tau}_{\tau_0}\cap\{r\le R\} } \langle [ \gamma^\mu\nabla_\mu,\mathscr L_{{L}} ] \psi, \mathscr L_{{L}}\psi\rangle\,dV+  \int_{D^{\tau}_{\tau_0}\cap\{r\ge R\} } \langle [ \gamma^\mu\nabla_\mu,\mathscr L_{{L}} ] \psi, \mathscr L_{{L}}\psi\rangle\,dV.
    \end{align*}
    On the far region, we have
    \begin{align*}
        \left|\int_{D^{\tau}_{\tau_0}\cap\{r\ge R\} } \langle [ \gamma^\mu\nabla_\mu,\mathscr L_{{L}} ] \psi, \mathscr L_{{L}}\psi\rangle\,dV \right| & = \left|\int_{D^{\tau}_{\tau_0}\cap\{r\ge R\} } \langle [ \gamma^\mu\nabla_\mu,\mathscr L_{{L}} ] \psi, \mathscr L_{{L}}\psi\rangle\,dV \right| \\
        & \lesssim M \int_{D^{\tau}_{\tau_0}\cap\{r\ge R\} } \frac1{r^2} |\slashed\nabla\psi| |\mathscr L_L\psi|\,dV +  \mu\int_{D^{\tau}_{\tau_0}\cap\{r\ge R\} } \frac1{r} |\slashed\nabla\psi| |\mathscr L_L\psi|\,dV. 
    \end{align*}
    Then we see that
    \begin{align*}
    	M \int_{D^{\tau}_{\tau_0}\cap\{r\ge R\} } \frac1{r^2} |\slashed\nabla\psi| |\mathscr L_L\psi|\,dV & \lesssim M \langle R\rangle^{-1+\delta} \int_{ D^{\tau}_{\tau_0}\cap\{r\ge R\} } \frac1{r^{1+\delta}} \left( |\nabla_L\psi|^2+|\slashed\nabla\psi|^2\right)\,dV \\
    	& \lesssim M\langle R\rangle^{-1+\delta} \left( \mathcal E^D[ \psi ]^2(\tau_0)+\sup_{v:\tau\le\tau(v)}\mathcal F^D[\psi]^2(v,\tau_0,\tau) \right).
    \end{align*}
    where we used the integrated local energy decay estimates away from the trapping region $r=3M$ to obtain the last inequality. On the other hand,
    \begin{align*}
    	\mu\int_{D^{\tau}_{\tau_0}\cap\{r\ge R\} } \frac1{r} |\slashed\nabla\psi| |\mathscr L_L\psi|\,dV & \lesssim \mu \int_{D^{\tau}_{\tau_0}\cap\{r\ge R\} } \frac1{r^{1-\delta}}|\slashed\nabla\psi|^2\,dV + \mu \int_{D^{\tau}_{\tau_0}\cap\{r\ge R\} } \frac1{r^{1+\delta}}|\nabla_L\psi|^2\,dV   \\
    	& \lesssim \mu \,  {}^{(\delta)}\mathcal E^D[\psi]^2(\tau_0)+ \mu \left( \mathcal E^D[ \psi ]^2(\tau_0)+\sup_{v:\tau\le\tau(v)}\mathcal F^D[\psi]^2(v,\tau_0,\tau) \right).
    \end{align*}
   We now consider the commutator with $L$ in the interior region $r\le R$. Away from the event horizon, the commutations with $L$ and $\underline L$ reduce to compositions of the Killing fields $T$ and $\Omega$, so the commutator vanishes; see Section~\ref{sec:interior}. Near the event horizon, we use $T$, $\Omega$, and the transversal field $Y$. The resulting spacetime integrals are controlled by the Morawetz estimates, as in the far region. We omit the details.
\end{proof}
\section{Application to nonlinear problems}\label{sec:tensor-dirac}
We consider the tensorial wave-Dirac system given by
\begin{align}\label{tensor-dirac}
    \begin{aligned}
        i\gamma^\mu\nabla_\mu\psi = \gamma^\nu \psi(  \nabla_\nu \mathrm{Tr}\,h), \\
    \Box_g h_{\mu\nu} = \mathscr T_{\mu\nu}[\psi],
    \end{aligned}
\end{align}
where $\mathscr T_{\mu\nu}[\psi]$ is the Dirac stress-energy tensor, defined by
\begin{align}
    \mathscr T_{\mu\nu}(\psi) & = \frac{i}{4} \left( \langle \psi, \gamma_\mu \nabla_\nu\psi\rangle - \langle \nabla_\nu\psi, \gamma_\mu\psi\rangle + \langle \psi,\gamma_\nu\nabla_\mu\psi\rangle - \langle \nabla_\mu\psi, \gamma_\nu\psi\rangle \right).
\end{align}
This system is motivated by the wave-gauge formulation of the Einstein--Dirac system. It isolates the spinor--tensor coupling while suppressing the quasilinear terms and the nonlinearities that depend only on the metric $g$. We also note that
\begin{align*}
    \mathrm{Tr}h := g^{\mu\nu}h_{\mu\nu} = -\frac1\mu h_{L\underline L}+ \slashed{g}^{AB}h_{AB}.
\end{align*}

\subsection{Derivation of the toy model}
We first consider the Einstein-Dirac system:
\begin{align}
\begin{aligned}
\mathcal R_{\mu\nu}-\frac12 \mathcal R g_{\mu\nu} = \mathscr T_{\mu\nu}[\psi], \\
\mathscr D\psi = 0.
\end{aligned}
\end{align}
We now impose the wave gauge and write $g=m+h$, where $m$ is the Minkowski metric. The wave-coordinate condition is
\begin{align}
g^{\alpha\beta}\Gamma^\mu_{\alpha\beta} = 0,
\end{align}
or equivalently,
\begin{align}
\Box_g x^\mu = 0.
\end{align}
We linearise it and obtain
\begin{align}
\partial^\alpha h_{\alpha\mu}-\frac12 \partial_\mu \mathrm{Tr}h = 0.
\end{align}
Under the wave gauge, the Einstein equation can be written as a system of nonlinear wave equations:
\begin{align}
g^{\alpha\beta}\partial_\alpha\partial_\beta h_{\mu\nu} = F_{\mu\nu}(g) (\partial h,\partial h) -2 \left( \mathscr T_{\mu\nu}-\frac12 g_{\mu\nu}\mathrm{Tr}\mathscr T  \right),
\end{align}
where $F_{\mu\nu}(g)$ is a quadratic term.
Since the energy-momentum tensor for the spinor field is traceless for the massless case, we have
\begin{align}
g^{\alpha\beta}\partial_\alpha\partial_\beta h_{\mu\nu} = F_{\mu\nu}(g) (\partial h,\partial h) -2 \mathscr T_{\mu\nu}.
\end{align}
Now we consider the Dirac equation:
\begin{align}
\mathscr D_g \psi := i\gamma^\mu(g) \nabla_\mu^g \psi,
\end{align}
and using the vierbein formalism, we write
\begin{align}
\nabla^g_{\mu}\psi = \partial_\mu\psi +\frac14 \omega_{\mu a b}(g)\gamma^a\gamma^b \psi.
\end{align}
Thus we may write schematically
\begin{align}
\mathscr D_g\psi \approx \mathscr D_m\psi + h \nabla \psi + (\nabla h)\psi + O(h^2 \nabla\psi)+O(h\nabla h \psi).
\end{align}
More precisely, under the metric variation $g=\bar g+\varepsilon h$, we have
\begin{align}
\delta\mathscr D_h\psi
={}&
-\frac{i}{2}
h^{\mu\nu}
\bar\gamma_\mu
\bar\nabla_\nu\psi+
\frac{i}{4}
\bar\gamma^\nu
\left(
\bar\nabla_\nu\mathrm{Tr}_{\bar g}h
-
\bar\nabla^\mu h_{\mu\nu}
\right)\psi .
\end{align}
Thus we have
\begin{align}
i\bar\gamma^\mu\bar\nabla_\mu\psi
={}
\frac{i}{2}
h^{\mu\nu}\bar\gamma_\mu\bar\nabla_\nu\psi-
\frac{i}{4}
\bar\gamma^\nu
\left(
\bar\nabla_\nu\mathrm{Tr} h
-
\bar\nabla^\mu h_{\mu\nu}
\right)\psi+
O(h^2\bar\nabla\psi)
+
O(h\bar\nabla h\,\psi).
\end{align}
Now by imposing the wave gauge in terms of the linearisation, we have
\begin{align}
i\bar\gamma^\mu\bar\nabla_\mu\psi
={}
\frac{i}{2}
h^{\mu\nu}\bar\gamma_\mu\bar\nabla_\nu\psi-
\frac{i}{8}
\bar\gamma^\nu
\bar\nabla_\nu\mathrm{Tr} h
\psi+
O(h^2\bar\nabla\psi)
+
O(h\bar\nabla h\,\psi).
\end{align}
Finally, we discard the quasilinear term $h\nabla\psi$, the terms that are quadratic solely in the metric perturbation, and the lower-order terms. This yields
\begin{align}
    \begin{aligned}
        i\gamma^\mu\nabla_\mu \psi &= -\frac{i}8 \nabla_\nu\mathrm{Tr}h \gamma^\nu\psi, \\
        \Box_g h_{\mu\nu}  &= \mathscr T_{\mu\nu}[\psi].
    \end{aligned}
\end{align}
\begin{rem}
    We emphasise that the present toy model is derived
    from the Einstein--Dirac system perturbatively around
    Minkowski spacetime, rather than Schwarzschild.
    Consequently, the tensorial wave equation does not
    contain the additional first-order linear terms
    arising from the nontrivial background geometry,
    which would require separate treatment in the
    energy estimates on Schwarzschild.

    The purpose of this toy model is to isolate the
    essential semilinear coupling between the metric
    perturbation and the spinor field, and, in particular,
    to investigate how the apparent loss of derivatives
    arising from the mismatch in differential order
    between the wave and Dirac equations can be overcome.
\end{rem}

\subsection{Main results}
We begin by establishing the local well-posedness of \eqref{tensor-dirac}. To do this, we introduce the notation:
\begin{align}
\begin{aligned}
 \mathcal X^D_{k}[\psi]^2(\tau) := {}_{\rm int}\mathcal E^D_{\le k-1}[\psi]^2(\tau) + \mathscr E^D_{\le k}[\psi]^2(\tau)+\sup_{v:\tau\le \tau(v)} \mathscr F^D_{\le k}[\psi]^2(v,0,\tau), \\
  \mathcal X^W_{k}[h]^2(\tau) := {}_{\rm int}\mathcal E^W_{\le k-1}[h]^2(\tau) + \mathscr E^W_{\le k}[h]^2(\tau)+\sup_{v:\tau\le \tau(v)} \mathscr F^W_{\le k}[h]^2(v,0,\tau).
  \end{aligned}
\end{align}
In particular, if $\tau=0$, then the energy along the ingoing cone vanishes and hence $\mathcal X^D_{k}[\psi]^2(0) := {}_{\rm int}\mathcal E^D_{\le k-1}[\psi]^2(0) + \mathscr E^D_{\le k}[\psi]^2(0)$.
\begin{thm}[Local well-posedness] \label{thm-lwp}
    Let $N\ge17$ and $0<\epsilon_0<1$. Consider the Cauchy problem for the tensorial wave--Dirac system \eqref{tensor-dirac} with initial data $(\psi,h)|_{\Sigma_{\tau=0}}:=(\psi_0,h_0)$ satisfying the smallness condition
     \begin{align}
       \mathcal X^D_{N}[\psi]^2(0)+\mathcal X^W_{N}[h]^2(0) \le (\epsilon_0)^2.
    \end{align}
    Then the system \eqref{tensor-dirac} is locally well-posed. More precisely, there exists $\tau=\tau(\epsilon_0)>0$ such that it admits a unique solution $(\psi,h)$ satisfying
    \begin{align}
        \psi \in C( [ 0, \tau]; H^N)\cap C^1((0,\tau); H^{N-1}) ,\quad h\in C( [0,\tau]; H^N) \cap C^1((0,\tau); H^{N-1})
    \end{align}
    with the energy estimates
    \begin{align}
      \sup_{0\le\tau'\le\tau}\left(   \mathcal X^D_{N}[\psi]^2(\tau')+\mathcal X^W_{N}[h]^2(\tau')\right)\le 2(\epsilon_0)^2.
    \end{align}
\end{thm}
The proof follows the standard iteration scheme based on the energy estimates and is similar to that in our previous paper; see its appendix \cite{hong}.

\begin{prop}
    Let $\mathcal T^*>0$ be a maximal existence time of local solutions to \eqref{tensor-dirac}. If $\mathcal T^*$ is finite, then the energy blows up, i.e., we have
    \begin{align*}
        \limsup_{\tau  \to \mathcal T^*} \left( \mathcal X^D_{N}[\psi]^2(\tau)+\mathcal X^W_{N}[h]^2(\tau) \right) = \infty.
    \end{align*}
    Equivalently,  if
     \begin{align*}
         \sup_{ 0\le \tau < \mathcal T^*} \left( \mathcal X^D_{N}[\psi]^2(\tau)+\mathcal X^W_{N}[h]^2(\tau) \right) < \infty,
     \end{align*}
     then solutions can be extended beyond $\mathcal T^*$.
\end{prop}

The continuation criterion above shows that the proof of global existence reduces to establishing uniform energy bounds. We are now ready to state the main theorem, the full version of Theorem \ref{main-thm-informal}.

\begin{thm}\label{main-thm-formal}
    Consider the tensorial wave--Dirac system \eqref{tensor-dirac} on the Schwarzschild manifold $(\mathcal M,g)$, foliated by hypersurfaces $\Sigma_\tau$. Fix $N\ge17$, $0<\epsilon_0<1$, and $0<\delta<\frac1{10}$. Assume that the initial data $(\psi,h)|_{\Sigma_{\tau=0}}=(\psi_0,h_0)$ satisfy the smallness condition
    \begin{align}
        \mathcal X^D_{N}[\psi]^2(0)+\mathcal X^W_{N}[h]^2(0) \le (\epsilon_0)^2.
    \end{align}
    Then there exist $0\le\varepsilon<\epsilon_0$ sufficiently small and $C\ge1$ sufficiently large such that the following energy estimates hold:
    \begin{align}
    \mathcal X^D_{N}[\psi]^2(\tau)+\mathcal X^W_{N}[h]^2(\tau) \le C^2\varepsilon^2 \tau^{\epsilon},
    \end{align}
 for arbitrarily small $\epsilon>0$.
 
    The $r^p$-weighted energies of the spinor and tensor fields satisfy
    \begin{align}
      {}^{(2-\delta)}\mathcal E^D_{\le N-1} [ \psi ]^2(\tau) + {}^{(2-\delta)}\mathcal E^W_{\le N-1}[h_{LL}]+ {}^{(2-\delta)}\mathcal E^W_{\le N-1}[h_{AB}]+ {}^{(2-\delta)}\mathcal E^W_{\le N-1}[h_{L\underline L}]+ {}^{(2-\delta)}\mathcal E^W_{\le N-1}[h_{AL}] \le C^2\varepsilon^2,
    \end{align}
    and a polynomial growth of the weighted wave energy is allowed only for the components $h_{A\underline L}$ and $h_{\underline L\underline L}$:
    \begin{align}
        {}^{(2-\delta)}\mathcal E^W_{\le N-1}[h_{A\underline L}]^2(\tau) +  {}^{(2-\delta)}\mathcal E^W_{\le N-2}[h_{\underline L\underline L}]^2(\tau) \le C^2\varepsilon^2 \tau^{\frac12+\delta}.
    \end{align}
    Further decay estimates for $\psi$ and $h$ are given in Sections~\ref{sec:decay} and~\ref{sec:interior}.
\end{thm}
A key step in the proof is to close the weighted spinorial wave energy ${}^{(2-\delta)}\mathcal E^D_{\le N-1}[\psi]$ at the highest level of the weighted energy hierarchy. Once the weighted energy estimates for the Dirac field are closed, we proceed to establish the corresponding estimates for the tensorial wave field. Throughout this argument, the Dirac current energy plays an essential role at both lower and higher derivative orders in controlling the nonlinear interactions. However, the top-order Dirac current energy is not required at any stage of the bootstrap argument. This structural feature allows us to avoid the apparent loss of derivatives that would otherwise arise in the top-order Dirac current estimate, while fully exploiting the complementary roles of the spinorial wave energy and the Dirac current energy.

\subsection{$r^p$-hierarchy}
We now establish the $r^p$-weighted energy hierarchy for the nonlinear Dirac equation. Recall that the corresponding spinorial wave equation takes the form
\begin{align}
    g^{\mu\nu}\nabla_\mu\nabla_\nu\psi
    &=
    \nabla_\nu\mathrm{Tr}h\,
    \gamma^\lambda\gamma^\nu\nabla_\lambda\psi
    +\nabla_\lambda\nabla_\nu\mathrm{Tr}h\,
    \gamma^\lambda\gamma^\nu\psi.
\end{align}
At first sight, the second term on the right-hand side appears to cause a loss of derivatives when commuting the equation with $N-1$ vector fields. However, this apparent derivative loss is eliminated by the Clifford-algebraic structure. Indeed, since $\mathrm{Tr}h$ is a scalar function, the symmetry of its Hessian and the Clifford relations yield
\begin{align}
    \gamma^\lambda\gamma^\nu
    \nabla_\lambda\nabla_\nu\mathrm{Tr}h
    =\Box_g\mathrm{Tr}h.
\end{align}
Consequently, the spinorial wave equation reduces to
\begin{align}
    g^{\mu\nu}\nabla_\mu\nabla_\nu\psi
    =
    \nabla_\nu\mathrm{Tr}h\,
    \gamma^\lambda\gamma^\nu\nabla_\lambda\psi
    +(\Box_g\mathrm{Tr}h)\psi.
\end{align}
By the tensorial wave equation, the latter term can be expressed in terms of $\mathscr T[\psi]$, thereby avoiding the apparent loss of derivatives. Thus, unlike the Dirac current estimates, the spinorial wave energy estimates do not suffer from this derivative-loss mechanism at the top-order.

On the far region $r\ge R$, the $r^p$-weighted energy inequality is given by
\begin{align}
\begin{aligned}
& {}^{(p)}\mathcal E^D_{\le k}[\psi]^2(\tau_{i+1}) + \int_{\tau_i}^{\tau_{i+1}} {}^{(p-1)}\mathcal E^D_{\le k}[\psi]^2(\tau)\,d\tau + {}^{(p)}\mathcal F^D_{\le k}[\psi](v=\infty, \tau_{i},\tau_{i+1}) \\
& \lesssim  {}^{(p)}\mathcal E^D_{\le k}[\psi]^2(\tau_{i})+ {}_{\rm int}\mathcal E^D_{\le k}[\psi]^2(\tau_i,|r-3M|>1) + \left| \int_{\mathcal D^{\tau_{i+1}}_{\tau_i}} r^p \langle \mathscr L_Z^{\le k} (  \nabla_\nu \mathrm{Tr}h  \gamma^\lambda \gamma^\nu \nabla_\lambda \psi ) , \gamma^T \nabla_L\mathscr L_Z^{\le k}\psi\rangle\,dV    \right| \\
& \qquad + \left| \int_{\mathcal D^{\tau_{i+1}}_{\tau_i}} r^p \langle \mathscr L_Z^{\le k} (  ( \mathscr{T}_{L\underline L}[\psi]+ \mathscr T_{AB}[\psi]) \psi  ) , \gamma^T \nabla_L\mathscr L_Z^{\le k}\psi\rangle\,dV    \right|. 
\end{aligned}
\end{align}
Here $p$ is chosen to be $p=2-\delta$, $p=1$, and $p=0$ for $0\le k\le N-1$. For lower-order cases $k\le N-9$, we also choose $p=2$. The second spacetime integral contains a cubic nonlinearity in $\psi$ and is therefore easier to estimate than the first.
\subsection{Auxiliary estimates}

\begin{prop}
    We have
    \begin{align}
        \begin{aligned}
            \nabla_L\psi_+ = -\frac12 \gamma^T\gamma^{e_A}\nabla_{e_A}\psi_-+\frac1{2r}\psi + (\nabla_L \mathrm{Tr}h)\psi_+ + (\nabla_{e_A}\mathrm{Tr}h)\gamma^{e_A}\gamma^L \psi, \\
            \nabla_{\underline L}\psi_- =  -\frac12 \gamma^T\gamma^{e_A}\nabla_{e_A}\psi_+-\frac1{2r}\psi + (\nabla_{\underline L}\mathrm{Tr}h)\psi_- +(\nabla_{e_A}\mathrm{Tr}h)\gamma^{e_A}\gamma^{\underline L}\psi .
        \end{aligned}
    \end{align}
\end{prop}

\begin{prop}
    We have
    \begin{align}
        \nabla_L\nabla_{\underline L}\psi = \mu \slashed\Delta\psi + \frac{\mu}{r}(\nabla_L-\nabla_{\underline L})\psi -\mu (\Box_g \mathrm{Tr}h ) \psi - \mu \nabla_\nu \mathrm{Tr}h \gamma^\lambda\gamma^\nu \nabla_\lambda \psi +\frac18 R_{L\underline L\rho\sigma}\gamma^\rho\gamma^\sigma \psi,\\
        \nabla_L\nabla_{\underline L}h_{\rho\sigma} = \mu \slashed\Delta h_{\rho\sigma} +\frac{\mu}{r}(\nabla_L-\nabla_{\underline L})h_{\rho\sigma}-\mu \mathscr T_{\rho\sigma}(\psi)-{R_{L\underline L\rho}}^{\!\alpha}h_{\alpha\sigma}-{R_{L\underline L\sigma}}^{\!\alpha}h_{\rho\alpha},
    \end{align}
    where $\mu = 1-\frac{2M}{r}$.
\end{prop}
We note that
\begin{align*}
    \mathrm{Tr}h = 2g^{L\underline L}h_{L\underline L}+\slashed{g}^{AB}h_{AB}, 
\end{align*}
we also have
\begin{align}
    \nabla_L\nabla_{\underline L}\psi = \mu\slashed\Delta\psi +\frac{\mu}{r}(\nabla_L-\nabla_{\underline L})\psi + \mathscr T_{L\underline L}(\psi) \psi - \mu \slashed{g}^{AB}\mathscr T_{AB}(\psi)\psi + \mathcal R h\cdot \psi - \mu \nabla_\nu \mathrm{Tr}h \gamma^\lambda\gamma^\nu \nabla_\lambda \psi +\frac18 R_{L\underline L\rho\sigma}\gamma^\rho\gamma^\sigma \psi,
\end{align}
where $\mathcal R$ is an abbreviation of the curvature tensor.
\subsection{Bootstrap assumptions}

\subsubsection{Assumed estimates for the spinor fields}
We assume the following:
\begin{align}
    \begin{aligned}
        {}^{(2-\delta)}\mathcal E^D_{\le N-4}[\psi]^2(\tau) \le C^2\varepsilon^2 \tau^{\frac12} , \quad {}^{(2-\delta)}\mathcal E^D_{\le N-1}[\psi]^2(\tau) \le C^2\varepsilon^2 \tau^{1-\epsilon}, \\
        {}^{(1)}\mathcal E^D_{\le N-6}[\psi]^2(\tau) \le C^2\varepsilon^2\tau^{-1+\delta}, \quad {}^{(1)}\mathcal E^D_{\le N-1}[\psi]^2(\tau) \le C^2\varepsilon^2 \tau^{\delta} , \\
        {}^{(0)}\mathcal E^D_{\le N-6}[\psi]^2(\tau) \le C^2\varepsilon^2 \tau^{-1-\epsilon} ,\quad {}^{(0)}\mathcal E^D_{\le N-1}[\psi]^2(\tau) \le C^2\varepsilon^2 \tau^{-\frac12}.
    \end{aligned}
\end{align}
We also assume that
\begin{align}
    {}_{\rm int}\mathcal E^D_{\le N-1}[\psi]^2(\tau)+ {}_{\rm int}\mathscr E^D_{\le N-1}[\psi]^2(\tau) +\mathscr E^D_{\le N-1}[\psi]^2(\tau)+\sup_{v:\tau\le \tau(v)} \mathscr F^D_{\le N-1}[\psi]^2(v,\tau,2\tau) \le C^2\varepsilon^2 \tau^{\delta}.
\end{align}
In particular, we assume
\begin{align}
    {}_{\rm int} \mathcal E^D_{\le N-1}[\psi]^2(\tau; |r-3M|>1 ) \le C^2\varepsilon^2 \tau^{-\frac12}.
\end{align}
\subsubsection{Assumed estimates for the tensor fields}
\begin{enumerate}
    \item Assumed estimates for $h_{LL}, h_{AB}, h_{L\underline L}$:
    We assume the following
    \begin{align}
    \begin{aligned}
        {}^{(2-\delta)}\mathcal E^W_{\le N-4}[h_{LL}]^2(\tau)+ {}^{(2-\delta)}\mathcal E^W_{\le N-4}[h_{AB}]^2(\tau)+ {}^{(2-\delta)}\mathcal E^W_{\le N-4}[h_{L\underline L}]^2(\tau) \le C^2\varepsilon^2, \\
         {}^{(2-\delta)}\mathcal E^W_{\le N-1}[h_{LL}]^2(\tau)+ {}^{(2-\delta)}\mathcal E^W_{\le N-1}[h_{AB}]^2(\tau)+ {}^{(2-\delta)}\mathcal E^W_{\le N-1}[h_{L\underline L}]^2(\tau) \le C^2\varepsilon^2 \tau^{\delta}, \\
          {}^{(1)}\mathcal E^W_{\le N-6}[h_{LL}]^2(\tau)+ {}^{(1)}\mathcal E^W_{\le N-6}[h_{AB}]^2(\tau)+ {}^{(1)}\mathcal E^W_{\le N-6}[h_{L\underline L}]^2(\tau) \le C^2\varepsilon^2 \tau^{-1+\delta}, \\
           {}^{(0)}\mathcal E^W_{\le N-1}[h_{LL}]^2(\tau)+ {}^{(0)}\mathcal E^W_{\le N-1}[h_{AB}]^2(\tau)+ {}^{(0)}\mathcal E^W_{\le N-1}[h_{L\underline L}]^2(\tau) \le C^2\varepsilon^2 \tau^{-\frac12}.
          \end{aligned}
    \end{align}
    \item Assumed estimates for $h_{AL}$: We assume the following:
    \begin{align}
        \begin{aligned}
            {}^{(2-\delta)}\mathcal E^W_{\le N-4}[h_{AL}]^2(\tau) \le C^2\varepsilon^2, \quad {}^{(2-\delta)}\mathcal E^W_{\le N-1}[h_{AL}]^2(\tau) \le C^2\varepsilon^2 \tau^{2\delta}, \\
             {}^{(1)}\mathcal E^W_{\le N-4}[h_{AL}]^2(\tau) \le C^2\varepsilon^2, \quad {}^{(1)}\mathcal E^W_{\le N-1}[h_{AL}]^2(\tau) \le C^2\varepsilon^2 \tau^{\delta}.
        \end{aligned}
    \end{align}
    \item Assumed estimates for $h_{A\underline L}$:
    \begin{align}
        \begin{aligned}
             {}^{(2-\delta)}\mathcal E^W_{\le N-4}[h_{A\underline L}]^2(\tau) \le C^2\varepsilon^2, \quad  {}^{(2-\delta)}\mathcal E^W_{\le N-1}[h_{A\underline L}]^2(\tau) \le C^2\varepsilon^2\tau^{1-\epsilon}, \\
              {}^{(1)}\mathcal E^W_{\le N-4}[h_{A\underline L}]^2(\tau) \le C^2\varepsilon^2, \quad  {}^{(1)}\mathcal E^W_{\le N-1}[h_{A\underline L}]^2(\tau) \le C^2\varepsilon^2\tau^{2\delta}.
        \end{aligned}
    \end{align}
    \item Assumed estimates for $h_{\underline L\underline L}$:
    \begin{align}
        \begin{aligned}
            {}^{(0)}\mathcal E^W_{\le N-1}[h_{\underline L\underline L}]^2(\tau) \le C^2\varepsilon^2 \tau^{2\delta}, \quad  {}^{(1)}\mathcal E^W_{\le N-2}[h_{\underline L\underline L}]^2(\tau) \le C^2\varepsilon^2 \tau^{2\delta}, \quad  {}^{(2-\delta)}\mathcal E^W_{\le N-3}[h_{\underline L\underline L}]^2(\tau) \le C^2\varepsilon^2 \tau^{1-\epsilon}.
        \end{aligned}
    \end{align}
\end{enumerate}
\section{Boundedness of the weighted energy for the spinor fields}\label{sec:bdd-weight-dirac}
The purpose of this section is to control the top-order weighted energy of the spinor field. To this end,
we need to control the following spacetime integrals: for $|I|\le N-1$,
\begin{align}
   \sum_{|I_1|+|I_2|\le |I|} \int_{\mathcal D^{\tau_{i+1}}_{\tau_i}} r^{2-\delta} \langle  \nabla_\nu \mathcal L_Z^{I_1}\mathrm{Tr}h \gamma^\lambda\gamma^\nu \nabla_\lambda \mathscr L_Z^{I_2}\psi, \gamma^T \nabla_L\mathscr L_Z^I\psi\rangle\,dV.
\end{align}
We must also estimate the spacetime integrals involving the cubic nonlinearities $\mathscr T_{L\underline L}[\psi]\psi$ and $\mathscr T_{AB}[\psi]$. These terms are easier to control than the preceding integrals, so we omit the details.

Therefore, in what follows, we exclusively consider the following integral:
\begin{align}
\sum_{\substack{|I_1|+|I_2|\le |I|  }}\int_{\mathcal D^{\tau_{i+1}}_{\tau_i}} r^{2-\delta} \langle  \nabla_\nu \mathcal L_Z^{I_1}h_{L\underline L} \gamma^\lambda \gamma^\nu \nabla_\lambda  \mathscr L_Z^{I_2}\psi, \gamma^T \nabla_L\mathscr L_Z^{I}\psi\rangle \,dV	,
\end{align}
with $|I|\le N-1$. By the identity $\mathrm{Tr}h = -\frac1\mu h_{L\underline L}+\slashed{g}^{AB}h_{AB}$, we must also consider the analogous integral involving $h_{AB}$. It is estimated in the same way, so we omit the details. 

\begin{prop}
    For $|I|=k\le N-1$, we have
    \begin{align}
        \begin{aligned}
            & \sum_{\substack{|I_1|+|I_2|\le |I|  }}\left|  \int_{\mathcal D^{\tau_{i+1}}_{\tau_i}} r^{2-\delta} \langle  \nabla_\nu \mathcal L_Z^{I_1}h_{L\underline L} \gamma^\lambda \gamma^\nu \nabla_\lambda  \mathscr L_Z^{I_2}\psi, \gamma^T \nabla_L\mathscr L_Z^{I}\psi\rangle \,dV\right|	 \\
            & \lesssim (\tau_{i+1})^\frac12  \sup_{v:\tau\le\tau(v)}{}^{(0)}\mathcal F^D_{\le k}[\psi] \left( \int_{\tau_i}^{\tau_{i+1}}  {}^{(1-\delta)} \mathcal E^W_{\ll k}[h_{L\underline L}]^2(\tau)\,d\tau  \right)^\frac12 \left( \int_{\tau_i}^{\tau_{i+1}} {}^{(1-\delta)}\mathcal E^D_{\le k}[\psi]^2(\tau)\,d\tau  \right)^\frac12.
        \end{aligned}
    \end{align}
    In particular, following the above bootstrap assumptions, we have
    \begin{align}
        \begin{aligned}
            & \sum_{\substack{|I_1|+|I_2|\le N-1  }}\left|  \int_{\mathcal D^{\tau_{i+1}}_{\tau_i}} r^{2-\delta} \langle  \nabla_\nu \mathcal L_Z^{I_1}h_{L\underline L} \gamma^\lambda \gamma^\nu \nabla_\lambda  \mathscr L_Z^{I_2}\psi, \gamma^T \nabla_L\mathscr L_Z^{\le N-1}\psi\rangle \,dV\right|	\\
            & \lesssim C^3\varepsilon^3 (\tau_i)^{\frac34+\frac{\delta-\epsilon}{2}},
        \end{aligned}
    \end{align}
    which improves the bootstrap assumption for ${}^{(2-\delta)}\mathcal E^D_{\le N-1}[\psi]^2(\tau)$. Furthermore, we have
    \begin{align}
        \begin{aligned}
           {}^{(2-\delta)} \mathcal E^D_{\le k}[\psi]^2(\tau) \le \begin{cases}
               \tau^{\frac58+\frac34\delta-\frac\epsilon4}, \quad k= N-2, \\
               \tau^{\frac{9}{16}+\frac78\delta-\frac\epsilon8}, \quad k= N-3.
           \end{cases}
        \end{aligned}
    \end{align}
    For lower-order cases, namely $k\le N-4$, we have
    \begin{align}
        \begin{aligned}
            & \sum_{\substack{|I_1|+|I_2|\le |I|  }}\left|  \int_{\mathcal D^{\tau_{i+1}}_{\tau_i}} r^{2-\delta} \langle  \nabla_\nu \mathcal L_Z^{I_1}h_{L\underline L} \gamma^\lambda \gamma^\nu \nabla_\lambda  \mathscr L_Z^{I_2}\psi, \gamma^T \nabla_L\mathscr L_Z^{I}\psi\rangle \,dV\right|	 \\
            & \lesssim \left( \int_{\tau_i}^{\tau_{i+1}}  {}^{(0)} \mathcal E^W_{\ll k}[h_{L\underline L}]^2(\tau) {}^{(1-\delta)}\mathcal E^D_{\le k+1}[\psi]^2(\tau) \,d\tau  \right)^\frac12 \left( \int_{\tau_i}^{\tau_{i+1}} {}^{(1-\delta)}\mathcal E^D_{\le k+2}[\psi]^2(\tau)\,d\tau  \right)^\frac12,
        \end{aligned}
    \end{align}
    which gives 
    \begin{align}
        \begin{aligned}
           {}^{(2-\delta)} \mathcal E^D_{\le k}[\psi]^2(\tau) \lesssim \begin{cases}
               \tau^{\frac{1}{32}+\frac{7}{16}\delta-\frac{9}{16}\epsilon}, \quad k = N-5, \\
               1, \quad k \le N-6.
           \end{cases}
        \end{aligned}
    \end{align}
\end{prop}
Iterating these estimates yields stronger bounds.

To prove the above nonlinear estimates, we separately consider the cases $\nu=L,\underline L,$ and $\nu = A$. Indeed, we write
\begin{enumerate}
    \item $\nu=L$:
    \begin{align}
        & \int_{\mathcal D^{\tau_{i+1}}_{\tau_i}} r^{2-\delta} \langle  \nabla_{L} \mathcal L_Z^{I_1}h_{L\underline L} \gamma^{\underline L} \gamma^{L} \nabla_{\underline L}  \mathscr L_Z^{I_2}\psi, \gamma^T \nabla_L\mathscr L_Z^{I}\psi\rangle \,dV	\\
        & \qquad +   \int_{\mathcal D^{\tau_{i+1}}_{\tau_i}} r^{2-\delta} \langle  \nabla_{L} \mathcal L_Z^{I_1}h_{L\underline L} \gamma^{e_A} \gamma^{L} \nabla_{e_A}  \mathscr L_Z^{I_2}\psi, \gamma^T \nabla_L\mathscr L_Z^{I}\psi\rangle \,dV.
    \end{align}
    \item $\nu = A$:
     \begin{align}
        & \int_{\mathcal D^{\tau_{i+1}}_{\tau_i}} r^{2-\delta} \langle  \nabla_{e_A} \mathcal L_Z^{I_1}h_{L\underline L} \gamma^{ L} \gamma^{e_A} \nabla_{L}  \mathscr L_Z^{I_2}\psi, \gamma^T \nabla_L\mathscr L_Z^{I}\psi\rangle \,dV	\\
        & \qquad +   \int_{\mathcal D^{\tau_{i+1}}_{\tau_i}} r^{2-\delta} \langle  \nabla_{e_A} \mathcal L_Z^{I_1}h_{L\underline L} \gamma^{e_B} \gamma^{e_A} \nabla_{e_B}  \mathscr L_Z^{I_2}\psi, \gamma^T \nabla_L\mathscr L_Z^{I}\psi\rangle \,dV , \\
        & \qquad \int_{\mathcal D^{\tau_{i+1}}_{\tau_i}} r^{2-\delta} \langle  \nabla_{e_A} \mathcal L_Z^{I_1}h_{L\underline L} \gamma^{\underline L} \gamma^{e_A} \nabla_{\underline L}  \mathscr L_Z^{I_2}\psi, \gamma^T \nabla_L\mathscr L_Z^{I}\psi\rangle \,dV .
    \end{align}
    \item $\nu = \underline L$:
     \begin{align}
        & \int_{\mathcal D^{\tau_{i+1}}_{\tau_i}} r^{2-\delta} \langle  \nabla_{\underline L} \mathcal L_Z^{I_1}h_{L\underline L} \gamma^{L} \gamma^{\underline L} \nabla_{ L}  \mathscr L_Z^{I_2}\psi, \gamma^T \nabla_L\mathscr L_Z^{I}\psi\rangle \,dV	\\
        & \qquad +   \int_{\mathcal D^{\tau_{i+1}}_{\tau_i}} r^{2-\delta} \langle  \nabla_{\underline L} \mathcal L_Z^{I_1}h_{L\underline L} \gamma^{e_A} \gamma^{\underline L} \nabla_{e_A}  \mathscr L_Z^{I_2}\psi, \gamma^T \nabla_L\mathscr L_Z^{I}\psi\rangle \,dV.
    \end{align}
\end{enumerate}

\subsection{$\nu=L$}
We consider only the low--high and high--low interactions:
\begin{align}
    & \int_{\mathcal D^{\tau_{i+1}}_{\tau_i}} r^{2-\delta} \langle  \nabla_{L} \mathcal L_Z^{\le N-11}h_{L\underline L} \gamma^{\underline L} \gamma^{L} \nabla_{\underline L}  \mathscr L_Z^{\le N-1}\psi, \gamma^T \nabla_L\mathscr L_Z^{\le N-1}\psi\rangle \,dV, \\
    & \int_{\mathcal D^{\tau_{i+1}}_{\tau_i}} r^{2-\delta} \langle  \nabla_{L} \mathcal L_Z^{\le N-1}h_{L\underline L} \gamma^{\underline L} \gamma^{L} \nabla_{\underline L}  \mathscr L_Z^{\le N-11}\psi, \gamma^T \nabla_L\mathscr L_Z^{\le N-1}\psi\rangle \,dV.
\end{align}
For the low--high interaction, we integrate by parts in the $\underline L$ direction to obtain
\begin{align}
    \begin{aligned}
        & \int_{\mathcal D^{\tau_{i+1}}_{\tau_i}} r^{2-\delta} \langle  \nabla_{L} \mathcal L_Z^{\le N-11}h_{L\underline L} \gamma^{\underline L} \gamma^{L} \nabla_{\underline L}  \mathscr L_Z^{\le N-1}\psi, \gamma^T \nabla_L\mathscr L_Z^{\le N-1}\psi\rangle \,dV \\
        & = \int_{\substack{u=\tau_{i+1}-R \\ v\ge \tau_{i+1}+R } } r^{2-\delta} \langle   \nabla_{L} \mathcal L_Z^{\le N-11}h_{L\underline L} \gamma^{\underline L} \gamma^{L}   \mathscr L_Z^{\le N-1}\psi, \gamma^T \nabla_L\mathscr L_Z^{\le N-1}\psi\rangle \,d\sigma  \\
        & \qquad - \int_{\substack{u=\tau_{i}-R \\ v\ge \tau_{i}+R } } r^{2-\delta} \langle   \nabla_{L} \mathcal L_Z^{\le N-11}h_{L\underline L} \gamma^{\underline L} \gamma^{L}   \mathscr L_Z^{\le N-1}\psi, \gamma^T \nabla_L\mathscr L_Z^{\le N-1}\psi\rangle \,d\sigma  \\
        & + \int_{r=R } r^{2-\delta} \langle   \nabla_{L} \mathcal L_Z^{\le N-11}h_{L\underline L} \gamma^{\underline L} \gamma^{L}   \mathscr L_Z^{\le N-1}\psi, \gamma^T \nabla_L\mathscr L_Z^{\le N-1}\psi\rangle \,d\sigma \\
        & + \int_{\mathcal D^{\tau_{i+1}}_{\tau_i} } r^{2-\delta} \langle  \nabla_{\underline L} \nabla_{L} \mathcal L_Z^{\le N-11}h_{L\underline L} \gamma^{\underline L} \gamma^{L}   \mathscr L_Z^{\le N-1}\psi, \gamma^T \nabla_L\mathscr L_Z^{\le N-1}\psi\rangle \,dV \\
        & \qquad +  \int_{\mathcal D^{\tau_{i+1}}_{\tau_i} } r^{2-\delta} \langle   \nabla_{L} \mathcal L_Z^{\le N-11}h_{L\underline L} \gamma^{\underline L} \gamma^{L}   \mathscr L_Z^{\le N-1}\psi, \gamma^T \nabla_{\underline L} \nabla_L\mathscr L_Z^{\le N-1}\psi\rangle \,dV .
    \end{aligned}
\end{align}
The boundary terms are controlled directly. In particular, for an integral over an outgoing null hypersurface, H\"older's inequality and the weighted Sobolev inequality applied to the lower-order term yield
\begin{align*}
    & \left| \int_{\substack{u=\tau_{i+1}-R \\ v\ge \tau_{i+1}+R } } r^{2-\delta} \langle   \nabla_{L} \mathcal L_Z^{\le N-11}h_{L\underline L} \gamma^{\underline L} \gamma^{L}   \mathscr L_Z^{\le N-1}\psi, \gamma^T \nabla_L\mathscr L_Z^{\le N-1}\psi\rangle \,d\sigma  \right| \\
    & \lesssim {}^{(1)}\mathcal E^W_{\le N-8}[h_{L\underline L}](\tau_{i+1}) \mathscr E^D_{\le N-1}[\psi](\tau_{i+1}){}^{(2-\delta)}\mathcal E^D_{\le N-1}[\psi](\tau_{i+1}).
\end{align*}
For the boundary integral along the timelike hypersurface $\{r=R\}$, we use an averaging argument on the interior domain away from the trapping region, which is possible by choosing $R$ sufficiently large so that $R-1>3M+1$, and we obtain
\begin{align*}
    & \left| \int_{r=R } r^{2-\delta} \langle   \nabla_{L} \mathcal L_Z^{\le N-11}h_{L\underline L} \gamma^{\underline L} \gamma^{L}   \mathscr L_Z^{\le N-1}\psi, \gamma^T \nabla_L\mathscr L_Z^{\le N-1}\psi\rangle \,d\sigma \right| \\
    & \lesssim {}_{\rm int}\mathcal E^W_{\le N-8}[h_{L\underline L}](\tau; |r-3M|>1) {}_{\rm int}\mathcal E^D_{\le N-2}[\psi](\tau;|r-3M|>1){}_{\rm int}\mathcal E^D_{\le N-1}[\psi](\tau;|r-3M|>1).
\end{align*}
It remains to estimate the spacetime integrals.
For the spacetime integral in which $\nabla_{\underline L}$ acts on the lower-order term, we use the wave equation for $h_{L\underline L}$ to obtain
\begin{align}
    \begin{aligned}
        & \left| \int_{\mathcal D^{\tau_{i+1}}_{\tau_i} } r^{2-\delta} \langle  \nabla_{\underline L} \nabla_{L} \mathcal L_Z^{\le N-11}h_{L\underline L} \gamma^{\underline L} \gamma^{L}   \mathscr L_Z^{\le N-1}\psi, \gamma^T \nabla_L\mathscr L_Z^{\le N-1}\psi\rangle \,dV \right| \\
        & \lesssim \int_{\mathcal D^{\tau_{i+1}}_{\tau_i} } r^{2-\delta} |\slashed\Delta \mathcal L_Z^{\le N-11}h_{L\underline L} | |   \mathscr L_Z^{\le N-1}\psi_+|  | \nabla_L\mathscr L_Z^{\le N-1}\psi_+ | \,dV \\
        & \qquad + \int_{\mathcal D^{\tau_{i+1}}_{\tau_i} } r^{1-\delta} |\nabla_{\underline L} \mathcal L_Z^{\le N-11}h_{L\underline L} | |   \mathscr L_Z^{\le N-1}\psi_+|  | \nabla_L\mathscr L_Z^{\le N-1}\psi_+ | \,dV \\
        & \qquad +  \int_{\mathcal D^{\tau_{i+1}}_{\tau_i} } r^{1-\delta} |\nabla_{ L} \mathcal L_Z^{\le N-11}h_{L\underline L} | |   \mathscr L_Z^{\le N-1}\psi_+|  | \nabla_L\mathscr L_Z^{\le N-1}\psi_+ | \,dV \\
        & \qquad +  \int_{\mathcal D^{\tau_{i+1}}_{\tau_i} } r^{2-\delta} | \mathcal L_Z^{\le N-11}\mathscr T_{L\underline L}[\psi]  | |   \mathscr L_Z^{\le N-1}\psi_+|  | \nabla_L\mathscr L_Z^{\le N-1}\psi_+ | \,dV.
    \end{aligned}
\end{align}
It remains to apply the weighted Sobolev inequality on the outgoing null hypersurfaces to the lower-order term, which yields the required decay. We omit the details.

For the spacetime integral in which $\nabla_{\underline L}$ acts on the higher-order term, a derivative loss appears. We have
\begin{align}
    \begin{aligned}
        & \left| \int_{\mathcal D^{\tau_{i+1}}_{\tau_i} } r^{2-\delta} \langle   \nabla_{L} \mathcal L_Z^{\le N-11}h_{L\underline L} \gamma^{\underline L} \gamma^{L}   \mathscr L_Z^{\le N-1}\psi, \gamma^T \nabla_{\underline L} \nabla_L\mathscr L_Z^{\le N-1}\psi\rangle \,dV  \right| \\
        & \lesssim \int_{\mathcal D^{\tau_{i+1}}_{\tau_i} } r^{2-\delta} |   \nabla_{L} \mathcal L_Z^{\le N-11}h_{L\underline L} | |  \mathscr L_Z^{\le N-1}\psi_+|  | \slashed\Delta\mathscr L_Z^{\le N-1}\psi_+ | \,dV \\
        & \qquad + \int_{\mathcal D^{\tau_{i+1}}_{\tau_i} } r^{1-\delta} |   \nabla_{L} \mathcal L_Z^{\le N-11}h_{L\underline L} | |  \mathscr L_Z^{\le N-1}\psi_+|  | \nabla_L \mathscr L_Z^{\le N-1}\psi_+ | \,dV \\
        & \qquad + \int_{\mathcal D^{\tau_{i+1}}_{\tau_i} } r^{1-\delta} |   \nabla_{L} \mathcal L_Z^{\le N-11}h_{L\underline L} | |  \mathscr L_Z^{\le N-1}\psi_+|  | \nabla_{\underline L}\mathscr L_Z^{\le N-1}\psi_+ | \,dV \\
        & \qquad + \int_{\mathcal D^{\tau_{i+1}}_{\tau_i} } r^{2-\delta} |   \nabla_{L} \mathcal L_Z^{\le N-11}h_{L\underline L} | |  \mathscr L_Z^{\le N-1}\psi_+|  | \mathscr L_Z^{\le N-1}(\nabla_\nu \mathrm{Tr}h \gamma^\lambda\gamma^\nu \nabla_\lambda\psi)   | \,dV
    \end{aligned}
\end{align}
For the integral involving $\slashed\Delta \mathscr L_Z^{\le N-1}\psi$, which may cause a derivative loss, we integrate by parts in the angular variables to obtain
\begin{align*}
    &  \int_{\mathcal D^{\tau_{i+1}}_{\tau_i} } r^{2-\delta} |   \nabla_{L} \mathcal L_Z^{\le N-11}h_{L\underline L} | |  \mathscr L_Z^{\le N-1}\psi_+|  | \slashed\Delta\mathscr L_Z^{\le N-1}\psi_+ | \,dV \\
    & \lesssim \sup_{\tau\in[\tau_i,\tau_{i+1}]} \mathscr E^D_{\le N-1}[\psi](\tau) \left( \int_{\tau_i}^{\tau_{i+1}} {}^{(0)}\mathcal E^W_{\le N-8}[h_{L\underline L}]^2(\tau)\,d\tau \right)^\frac12  \left( \int_{\tau_i}^{\tau_{i+1}} {}^{(1-\delta)}\mathcal E^D_{\le N-1}[\psi]^2(\tau)\,d\tau \right)^\frac12 \\
    & \qquad +\sup_{\tau\in[\tau_i,\tau_{i+1}]} {}^{(1)}\mathcal E^W_{\le N-8}[h_{L\underline L}](\tau)\int_{\tau_i}^{\tau_{i+1}} {}^{(1-\delta)}\mathcal E^D_{\le N-1}[\psi]^2(\tau)\,d\tau.
\end{align*}
The second integral on the right-hand side is controlled similarly, while the fourth is easier because it contains an additional nonlinear factor. We omit the details.
We now consider the second integral, which contains the higher-order term $\nabla_{\underline L}\mathscr L_Z^{\le N-1}\psi_+$ and is the most delicate case for $\nu=L$. We use the foliation by the ingoing null cones $\underline{\mathcal C_v}\cap\mathcal D^{\tau_{i+1}}_{\tau_i}$ to obtain
\begin{align*}
    & \int_{\mathcal D^{\tau_{i+1}}_{\tau_i} } r^{1-\delta} |   \nabla_{L} \mathcal L_Z^{\le N-11}h_{L\underline L} | |  \mathscr L_Z^{\le N-1}\psi_+|  | \nabla_{\underline L}\mathscr L_Z^{\le N-1}\psi_+ | \,dV \\
    & \lesssim \int r^{1-\delta} \|\nabla_L\mathcal L_Z^{\le N-11}h_{L\underline L}\|_{L^\infty(\underline{\mathcal C_v})} \| \mathscr L_Z^{\le N-1}\psi_+\|_{L^2(\underline{\mathcal C_v})} \|\nabla_{\underline L}\mathscr L_Z^{\le N-1}\psi_+\|_{L^2(\underline{\mathcal C_v})} \,dv \\
    & \lesssim \sup_{v:\tau\le\tau(v)}{}^{(0)}\mathcal F^D_{\le N-1}[\psi] \int r^{\frac12-\delta} \|\nabla_{\underline L}\nabla_L\mathcal L_Z^{\le N-9}h_{L\underline L}\|_{L^\infty(\underline{\mathcal C_v})} \| \mathscr L_Z^{\le N-1}\psi_+\|_{L^2(\underline{\mathcal C_v})} \\
    & \lesssim \sup_{v:\tau\le\tau(v)}{}^{(0)}\mathcal F^D_{\le N-1}[\psi] \int r^{-\frac12-\delta} \|\slashed\nabla\mathcal L_Z^{\le N-8}h_{L\underline L}\|_{L^2(\underline{\mathcal C_v})} \| \mathscr L_Z^{\le N-1}\psi_+\|_{L^2(\underline{\mathcal C_v})} \\
    & \qquad + \sup_{v:\tau\le\tau(v)}{}^{(0)}\mathcal F^D_{\le N-1}[\psi] \int r^{-\frac12-\delta} \|\nabla_{\underline L}\mathcal L_Z^{\le N-9}h_{L\underline L}\|_{L^2(\underline{\mathcal C_v})} \| \mathscr L_Z^{\le N-1}\psi_+\|_{L^2(\underline{\mathcal C_v})} \\
    & \qquad + \sup_{v:\tau\le\tau(v)}{}^{(0)}\mathcal F^D_{\le N-1}[\psi] \int r^{-\frac12-\delta} \|\nabla_{L}\mathcal L_Z^{\le N-9}h_{L\underline L}\|_{L^2(\underline{\mathcal C_v})} \| \mathscr L_Z^{\le N-1}\psi_+\|_{L^2(\underline{\mathcal C_v})} \\
    & \qquad + \sup_{v:\tau\le\tau(v)}{}^{(0)}\mathcal F^D_{\le N-1}[\psi] \int r^{\frac12-\delta} \|\mathcal L_Z^{\le N-9}\mathscr T_{L\underline L}[\psi] \|_{L^2(\underline{\mathcal C_v})} \| \mathscr L_Z^{\le N-1}\psi_+\|_{L^2(\underline{\mathcal C_v})}.
\end{align*}
We focus on the second integral in the last inequality, since the other integrals yield better decay. Applying the Hardy and Cauchy--Schwarz inequalities to this term gives
\begin{align*}
    & \sup_{v:\tau\le\tau(v)}{}^{(0)}\mathcal F^D_{\le N-1}[\psi] \int r^{-\frac12-\delta} \|\nabla_{\underline L}\mathcal L_Z^{\le N-9}h_{L\underline L}\|_{L^2(\underline{\mathcal C_v})} \| \mathscr L_Z^{\le N-1}\psi_+\|_{L^2(\underline{\mathcal C_v})}  \\
    & \lesssim \sup_{v:\tau\le\tau(v)}{}^{(0)}\mathcal F^D_{\le N-1}[\psi] \left( \int_{\tau_i}^{\tau_{i+1}} r^{} {}^{(1-\delta)} \mathcal E^W_{\le N-8}[h_{L\underline L}]^2(\tau)\,d\tau  \right)^\frac12 \left( \int_{\tau_i}^{\tau_{i+1}} {}^{(1-\delta)}\mathcal E^D_{\le N-1}[\psi]^2(\tau)\,d\tau  \right)^\frac12 \\
    & \lesssim  (\tau_{i+1})^\frac12  \sup_{v:\tau\le\tau(v)}{}^{(0)}\mathcal F^D_{\le N-1}[\psi] \left( \int_{\tau_i}^{\tau_{i+1}}  {}^{(1-\delta)} \mathcal E^W_{\le N-8}[h_{L\underline L}]^2(\tau)\,d\tau  \right)^\frac12 \left( \int_{\tau_i}^{\tau_{i+1}} {}^{(1-\delta)}\mathcal E^D_{\le N-1}[\psi]^2(\tau)\,d\tau  \right)^\frac12.
\end{align*}
The ingoing flux \({}^{(0)}\mathcal F^D_{\le N-1}[\psi]\) inherits the decay of the weighted exterior energy \({}^{(\delta)}\mathcal E^D_{\le N-1}[\psi]\) and the degenerate interior energy \(\mathcal E^D_{\le N-1}[\chi\psi]\). Indeed, the boundary flux at \(r=R\) is controlled by an averaging argument over the compact region \(R-1\le r\le R\), where the cutoff \(\chi\) is nondegenerate. Thus, the degeneracy of \(\chi\) near the photon sphere does not affect this estimate.

\begin{rem}\label{rem-energy-hierarchy}
We point out that the above integral is precisely one of the situations in which the particular commutation configuration $\mathscr L_{\underline L}^{\le N}\psi_+$ can arise. Although this term could in principle be estimated directly through the Dirac current energy $\mathscr E^D[\mathscr L_{\underline L}^{\le N}\psi]$, we instead exploit the vector-field/energy hierarchy and control it by the ingoing flux ${}^{(0)}\mathcal F^D[\mathscr L_{\underline L}^{\le N-1}\psi]$. In this way, the estimate can be formulated entirely in terms of the full spinor flux, without keeping track of the individual $+$ and $-$ components.
\end{rem}


For the high-low interaction, we simply use H\"older's inequality and the Sobolev embedding to get
\begin{align}
    \begin{aligned}
        & \left| \int_{\mathcal D^{\tau_{i+1}}_{\tau_i}} r^{2-\delta} \langle  \nabla_{L} \mathcal L_Z^{\le N-1}h_{L\underline L} \gamma^{\underline L} \gamma^{L} \nabla_{\underline L}  \mathscr L_Z^{\le N-11}\psi, \gamma^T \nabla_L\mathscr L_Z^{\le N-1}\psi\rangle \,dV \right| \\
        & \lesssim \int_{\tau_i}^{\tau_{i+1}} r^{\frac32-\delta} \mathcal E^W_{\le N-1}[h_{L\underline L}] \|\nabla_L\nabla_{\underline L}\mathscr L_Z^{\le N-9}\psi_+\|_{L^2(\Sigma_\tau)} \mathcal E^D_{\le N-1}[\psi]\,d\tau \\
        & \lesssim \int_{\tau_i}^{\tau_{i+1}} r^{\frac12-\delta} \mathcal E^W_{\le N-1}[h_{L\underline L}] \|\slashed\nabla\mathscr L_Z^{\le N-8}\psi_+\|_{L^2(\Sigma_\tau)} \mathcal E^D_{\le N-1}[\psi]\,d\tau \\
        & \qquad + \int_{\tau_i}^{\tau_{i+1}} r^{\frac12-\delta} \mathcal E^W_{\le N-1}[h_{L\underline L}] \|\nabla_{\underline L}\mathscr L_Z^{\le N-9}\psi_+\|_{L^2(\Sigma_\tau)} \mathcal E^D_{\le N-1}[\psi]\,d\tau \\
        & \qquad + \int_{\tau_i}^{\tau_{i+1}} r^{\frac12-\delta} \mathcal E^W_{\le N-1}[h_{L\underline L}] \|\nabla_L\mathscr L_Z^{\le N-9}\psi_+\|_{L^2(\Sigma_\tau)} \mathcal E^D_{\le N-1}[\psi]\,d\tau \\
        & \qquad + \int_{\tau_i}^{\tau_{i+1}} r^{\frac32-\delta} \mathcal E^W_{\le N-1}[h_{L\underline L}] \|\mathscr L_Z^{\le N-9}(\nabla_\nu \mathrm{Tr}h \gamma^\lambda\gamma^\nu \nabla_\lambda\psi)  \|_{L^2(\Sigma_\tau)} \mathcal E^D_{\le N-1}[\psi]\,d\tau \\
        & \qquad + \int_{\tau_i}^{\tau_{i+1}} r^{\frac32-\delta} \mathcal E^W_{\le N-1}[h_{L\underline L}] \|\mathscr L_Z^{\le N-9} (\Box_g \mathrm{Tr}h \cdot \psi )\|_{L^2(\Sigma_\tau)} \mathcal E^D_{\le N-1}[\psi]\,d\tau .
    \end{aligned}
\end{align}
We estimate the second integral in the last inequality as follows:
\begin{align*}
    & \int_{\tau_i}^{\tau_{i+1}} r^{\frac12-\delta} \mathcal E^D_{\le N-1}[h_{L\underline L}] \|\nabla_{\underline L}\mathscr L_Z^{\le N-9}\psi_+\|_{L^2(\Sigma_\tau)} \mathcal E^D_{\le N-1}[\psi]\,d\tau \\
    & \lesssim \sup_{\tau\in[\tau_i,\tau_{i+1}]} \mathscr E^D_{\le N-8}[\psi](\tau) \left( \int_{\tau_i}^{\tau_{i+1}} r^{-\delta}\, {}^{(0)}\mathcal E^W_{\le N-1}[h_{L\underline L}]^2(\tau)\,d\tau  \right)^\frac12  \left( \int_{\tau_i}^{\tau_{i+1}}  {}^{(1-\delta)}\mathcal E^D_{\le N-1}[\psi]^2(\tau)\,d\tau  \right)^\frac12, 
\end{align*}
and the control of the first and third integrals is also similar and gives better bounds. For the fourth integral, we write
\begin{align}
    \begin{aligned}
        & \int_{\tau_i}^{\tau_{i+1}} r^{\frac32-\delta} \mathcal E^W_{\le N-1}[h_{L\underline L}] \|\mathscr L_Z^{\le N-9}(\nabla_\nu \mathrm{Tr}h \gamma^\lambda\gamma^\nu \nabla_\lambda\psi)  \|_{L^2(\Sigma_\tau)} \mathcal E^D_{\le N-1}[\psi]\,d\tau  \\
        & \lesssim \int_{\tau_i}^{\tau_{i+1}} r^{\frac32-\delta} \mathcal E^W_{\le N-1}[h_{L\underline L}] \|\mathscr L_Z^{\le N-9}(\nabla_L \mathrm{Tr}h  \nabla_{\underline L}\psi_+)  \|_{L^2(\Sigma_\tau)} \mathcal E^D_{\le N-1}[\psi]\,d\tau \\
        & \qquad + \int_{\tau_i}^{\tau_{i+1}} r^{\frac32-\delta} \mathcal E^W_{\le N-1}[h_{L\underline L}] \|\mathscr L_Z^{\le N-9}(\nabla_L \mathrm{Tr}h  \slashed\nabla\psi_+)  \|_{L^2(\Sigma_\tau)} \mathcal E^D_{\le N-1}[\psi]\,d\tau \\
        & \qquad + \int_{\tau_i}^{\tau_{i+1}} r^{\frac32-\delta} \mathcal E^W_{\le N-1}[h_{L\underline L}] \|\mathscr L_Z^{\le N-9}(\slashed\nabla \mathrm{Tr}h  \nabla_L\psi_+)  \|_{L^2(\Sigma_\tau)} \mathcal E^D_{\le N-1}[\psi]\,d\tau \\
        & \qquad + \int_{\tau_i}^{\tau_{i+1}} r^{\frac32-\delta} \mathcal E^W_{\le N-1}[h_{L\underline L}] \|\mathscr L_Z^{\le N-9}(\slashed\nabla \mathrm{Tr}h  \nabla_{\underline L}\psi_-)  \|_{L^2(\Sigma_\tau)} \mathcal E^D_{\le N-1}[\psi]\,d\tau \\
        & \qquad + \int_{\tau_i}^{\tau_{i+1}} r^{\frac32-\delta} \mathcal E^W_{\le N-1}[h_{L\underline L}] \|\mathscr L_Z^{\le N-9}(\slashed\nabla \mathrm{Tr}h  \slashed\nabla\psi_\pm)  \|_{L^2(\Sigma_\tau)} \mathcal E^D_{\le N-1}[\psi]\,d\tau \\
        & \qquad + \int_{\tau_i}^{\tau_{i+1}} r^{\frac32-\delta} \mathcal E^W_{\le N-1}[h_{L\underline L}] \|\mathscr L_Z^{\le N-9}(\nabla_{\underline L} \mathrm{Tr}h  \nabla_{ L}\psi_-)  \|_{L^2(\Sigma_\tau)} \mathcal E^D_{\le N-1}[\psi]\,d\tau \\
        & \qquad + \int_{\tau_i}^{\tau_{i+1}} r^{\frac32-\delta} \mathcal E^W_{\le N-1}[h_{L\underline L}] \|\mathscr L_Z^{\le N-9}(\nabla_{\underline L} \mathrm{Tr}h  \slashed\nabla\psi_-)  \|_{L^2(\Sigma_\tau)} \mathcal E^D_{\le N-1}[\psi]\,d\tau.
    \end{aligned}
\end{align}
Then
\begin{align*}
    & \int_{\tau_i}^{\tau_{i+1}} r^{\frac32-\delta} \mathcal E^W_{\le N-1}[h_{L\underline L}] \|\mathscr L_Z^{\le N-9}(\nabla_L \mathrm{Tr}h  \nabla_{\underline L}\psi_+)  \|_{L^2(\Sigma_\tau)} \mathcal E^D_{\le N-1}[\psi]\,d\tau \\
    & \lesssim \int_{\tau_i}^{\tau_{i+1}} r^{1 -\delta} \mathcal E^W_{\le N-1}[h_{L\underline L}] \mathcal E^W_{\le N-9}[h_{L\underline L}] \mathcal E^D_{\le N-6}[\psi]  \mathcal E^D_{\le N-1}[\psi]\,d\tau \\
    & \lesssim \left( \int_{\tau_i}^{\tau_{i+1}} {}^{(0)}\mathcal E^D_{\le N-6}[\psi]^2(\tau) {}^{(1)}\mathcal E^W_{\le N-9}[h_{L\underline L}]^2(\tau) {}^{(0)}\mathcal E^W_{\le N-1}[h_{L\underline L}]^2(\tau)\,d\tau   \right)^\frac12 \left( \int_{\tau_i}^{\tau_{i+1}} {}^{(1-\delta)}\mathcal E^D_{\le N-1}[\psi]^2(\tau)\,d\tau \right)^\frac12,
\end{align*}
and
\begin{align*}
    & \int_{\tau_i}^{\tau_{i+1}} r^{\frac32-\delta} \mathcal E^W_{\le N-1}[h_{L\underline L}] \|\mathscr L_Z^{\le N-9}(\nabla_L \mathrm{Tr}h  \slashed\nabla\psi_+)  \|_{L^2(\Sigma_\tau)} \mathcal E^D_{\le N-1}[\psi]\,d\tau  \\
    & \lesssim \int_{\tau_i}^{\tau_{i+1}} r^{\frac12-\delta} \mathcal E^W_{\le N-1}[h_{L\underline L}] \|\mathscr L_Z^{\le N-9}(\nabla_L \mathrm{Tr}h  \mathscr L_{\Omega}^{\le1}\psi_+)  \|_{L^2(\Sigma_\tau)} \mathcal E^D_{\le N-1}[\psi]\,d\tau \\
    & \lesssim \int_{\tau_i}^{\tau_{i+1}} r^{1 -\delta} \mathcal E^W_{\le N-1}[h_{L\underline L}] \mathcal E^W_{\le N-9}[h_{L\underline L}] \mathcal E^D_{\le N-5}[\psi]  \mathcal E^D_{\le N-1}[\psi]\,d\tau \\
    & \lesssim  \left( \int_{\tau_i}^{\tau_{i+1}} {}^{(0)}\mathcal E^D_{\le N-5}[\psi]^2(\tau) {}^{(1)}\mathcal E^W_{\le N-9}[h_{L\underline L}]^2(\tau) {}^{(0)}\mathcal E^W_{\le N-1}[h_{L\underline L}]^2(\tau)\,d\tau   \right)^\frac12 \left( \int_{\tau_i}^{\tau_{i+1}} {}^{(1-\delta)}\mathcal E^D_{\le N-1}[\psi]^2(\tau)\,d\tau \right)^\frac12,
\end{align*}
The remaining integrals are estimated similarly. The term involving the additional nonlinearity $\Box_g\mathrm{Tr}h\cdot\psi$ is easier to control: replacing $\Box_g\mathrm{Tr}h$ by $\mathscr T_{L\underline L}[\psi]$ and $\mathscr T_{AB}[\psi]$ yields cubic terms. We omit the details.

To complete the estimate of the spacetime integral of the nonlinearity for $\nu=L$, it remains to consider:
\begin{align}
    \int_{\mathcal D^{\tau_{i+1}}_{\tau_i}} r^{2-\delta} \langle  \nabla_{L} \mathcal L_Z^{I_1}h_{L\underline L} \gamma^{e_A} \gamma^{L} \nabla_{e_A}  \mathscr L_Z^{I_2}\psi, \gamma^T \nabla_L\mathscr L_Z^{I}\psi\rangle \,dV,
\end{align}
This estimate is straightforward. For the high-low interaction, we absorb the weight $r$ into the lower-order term and apply the weighted Sobolev and Cauchy--Schwarz inequalities. For the low-high interaction, the weighted Sobolev inequality applied to the lower-order term gives the required decay. We omit the details.
\subsection{$\nu=\underline L$}
We now consider the case $\nu=\underline L$.
As in the previous case, it suffices to consider the high-low and low-high interactions:
\begin{align}
	& \int_{\mathcal D^{\tau_{i+1}}_{\tau_i}} r^{2-\delta} \nabla_{\underline L}\mathcal L_Z^{\le N-1}h_{L\underline L} \gamma^{L}\gamma^{\underline L} \nabla_L \mathscr L_Z^{\le N-11} \psi, \gamma^T \nabla_L\mathscr L_Z^{\le N-1}\psi\rangle \,dV , \\
    &  \int_{\mathcal D^{\tau_{i+1}}_{\tau_i}} r^{2-\delta} \nabla_{\underline L}\mathcal L_Z^{\le N-11}h_{L\underline L} \gamma^{L}\gamma^{\underline L} \nabla_L \mathscr L_Z^{\le N-1} \psi, \gamma^T \nabla_L\mathscr L_Z^{\le N-1}\psi\rangle \,dV .
\end{align}
We first consider the high-low interaction. Using the foliation by the ingoing cones $\underline{\mathcal C}_v\cap\mathcal D^{\tau_{i+1}}_{\tau_i}$ gives
\begin{align*}
	& \left| \int_{\mathcal D^{\tau_{i+1}}_{\tau_i}} r^{2-\delta} \nabla_{\underline L}\mathcal L_Z^{\le N-1}h_{L\underline L} \gamma^{L}\gamma^{\underline L} \nabla_L \mathscr L_Z^{\le N-11} \psi, \gamma^T \nabla_L\mathscr L_Z^{\le N-1}\psi\rangle \,dV  \right| \\
	& \lesssim \int_{\mathcal D^{\tau_{i+1}}_{\tau_i}} r^{2-\delta} |\nabla_{\underline L}\mathcal L_Z^{\le N-1}h_{L\underline L} | | \nabla_L \mathscr L_Z^{\le N-11} \psi_-| |\nabla_L\mathscr L_Z^{\le N-1}\psi_- |\,dV \\
	& \lesssim  \int  r^{2-\delta} \|\nabla_{\underline L}\mathcal L_Z^{\le N-1}h_{L\underline L} \|_{L^2(\underline{\mathcal C}_v)} \| \nabla_L \mathscr L_Z^{\le N-11} \psi_-\|_{L^\infty(\underline{\mathcal C}_v)} \|\nabla_L\mathscr L_Z^{\le N-1}\psi_- \|_{L^2(\underline{\mathcal C}_v)}\,dv \\
    & \lesssim \int r^{\frac32-\delta} \mathcal F^W_{\le N-1}[h_{L\underline L}](v,\tau_{i},\tau_{i+1}) \|\nabla_L \nabla_{\underline L}\mathscr L_Z^{\le N-9}\psi_-\|_{L^2(\underline{\mathcal C}_v)} \|\nabla_L\mathscr L_Z^{\le N-1}\psi_- \|_{L^2(\underline{\mathcal C}_v)}\,dv \\
    & \lesssim \int r^{\frac12-\delta} \mathcal F^W_{\le N-1}[h_{L\underline L}](v,\tau_{i},\tau_{i+1}) \|\nabla_L \mathscr L_Z^{\le N-8}\psi_+\|_{L^2(\underline{\mathcal C}_v)} \|\nabla_L\mathscr L_Z^{\le N-1}\psi_- \|_{L^2(\underline{\mathcal C}_v)}\,dv \\
    & \qquad +  \int r^{\frac32-\delta} \mathcal F^W_{\le N-1}[h_{L\underline L}](v,\tau_{i},\tau_{i+1}) \|\nabla_L \mathscr L_Z^{\le N-9}( \nabla_{\underline L} \mathrm{Tr}h \cdot \psi_- )\|_{L^2(\underline{\mathcal C}_v)} \|\nabla_L\mathscr L_Z^{\le N-1}\psi_- \|_{L^2(\underline{\mathcal C}_v)}\,dv \\
    & \qquad + \int r^{\frac32-\delta} \mathcal F^W_{\le N-1}[h_{L\underline L}](v,\tau_{i},\tau_{i+1}) \|\nabla_L \mathscr L_Z^{\le N-9}( \slashed\nabla \mathrm{Tr}h \cdot \psi_- )\|_{L^2(\underline{\mathcal C}_v)} \|\nabla_L\mathscr L_Z^{\le N-1}\psi_- \|_{L^2(\underline{\mathcal C}_v)}\,dv,
\end{align*}
where we used the null decomposition for the spinor fields. Now we have to treat the three integrals of the last inequality.
For the first integral, we see that
\begin{align*}
    & \int r^{\frac12-\delta} \mathcal F^W_{\le N-1}[h_{L\underline L}](v,\tau_{i},\tau_{i+1}) \|\nabla_L \mathscr L_Z^{\le N-8}\psi_+\|_{L^2(\underline{\mathcal C}_v)} \|\nabla_L\mathscr L_Z^{\le N-1}\psi_- \|_{L^2(\underline{\mathcal C}_v)}\,dv \\
    & \lesssim \sup_{v}{}^{(0)} \mathcal F^W_{\le N-1}[h_{L\underline L}](v,\tau_{i},\tau_{i+1}) \int r^{\frac12-\delta} \|\nabla_L \mathscr L_Z^{\le N-8}\psi_+\|_{L^2(\underline{\mathcal C}_v)} \|\nabla_L\mathscr L_Z^{\le N-1}\psi_- \|_{L^2(\underline{\mathcal C}_v)}\,dv \\
    & \lesssim \sup_{v}{}^{(0)}\mathcal F^W_{\le N-1}[h_{L\underline L}](v,\tau_{i},\tau_{i+1}) \left(\int \|\nabla_L \mathscr L_Z^{\le N-8}\psi_+\|_{L^2(\underline{\mathcal C}_v)}^2\,dv \right)^\frac12 \left( \int r^{1-2\delta} \|\nabla_L\mathscr L_Z^{\le N-1}\psi_- \|_{L^2(\underline{\mathcal C}_v)}^2\,dv\right)^\frac12 \\
    & \lesssim \sup_{v}{}^{(0)}\mathcal F^W_{\le N-1}[h_{L\underline L}](v,\tau_{i},\tau_{i+1}) \left(\int |\nabla_L \mathscr L_Z^{\le N-8}\psi_+|^2\,dV \right)^\frac12 \left( \int r^{1-2\delta} |\nabla_L\mathscr L_Z^{\le N-1}\psi_- |^2\,dV\right)^\frac12 \\
    & \lesssim \sup_{v}{}^{(0)}\mathcal F^W_{\le N-1}[h_{L\underline L}](v,\tau_{i},\tau_{i+1}) \left(\int_{\tau_i}^{\tau_{i+1}} {}^{(0)}\mathcal E^D_{\le N-8}[\psi]^2(\tau) \,d\tau  \right)^\frac12 \left( \int_{\tau_i}^{\tau_{i+1}} r^{-\delta} \,{}^{(1-\delta)} \mathcal E^D_{\le N-1}[\psi]^2(\tau) \,d\tau \right)^\frac12 ,
\end{align*}
We estimate the second and third integrals together as follows:
\begin{align*}
    &  \int r^{\frac32-\delta} {}^{(0)}\mathcal F^W_{\le N-1}[h_{L\underline L}](v,\tau_{i},\tau_{i+1}) \|\nabla_L \mathscr L_Z^{\le N-9}(\nabla h\cdot \psi)\|_{L^2(\underline{\mathcal C}_v)} \|\nabla_L\mathscr L_Z^{\le N-1}\psi_- \|_{L^2(\underline{\mathcal C}_v)}\,dv  \\
    & \lesssim \sup_{v}{}^{(0)}\mathcal F^W_{\le N-1}[h_{L\underline L}](v,\tau_i,\tau_{i+1}) \left( \int_{\mathcal D^{\tau_{i+1}}_{\tau_i}} r^{2-\delta} | \nabla_L \mathscr L_Z^{\le N-9}(\nabla h\cdot \psi)|^2\,dV \right)^\frac12 \left( \int_{\mathcal D^{\tau_{i+1}}_{\tau_i}} r^{1-\delta}|\nabla_L \mathscr L_Z^{\le N-1}\psi_-|^2\,dV \right)^\frac12 \\
    & \lesssim \sup_{v}\mathcal F^W_{\le N-1}[h_{L\underline L}](v,\tau_i,\tau_{i+1}) \left( \int_{\tau_i}^{\tau_{i+1}}  {}^{(1)}\mathcal E^W_{\le N-7}[h_{L\underline L}]^2(\tau) {}^{(1-\delta)}\mathcal E^D_{\le N-7}[\psi]^2(\tau)\,d\tau \right)^\frac12 \left( \int_{\tau_i}^{\tau_{i+1}} {}^{(1-\delta)}\mathcal E^D_{\le N-1}[\psi]^2(\tau)\,d\tau \right)^\frac12.
\end{align*}
For the low-high interaction, we apply H\"older's inequality and the weighted Sobolev embedding to get
\begin{align}
    \begin{aligned}
        & \left| \int_{\mathcal D^{\tau_{i+1}}_{\tau_i}} r^{2-\delta} \nabla_{\underline L}\mathcal L_Z^{\le N-11}h_{L\underline L} \gamma^{L}\gamma^{\underline L} \nabla_L \mathscr L_Z^{\le N-1} \psi, \gamma^T \nabla_L\mathscr L_Z^{\le N-1}\psi\rangle \,dV  \right| \\
        & \lesssim \int_{\tau_i}^{\tau_{i+1}}r^{\frac32-\delta} \|\nabla_L\nabla_{\underline L}\mathcal L_Z^{\le N-9}[h_{L\underline L}] \mathcal E^D_{\le N-1}[\psi]^2\,d\tau \\
        & \lesssim \int_{\tau_i}^{\tau_{i+1}}r^{\frac12-\delta} \mathcal E^W_{\le N-8}[h_{L\underline L}]  \mathcal E^D_{\le N-1}[\psi]^2\,d\tau  + \int_{\tau_i}^{\tau_{i+1}}r^{\frac12-\delta} \|\nabla_{\underline L}\mathcal L_Z^{\le N-9}h_{L\underline L}\|_{L^2(\Sigma_\tau)} \mathcal E^D_{\le N-1}[\psi]^2\,d\tau \\
        & \qquad + \int_{\tau_i}^{\tau_{i+1}}r^{\frac12-\delta} \|\nabla_L\mathcal L_Z^{\le N-9}h_{L\underline L}\|_{L^2(\Sigma_\tau)} \mathcal E^D_{\le N-1}[\psi]^2\,d\tau + \int_{\tau_i}^{\tau_{i+1}}r^{\frac32-\delta} \|\mathcal L_Z^{\le N-9}\mathscr T_{L\underline L}[\psi] \|_{L^2(\Sigma_\tau)}  \mathcal E^D_{\le N-1}[\psi]^2\,d\tau.
    \end{aligned}
\end{align}
It remains to estimate only the second integral, which contains $\nabla_{\underline L}\mathcal L_Z^{\le N-9}h_{L\underline L}$; the other integrals yield stronger decay.
For the second integral, we use the Hardy inequality and get
\begin{align*}
    & \int_{\tau_i}^{\tau_{i+1}}r^{\frac12-\delta} \|\nabla_{\underline L}\mathcal L_Z^{\le N-9}h_{L\underline L}\|_{L^2(\Sigma_\tau)} \mathcal E^D_{\le N-1}[\psi]^2\,d\tau  \\
    & \lesssim \int_{\tau_i}^{\tau_{i+1}}r^{-\delta}\, {}^{(1)} \mathcal E^W_{\le N-8}[h_{L\underline L}](\tau)  {}^{(1-\delta)} \mathcal E^D_{\le N-1}[\psi]^2(\tau) \,d\tau \\
    & \lesssim \sup_{\tau\in[\tau_i,\tau_{i+1}]} {}^{(1)} \mathcal E^W_{\le N-8}[h_{L\underline L}](\tau) \int_{\tau_i}^{\tau_{i+1}}r^{-\delta}\,  {}^{(1-\delta)} \mathcal E^D_{\le N-1}[\psi]^2(\tau) \,d\tau .
\end{align*}
This completes the estimate for $\nu=\underline L$. The case $\nu=A$ follows by the same argument, so we omit the details.
\section{Boundedness of the weighted energy for the tensor fields}\label{sec:bdd-tensor}
We now turn to the tensorial wave part:
\begin{align}
    \Box_g h_{\mu\nu} = \mathscr T_{\mu\nu}[\psi].
\end{align}
The null components of $h$ are as follows:
\begin{align}
    h_{LL}, \ h_{AB}, \ h_{L\underline L}, \ h_{AL}, \ h_{A\underline L},  \ h_{\underline L\underline L}.
\end{align}
In this section, we estimate the top-order weighted energies of all null components except $h_{\underline L\underline L}$. We consider the following spacetime integrals:
\begin{enumerate}
    \item Control of ${}^{(2-\delta)}\mathcal E^W_{\le N-1}[h_{LL}]$:
    \begin{align}
        \int_{\mathcal D^{\tau_{i+1}}_{\tau_i}} r^{2-\delta} \mathcal L_Z^{\le N-1}\mathscr T_{LL}[\psi] \nabla_L\mathcal L_Z^{\le N-1}h_{LL}\,dV.
    \end{align}
    \item Control of ${}^{(2-\delta)}\mathcal E^W_{\le N-1}[h_{AB}]$: 
    \begin{align}
        \int_{\mathcal D^{\tau_{i+1}}_{\tau_i}} r^{2-\delta} \mathcal L_Z^{\le N-1}\mathscr T_{AB}[\psi] \nabla_L\mathcal L_Z^{\le N-1}h_{AB}\,dV.
    \end{align}
    \item Control of ${}^{(2-\delta)}\mathcal E^W_{\le N-1}[h_{L\underline L}]$: 
    \begin{align}
        \int_{\mathcal D^{\tau_{i+1}}_{\tau_i}} r^{2-\delta} \mathcal L_Z^{\le N-1}\mathscr T_{L\underline L}[\psi] \nabla_L\mathcal L_Z^{\le N-1}h_{L\underline L}\,dV.
    \end{align}
    \item Control of ${}^{(2-\delta)}\mathcal E^W_{\le N-1}[h_{AL}]$: 
    \begin{align}
        \int_{\mathcal D^{\tau_{i+1}}_{\tau_i}} r^{2-\delta} \mathcal L_Z^{\le N-1}\mathscr T_{AL}[\psi] \nabla_L\mathcal L_Z^{\le N-1}h_{AL}\,dV.
    \end{align}
    \item Control of ${}^{(2-\delta)}\mathcal E^W_{\le N-1}[h_{A\underline L}]$: 
    \begin{align}
        \int_{\mathcal D^{\tau_{i+1}}_{\tau_i}} r^{2-\delta} \mathcal L_Z^{\le N-1}\mathscr T_{A\underline L}[\psi] \nabla_L\mathcal L_Z^{\le N-1}h_{A\underline L}\,dV.
    \end{align}
    \item Control of ${}^{(2-\delta)}\mathcal E^W_{\le N-3}[h_{L\underline L}]$: 
    \begin{align}
        \int_{\mathcal D^{\tau_{i+1}}_{\tau_i}} r^{2-\delta} \mathcal L_Z^{\le N-1}\mathscr T_{\underline L\underline L}[\psi] \nabla_L\mathcal L_Z^{\le N-1}h_{\underline L\underline L}\,dV.
    \end{align}
\end{enumerate}
\begin{prop}
    For the top-order, i.e., $k\le N-1$, we have
    \begin{align}
        \begin{aligned}
           & \left|  \int_{\mathcal D^{\tau_{i+1}}_{\tau_i}} r^{2-\delta} \mathcal L_Z^{\le k}\mathscr T_{LL}[\psi] \nabla_L\mathcal L_Z^{\le k}h_{LL}\,dV \right| \\
           & \lesssim \sup_{\tau\in[\tau_i,\tau_{i+1}] }{}^{(1)}\mathcal E^D_{\ll k}[\psi](\tau) \left( \int_{\tau_i}^{\tau_{i+1}}  {}^{(1-\delta)}\mathcal E^D_{\le k}[\psi]^2(\tau)\,d\tau \right)^\frac12 \left( \int_{\tau_i}^{\tau_{i+1}} {}^{(1-\delta)}\mathcal E^W_{\le k}[h_{LL}]^2(\tau)\,d\tau  \right)^\frac12,
        \end{aligned}
    \end{align}
    and
    \begin{align}
        \begin{aligned}
            & \left|  \int_{\mathcal D^{\tau_{i+1}}_{\tau_i}} r^{2-\delta} \mathcal L_Z^{\le k}\mathscr T_{AB}[\psi] \nabla_L\mathcal L_Z^{\le k}h_{AB}\,dV \right| \\
            & \lesssim \sup_{\tau\in[\tau_i,\tau_{i+1}]}{}^{(1)}\mathcal E^D_{\ll k}[\psi](\tau) \left( \int_{\tau_i}^{\tau_{i+1}}  {}^{(1-\delta)}\mathcal E^D_{\le k}[\psi]^2(\tau)\,d\tau \right)^\frac12 \left(  \int_{\tau_i}^{\tau_{i+1}}{}^{(1-\delta)}\mathcal E^W_{\le k}[h_{AB}]^2(\tau)\,d\tau \right)^\frac12,
        \end{aligned}
    \end{align}
    and
    \begin{align}
        \begin{aligned}
            & \left| \int_{\mathcal D^{\tau_{i+1}}_{\tau_i}} r^{2-\delta} \mathcal L_Z^{\le k}\mathscr T_{L\underline L}[\psi] \nabla_L\mathcal L_Z^{\le k}h_{L\underline L}\,dV \right| \\
            & \lesssim  \left( \int_{\tau_i}^{\tau_{i+1}}{}^{(1)}\mathcal E^D_{\ll k}[\psi](\tau) {}^{(1-\delta)}\mathcal E^D_{\le k}[\psi]^2(\tau)\,d\tau  \right)^\frac12 \left( \int_{\tau_i}^{\tau_{i+1}} {}^{(1-\delta)}\mathcal E^W_{\le k}[h_{L\underline L}]^2(\tau)\,d\tau  \right)^\frac12,
        \end{aligned}
    \end{align}
    and
    \begin{align}
        \begin{aligned}
             & \left| \int_{\mathcal D^{\tau_{i+1}}_{\tau_i}} r^{2-\delta} \mathcal L_Z^{\le k}\mathscr T_{AL}[\psi] \nabla_L\mathcal L_Z^{\le k}h_{AL}\,dV \right| \\
             & \lesssim \sup_{v: \tau\le\tau(v) }\mathscr F^D_{\ll k}[\psi](v,\tau_i,\tau_{i+1}) \left( \int_{\tau_i}^{\tau_{i+1}} {}^{(0)}\mathcal E^D_{\le k}[\psi]^2(\tau) \,d\tau \right)^\frac12 \left( \int_{\tau_i}^{\tau_{i+1}}  {}^{(1-\delta)}\mathcal E^W_{\le k}[h_{AL}]^2(\tau) \,d\tau \right)^\frac12,
        \end{aligned}
    \end{align}
    and
    \begin{align}
        \begin{aligned}
             & \left| \int_{\mathcal D^{\tau_{i+1}}_{\tau_i}} r^{2-\delta} \mathcal L_Z^{\le k}\mathscr T_{A\underline L}[\psi] \nabla_L\mathcal L_Z^{\le k}h_{A\underline L}\,dV \right| \\ 
             & \lesssim (\tau_{i+1})^\frac12 \left( \int_{\tau_i}^{\tau_{i+1}} \mathscr E^D_{\ll k}[\psi]^2(\tau) {}^{(1-\delta)}\mathcal E^D_{\le k-1}[\psi]^2(\tau)\,d\tau \right)^\frac12 \left( \int_{\tau_i}^{\tau_{i+1}} {}^{(1-\delta)}\mathcal E^W_{\le k-1}[h_{A\underline L}]^2(\tau)\,d\tau\right)^\frac12,
        \end{aligned}
    \end{align}
    For $h_{\underline L\underline L}$, if $k\le N-3$, we have
    \begin{align}
        \begin{aligned}
             & \left| \int_{\mathcal D^{\tau_{i+1}}_{\tau_i}} r^{2-\delta} \mathcal L_Z^{\le k}\mathscr T_{\underline L\underline L}[\psi] \nabla_L\mathcal L_Z^{\le k}h_{\underline L\underline L}\,dV \right| \\ 
             & \lesssim (\tau_{i+1})^\frac12 \left( \int_{\tau_i}^{\tau_{i+1}} \mathscr E^D_{\ll k}[\psi]^2(\tau) {}^{(1-\delta)}\mathcal E^D_{\le k-1}[\psi]^2(\tau)\,d\tau \right)^\frac12 \left( \int_{\tau_i}^{\tau_{i+1}} {}^{(1-\delta)}\mathcal E^W_{\le k-1}[h_{A\underline L}]^2(\tau)\,d\tau\right)^\frac12,
        \end{aligned}
    \end{align}
\end{prop}

\subsection{Control of $h_{LL}$}
We consider the integrals
\begin{align}
    \int_{\mathcal D^{\tau_{i+1}}_{\tau_i}} r^{2-\delta}\langle \nabla_L \mathscr L_Z^{\le N-11} \psi, \gamma^{\underline L}\mathscr L_Z^{\le N-1}\psi\rangle \nabla_L\mathcal L_Z^{\le N-1}h_{LL}\,dV 
\end{align}
and
\begin{align}
    \int_{\mathcal D^{\tau_{i+1}}_{\tau_i}} r^{2-\delta}\langle \mathscr L_Z^{\le N-11}\psi, \gamma^{\underline L}\nabla_L \mathscr L_Z^{\le N-1}\psi\rangle \nabla_L\mathcal L_Z^{\le N-1}h_{LL}\,dV .
\end{align}
Then we have
\begin{align*}
    & \left|  \int_{\mathcal D^{\tau_{i+1}}_{\tau_i}} r^{2-\delta}\langle \nabla_L \mathscr L_Z^{\le N-11}\psi, \gamma^{\underline L}\mathscr L_Z^{\le N-1}\psi\rangle \nabla_L\mathcal L_Z^{\le N-1}h_{LL}\,dV  \right| \\
    & \lesssim  \int_{\mathcal D^{\tau_{i+1}}_{\tau_i}} r^{2-\delta} | \nabla_L \mathscr L_Z^{\le N-11}\psi_-| |\mathscr L_Z^{\le N-1}\psi_-| | \nabla_L\mathcal L_Z^{\le N-1}h_{LL}|\,dV \\
    & \lesssim \int r^{2-\delta} \|\nabla_{\underline L}\nabla_L \mathscr L_Z^{\le N-9} \psi_-\|_{L^2(\underline{\mathcal C}_v)} \|\mathscr L_Z^{\le N-1}\psi_-\|_{L^2(\underline{\mathcal C}_v)} \|\nabla_L\mathcal L_Z^{\le N-1}h_{LL}\|_{L^2(\underline{\mathcal C}_v)} \,dv \\
    & \lesssim \int r^{\frac12-\delta} \|\nabla_L\mathscr L_Z^{\le N-8}\psi_+\|_{L^2(\underline{\mathcal C}_v)} \mathscr F^D_{\le N-1}[\psi] \|\nabla_L\mathcal L_Z^{\le N-1}h_{LL}\|_{L^2(\underline{\mathcal C}_v)} \,dv \\
    & \lesssim \sup_{v: \tau\le\tau(v) }\mathscr F^D_{\le N-1}[\psi](v,\tau_i,\tau_{i+1}) \left( \int \|\nabla_L\mathscr L_Z^{\le N-8}\psi_+\|_{L^2(\underline{\mathcal C}_v)}^2\,dv \right)^\frac12 \left( \int r^{1-2\delta} \|\nabla_L\mathcal L_Z^{\le N-1}h_{LL}\|_{L^2(\underline{\mathcal C}_v)}^2 \,dv \right)^\frac12 \\
    & \lesssim \sup_{v: \tau\le\tau(v) }\mathscr F^D_{\le N-1}[\psi](v,\tau_i,\tau_{i+1}) \left( \int_{\tau_i}^{\tau_{i+1}} {}^{(0)}\mathcal E^D_{\le N-8}[\psi]^2(\tau) \,d\tau \right)^\frac12 \left( \int_{\tau_i}^{\tau_{i+1}} r^{-\delta}\, {}^{(1-\delta)}\mathcal E^W_{\le N-1}[h_{LL}]^2(\tau)  \,d\tau \right)^\frac12,
\end{align*}
The analogous term involving the additional nonlinearity $\nabla\mathrm{Tr}h\cdot\nabla\psi$ is easier to estimate; we omit the details, and 
\begin{align*}
    & \left|  \int_{\mathcal D^{\tau_{i+1}}_{\tau_i}} r^{2-\delta}\langle \mathscr L_Z^{\le N-11} \psi, \gamma^{\underline L}\nabla_L \mathscr L_Z^{\le N-1}\psi\rangle \nabla_L\mathcal L_Z^{\le N-1}h_{LL}\,dV  \right| \\
    & \lesssim  \int_{\mathcal D^{\tau_{i+1}}_{\tau_i}} r^{2-\delta} |\mathscr L_Z^{\le N-11} \psi_-| |\nabla_L \mathscr L_Z^{\le N-1}\psi_-| | \nabla_L\mathcal L_Z^{\le N-1}h_{LL}| \,dV \\
    & \lesssim \int_{\tau_i}^{\tau_{i+1}} r^{\frac32-\delta} \mathcal E^D_{\le N-9}[\psi]\mathcal E^D_{\le N-1}[\psi]\mathcal E^W_{\le N-1}[h_{LL}]\,d\tau \\
    & \lesssim \left( \int_{\tau_i}^{\tau_{i+1}} {}^{(1)}\mathcal E^D_{\le N-9}[\psi]^2(\tau) {}^{(1-\delta)}\mathcal E^D_{\le N-1}[\psi]^2(\tau)\,d\tau \right)^\frac12 \left( \int_{\tau_i}^{\tau_{i+1}} {}^{(1-\delta)}\mathcal E^W_{\le N-1}[h_{LL}]^2(\tau)\,d\tau  \right)^\frac12 \\
    & \lesssim \sup_{\tau\in[\tau_i,\tau_{i+1}] }{}^{(1)}\mathcal E^D_{\le N-9}[\psi]^2(\tau) \left( \int_{\tau_i}^{\tau_{i+1}}  {}^{(1-\delta)}\mathcal E^D_{\le N-1}[\psi]^2(\tau)\,d\tau \right)^\frac12 \left( \int_{\tau_i}^{\tau_{i+1}} {}^{(1-\delta)}\mathcal E^W_{\le N-1}[h_{LL}]^2(\tau)\,d\tau  \right)^\frac12.
\end{align*}
\subsection{Control of $h_{AB}$}
We consider the integral of the form
\begin{align}
    \int_{\mathcal D^{\tau_{i+1}}_{\tau_i}}r^{2-\delta} \langle \mathscr L_Z^{\le N-11} \psi, \gamma^{e_A}\nabla_{e_B}\mathscr L_Z^{\le N-1}\psi\rangle \nabla_L\mathcal L_Z^{\le N-1}h_{AB}\,dV.
\end{align}
The preceding integral is straightforward to estimate. In particular,
\begin{align*}
    & \left|   \int_{\mathcal D^{\tau_{i+1}}_{\tau_i}}r^{2-\delta} \langle \mathscr L_Z^{\le N-11} \psi, \gamma^{e_A}\nabla_{e_B}\mathscr L_Z^{\le N-1}\psi\rangle \nabla_L\mathcal L_Z^{\le N-1}h_{AB}\,dV \right| \\
    & \lesssim   \int_{\mathcal D^{\tau_{i+1}}_{\tau_i}}r^{2-\delta} | \mathscr L_Z^{\le N-11}\psi_\pm|  |\slashed \nabla\mathscr L_Z^{\le N-1}\psi_\mp| | \nabla_L\mathcal L_Z^{\le N-1}h_{AB}|\,dV \\
    & \lesssim \int_{\tau_i}^{\tau_{i+1}} r^{\frac32-\delta} \mathcal E^D_{\le N-9}[\psi] \mathcal E^D_{\le N-1}[\psi]\mathcal E^W_{\le N-1} [h_{AB}]\,d\tau \\
    & \lesssim \left( \int_{\tau_i}^{\tau_{i+1}} {}^{(1)}\mathcal E^D_{\le N-9}[\psi]^2(\tau) {}^{(1-\delta)}\mathcal E^D_{\le N-1}[\psi]^2(\tau)\,d\tau \right)^\frac12 \left(  \int_{\tau_i}^{\tau_{i+1}}{}^{(1-\delta)}\mathcal E^W_{\le N-1}[h_{AB}]^2(\tau)\,d\tau \right)^\frac12 \\
    & \lesssim \sup_{\tau\in[\tau_i,\tau_{i+1}]}{}^{(1)}\mathcal E^D_{\le N-9}[\psi](\tau) \left( \int_{\tau_i}^{\tau_{i+1}}  {}^{(1-\delta)}\mathcal E^D_{\le N-1}[\psi]^2(\tau)\,d\tau \right)^\frac12 \left(  \int_{\tau_i}^{\tau_{i+1}}{}^{(1-\delta)}\mathcal E^W_{\le N-1}[h_{AB}]^2(\tau)\,d\tau \right)^\frac12.
\end{align*}
\subsection{Control of $h_{L\underline L}$}
To estimate the nonlinearity $\mathcal L_Z^{\le N-1}\mathscr T_{L\underline L}[\psi]$, it suffices to control the following integrals:
\begin{align}
    \begin{aligned}
    &  \int_{\mathcal D^{\tau_{i+1}}_{\tau_i}} r^{2-\delta} \langle \mathscr L_Z^{\le N-11}\psi, \gamma^{ L}\nabla_{ L}\mathscr L_Z^{\le N-1}\psi\rangle \nabla_L \mathcal L_Z^{\le N-1}h_{L\underline L}\,dV, \\
        &  \int_{\mathcal D^{\tau_{i+1}}_{\tau_i}} r^{2-\delta} \langle \mathscr L_Z^{\le N-1}\psi, \gamma^{ L}\nabla_{ L}\mathscr L_Z^{\le N-11}\psi\rangle \nabla_L \mathcal L_Z^{\le N-1}h_{L\underline L}\,dV, \\
        & \int_{\mathcal D^{\tau_{i+1}}_{\tau_i}} r^{2-\delta} \langle \mathscr L_Z^{\le N-11}\psi, \gamma^{\underline L}\nabla_{\underline L}\mathscr L_Z^{\le N-1}\psi\rangle \nabla_L \mathcal L_Z^{\le N-1}h_{L\underline L}\,dV, \\
        & \int_{\mathcal D^{\tau_{i+1}}_{\tau_i}} r^{2-\delta} \langle \mathscr L_Z^{\le N-1}\psi, \gamma^{\underline L}\nabla_{\underline L}\mathscr L_Z^{\le N-11}\psi\rangle \nabla_L \mathcal L_Z^{\le N-1}h_{L\underline L}\,dV.
    \end{aligned}
\end{align}
For the first integral, the Sobolev embedding gives
\begin{align*}
    & \left| \int_{\mathcal D^{\tau_{i+1}}_{\tau_i}} r^{2-\delta} \langle \mathscr L_Z^{\le N-11}\psi, \gamma^{ L}\nabla_{ L}\mathscr L_Z^{\le N-1}\psi\rangle \nabla_L \mathcal L_Z^{\le N-1}h_{L\underline L}\,dV \right| \\
    & \lesssim \int_{\tau_i}^{\tau_{i+1}} r^{\frac32-\delta} \mathcal E^D_{\le N-9}[\psi]\mathcal E^D_{\le N-1}[\psi]\mathcal E^W_{\le N-1}[h_{L\underline L}]\,d\tau \\
    & \lesssim \left( \int_{\tau_i}^{\tau_{i+1}}{}^{(1)}\mathcal E^D_{\le N-9}[\psi]^2(\tau) {}^{(1-\delta)}\mathcal E^D_{\le N-1}[\psi]^2(\tau)\,d\tau  \right)^\frac12 \left( \int_{\tau_i}^{\tau_{i+1}} {}^{(1-\delta)}\mathcal E^W_{\le N-1}[h_{L\underline L}]^2(\tau)\,d\tau  \right)^\frac12.
\end{align*}

For the second integral, we write
\begin{align*}
    & \left| \int_{\mathcal D^{\tau_{i+1}}_{\tau_i}} r^{2-\delta} \langle \mathscr L_Z^{\le N-1}\psi, \gamma^{ L}\nabla_{ L}\mathscr L_Z^{\le N-11}\psi\rangle \nabla_L \mathcal L_Z^{\le N-1}h_{L\underline L}\,dV \right| \\
    & \lesssim \int_{\mathcal D^{\tau_{i+1}}_{\tau_i}} r^{2-\delta} | \mathscr L_Z^{\le N-1}\psi_+| |\nabla_{ L}\mathscr L_Z^{\le N-11}\psi_+| | \nabla_L \mathcal L_Z^{\le N-1}h_{L\underline L}|\,dV \\
    & \lesssim \int_{\mathcal D^{\tau_{i+1}}_{\tau_i}} r^{2-\delta} | \mathscr L_Z^{\le N-1}\psi_+| |\slashed\nabla\mathscr L_Z^{\le N-11}\psi_-| | \nabla_L \mathcal L_Z^{\le N-1}h_{L\underline L}|\,dV \\
    & \qquad + \int_{\mathcal D^{\tau_{i+1}}_{\tau_i}} r^{2-\delta} | \mathscr L_Z^{\le N-1}\psi_+| |\mathscr L_Z^{\le N-11}(\nabla_L \mathrm{Tr}h \cdot \psi_+)| | \nabla_L \mathcal L_Z^{\le N-1}h_{L\underline L}|\,dV \\
    & \qquad +  \int_{\mathcal D^{\tau_{i+1}}_{\tau_i}} r^{2-\delta} | \mathscr L_Z^{\le N-1}\psi_+| |\mathscr L_Z^{\le N-11}(\slashed\nabla \mathrm{Tr}h\cdot \psi_+)| | \nabla_L \mathcal L_Z^{\le N-1}h_{L\underline L}|\,dV,
\end{align*}
and
\begin{align*}
    &  \int_{\mathcal D^{\tau_{i+1}}_{\tau_i}} r^{2-\delta} | \mathscr L_Z^{\le N-1}\psi_+| |\slashed\nabla\mathscr L_Z^{\le N-11}\psi_-| | \nabla_L \mathcal L_Z^{\le N-1}h_{L\underline L}|\,dV  \\
    & \lesssim  \int_{\mathcal D^{\tau_{i+1}}_{\tau_i}} r^{1-\delta} |\mathscr L_Z^{\le N-1} \psi_+| |\mathscr L_Z^{\le N-10}\psi_-| | \nabla_L \mathcal L_Z^{\le N-1}h_{L\underline L}|\,dV \\
    & \lesssim \int_{\tau_i}^{\tau_{i+1}} r^{\frac12-\delta} \mathscr E^D_{\le N-1}[\psi] \mathcal E^D_{\le N-8}[\psi]\mathcal E^W_{\le N-1}[h_{L\underline L}]   \,d\tau \\
& \lesssim \sup_{\tau\in[\tau_i,\tau_{i+1}]} \left( \int_{\tau_i}^{\tau_{i+1}} {}^{(0)}\mathcal E^D_{\le N-8}[\psi]^2(\tau)  \right)^\frac12 \left( \int_{\tau_i}^{\tau_{i+1}} {}^{(1-\delta)}\mathcal E^W_{\le N-1}[h_{L\underline L}]^2(\tau)  \right)^\frac12.
\end{align*}
For the integrals containing the additional nonlinearities $\nabla \mathrm{Tr}h\cdot\psi$, we consider only the low--high and high--low interactions. We write
\begin{align*}
    & \int_{\mathcal D^{\tau_{i+1}}_{\tau_i}} r^{2-\delta} | \mathscr L_Z^{\le N-1} \psi_+| |\mathscr L_Z^{\le N-11}(\nabla_L \mathrm{Tr}h\cdot \psi_+)| | \nabla_L \mathcal L_Z^{\le N-1}h_{L\underline L}|\,dV  \\
    & \lesssim \int_{\tau_i}^{\tau_{i+1}}r^{1-\delta} \mathscr E^D_{\le N-1}[\psi] \mathcal E^W_{\le N-9}[h_{L\underline L}]\mathcal E^D_{\le N-9}[\psi] \mathcal E^W_{\le N-1}[h_{L\underline L}]\,d\tau \\
    & \lesssim \sup_{\tau\in[\tau_i,\tau_{i+1}]} \mathscr E^D_{\le N-1}[\psi]  \left( \int_{\tau_i}^{\tau_{i+1}} {}^{(0)}\mathcal E^W_{\le N-9}[h_{L\underline L}]^2(\tau){}^{(1)}\mathcal E^D_{\le N-9}[\psi]^2(\tau)\,d\tau \right)^\frac12  \left( \int_{\tau_i}^{\tau_{i+1}} {}^{(1-\delta)}\mathcal E^W_{\le N-1}[h_{L\underline L}]^2(\tau)\,d\tau \right)^\frac12. 
\end{align*}
The integral involving $\slashed\nabla\mathrm{Tr}h\cdot\psi$ is easier to estimate; we omit the details.
We also have to deal with the integral
\begin{align}
     \int_{\mathcal D^{\tau_{i+1}}_{\tau_i}} r^{2-\delta} \langle \mathscr L_Z^{\le N-11} \psi, \gamma^{\underline L}\nabla_{\underline L}\mathscr L_Z^{\le N-1}\psi\rangle \nabla_L \mathcal L_Z^{\le N-1}h_{L\underline L}\,dV.
\end{align}
To do this, we write
\begin{align*}
    & \left|  \int_{\mathcal D^{\tau_{i+1}}_{\tau_i}} r^{2-\delta} \langle \mathscr L_Z^{\le N-11} \psi, \gamma^{\underline L}\nabla_{\underline L}\mathscr L_Z^{\le N-1}\psi\rangle \nabla_L \mathcal L_Z^{\le N-1}h_{L\underline L}\,dV \right| \\
    & \lesssim  \int_{\mathcal D^{\tau_{i+1}}_{\tau_i}} r^{2-\delta}  |\mathscr L_Z^{\le N-11}\psi_-| |\nabla_{\underline L}\mathscr L_Z^{\le N-1}\psi_-| | \nabla_L \mathcal L_Z^{\le N-1}h_{L\underline L}|\,dV \\
    & \lesssim  \int_{\mathcal D^{\tau_{i+1}}_{\tau_i}} r^{2-\delta}  |\mathscr L_Z^{\le N-11}\psi_-| |\slashed\nabla\mathscr L_Z^{\le N-1}\psi_+| | \nabla_L \mathcal L_Z^{\le N-1}h_{L\underline L}|\,dV \\
    & \qquad +  \int_{\mathcal D^{\tau_{i+1}}_{\tau_i}} r^{2-\delta}  |\mathscr L_Z^{\le N-11}\psi_-| |\mathscr L_Z^{\le N-1}( \nabla_{\underline L}\mathrm{Tr}h \cdot \psi_-)| | \nabla_L \mathcal L_Z^{\le N-1}h_{L\underline L}|\,dV \\
    & \qquad + \int_{\mathcal D^{\tau_{i+1}}_{\tau_i}} r^{2-\delta}  |\mathscr L_Z^{\le N-11}\psi_-| |\mathscr L_Z^{\le N-1}( \slashed\nabla\mathrm{Tr}h \cdot \psi_-)| | \nabla_L \mathcal L_Z^{\le N-1}h_{L\underline L}|\,dV .
\end{align*}
For the first integral, we write
\begin{align*}
    & \int_{\mathcal D^{\tau_{i+1}}_{\tau_i}} r^{2-\delta}  |\mathscr L_Z^{\le N-11}\psi_-| |\slashed\nabla\mathscr L_Z^{\le N-1}\psi_+| | \nabla_L \mathcal L_Z^{\le N-1}h_{L\underline L}|\,dV \\
    & \lesssim \int_{\mathcal D^{\tau_{i+1}}_{\tau_i}} r^{1-\delta}  |\mathscr L_Z^{\le N-11}\psi_-| |\mathscr L_Z^{\le N}\psi_+| | \nabla_L \mathcal L_Z^{\le N-1}h_{L\underline L}|\,dV \\
    & \lesssim \int_{\tau_i}^{\tau_{i+1}}r^{\frac12-\delta} \mathcal E^D_{\le N-9}[\psi]\mathscr E^D_{\le N}[\psi]\mathcal E^W_{\le N-1}[h_{L\underline L}]\,d\tau \\
    & \lesssim \sup_{\tau\in[\tau_i,\tau_{i+1}]} \mathscr E^D_{\le N}[\psi](\tau) \left( \int_{\tau_i}^{\tau_{i+1}} {}^{(0)}\mathcal E^D_{\le N-9}[\psi]^2(\tau)\,d\tau \right)^\frac12 \left( \int_{\tau_i}^{\tau_{i+1}} {}^{(1-\delta)}\mathcal E^W_{\le N-1}[h_{L\underline L}]^2(\tau)\,d\tau \right)^\frac12 .
\end{align*}
For the second integral, we consider the low-high and high-low interactions as follows:
\begin{align*}
    & \int_{\mathcal D^{\tau_{i+1}}_{\tau_i}} r^{2-\delta}  |\mathscr L_Z^{\le N-11}\psi_-| |\mathscr L_Z^{\le N-1}( \nabla_{\underline L}\mathrm{Tr}h \cdot \psi_-)| | \nabla_L \mathcal L_Z^{\le N-1}h_{L\underline L}|\,dV \\
    & \lesssim \int_{\mathcal D^{\tau_{i+1}}_{\tau_i}} r^{2-\delta}  |\mathscr L_Z^{\le N-11}\psi_-| | \nabla_{\underline L}\mathcal L_Z^{\le N-1}\mathrm{Tr}h| | \mathscr L_Z^{\le N-11}  \psi_-| | \nabla_L \mathcal L_Z^{\le N-1}h_{L\underline L}|\,dV \\
    & \qquad + \int_{\mathcal D^{\tau_{i+1}}_{\tau_i}} r^{2-\delta}  |\mathscr L_Z^{\le N-11}\psi_-| | \nabla_{\underline L}\mathcal L_Z^{\le N-11}\mathrm{Tr}h| | \mathscr L_Z^{\le N-1}  \psi_-| | \nabla_L \mathcal L_Z^{\le N-1}h_{L\underline L}|\,dV
\end{align*}
For the high-low interaction, we see that
\begin{align*}
    & \int_{\mathcal D^{\tau_{i+1}}_{\tau_i}} r^{2-\delta}  |\mathscr L_Z^{\le N-11}\psi_-| | \nabla_{\underline L}\mathcal L_Z^{\le N-1}\mathrm{Tr}h| | \mathscr L_Z^{\le N-11}  \psi_-| | \nabla_L \mathcal L_Z^{\le N-1}h_{L\underline L}|\,dV \\
    & \lesssim \int r^{1-\delta} \|\nabla_{\underline L}\mathscr L_Z^{\le N-9}\psi_-\|_{L^2(\underline{\mathcal C_v})}^2 \|\nabla_{\underline L}\mathcal L_Z^{\le N-1}\mathrm{Tr}h\|_{L^2(\underline{\mathcal C_v})} \|\nabla_L\mathcal L_Z^{\le N-1}h_{L\underline L}\|_{L^2(\underline{\mathcal C_v})}\,dv \\
    & \lesssim \sup_{v:\tau\le\tau(v)}{}^{(0)} \mathcal F^W_{\le N-1}[h_{L\underline L}](v,\tau_i,\tau_{i+1}){}^{(0)} \mathcal F^D_{\le N-8}[\psi](v,\tau_i,\tau_{i+1}) \\
    & \times\left( \int_{\mathcal D^{\tau_{i+1}}_{\tau_i}} r^{1-\delta} |\nabla_{\underline L}\mathscr L_Z^{\le N-9}\psi_-|^2 \,dV\right)^\frac12 \left( \int_{\mathcal D^{\tau_{i+1}}_{\tau_i}} r^{1-\delta} |\nabla_L\mathcal L_Z^{\le N-1}h_{L\underline L}|^2\,dV  \right)^\frac12 .
\end{align*}
Then we have
\begin{align*}
    &  \left( \int_{\mathcal D^{\tau_{i+1}}_{\tau_i}} r^{1-\delta} |\nabla_{\underline L}\mathscr L_Z^{\le N-9}\psi_-|^2 \,dV\right)^\frac12 \left( \int_{\mathcal D^{\tau_{i+1}}_{\tau_i}} r^{1-\delta} |\nabla_L\mathcal L_Z^{\le N-1}h_{L\underline L}|^2\,dV  \right)^\frac12 \\
    & \lesssim  \left( \int_{\tau_i}^{\tau_{i+1}} r^{1-\delta} \|\slashed\nabla\mathscr L_Z^{\le N-9}\psi_+\|_{L^2(\Sigma_\tau)}^2 \,d\tau \right)^\frac12 \left( \int_{\tau_i}^{\tau_{i+1}} {}^{(1-\delta)}\mathcal E^W_{\le N-1}[h_{L\underline L}]^2(\tau)\,d\tau   \right)^\frac12 \\
    & \qquad +  \left( \int_{\tau_i}^{\tau_{i+1}} r^{1-\delta}\| \mathscr L_Z^{\le N-9}(\nabla_{\underline L}\mathrm{Tr}h \cdot\psi_- )\|_{L^2(\Sigma_\tau)}^2 \,d\tau \right)^\frac12 \left( \int_{\tau_i}^{\tau_{i+1}} {}^{(1-\delta)}\mathcal E^W_{\le N-1}[h_{L\underline L}]^2(\tau)\,d\tau   \right)^\frac12 \\
    & \qquad +  \left( \int_{\tau_i}^{\tau_{i+1}} r^{1-\delta} \| \mathscr L_Z^{\le N-9}(\slashed\nabla\mathrm{Tr}h \cdot\psi_- )\|_{L^2(\Sigma_\tau)}^2 \,d\tau \right)^\frac12 \left( \int_{\tau_i}^{\tau_{i+1}} {}^{(1-\delta)}\mathcal E^W_{\le N-1}[h_{L\underline L}]^2(\tau)\,d\tau   \right)^\frac12 \\
    & \lesssim  \left( \int_{\tau_i}^{\tau_{i+1}} {}^{(1-\delta)}\mathcal E^D_{\le N-9}[\psi]^2(\tau)  \,d\tau \right)^\frac12 \left( \int_{\tau_i}^{\tau_{i+1}} {}^{(1-\delta)}\mathcal E^W_{\le N-1}[h_{L\underline L}]^2(\tau)\,d\tau   \right)^\frac12 \\
    & \qquad + \sup_{\tau\in[\tau_{i},\tau_{i+1}]}{}^{(1)}\mathcal E^D_{\le N-7}[\psi](\tau)  \left( \int_{\tau_i}^{\tau_{i+1}} {}^{(0)}\mathcal E^W_{\le N-7}[h]^2(\tau)  \,d\tau \right)^\frac12 \left( \int_{\tau_i}^{\tau_{i+1}} {}^{(1-\delta)}\mathcal E^W_{\le N-1}[h_{L\underline L}]^2(\tau)\,d\tau   \right)^\frac12.
\end{align*}
For the low-high interaction, we have
\begin{align*}
    & \int_{\mathcal D^{\tau_{i+1}}_{\tau_i}} r^{2-\delta}  |\mathscr L_Z^{\le N-11}\psi_-| | \nabla_{\underline L}\mathcal L_Z^{\le N-11}\mathrm{Tr}h| | \mathscr L_Z^{\le N-1}  \psi_-| | \nabla_L \mathcal L_Z^{\le N-1}h_{L\underline L}|\,dV \\
    & \lesssim \int r^{2-\delta} \|\mathscr L_Z^{\le N-11}\psi_-\|_{L^\infty(\underline{\mathcal C}_v)} \|\nabla_{\underline L}\mathscr L_Z^{\le N-11}\mathrm{Tr}h\|_{L^\infty(\underline{\mathcal C}_v)} \|\mathscr L_Z^{\le N-1}\psi_-\|_{L^2(\underline{\mathcal C}_v)}  \|\nabla_L\mathcal L_Z^{\le N-1}h_{L\underline L}\|_{L^2(\underline{\mathcal C}_v)} \,dv \\
    & \lesssim \sup_{v:\tau\le\tau(v)}{}^{(0)} \mathcal F^W_{\le N-8}[h](v,\tau_i,\tau_{i+1})\mathscr F^D_{\le N-1}[\psi](v,\tau_i,\tau_{i+1}) \int r^{1-\delta} \|\nabla_{\underline L}\mathscr L_Z^{\le N-9}\psi_-\|_{L^2(\underline{\mathcal C}_v)}   \|\nabla_L\mathcal L_Z^{\le N-1}h_{L\underline L}\|_{L^2(\underline{\mathcal C}_v)} \,dv,
\end{align*}
and the remaining estimate is analogous to the high-low interaction. It remains to estimate the integral
\begin{align}
     \int_{\mathcal D^{\tau_{i+1}}_{\tau_i}} r^{2-\delta} \langle \mathscr L_Z^{\le N-1} \psi, \gamma^{\underline L}\nabla_{\underline L}\mathscr L_Z^{\le N-11}\psi\rangle \nabla_L \mathcal L_Z^{\le N-1}h_{L\underline L}\,dV.
\end{align}
We see that
\begin{align*}
    & \left|  \int_{\mathcal D^{\tau_{i+1}}_{\tau_i}} r^{2-\delta} \langle \mathscr L_Z^{\le N-1} \psi, \gamma^{\underline L}\nabla_{\underline L}\mathscr L_Z^{\le N-11}\psi\rangle \nabla_L \mathcal L_Z^{\le N-1}h_{L\underline L}\,dV \right| \\
    & \lesssim  \int_{\mathcal D^{\tau_{i+1}}_{\tau_i}} r^{2-\delta} |\slashed\nabla \mathscr L_Z^{\le N-11} \psi_+|  |\mathscr L_Z^{\le N-1}\psi_-| | \nabla_L \mathcal L_Z^{\le N-1}h_{L\underline L}|\,dV \\
    & \qquad +  \int_{\mathcal D^{\tau_{i+1}}_{\tau_i}} r^{2-\delta} | \mathscr L_Z^{\le N-11} ( \nabla_{\underline L}\mathrm{Tr}h\cdot \psi_-)|  |\mathscr L_Z^{\le N-1}\psi_-| | \nabla_L \mathcal L_Z^{\le N-1}h_{L\underline L}|\,dV \\
    & \qquad +  \int_{\mathcal D^{\tau_{i+1}}_{\tau_i}} r^{2-\delta} | \mathscr L_Z^{\le N-11} ( \slashed\nabla\mathrm{Tr}h\cdot \psi_-)|  |\mathscr L_Z^{\le N-1}\psi_-| | \nabla_L \mathcal L_Z^{\le N-1}h_{L\underline L}|\,dV .
\end{align*}
Then we use the Hardy inequality to get
\begin{align*}
    &  \int_{\mathcal D^{\tau_{i+1}}_{\tau_i}} r^{2-\delta} |\slashed\nabla \mathscr L_Z^{\le N-11} \psi_+|  |\mathscr L_Z^{\le N-1}\psi_-| | \nabla_L \mathcal L_Z^{\le N-1}h_{L\underline L}|\,dV \\
    & \lesssim  \int_{\mathcal D^{\tau_{i+1}}_{\tau_i}} r^{1-\delta} |\slashed\nabla \mathscr L_Z^{\le N-10} \psi_+|  |\mathscr L_Z^{\le N-1}\psi_-| | \nabla_L \mathcal L_Z^{\le N-1}h_{L\underline L}|\,dV \\
    & \lesssim \int_{\tau_i}^{\tau_{i+1}} r^{\frac32-\delta} \mathcal E^D_{\le N-8}[\psi] \mathcal E^D_{\le N-1}[\psi]\mathcal E^W_{\le N-1}[h_{L\underline L}]\,d\tau \\
    & \lesssim \left( \int_{\tau_i}^{\tau_{i+1}} {}^{(1)}\mathcal E^D_{\le N-8}[\psi]^2(\tau){}^{(1-\delta)} \mathcal E^D_{\le N-1}[\psi]^2(\tau)\,d\tau \right)^\frac12 \left( \int_{\tau_i}^{\tau_{i+1}} {}^{(1-\delta)} \mathcal E^W_{\le N-1}[h_{L\underline L}]^2(\tau)\,d\tau \right)^\frac12,
\end{align*}
and
\begin{align*}
    & \int_{\mathcal D^{\tau_{i+1}}_{\tau_i}} r^{2-\delta} | \mathscr L_Z^{\le N-11} ( \nabla_{\underline L}\mathrm{Tr}h\cdot \psi_-)|  |\mathscr L_Z^{\le N-1}\psi_-| | \nabla_L \mathcal L_Z^{\le N-1}h_{L\underline L}|\,dV \\
    & \lesssim \int_{\tau_i}^{\tau_{i+1}} r^{1-\delta} \mathcal E^W_{\le N-8}[h_{L\underline L}]\mathcal E^D_{\le N-9}[\psi] \|\mathscr L_Z^{\le N-1}\psi_-\|_{L^2(\Sigma_\tau)}\mathcal E^W_{\le N-1}[h_{L\underline L}]\,d\tau \\
    & \lesssim \int_{\tau_i}^{\tau_{i+1}} r^{2-\delta} \mathcal E^W_{\le N-8}[h_{L\underline L}]\mathcal E^D_{\le N-9}[\psi] \mathcal E^D_{\le N-1}[\psi] \mathcal E^W_{\le N-1}[h_{L\underline L}]\,d\tau \\
    & \lesssim \left( \int_{\tau_i}^{\tau_{i+1}} {}^{(1)}\mathcal E^D_{\le N-8}[\psi]^2(\tau){}^{(1)}\mathcal E^W_{\le N-8}[h_{L\underline L}]^2(\tau){}^{(1-\delta)} \mathcal E^D_{\le N-1}[\psi]^2(\tau)\,d\tau \right)^\frac12 \left( \int_{\tau_i}^{\tau_{i+1}} {}^{(1-\delta)} \mathcal E^W_{\le N-1}[h_{L\underline L}]^2(\tau)\,d\tau \right)^\frac12.
\end{align*}
The spacetime integral involving $\slashed\nabla\mathrm{Tr}h\cdot\psi_-$ is estimated similarly; we omit the details.
\subsection{Control of $h_{AL}$}
We consider the integral
\begin{align}
    \int_{\mathcal D^{\tau_{i+1}}_{\tau_i}} r^{2-\delta} \mathcal L_Z^{\le N-1} \langle \psi, \gamma^{e_A}\nabla_L  \psi\rangle \nabla_L \mathcal L_Z^{\le N-1} h_{AL}\,dV +  \int_{\mathcal D^{\tau_{i+1}}_{\tau_i}} r^{2-\delta} \mathcal L_Z^{\le N-1} \langle \psi, \gamma^{\underline L}\nabla_{e_A}  \psi\rangle \nabla_L \mathcal L_Z^{\le N-1} h_{AL}\,dV 
\end{align}
The second integral is easier than the first integral. Indeed, we have
\begin{align*}
    & \left| \int_{\mathcal D^{\tau_{i+1}}_{\tau_i}} r^{2-\delta}  \langle \mathscr L_Z^{\le N-11} \psi, \gamma^{\underline L}\nabla_{e_A} \mathscr L_Z^{\le N-1}  \psi\rangle \nabla_L \mathcal L_Z^{\le N-1} h_{AL}\,dV \right| \\
     & \lesssim \int_{\tau_i}^{\tau_{i+1}} r^{\frac32-\delta} \mathcal E^D_{\le N-9}[\psi]\mathcal E^D_{\le N-1}[\psi]\mathcal E^W_{\le N-1}[h_{AL}]\,d\tau \\
    & \lesssim \left( \int_{\tau_i}^{\tau_{i+1}}{}^{(1)}\mathcal E^D_{\le N-9}[\psi]^2(\tau) {}^{(1-\delta)}\mathcal E^D_{\le N-1}[\psi]^2(\tau)\,d\tau  \right)^\frac12 \left( \int_{\tau_i}^{\tau_{i+1}} {}^{(1-\delta)}\mathcal E^W_{\le N-1}[h_{AL}]^2(\tau)\,d\tau  \right)^\frac12,
\end{align*}
and the Hardy inequality yields
\begin{align*}
    & \left| \int_{\mathcal D^{\tau_{i+1}}_{\tau_i}} r^{2-\delta}  \langle \mathscr L_Z^{\le N-1} \psi, \gamma^{\underline L}\nabla_{e_A} \mathscr L_Z^{\le N-11}  \psi\rangle \nabla_L \mathcal L_Z^{\le N-1} h_{AL}\,dV \right| \\
    & \lesssim \int_{\mathcal D^{\tau_{i+1}}_{\tau_i}} r^{1-\delta} | \mathscr L_Z^{\le N-1} \psi_-| | \mathscr L_Z^{\le N-10}  \psi| |\nabla_L \mathcal L_Z^{\le N-1} h_{AL}|\,dV \\
    & \lesssim \int_{\tau_i}^{\tau_{i+1}} r^{\frac32-\delta} \mathcal E^D_{\le N-1}[\psi]\mathcal E^D_{\le N-8}[\psi]\mathcal E^W_{\le N-1}[h_{AL}]\,d\tau \\
    & \lesssim \left( \int_{\tau_i}^{\tau_{i+1}}{}^{(1)}\mathcal E^D_{\le N-8}[\psi]^2(\tau) {}^{(1-\delta)}\mathcal E^D_{\le N-1}[\psi]^2(\tau)\,d\tau  \right)^\frac12 \left( \int_{\tau_i}^{\tau_{i+1}} {}^{(1-\delta)}\mathcal E^W_{\le N-1}[h_{AL}]^2(\tau)\,d\tau  \right)^\frac12.
\end{align*}
For the first integral, it suffices to consider a term of the form:
\begin{align}
     \int_{\mathcal D^{\tau_{i+1}}_{\tau_i}} r^{2-\delta} |\mathscr L_Z^{\le N-1} \psi_+| |\nabla_{ L} \mathscr L_Z^{\le N-11} \psi_-| | \nabla_L \mathcal L_Z^{\le N-1} h_{A L}|\,dV ,
\end{align}
since the integral in which $\nabla_L$ acts on the higher-order term is straightforward to estimate.

Foliating the domain by the ingoing cones $\underline{\mathcal C}_v$ gives
\begin{align}\label{int-hAL1}
\begin{aligned}
    &  \int_{\mathcal D^{\tau_{i+1}}_{\tau_i}} r^{2-\delta} |\mathscr L_Z^{\le N-1} \psi_+| |\nabla_{ L}  \mathscr L_Z^{\le N-11}\psi_-| | \nabla_L \mathcal L_Z^{\le N-1} h_{A L}|\,dV  \\
    & \lesssim \int r^{\frac32-\delta} \|\mathscr L_Z^{\le N-1} \psi_+\|_{L^2(\underline{\mathcal C}_v)} \|\nabla_{\underline L}\nabla_L\mathscr L_Z^{\le N-9}\psi_-\|_{L^2(\underline{\mathcal C}_v)} \| \nabla_L \mathcal L_Z^{\le N-1} h_{A L}\|_{L^2(\underline{\mathcal C}_v)}\,dv \\
    & \lesssim \int r^{\frac12-\delta} \|\mathscr L_Z^{\le N-1} \psi_+\|_{L^2(\underline{\mathcal C}_v)} \|\nabla_L\mathscr L_Z^{\le N-8}\psi_+\|_{L^2(\underline{\mathcal C}_v)} \| \nabla_L \mathcal L_Z^{\le N-1} h_{A L}\|_{L^2(\underline{\mathcal C}_v)}\,dv \\
    & \qquad + \int r^{\frac32-\delta} \|\mathscr L_Z^{\le N-1} \psi_+\|_{L^2(\underline{\mathcal C}_v)} \|\nabla_L\mathscr L_Z^{\le N-9}( \nabla_{\underline L}\mathrm{Tr}h \cdot \psi_-)  \|_{L^2(\underline{\mathcal C}_v)} \| \nabla_L \mathcal L_Z^{\le N-1} h_{A L}\|_{L^2(\underline{\mathcal C}_v)}\,dv \\
     & \qquad + \int r^{\frac32-\delta} \|\mathscr L_Z^{\le N-1} \psi_+\|_{L^2(\underline{\mathcal C}_v)} \|\nabla_L\mathscr L_Z^{\le N-9}( \slashed\nabla \mathrm{Tr}h \cdot \psi_-)  \|_{L^2(\underline{\mathcal C}_v)} \| \nabla_L \mathcal L_Z^{\le N-1} h_{A L}\|_{L^2(\underline{\mathcal C}_v)}\,dv.
     \end{aligned}
\end{align}
Then we use the null decomposition for the spinor fields to get
\begin{align*}
    & \int r^{\frac12-\delta} \|\mathscr L_Z^{\le N-1} \psi_+\|_{L^2(\underline{\mathcal C}_v)} \|\nabla_L\mathscr L_Z^{\le N-8}\psi_+\|_{L^2(\underline{\mathcal C}_v)} \| \nabla_L \mathcal L_Z^{\le N-1} h_{A L}\|_{L^2(\underline{\mathcal C}_v)}\,dv \\
    & \lesssim \int r^{\frac12-\delta} \|\mathscr L_Z^{\le N-1} \psi_+\|_{L^2(\underline{\mathcal C}_v)} \|\slashed\nabla\mathscr L_Z^{\le N-8}\psi_-\|_{L^2(\underline{\mathcal C}_v)} \| \nabla_L \mathcal L_Z^{\le N-1} h_{A L}\|_{L^2(\underline{\mathcal C}_v)}\,dv \\
    & \qquad + \int r^{\frac12-\delta} \|\mathscr L_Z^{\le N-1} \psi_+\|_{L^2(\underline{\mathcal C}_v)} \|\mathscr L_Z^{\le N-8}( \nabla_L \mathrm{Tr}h \cdot \psi_+) \|_{L^2(\underline{\mathcal C}_v)} \| \nabla_L \mathcal L_Z^{\le N-1} h_{A L}\|_{L^2(\underline{\mathcal C}_v)}\,dv \\
    & \qquad + \int r^{\frac12-\delta} \|\mathscr L_Z^{\le N-1} \psi_+\|_{L^2(\underline{\mathcal C}_v)} \|\mathscr L_Z^{\le N-8}( \slashed\nabla \mathrm{Tr}h \cdot \psi_+) \|_{L^2(\underline{\mathcal C}_v)} \| \nabla_L \mathcal L_Z^{\le N-1} h_{A L}\|_{L^2(\underline{\mathcal C}_v)}\,dv.
\end{align*}
For the first integral, we see that
\begin{align*}
    &  \int r^{\frac12-\delta} \|\mathscr L_Z^{\le N-1} \psi_+\|_{L^2(\underline{\mathcal C}_v)} \|\slashed\nabla\mathscr L_Z^{\le N-8}\psi_-\|_{L^2(\underline{\mathcal C}_v)} \| \nabla_L \mathcal L_Z^{\le N-1} h_{A L}\|_{L^2(\underline{\mathcal C}_v)}\,dv  \\
    & \lesssim \int r^{-\frac12-\delta} \|\mathscr L_Z^{\le N-1} \psi_+\|_{L^2(\underline{\mathcal C}_v)} \|\mathscr L_Z^{\le N-7}\psi_-\|_{L^2(\underline{\mathcal C}_v)} \| \nabla_L \mathcal L_Z^{\le N-1} h_{A L}\|_{L^2(\underline{\mathcal C}_v)}\,dv \\
    & \lesssim \sup_{v: \tau\le\tau(v) }\mathscr F^D_{\le  N-7}[\psi](v,\tau_i,\tau_{i+1}) \int r^{-\frac12-\delta}\|\mathscr L_Z^{\le N-1} \psi_+\|_{L^2(\underline{\mathcal C}_v)} \| \nabla_L \mathcal L_Z^{\le N-1} h_{A L}\|_{L^2(\underline{\mathcal C}_v)}\,dv \\
    & \lesssim  \sup_{v: \tau\le\tau(v) }\mathscr F^D_{\le N-7}[\psi](v,\tau_i,\tau_{i+1}) \left( \int_{\tau_i}^{\tau_{i+1}} r^{-2} \|\mathscr L_Z^{\le N-1}\psi \|_{L^2(\Sigma_\tau)}^2\,d\tau \right)^\frac12 \left( \int_{\tau_i}^{\tau_{i+1}} r^{1-2\delta} \| \nabla_L \mathcal L_Z^{\le N-1}h_{AL}\|_{L^2(\Sigma_\tau)}^2\,d\tau \right)^\frac12 \\
    & \lesssim  \sup_{v: \tau\le\tau(v) }\mathscr F^D_{\le N-7}[\psi](v,\tau_i,\tau_{i+1}) \left( \int_{\tau_i}^{\tau_{i+1}} {}^{(0)}\mathcal E^D_{\le N-1}[\psi]^2(\tau) \,d\tau \right)^\frac12 \left( \int_{\tau_i}^{\tau_{i+1}} r^{-\delta}\, {}^{(1-\delta)}\mathcal E^W_{\le N-1}[h_{AL}]^2(\tau) \,d\tau \right)^\frac12 .
\end{align*}
For the second integral, we apply H\"older's inequality to the factor $\nabla_L\mathrm{Tr}h\cdot\psi_+$ to obtain an $L^4\times L^4$ product, then apply the weighted Sobolev inequality along $\underline{\mathcal C}_v$ to obtain
\begin{align*}
    & \int r^{\frac12-\delta} \|\mathscr L_Z^{\le N-1} \psi_+\|_{L^2(\underline{\mathcal C}_v)} \|\mathscr L_Z^{\le N-8}( \nabla_L \mathrm{Tr}h \cdot \psi_+) \|_{L^2(\underline{\mathcal C}_v)} \| \nabla_L \mathcal L_Z^{\le N-1} h_{A L}\|_{L^2(\underline{\mathcal C}_v)}\,dv \\
    & \lesssim \int r^{\frac12-\delta} \|\mathscr L_Z^{\le N-1} \psi_+\|_{L^2(\underline{\mathcal C}_v)} \|\nabla_{\underline L}\nabla_{L}\mathscr L_Z^{\le N-7}\mathrm{Tr}h \|_{L^2(\underline{\mathcal C}_v)}  \| \nabla_{\underline L}\mathscr L_Z^{\le N-7} \psi_+ \|_{L^2(\underline{\mathcal C}_v)} \| \nabla_L \mathcal L_Z^{\le N-1} h_{A L}\|_{L^2(\underline{\mathcal C}_v)}\,dv \\
    & \lesssim \int r^{\frac12-\delta} \|\mathscr L_Z^{\le N-1} \psi_+\|_{L^2(\underline{\mathcal C}_v)} \|\slashed\Delta\mathscr L_Z^{\le N-7}\mathrm{Tr}h \|_{L^2(\underline{\mathcal C}_v)}  \| \nabla_{\underline L}\mathscr L_Z^{\le N-7} \psi_+ \|_{L^2(\underline{\mathcal C}_v)} \| \nabla_L \mathcal L_Z^{\le N-1} h_{A L}\|_{L^2(\underline{\mathcal C}_v)}\,dv \\
    & \qquad + \int r^{-\frac12-\delta} \|\mathscr L_Z^{\le N-1} \psi_+\|_{L^2(\underline{\mathcal C}_v)} \|\nabla_{\underline L}\mathscr L_Z^{\le N-7}\mathrm{Tr}h \|_{L^2(\underline{\mathcal C}_v)}  \| \nabla_{\underline L}\mathscr L_Z^{\le N-7} \psi_+ \|_{L^2(\underline{\mathcal C}_v)} \| \nabla_L \mathcal L_Z^{\le N-1} h_{A L}\|_{L^2(\underline{\mathcal C}_v)}\,dv \\
    & \qquad + \int r^{-\frac12-\delta} \|\mathscr L_Z^{\le N-1} \psi_+\|_{L^2(\underline{\mathcal C}_v)} \|\nabla_{L}\mathscr L_Z^{\le N-7}\mathrm{Tr}h \|_{L^2(\underline{\mathcal C}_v)}  \| \nabla_{\underline L}\mathscr L_Z^{\le N-7} \psi_+ \|_{L^2(\underline{\mathcal C}_v)} \| \nabla_L \mathcal L_Z^{\le N-1} h_{A L}\|_{L^2(\underline{\mathcal C}_v)}\,dv \\
    & \qquad + \int r^{\frac12-\delta} \|\mathscr L_Z^{\le N-1} \psi_+\|_{L^2(\underline{\mathcal C}_v)} \|\mathcal L_Z^{\le N-7} \mathscr T_{L\underline L}[\psi] \|_{L^2(\underline{\mathcal C}_v)}  \| \nabla_{\underline L}\mathscr L_Z^{\le N-7} \psi_+ \|_{L^2(\underline{\mathcal C}_v)} \| \nabla_L \mathcal L_Z^{\le N-1} h_{A L}\|_{L^2(\underline{\mathcal C}_v)}\,dv \\
    & \qquad + \int r^{\frac12-\delta} \|\mathscr L_Z^{\le N-1} \psi_+\|_{L^2(\underline{\mathcal C}_v)} \|\mathcal L_Z^{\le N-7} \mathscr T_{AB}[\psi] \|_{L^2(\underline{\mathcal C}_v)}  \| \nabla_{\underline L}\mathscr L_Z^{\le N-7} \psi_+ \|_{L^2(\underline{\mathcal C}_v)} \| \nabla_L \mathcal L_Z^{\le N-1} h_{A L}\|_{L^2(\underline{\mathcal C}_v)}\,dv .
\end{align*}
Then we have
\begin{align*}
   &  \int r^{\frac12-\delta} \|\mathscr L_Z^{\le N-1} \psi_+\|_{L^2(\underline{\mathcal C}_v)} \|\slashed\Delta\mathscr L_Z^{\le N-7}\mathrm{Tr}h \|_{L^2(\underline{\mathcal C}_v)}  \| \nabla_{\underline L}\mathscr L_Z^{\le N-7} \psi_+ \|_{L^2(\underline{\mathcal C}_v)} \| \nabla_L \mathcal L_Z^{\le N-1} h_{A L}\|_{L^2(\underline{\mathcal C}_v)}\,dv\\
   & \lesssim  \int r^{-\frac12-\delta} \|\mathscr L_Z^{\le N-1} \psi_+\|_{L^2(\underline{\mathcal C}_v)} \|\slashed\nabla\mathscr L_Z^{\le N-6}\mathrm{Tr}h \|_{L^2(\underline{\mathcal C}_v)}  \| \nabla_{\underline L}\mathscr L_Z^{\le N-7} \psi_+ \|_{L^2(\underline{\mathcal C}_v)} \| \nabla_L \mathcal L_Z^{\le N-1} h_{A L}\|_{L^2(\underline{\mathcal C}_v)}\,dv \\
   & \lesssim \sup_{v:\tau\le\tau(v)}{}^{(0)} \mathcal F^W_{\le N-6}[h_{L\underline L}] {}^{(0)} \mathcal F^D_{\le N-7}[\psi] \int r^{-\frac12-\delta} \|\mathscr L_Z^{\le N-1} \psi_+\|_{L^2(\underline{\mathcal C}_v)}  \| \nabla_L \mathcal L_Z^{\le N-1} h_{A L}\|_{L^2(\underline{\mathcal C}_v)}\,dv \\
   & \lesssim \sup_{v:\tau\le\tau(v)}{}^{(0)} \mathcal F^W_{\le N-6}[h_{L\underline L}] {}^{(0)} \mathcal F^D_{\le N-7}[\psi]  \left( \int_{\tau_i}^{\tau_{i+1}} {}^{(0)}\mathcal E^D_{\le N-1}[\psi]^2(\tau) \,d\tau \right)^\frac12 \left( \int_{\tau_i}^{\tau_{i+1}} r^{-\delta}\, {}^{(1-\delta)}\mathcal E^W_{\le N-1}[h_{AL}]^2(\tau) \,d\tau \right)^\frac12  ,
\end{align*}
The second integral is estimated similarly.

For the third integral, we consider the specific configuration of the vector fields of the term $\mathscr L_Z^{\le N-1}\psi_+$. If the vector fields contain at least one $\underline L$, then we immediately have $\|\mathscr L_Z^{\le N-1}\psi_+\|_{L^2(\underline{\mathcal C}_v)}\lesssim \mathcal F^D_{\le N-2}[\psi] $. If the vector fields contain one $L$, using the null decomposition, we have $\|\mathscr L_Z^{\le N-1}\psi_+\|_{L^2(\underline{\mathcal C}_v)}\lesssim \|\slashed\nabla\mathscr L_Z^{\le N-2}\psi_-\|_{L^2(\underline{\mathcal C}_v)} $ up to lower-order terms. If the vector fields consist of only rotation fields, then $\|\mathscr L_Z^{\le N-1}\psi_+\|_{L^2(\underline{\mathcal C}_v)}\lesssim r \|\slashed\nabla \mathscr L_Z^{\le N-2}\psi_+\|_{L^2(\underline{\mathcal C}_v)} \lesssim r^{\frac\delta2} {}^{(2-\delta)}\mathcal F^D_{\le N-2}[\psi] $. Therefore, we see that
\begin{align*}
    &  \int r^{-\frac12-\delta} \|\mathscr L_Z^{\le N-1} \psi_+\|_{L^2(\underline{\mathcal C}_v)} \|\nabla_{L}\mathscr L_Z^{\le N-7}\mathrm{Tr}h \|_{L^2(\underline{\mathcal C}_v)}  \| \nabla_{\underline L}\mathscr L_Z^{\le N-7} \psi_+ \|_{L^2(\underline{\mathcal C}_v)} \| \nabla_L \mathcal L_Z^{\le N-1} h_{A L}\|_{L^2(\underline{\mathcal C}_v)}\,dv  \\
    & \lesssim \sup_{v}  {}^{(2-\delta)}\mathcal F^D_{\le N-2}[\psi] {}^{(0)} \mathcal F^D_{\le N-7}[\psi]  \int r^{-\frac12-\frac\delta2}  \|\nabla_{L}\mathscr L_Z^{\le N-7}\mathrm{Tr}h \|_{L^2(\underline{\mathcal C}_v)}  \| \nabla_L \mathcal L_Z^{\le N-1} h_{A L}\|_{L^2(\underline{\mathcal C}_v)}\,dv \\
    & \lesssim  \sup_{v}  {}^{(2-\delta)}\mathcal F^D_{\le N-2}[\psi] {}^{(0)} \mathcal F^D_{\le N-7}[\psi]  \left( \int_{\tau_i}^{\tau_{i+1}} {}^{(0)}\mathcal E^W_{\le N-7}[h_{L\underline L}]^2(\tau)\,d\tau \right)^\frac12  \left( \int_{\tau_i}^{\tau_{i+1}} {}^{(0)}\mathcal E^W_{\le N-1}[h_{AL}]^2(\tau)\,d\tau \right)^\frac12 .
\end{align*}
For the fourth integral including $\mathscr T_{L\underline L}$, we write
\begin{align*}
    & \int r^{\frac12-\delta} \|\mathscr L_Z^{\le N-1} \psi_+\|_{L^2(\underline{\mathcal C}_v)} \|\mathcal L_Z^{\le N-7} \mathscr T_{L\underline L}[\psi] \|_{L^2(\underline{\mathcal C}_v)}  \| \nabla_{\underline L}\mathscr L_Z^{\le N-7} \psi_+ \|_{L^2(\underline{\mathcal C}_v)} \| \nabla_L \mathcal L_Z^{\le N-1} h_{A L}\|_{L^2(\underline{\mathcal C}_v)}\,dv  \\
    & \lesssim \int r^{\frac12-\delta} \|\mathscr L_Z^{\le N-1} \psi_+\|_{L^2(\underline{\mathcal C}_v)} \|\mathcal L_Z^{\le N-7} ( \psi_- \nabla_{\underline L}\psi_-)  \|_{L^2(\underline{\mathcal C}_v)}  \| \nabla_{\underline L}\mathscr L_Z^{\le N-7} \psi_+ \|_{L^2(\underline{\mathcal C}_v)} \| \nabla_L \mathcal L_Z^{\le N-1} h_{A L}\|_{L^2(\underline{\mathcal C}_v)}\,dv \\
    & \qquad + \int r^{\frac12-\delta} \|\mathscr L_Z^{\le N-1} \psi_+\|_{L^2(\underline{\mathcal C}_v)} \|\mathcal L_Z^{\le N-7} (\psi_+\nabla_L\psi_+) \|_{L^2(\underline{\mathcal C}_v)}  \| \nabla_{\underline L}\mathscr L_Z^{\le N-7} \psi_+ \|_{L^2(\underline{\mathcal C}_v)} \| \nabla_L \mathcal L_Z^{\le N-1} h_{A L}\|_{L^2(\underline{\mathcal C}_v)}\,dv .
\end{align*}

Then we simply use the Sobolev inequality to get
\begin{align*}
    & \int r^{\frac12-\delta} \|\mathscr L_Z^{\le N-1} \psi_+\|_{L^2(\underline{\mathcal C}_v)} \|\mathcal L_Z^{\le N-7} ( \psi_- \nabla_{\underline L}\psi_-)  \|_{L^2(\underline{\mathcal C}_v)}  \| \nabla_{\underline L}\mathscr L_Z^{\le N-7} \psi_+ \|_{L^2(\underline{\mathcal C}_v)} \| \nabla_L \mathcal L_Z^{\le N-1} h_{A L}\|_{L^2(\underline{\mathcal C}_v)}\,dv \\
    & \lesssim \sup_{v:\tau\le\tau(v)}{}^{(0)} \mathcal F^D_{\le N-6}[\psi] \mathscr F^D_{\le N-6}[\psi] \int r^{-\delta} \|\mathscr L_Z^{\le N-1} \psi_+\|_{L^2(\underline{\mathcal C}_v)} \| \nabla_{\underline L}\mathscr L_Z^{\le N-7} \psi_+ \|_{L^2(\underline{\mathcal C}_v)} \| \nabla_L \mathcal L_Z^{\le N-1} h_{A L}\|_{L^2(\underline{\mathcal C}_v)}\,dv \\
    & \lesssim \sup_{v:\tau\le\tau(v)}{}^{(0)} \mathcal F^D_{\le N-6}[\psi]^2  \mathscr F^D_{\le N-6}[\psi] \int r^{1-\delta} \|r^{-1} \mathscr L_Z^{\le N-1} \psi_+\|_{L^2(\underline{\mathcal C}_v)}  \| \nabla_L \mathcal L_Z^{\le N-1} h_{A L}\|_{L^2(\underline{\mathcal C}_v)}\,dv \\
    & \lesssim \sup_{v:\tau\le\tau(v)}{}^{(0)} \mathcal F^D_{\le N-6}[\psi]^2 \mathscr F^D_{\le N-6}[\psi] \left( \int_{\tau_i}^{\tau_{i+1}} {}^{(1-\delta)}\mathcal E^D_{\le N-1}[\psi]^2(\tau)\,d\tau\right)^\frac12 \left( \int_{\tau_i}^{\tau_{i+1}} {}^{(1-\delta)}\mathcal E^W_{\le N-1}[h_{AL}]^2(\tau)\,d\tau\right)^\frac12. 
\end{align*}
Now we consider the integral
\begin{align}
    \int r^{\frac12-\delta} \|\mathscr L_Z^{\le N-1} \psi_+\|_{L^2(\underline{\mathcal C}_v)} \|\mathscr L_Z^{\le N-8}( \slashed\nabla \mathrm{Tr}h \cdot \psi_+) \|_{L^2(\underline{\mathcal C}_v)} \| \nabla_L \mathcal L_Z^{\le N-1} h_{A L}\|_{L^2(\underline{\mathcal C}_v)}\,dv.
\end{align}
This is much simpler. Indeed, we have
\begin{align*}
    & \int r^{\frac12-\delta} \|\mathscr L_Z^{\le N-1} \psi_+\|_{L^2(\underline{\mathcal C}_v)} \|\mathscr L_Z^{\le N-8}( \slashed\nabla \mathrm{Tr}h \cdot \psi_+) \|_{L^2(\underline{\mathcal C}_v)} \| \nabla_L \mathcal L_Z^{\le N-1} h_{A L}\|_{L^2(\underline{\mathcal C}_v)}\,dv. \\
    & \lesssim \int r^{-\frac12-\delta} \|\mathscr L_Z^{\le N-1} \psi_+\|_{L^2(\underline{\mathcal C}_v)} \|\mathscr L_Z^{\le N-8}( \mathcal L_\Omega^{\le1} \mathrm{Tr}h \cdot \psi_+) \|_{L^2(\underline{\mathcal C}_v)} \| \nabla_L \mathcal L_Z^{\le N-1} h_{A L}\|_{L^2(\underline{\mathcal C}_v)}\,dv \\
    & \lesssim \sup_{v:\tau\le\tau(v) } {}^{(0)} \mathcal F^W_{\le N-9}[h_{L\underline L}]{}^{(0)} \mathcal F^D_{\le N-9}[\psi]  \int r^{-\frac12-\delta} \|\mathscr L_Z^{\le N-1} \psi_+\|_{L^2(\underline{\mathcal C}_v)}  \| \nabla_L \mathcal L_Z^{\le N-1} h_{A L}\|_{L^2(\underline{\mathcal C}_v)}\,dv.
\end{align*}
It remains to estimate the second integral in the last inequality of \eqref{int-hAL1}:
\begin{align}
    \int r^{\frac32-\delta} \|\mathscr L_Z^{\le N-1} \psi_+\|_{L^2(\underline{\mathcal C}_v)} \|\nabla_L\mathscr L_Z^{\le N-9}( \nabla_{\underline L}\mathrm{Tr}h \cdot \psi_-)  \|_{L^2(\underline{\mathcal C}_v)} \| \nabla_L \mathcal L_Z^{\le N-1} h_{A L}\|_{L^2(\underline{\mathcal C}_v)}\,dv.
\end{align}
Then we write
\begin{align*}
    & \int r^{\frac32-\delta} \|\mathscr L_Z^{\le N-1} \psi_+\|_{L^2(\underline{\mathcal C}_v)} \|\nabla_L\mathscr L_Z^{\le N-9}( \nabla_{\underline L}\mathrm{Tr}h \cdot \psi_-)  \|_{L^2(\underline{\mathcal C}_v)} \| \nabla_L \mathcal L_Z^{\le N-1} h_{A L}\|_{L^2(\underline{\mathcal C}_v)}\,dv \\
    & \lesssim \int r^{\frac32-\delta} \|\mathscr L_Z^{\le N-1} \psi_+\|_{L^2(\underline{\mathcal C}_v)} \|\mathscr L_Z^{\le N-9}( \nabla_L\nabla_{\underline L}\mathrm{Tr}h \cdot \psi_-)  \|_{L^2(\underline{\mathcal C}_v)} \| \nabla_L \mathcal L_Z^{\le N-1} h_{A L}\|_{L^2(\underline{\mathcal C}_v)}\,dv \\
    & \qquad + \int r^{\frac32-\delta} \|\mathscr L_Z^{\le N-1} \psi_+\|_{L^2(\underline{\mathcal C}_v)} \|\mathscr L_Z^{\le N-9}( \nabla_{\underline L}\mathrm{Tr}h \cdot \nabla_L\psi_-)  \|_{L^2(\underline{\mathcal C}_v)} \| \nabla_L \mathcal L_Z^{\le N-1} h_{A L}\|_{L^2(\underline{\mathcal C}_v)}\,dv.
\end{align*}
The first integral can be decomposed as follows:
\begin{align*}
   & \int r^{\frac32-\delta} \|\mathscr L_Z^{\le N-1} \psi_+\|_{L^2(\underline{\mathcal C}_v)} \|\mathscr L_Z^{\le N-9}( \nabla_L\nabla_{\underline L}\mathrm{Tr}h \cdot \psi_-)  \|_{L^2(\underline{\mathcal C}_v)} \| \nabla_L \mathcal L_Z^{\le N-1} h_{A L}\|_{L^2(\underline{\mathcal C}_v)}\,dv\\
   & \lesssim \int r^{\frac32-\delta} \|\mathscr L_Z^{\le N-1} \psi_+\|_{L^2(\underline{\mathcal C}_v)} \|\mathscr L_Z^{\le N-9}( \slashed\Delta\mathrm{Tr}h \cdot \psi_-)  \|_{L^2(\underline{\mathcal C}_v)} \| \nabla_L \mathcal L_Z^{\le N-1} h_{A L}\|_{L^2(\underline{\mathcal C}_v)}\,dv \\
   & \qquad + \int r^{\frac12-\delta} \|\mathscr L_Z^{\le N-1} \psi_+\|_{L^2(\underline{\mathcal C}_v)} \|\mathscr L_Z^{\le N-9}( \nabla_{\underline L}\mathrm{Tr}h \cdot \psi_-)  \|_{L^2(\underline{\mathcal C}_v)} \| \nabla_L \mathcal L_Z^{\le N-1} h_{A L}\|_{L^2(\underline{\mathcal C}_v)}\,dv \\
   & \qquad + \int r^{\frac12-\delta} \|\mathscr L_Z^{\le N-1} \psi_+\|_{L^2(\underline{\mathcal C}_v)} \|\mathscr L_Z^{\le N-9}( \nabla_L\mathrm{Tr}h \cdot \psi_-)  \|_{L^2(\underline{\mathcal C}_v)} \| \nabla_L \mathcal L_Z^{\le N-1} h_{A L}\|_{L^2(\underline{\mathcal C}_v)}\,dv \\
   & \qquad + \int r^{\frac32-\delta} \|\mathscr L_Z^{\le N-1} \psi_+\|_{L^2(\underline{\mathcal C}_v)} \|\mathscr L_Z^{\le N-9}( \mathscr T_{L\underline L}[\psi]  \cdot \psi_-)  \|_{L^2(\underline{\mathcal C}_v)} \| \nabla_L \mathcal L_Z^{\le N-1} h_{A L}\|_{L^2(\underline{\mathcal C}_v)}\,dv \\
   & \qquad + \int r^{\frac32-\delta} \|\mathscr L_Z^{\le N-1} \psi_+\|_{L^2(\underline{\mathcal C}_v)} \|\mathscr L_Z^{\le N-9}( \mathscr T_{AB}[\psi] \cdot \psi_-)  \|_{L^2(\underline{\mathcal C}_v)} \| \nabla_L \mathcal L_Z^{\le N-1} h_{A L}\|_{L^2(\underline{\mathcal C}_v)}\,dv.
\end{align*}
Then we have
\begin{align*}
    & \int r^{\frac32-\delta} \|\mathscr L_Z^{\le N-1} \psi_+\|_{L^2(\underline{\mathcal C}_v)} \|\mathscr L_Z^{\le N-9}( \slashed\Delta\mathrm{Tr}h \cdot \psi_-)  \|_{L^2(\underline{\mathcal C}_v)} \| \nabla_L \mathcal L_Z^{\le N-1} h_{A L}\|_{L^2(\underline{\mathcal C}_v)}\,dv \\
    & \lesssim \int r^{-\frac12-\delta} \|\mathscr L_Z^{\le N-1} \psi_+\|_{L^2(\underline{\mathcal C}_v)} \|\mathscr L_Z^{\le N-9}( \mathcal L_\Omega^{\le2}\mathrm{Tr}h \cdot \psi_-)  \|_{L^2(\underline{\mathcal C}_v)} \| \nabla_L \mathcal L_Z^{\le N-1} h_{A L}\|_{L^2(\underline{\mathcal C}_v)}\,dv \\
    & \lesssim \sup_{v:\tau\le\tau(v)}{}^{(0)} \mathcal F^W_{\le N-7}[h_{L\underline L}]\mathscr F^D_{\le N-9}[\psi]  \int r^{-\frac12-\delta} \|\mathscr L_Z^{\le N-1} \psi_+\|_{L^2(\underline{\mathcal C}_v)}  \| \nabla_L \mathcal L_Z^{\le N-1} h_{A L}\|_{L^2(\underline{\mathcal C}_v)}\,dv.
\end{align*}
We then apply the Hardy and Cauchy--Schwarz inequalities, as in the estimates of the previous integrals, and omit the details.

For the second integral, we take the $L^2\times L^\infty$ product for $\nabla_{\underline L}\mathrm{Tr}h\cdot\psi_-$ and get
\begin{align*}
    & \int r^{\frac12-\delta} \|\mathscr L_Z^{\le N-1} \psi_+\|_{L^2(\underline{\mathcal C}_v)} \|\mathscr L_Z^{\le N-9}( \nabla_{\underline L}\mathrm{Tr}h \cdot \psi_-)  \|_{L^2(\underline{\mathcal C}_v)} \| \nabla_L \mathcal L_Z^{\le N-1} h_{A L}\|_{L^2(\underline{\mathcal C}_v)}\,dv  \\
    & \lesssim \int r^{-\delta} \|\mathscr L_Z^{\le N-1} \psi_+\|_{L^2(\underline{\mathcal C}_v)} \| \nabla_{\underline L}\mathcal L_Z^{\le N-9} \mathrm{Tr}h \|_{L^2(\underline{\mathcal C}_v)} \| \nabla_{\underline L}\mathscr L_Z^{\le N-7}\psi_-  \|_{L^2(\underline{\mathcal C}_v)} \| \nabla_L \mathcal L_Z^{\le N-1} h_{A L}\|_{L^2(\underline{\mathcal C}_v)}\,dv  \\
    & \lesssim  \int r^{-\delta} \|\mathscr L_Z^{\le N-1} \psi_+\|_{L^2(\underline{\mathcal C}_v)} \| \nabla_{\underline L}\mathcal L_Z^{\le N-9} \mathrm{Tr}h \|_{L^2(\underline{\mathcal C}_v)} \| \slashed\nabla\mathscr L_Z^{\le N-7}\psi_+  \|_{L^2(\underline{\mathcal C}_v)} \| \nabla_L \mathcal L_Z^{\le N-1} h_{A L}\|_{L^2(\underline{\mathcal C}_v)}\,dv \\
    & \qquad +  \int r^{-\delta} \|\mathscr L_Z^{\le N-1} \psi_+\|_{L^2(\underline{\mathcal C}_v)} \| \nabla_{\underline L}\mathcal L_Z^{\le N-9} \mathrm{Tr}h \|_{L^2(\underline{\mathcal C}_v)} \| \mathscr L_Z^{\le N-7}(\nabla_{\underline L}\mathrm{Tr}h \cdot\psi_- )  \|_{L^2(\underline{\mathcal C}_v)} \| \nabla_L \mathcal L_Z^{\le N-1} h_{A L}\|_{L^2(\underline{\mathcal C}_v)}\,dv  \\
    & \qquad +  \int r^{-\delta} \|\mathscr L_Z^{\le N-1} \psi_+\|_{L^2(\underline{\mathcal C}_v)} \| \nabla_{\underline L}\mathcal L_Z^{\le N-9} \mathrm{Tr}h \|_{L^2(\underline{\mathcal C}_v)} \| \mathscr L_Z^{\le N-7}( \slashed\nabla\mathrm{Tr}h\cdot \psi_-)  \|_{L^2(\underline{\mathcal C}_v)} \| \nabla_L \mathcal L_Z^{\le N-1} h_{A L}\|_{L^2(\underline{\mathcal C}_v)}\,dv .
\end{align*}
Then we see that
\begin{align*}
    & \int r^{-\delta} \|\mathscr L_Z^{\le N-1} \psi_+\|_{L^2(\underline{\mathcal C}_v)} \| \nabla_{\underline L}\mathcal L_Z^{\le N-9} \mathrm{Tr}h \|_{L^2(\underline{\mathcal C}_v)} \| \slashed\nabla\mathscr L_Z^{\le N-7}\psi_+  \|_{L^2(\underline{\mathcal C}_v)} \| \nabla_L \mathcal L_Z^{\le N-1} h_{A L}\|_{L^2(\underline{\mathcal C}_v)}\,dv  \\
    & \lesssim \sup_{v:\tau\le\tau(v) } {}^{(0)}\mathcal F^W_{\le N-9}[h]{}^{(2-\delta)}\mathcal F^D_{\le N-7}[\psi] \int r^{-1-\frac\delta2} \|\mathscr L_Z^{\le N-1} \psi_+\|_{L^2(\underline{\mathcal C}_v)} \| \nabla_L \mathcal L_Z^{\le N-1} h_{A L}\|_{L^2(\underline{\mathcal C}_v)}\,dv .
\end{align*}
The second and third integrals are estimated similarly. For the term involving $\nabla_L\mathrm{Tr}h\cdot\psi_-$, we apply H\"older's inequality to obtain an $L^4\times L^4$ product, followed by the Sobolev inequality. This yields $\nabla_{\underline L}\nabla_L\mathrm{Tr}h$, which is handled as above. The terms involving $\mathscr T_{L\underline L}$ and $\mathscr T_{AB}$ are also estimated similarly. It remains to control the third integral in the last inequality of \eqref{int-hAL1}:
\begin{align}
     \int r^{\frac32-\delta} \|\mathscr L_Z^{\le N-1} \psi_+\|_{L^2(\underline{\mathcal C}_v)} \|\nabla_L\mathscr L_Z^{\le2}( \slashed\nabla \mathrm{Tr}h \cdot \psi_-)  \|_{L^2(\underline{\mathcal C}_v)} \| \nabla_L \mathcal L_Z^{\le N-1} h_{A L}\|_{L^2(\underline{\mathcal C}_v)}\,dv.
\end{align}
This term is much simpler than the preceding one, so we omit the details.
\subsection{Control of $h_{A\underline L}$ }
We consider the integral
\begin{align}
    \int_{\mathcal D^{\tau_{i+1}}_{\tau_i}} r^{2-\delta} \mathcal L_Z^{\le N-1} \langle \psi, \gamma^{e_A}\nabla_{\underline L}  \psi\rangle \nabla_L \mathcal L_Z^{\le N-1} h_{A\underline L}\,dV +  \int_{\mathcal D^{\tau_{i+1}}_{\tau_i}} r^{2-\delta} \mathcal L_Z^{\le N-1} \langle \psi, \gamma^{ L}\nabla_{e_A}  \psi\rangle \nabla_L \mathcal L_Z^{\le N-1} h_{A\underline L}\,dV .
\end{align}
We estimate the second integral as in the case of the component $h_{AL}$ and omit the details.

We now focus on the first integral.
It suffices to consider integrals of the following forms:
\begin{align}
    \int_{\mathcal D^{\tau_{i+1}}_{\tau_i}} r^{2-\delta} \nabla_{\underline L}\mathscr L_Z^{\le N-11} \psi_+ \cdot \mathscr L_Z^{\le N-1}\psi_- \cdot \nabla_L\mathcal L_Z^{\le N-1}h_{A\underline L}\,dV  , \\
     \int_{\mathcal D^{\tau_{i+1}}_{\tau_i}} r^{2-\delta} \nabla_{\underline L}\mathscr L_Z^{\le N-1} \psi_+ \cdot \mathscr L_Z^{\le N-11}\psi_- \cdot \nabla_L\mathcal L_Z^{\le N-1}h_{A\underline L}\,d\sigma, \\
      \int_{\mathcal D^{\tau_{i+1}}_{\tau_i}} r^{2-\delta} \mathscr L_Z^{\le N-11} \psi_+ \cdot \nabla_{\underline L}\mathscr L_Z^{\le N-1}\psi_- \cdot \nabla_L\mathcal L_Z^{\le N-1}h_{A\underline L}\,dV  , \\
     \int_{\mathcal D^{\tau_{i+1}}_{\tau_i}} r^{2-\delta} \mathscr L_Z^{\le N-1} \psi_+ \cdot \nabla_{\underline L}\mathscr L_Z^{\le N-11}\psi_- \cdot \nabla_L\mathcal L_Z^{\le N-1}h_{A\underline L}\,dV.
\end{align}
For the first integral, we integrate by parts in the $\underline L$ direction to obtain
\begin{align*}
   &  \int_{\mathcal D^{\tau_{i+1}}_{\tau_i}} r^{2-\delta} \nabla_{\underline L}\mathscr L_Z^{\le N-11} \psi_+ \cdot \mathscr L_Z^{\le N-1}\psi_- \cdot \nabla_L\mathcal L_Z^{\le N-1}h_{A\underline L}\,dV  \\
   & = \int_{\substack{u=\tau_{i+1}-R \\ v \ge \tau_{i+1}+R} }   r^{2-\delta} \mathscr L_Z^{\le N-11} \psi_+ \cdot \mathscr L_Z^{\le N-1}\psi_- \cdot \nabla_L\mathcal L_Z^{\le N-1}h_{A\underline L}\,d\sigma  - \int_{\substack{u=\tau_{i}-R \\ v \ge \tau_{i}+R} }   r^{2-\delta} \mathscr L_Z^{\le N-11} \psi_+ \cdot \mathscr L_Z^{\le N-1}\psi_- \cdot \nabla_L\mathcal L_Z^{\le N-1}h_{A\underline L}\,d\sigma  \\
   & + \int_{r=R }   r^{2-\delta} \mathscr L_Z^{\le N-11} \psi_+ \cdot \mathscr L_Z^{\le N-1}\psi_- \cdot \nabla_L\mathcal L_Z^{\le N-1}h_{A\underline L}\,dV \\
   & +  \int_{\mathcal D^{\tau_{i+1}}_{\tau_i}} r^{2-\delta}\mathscr L_Z^{\le N-11} \psi_+ \cdot  \nabla_{\underline L} \mathscr L_Z^{\le N-1}\psi_- \cdot \nabla_L\mathcal L_Z^{\le N-1}h_{A\underline L}\,dV  +  \int_{\mathcal D^{\tau_{i+1}}_{\tau_i}} r^{2-\delta}\mathscr L_Z^{\le N-11} \psi_+ \cdot   \mathscr L_Z^{\le N-1}\psi_- \cdot \nabla_{\underline L}\nabla_L\mathcal L_Z^{\le N-1}h_{A\underline L}\,dV .
\end{align*}
To estimate the boundary integrals along the null hypersurfaces, we use the weighted Sobolev embedding and Hardy inequality:
\begin{align*}
    & \left| \int_{\substack{u=\tau_{i+1}-R \\ v \ge \tau_{i+1}+R} }   r^{2-\delta} \mathscr L_Z^{\le N-11} \psi_+ \cdot \mathscr L_Z^{\le N-1}\psi_- \cdot \nabla_L\mathcal L_Z^{\le N-1}h_{A\underline L}\,d\sigma \right| \\
    & \lesssim {}^{(1)} \mathcal E^D_{\le N-9}[\psi](\tau_{i+1}) {}^{(2-\delta)}\mathcal E^D_{\le N-1}[\psi](\tau_{i+1}) {}^{(2-\delta)}\mathcal E^W_{\le N-1}[h_{A\underline L}](\tau_{i+1}).
\end{align*}
The boundary integral across the timelike surface is controlled directly; we omit the details.
For the spacetime integral in which $\underline L$ acts on $\mathscr L_Z^{\le N-1}\psi_-$, the null decomposition for the spinors gives
\begin{align*}
    & \left|  \int_{\mathcal D^{\tau_{i+1}}_{\tau_i}} r^{2-\delta}\mathscr L_Z^{\le N-11} \psi_+ \cdot  \nabla_{\underline L} \mathscr L_Z^{\le N-1}\psi_- \cdot \nabla_L\mathcal L_Z^{\le N-1}h_{A\underline L}\,dV  \right| \\
    & \lesssim  \int_{\mathcal D^{\tau_{i+1}}_{\tau_i}} r^{2-\delta} |\mathscr L_Z^{\le N-11} \psi_+ | | \slashed \nabla \mathscr L_Z^{\le N-1}\psi_+|   |\nabla_L\mathcal L_Z^{\le N-1}h_{A\underline L} |\,dV \\
    & \qquad + \int_{\mathcal D^{\tau_{i+1}}_{\tau_i}} r^{2-\delta} |\mathscr L_Z^{\le N-11} \psi_+ | | \mathscr L_Z^{\le N-1}(\nabla_{\underline L}\mathrm{Tr}h \cdot \psi_-)|   |\nabla_L\mathcal L_Z^{\le N-1}h_{A\underline L} |\,dV \\
    & \qquad + \int_{\mathcal D^{\tau_{i+1}}_{\tau_i}} r^{2-\delta} |\mathscr L_Z^{\le N-11} \psi_+ | | \mathscr L_Z^{\le N-1}(\slashed\nabla\mathrm{Tr}h \cdot \psi_-)|   |\nabla_L\mathcal L_Z^{\le N-1}h_{A\underline L} |\,dV.
\end{align*}
For the first integral, it follows that
\begin{align*}
    & \int_{\mathcal D^{\tau_{i+1}}_{\tau_i}} r^{2-\delta} |\mathscr L_Z^{\le N-11} \psi_+ | | \slashed \nabla \mathscr L_Z^{\le N-1}\psi_+|   |\nabla_L\mathcal L_Z^{\le N-1}h_{A\underline L} |\,dV \\
    & \lesssim \int_{\tau_i}^{\tau_{i+1}} r^{\frac32-\delta} \mathcal E^D_{\le N-9}[\psi]\mathcal E^D_{\le N-1}[\psi]\mathcal E^W_{\le N-1}[h_{A\underline L}]\,d\tau \\
    & \lesssim \left( \int_{\tau_i}^{\tau_{i+1}} {}^{(1)}\mathcal E^D_{\le N-9}[\psi]^2(\tau){}^{(1-\delta)}\mathcal E^D_{\le N-1}[\psi]^2(\tau)\,d\tau  \right)^\frac12  \left( \int_{\tau_i}^{\tau_{i+1}} {}^{(1-\delta)}\mathcal E^W_{\le N-1}[h_{A\underline L}]^2(\tau)\,d\tau  \right)^\frac12,
\end{align*}
and for the second integral, we only consider the high-low and low-high interactions:
\begin{align*}
    & \int_{\mathcal D^{\tau_{i+1}}_{\tau_i}} r^{2-\delta} |\mathscr L_Z^{\le N-11} \psi_+ | | \mathscr L_Z^{\le N-1}(\nabla_{\underline L}\mathrm{Tr}h \cdot \psi_-)|   |\nabla_L\mathcal L_Z^{\le N-1}h_{A\underline L} |\,dV \\
    & \lesssim \int_{\mathcal D^{\tau_{i+1}}_{\tau_i}} r^{2-\delta} |\mathscr L_Z^{\le N-11} \psi_+ | | \nabla_{\underline L}\mathcal L_Z^{\le N-1} \mathrm{Tr}h \cdot \mathscr L_Z^{\le N-11} \psi_-)| |\nabla_L\mathcal L_Z^{\le N-1}h_{A\underline L} |\,dV \\
    & \qquad +  \int_{\mathcal D^{\tau_{i+1}}_{\tau_i}} r^{2-\delta} |\mathscr L_Z^{\le N-11} \psi_+ | | \nabla_{\underline L}\mathcal L_Z^{\le N-11} \mathrm{Tr}h \cdot \mathscr L_Z^{\le N-1} \psi_-)| |\nabla_L\mathcal L_Z^{\le N-1}h_{A\underline L} |\,dV .
\end{align*}
For the high-low interaction, consider the commutation fields in $\nabla_{\underline L}\mathcal L_Z^{\le N-1}\mathrm{Tr}h$. If they include at least one $L$, then $\|\nabla_{\underline L}\mathcal L_Z^{\le N-1}\mathrm{Tr}h\|_{L^2(\Sigma_\tau)} \lesssim \mathcal E^W_{\le N-1}[h] $. If they include at least one rotation $\Omega$, then $\|\nabla_{\underline L}\mathcal L_Z^{\le N-1}\mathrm{Tr}h\|_{L^2(\Sigma_\tau)} \lesssim r\mathcal E^W_{\le N-1}[h] $. Thus we have
\begin{align*}
    & \int_{\mathcal D^{\tau_{i+1}}_{\tau_i}} r^{2-\delta} |\mathscr L_Z^{\le N-11} \psi_+ | | \nabla_{\underline L}\mathcal L_Z^{\le N-1} \mathrm{Tr}h \cdot \mathscr L_Z^{\le N-11} \psi_-)| |\nabla_L\mathcal L_Z^{\le N-1}h_{A\underline L} |\,dV \\
    & \lesssim \int_{\tau_i}^{\tau_{i+1}} {}^{(1)} \mathcal E^D_{\le N-9}[\psi]^2(\tau){}^{(1-\delta)} \mathcal E^W_{\le N-1}[h_{A\underline L}]^2(\tau)\,d\tau .
\end{align*}
If all commutation fields are $\underline L$, we foliate by the ingoing null cones $\underline{\mathcal C}_v$ and obtain
\begin{align*}
    & \int_{\mathcal D^{\tau_{i+1}}_{\tau_i}} r^{2-\delta} |\mathscr L_Z^{\le N-11} \psi_+ | | \nabla_{\underline L}\mathcal L_Z^{\le N-1} \mathrm{Tr}h \cdot \mathscr L_Z^{\le N-11} \psi_-)| |\nabla_L\mathcal L_Z^{\le N-1}h_{A\underline L} |\,dV \\
    & \lesssim \int r^{1-\delta} \mathcal F^D_{\le N-9}[\psi] \mathcal F^W_{\le N-1}[h_{L\underline L}] \|\slashed\nabla \mathscr L_Z^{\le N-9}\psi_+\|_{L^2(\underline{\mathcal C}_v)} \|\nabla_L \mathcal L_Z^{\le N-1}h_{A\underline L}\|_{L^2(\underline{\mathcal C}_v)}\,dv \\
    & \qquad + \int r^{1-\delta} \mathcal F^D_{\le N-9}[\psi] \mathcal F^W_{\le N-1}[h_{L\underline L}] \| \mathscr L_Z^{\le N-9}(\nabla_{\underline L}\mathrm{Tr}h \cdot \psi_-)\|_{L^2(\underline{\mathcal C}_v)} \|\nabla_L \mathcal L_Z^{\le N-1}h_{A\underline L}\|_{L^2(\underline{\mathcal C}_v)}\,dv \\
    & \qquad + \int r^{1-\delta} \mathcal F^D_{\le N-9}[\psi] \mathcal F^W_{\le N-1}[h_{L\underline L}] \| \mathscr L_Z^{\le N-9}(\slashed\nabla\mathrm{Tr}h \cdot \psi_-)\|_{L^2(\underline{\mathcal C}_v)} \|\nabla_L \mathcal L_Z^{\le N-1}h_{A\underline L}\|_{L^2(\underline{\mathcal C}_v)}\,dv.
\end{align*}
Then we see that
\begin{align*}
    & \int r^{1-\delta} \mathcal F^D_{\le N-9}[\psi] \mathcal F^W_{\le N-1}[h_{L\underline L}] \|\slashed\nabla \mathscr L_Z^{\le N-9}\psi_+\|_{L^2(\underline{\mathcal C}_v)} \|\nabla_L \mathcal L_Z^{\le N-1}h_{A\underline L}\|_{L^2(\underline{\mathcal C}_v)}\,dv \\
    & \lesssim \sup_{v:\tau\le\tau(v) }{}^{(0)} \mathcal F^D_{\le N-9}[\psi]{}^{(0)} \mathcal F^W_{\le N-1}[h_{L\underline L}] \left( \int_{\tau_i}^{\tau_{i+1}} r^{-\delta}\,  {}^{(1-\delta)}\mathcal E^D_{\le N-9}[\psi]^2(\tau) \right)^\frac12 \left( \int_{\tau_i}^{\tau_{i+1}}   {}^{(1-\delta)}\mathcal E^W_{\le N-1}[h_{A\underline L}]^2(\tau) \right)^\frac12 ,
\end{align*}
and
\begin{align*}
    & \int r^{1-\delta} \mathcal F^D_{\le N-9}[\psi] \mathcal F^W_{\le N-1}[h_{L\underline L}] \| \mathscr L_Z^{\le N-9}(\nabla_{\underline L}\mathrm{Tr}h \cdot \psi_-)\|_{L^2(\underline{\mathcal C}_v)} \|\nabla_L \mathcal L_Z^{\le N-1}h_{A\underline L}\|_{L^2(\underline{\mathcal C}_v)}\,dv \\
    & \lesssim \sup_{v:\tau\le\tau(v) }{}^{(0)}  \mathcal F^D_{\le N-9}[\psi]{}^{(0)} \mathcal F^W_{\le N-1}[h_{L\underline L}] \left( \int_{\tau_i}^{\tau_{i+1}} r  \| \mathscr L_Z^{\le N-9}(\nabla_{\underline L}\mathrm{Tr}h \cdot \psi_-)\|_{L^2(\Sigma_\tau)}^2\,d\tau      \right)^\frac12   \left( \int_{\tau_i}^{\tau_{i+1}}   {}^{(1-\delta)}\mathcal E^W_{\le N-1}[h_{A\underline L}]^2(\tau) \right)^\frac12 \\
    & \lesssim \sup_{v:\tau\le\tau(v) }{}^{(0)} \mathcal F^D_{\le N-9}[\psi] {}^{(0)} \mathcal F^W_{\le N-1}[h_{L\underline L}] \left( \int_{\tau_i}^{\tau_{i+1}} {}^{(1)}\mathcal E^D_{\le N-7}[\psi]^2(\tau){}^{(0)}\mathcal E^W_{\le N-7}[h_{L\underline L}]^2(\tau)   \,d\tau      \right)^\frac12   \left( \int_{\tau_i}^{\tau_{i+1}}   {}^{(1-\delta)}\mathcal E^W_{\le N-1}[h_{A\underline L}]^2(\tau) \right)^\frac12 ,
\end{align*}
and the integral including $\slashed\nabla\mathrm{Tr}h\cdot\psi_- $ can be controlled similarly. For the low-high interaction, we consider the vector fields of the term $\mathscr L_Z^{\le N-1}\psi_-$. If they contain at least one $L$ or rotation $\Omega$, then we have $\|\mathscr L_Z^{\le N-1}\psi_-\|_{L^2(\Sigma_\tau)}\lesssim r\mathcal E^D_{\le N-2}[\psi] $ and the remaining estimate is analogous to the high-low case. If all the commutation fields are $\underline L$, we have
\begin{align*}
    &  \int_{\mathcal D^{\tau_{i+1}}_{\tau_i}} r^{2-\delta} |\mathscr L_Z^{\le N-11} \psi_+ | | \nabla_{\underline L}\mathcal L_Z^{\le N-11} \mathrm{Tr}h \cdot \mathscr L_Z^{\le N-1} \psi_-)| |\nabla_L\mathcal L_Z^{\le N-1}h_{A\underline L} |\,dV \\
    & \lesssim \int_{\tau_i}^{\tau_{i+1}} r^{1-\delta} \mathcal E^D_{\le N-9}[\psi]\mathcal E^W_{\le N-8}[h] \|\slashed\nabla \mathscr L_Z^{\le N-2}\psi_+\|_{L^2(\Sigma_\tau)} \mathcal E^W_{\le N-1}[h_{A\underline L}]\,d\tau \\
    & \qquad + \int_{\tau_i}^{\tau_{i+1}} r^{1-\delta} \mathcal E^D_{\le N-9}[\psi]\mathcal E^W_{\le N-8}[h] \| \mathscr L_{\underline L}^{\le N-2}(\nabla_{\underline L}\mathrm{Tr}h\cdot\psi_- )\|_{L^2(\Sigma_\tau)} \mathcal E^W_{\le N-1}[h_{A\underline L}]\,d\tau \\
    & \qquad + \int_{\tau_i}^{\tau_{i+1}} r^{1-\delta} \mathcal E^D_{\le N-9}[\psi]\mathcal E^W_{\le N-8}[h] \| \mathscr L_{\underline L}^{\le N-2}(\slashed\nabla \mathrm{Tr}h\cdot\psi_- )\|_{L^2(\Sigma_\tau)} \mathcal E^W_{\le N-1}[h_{A\underline L}]\,d\tau.
\end{align*}
We estimate the first integral directly and focus on the second remaining integral. If at least one commutation field acts on $\psi_-$, we apply H\"older's inequality to obtain an $L^4\times L^4$ product, then use the weighted Sobolev inequality to bound it by $\mathcal E^W_{\le N-3}[h]\mathcal E^D_{\ge1}[\psi]$. It remains to consider the case in which all commutations are by $\underline L$ and act on $\nabla_{\underline L}\mathrm{Tr}h$. We then apply the Hardy inequality to obtain

\begin{align*}
    & \int_{\tau_i}^{\tau_{i+1}} r^{1-\delta} \mathcal E^D_{\le N-9}[\psi]\mathcal E^W_{\le N-8}[h] \| \nabla_{\underline L}\mathcal L_{\underline L}^{\le N-2}\mathrm{Tr}h\cdot\psi_- \|_{L^2(\Sigma_\tau)} \mathcal E^W_{\le N-1}[h_{A\underline L}]\,d\tau \\
    & \lesssim \int_{\tau_i}^{\tau_{i+1}} r^{\frac32-\delta} \mathcal E^D_{\le N-9}[\psi]^2\mathcal E^W_{\le N-8}[h] \| r^{-1}\nabla_{\underline L}\mathcal L_{\underline L}^{\le N-2}\mathrm{Tr}h \|_{L^2(\Sigma_\tau)} \mathcal E^W_{\le N-1}[h_{A\underline L}]\,d\tau \\
    & \lesssim \int_{\tau_i}^{\tau_{i+1}}{}^{(0)}\mathcal E^D_{\le N-9}[\psi]^2(\tau){}^{(1)} \mathcal E^W_{\le N-8}[h](\tau) {}^{(1-\delta)} \mathcal E^W_{\le N-1}[h_{A\underline L}]^2(\tau)\,d\tau.
\end{align*}
The third integral, which contains $\slashed\nabla\mathrm{Tr}h\cdot\psi_-$, is easier to estimate; we omit the details.

We now estimate the spacetime integral that exhibits derivative loss:
\begin{align}
    \int_{\mathcal D^{\tau_{i+1}}_{\tau_i}} r^{2-\delta}\mathscr L_Z^{\le N-11} \psi_+ \cdot   \mathscr L_Z^{\le N-1}\psi_- \cdot \nabla_{\underline L}\nabla_L\mathcal L_Z^{\le N-1}h_{A\underline L}\,dV.
\end{align}
Then we have
\begin{align*}
    & \int_{\mathcal D^{\tau_{i+1}}_{\tau_i}} r^{2-\delta}\mathscr L_Z^{\le N-11} \psi_+ \cdot   \mathscr L_Z^{\le N-1}\psi_- \cdot \nabla_{\underline L}\nabla_L\mathcal L_Z^{\le N-1}h_{A\underline L}\,dV \\
    & = \int_{\mathcal D^{\tau_{i+1}}_{\tau_i}} r^{2-\delta}\mathscr L_Z^{\le N-11} \psi_+ \cdot   \mathscr L_Z^{\le N-1}\psi_- \cdot \slashed\Delta\mathcal L_Z^{\le N-1}h_{A\underline L}\,dV \\
    & \qquad + \int_{\mathcal D^{\tau_{i+1}}_{\tau_i}} r^{1-\delta}\mathscr L_Z^{\le N-11} \psi_+ \cdot   \mathscr L_Z^{\le N-1}\psi_- \cdot \nabla_{\underline L}\mathcal L_Z^{\le N-1}h_{A\underline L}\,dV  \\
    & \qquad + \int_{\mathcal D^{\tau_{i+1}}_{\tau_i}} r^{1-\delta}\mathscr L_Z^{\le N-11} \psi_+ \cdot   \mathscr L_Z^{\le N-1}\psi_- \cdot \nabla_L\mathcal L_Z^{\le N-1}h_{A\underline L}\,dV \\
    & \qquad + \int_{\mathcal D^{\tau_{i+1}}_{\tau_i}} r^{2-\delta}\mathscr L_Z^{\le N-11} \psi_+ \cdot   \mathscr L_Z^{\le N-1}\psi_- \cdot \mathcal L_Z^{\le N-1}\mathscr T_{A\underline L}\,dV.
\end{align*}
For the first integral, we integrate by parts in the angular variables to obtain
\begin{align*}
    & \int_{\mathcal D^{\tau_{i+1}}_{\tau_i}} r^{2-\delta}\mathscr L_Z^{\le N-11} \psi_+ \cdot   \mathscr L_Z^{\le N-1}\psi_- \cdot \slashed\Delta\mathcal L_Z^{\le N-1}h_{A\underline L}\,dV  \\
    & = \int_{\mathcal D^{\tau_{i+1}}_{\tau_i}} r^{2-\delta}\slashed\nabla\mathscr L_Z^{\le N-11} \psi_+ \cdot   \mathscr L_Z^{\le N-1}\psi_- \cdot \slashed\nabla\mathcal L_Z^{\le N-1}h_{A\underline L}\,dV \\
    & \qquad + \int_{\mathcal D^{\tau_{i+1}}_{\tau_i}} r^{2-\delta}\mathscr L_Z^{\le N-11} \psi_+ \cdot   \slashed\nabla\mathscr L_Z^{\le N-1}\psi_- \cdot \slashed\nabla\mathcal L_Z^{\le N-1}h_{A\underline L}\,dV .
\end{align*}
Then we write
\begin{align*}
    & \left| \int_{\mathcal D^{\tau_{i+1}}_{\tau_i}} r^{2-\delta}\slashed\nabla\mathscr L_Z^{\le N-11} \psi_+ \cdot   \mathscr L_Z^{\le N-1}\psi_- \cdot \slashed\nabla\mathcal L_Z^{\le N-1}h_{A\underline L}\,dV \right| \\
    & \lesssim \int_{\tau_i}^{\tau_{i+1}} r^{\frac12-\delta} \mathcal E^D_{\le N-8}[\psi] \|\mathscr L_Z^{\le N-1}\psi_-\|_{L^2(\Sigma_\tau)}\mathcal E^W_{\le N-1}[h_{A\underline L}]\,d\tau. 
\end{align*}
If the configuration of the vector fields $\mathscr L_Z^{\le N-1}\psi_-$ contains at least one $L$ or $\Omega$, then we have
\begin{align*}
    & \int_{\tau_i}^{\tau_{i+1}} r^{\frac12-\delta} \mathcal E^D_{\le N-8}[\psi] \|\mathscr L_Z^{\le N-1}\psi_-\|_{L^2(\Sigma_\tau)}\mathcal E^W_{\le N-1}[h_{A\underline L}]\,d\tau \\
    & \lesssim \int_{\tau_i}^{\tau_{i+1}} r^{\frac32-\delta} \mathcal E^D_{\le N-8}[\psi] \mathcal E^D_{\le N-2}[\psi] \mathcal E^W_{\le N-1}[h_{A\underline L}]\,d\tau \\
    & \lesssim \left( \int_{\tau_i}^{\tau_{i+1}} {}^{(1)}\mathcal E^D_{\le N-8}[\psi]^2(\tau) {}^{(1-\delta)}\mathcal E^D_{\le N-2}[\psi]^2(\tau) \,d\tau \right)^\frac12 \left(\int_{\tau_i}^{\tau_{i+1}} {}^{(1-\delta)}\mathcal E^W_{\le N-2}[h_{A\underline L}]^2(\tau) \,d\tau \right)^\frac12  .
\end{align*}
If, on the other hand, all the vector fields are $\underline L$, then
\begin{align*}
    &  \int_{\tau_i}^{\tau_{i+1}} r^{\frac12-\delta} \mathcal E^D_{\le N-8}[\psi] \|\mathscr L_{\underline L}^{\le N-1}\psi_-\|_{L^2(\Sigma_\tau)}\mathcal E^W_{\le N-1}[h_{A\underline L}]\,d\tau  \\
    & \lesssim  \int_{\tau_i}^{\tau_{i+1}} r^{\frac12-\delta} \mathcal E^D_{\le N-8}[\psi] \|\slashed\nabla\mathscr L_{\underline L}^{\le N-2}\psi_+\|_{L^2(\Sigma_\tau)}\mathcal E^W_{\le N-1}[h_{A\underline L}]\,d\tau  \\
    & \qquad +  \int_{\tau_i}^{\tau_{i+1}} r^{\frac12-\delta} \mathcal E^D_{\le N-8}[\psi] \|\mathscr L_{\underline L}^{\le N-2}(\nabla_{\underline L}\mathrm{Tr}h\cdot \psi_-)\|_{L^2(\Sigma_\tau)}\mathcal E^W_{\le N-1}[h_{A\underline L}]\,d\tau \\
    & \qquad + \int_{\tau_i}^{\tau_{i+1}} r^{\frac12-\delta} \mathcal E^D_{\le N-8}[\psi] \|\mathscr L_{\underline L}^{\le N-2}(\slashed\nabla\mathrm{Tr}h\cdot \psi_-)\|_{L^2(\Sigma_\tau)}\mathcal E^W_{\le N-1}[h_{A\underline L}]\,d\tau ,
\end{align*}
then the remaining task follows the previous argument. When the angular derivative acts on the higher-order term $\mathscr L_Z^{\le N-1}\psi_-$, we simply write
\begin{align*}
    & \left|  \int_{\mathcal D^{\tau_{i+1}}_{\tau_i}} r^{2-\delta}\mathscr L_Z^{\le N-11} \psi_+ \cdot   \slashed\nabla\mathscr L_Z^{\le N-1}\psi_- \cdot \slashed\nabla\mathcal L_Z^{\le N-1}h_{A\underline L}\,dV \right| \\
    & \lesssim \int_{\tau_i}^{\tau_{i+1}} r^{\frac32-\delta} \mathcal E^D_{\le N-9}[\psi] \mathcal E^D_{\le N-1}[\psi] \mathcal E^W_{\le N-1}[h_{A\underline L}]\,d\tau \\
    & \lesssim \left( \int_{\tau_i}^{\tau_{i+1}} {}^{(1)}\mathcal E^D_{\le N-9}[\psi]^2(\tau) {}^{(1-\delta)}\mathcal E^D_{\le N-1}[\psi]^2(\tau) \,d\tau \right)^\frac12 \left(\int_{\tau_i}^{\tau_{i+1}} {}^{(1-\delta)}\mathcal E^W_{\le N-1}[h_{A\underline L}]^2(\tau) \,d\tau \right)^\frac12 .
\end{align*}
The spacetime integral involving $\nabla_L\mathcal L_Z^{\le N-1}h_{A\underline L}$ is easier to estimate. The integral involving the additional nonlinear term $\mathcal L_Z^{\le N-1}\mathscr T_{A\underline L}[\psi]$ is estimated in the same way as the preceding integrals.
We now focus on the following integral:
\begin{align}
    \int_{\mathcal D^{\tau_{i+1}}_{\tau_i}} r^{1-\delta}\mathscr L_Z^{\le N-11} \psi_+ \cdot   \mathscr L_Z^{\le N-1}\psi_- \cdot \nabla_{\underline L}\mathcal L_Z^{\le N-1}h_{A\underline L}\,dV.
\end{align}
We again integrate by parts in the $\nabla_{\underline L}$ direction to obtain
\begin{align*}
   &  \int_{\mathcal D^{\tau_{i+1}}_{\tau_i}} r^{1-\delta}\mathscr L_Z^{\le N-11} \psi_+ \cdot   \mathscr L_Z^{\le N-1}\psi_- \cdot \nabla_{\underline L}\mathcal L_Z^{\le N-1}h_{A\underline L}\,dV \\
   & = \int_{\substack{u=\tau_{i+1}-R \\ v\ge \tau_{i+1}+R } } r^{1-\delta}\mathscr L_Z^{\le N-11} \psi_+ \cdot   \mathscr L_Z^{\le N-1}\psi_- \cdot \mathcal L_Z^{\le N-1}h_{A\underline L}\,d\sigma- \int_{\substack{u=\tau_{i}-R \\ v\ge \tau_{i}+R } } r^{1-\delta}\mathscr L_Z^{\le N-11} \psi_+ \cdot   \mathscr L_Z^{\le N-1}\psi_- \cdot \mathcal L_Z^{\le N-1}h_{A\underline L}\,d\sigma \\
   & + \int_{r=R } r^{1-\delta}\mathscr L_Z^{\le N-11} \psi_+ \cdot   \mathscr L_Z^{\le N-1}\psi_- \cdot \mathcal L_Z^{\le N-1}h_{A\underline L}\,d\sigma \\
   & + \int_{\mathcal D^{\tau_{i+1}}_{\tau_i}} r^{1-\delta}\nabla_{\underline L}\mathscr L_Z^{\le N-11} \psi_+ \cdot   \mathscr L_Z^{\le N-1}\psi_- \cdot \mathcal L_Z^{\le N-1}h_{A\underline L}\,dV+  \int_{\mathcal D^{\tau_{i+1}}_{\tau_i}} r^{1-\delta} \mathscr L_Z^{\le N-11} \psi_+ \cdot  \nabla_{\underline L} \mathscr L_Z^{\le N-1}\psi_- \cdot \mathcal L_Z^{\le N-1}h_{A\underline L}\,dV .
\end{align*}
The boundary integrals are controlled directly, so we focus on the spacetime integrals.
When $\nabla_{\underline L}$ acts on the lower-order term $\mathscr L_Z^{\le N-11}\psi_+$, the decomposition gives
\begin{align*}
    & \left|  \int_{\mathcal D^{\tau_{i+1}}_{\tau_i}} r^{1-\delta}\nabla_{\underline L}\mathscr L_Z^{\le N-11} \psi_+ \cdot   \mathscr L_Z^{\le N-1}\psi_- \cdot \mathcal L_Z^{\le N-1}h_{A\underline L}\,dV \right| \\
    & \lesssim \int_{\tau_i}^{\tau_{i+1}} r^{\frac12-\delta} \|\nabla_L\nabla_{\underline L}\mathscr L_Z^{\le N-9}\psi_+\|_{L^2(\Sigma_\tau)} \|\mathscr L_Z^{\le N-1}\psi_-\|_{L^2(\Sigma_\tau)} \|\mathcal L_Z^{\le N-1}h_{A\underline L}\|_{L^2(\Sigma_\tau)}\,d\tau \\
    & \lesssim \int_{\tau_i}^{\tau_{i+1}} r^{-\frac12-\delta} \|\slashed\nabla\mathscr L_Z^{\le N-8}\psi_+\|_{L^2(\Sigma_\tau)} \|\mathscr L_Z^{\le N-1}\psi_-\|_{L^2(\Sigma_\tau)} \|\mathcal L_Z^{\le N-1}h_{A\underline L}\|_{L^2(\Sigma_\tau)}\,d\tau \\
    & \qquad + \int_{\tau_i}^{\tau_{i+1}} r^{-\frac12-\delta} \|\nabla_{\underline L}\mathscr L_Z^{\le N-9}\psi_+\|_{L^2(\Sigma_\tau)} \|\mathscr L_Z^{\le N-1}\psi_-\|_{L^2(\Sigma_\tau)} \|\mathcal L_Z^{\le N-1}h_{A\underline L}\|_{L^2(\Sigma_\tau)}\,d\tau \\
    & \qquad + \int_{\tau_i}^{\tau_{i+1}} r^{-\frac12-\delta} \|\nabla_L\mathscr L_Z^{\le N-9}\psi_+\|_{L^2(\Sigma_\tau)} \|\mathscr L_Z^{\le N-1}\psi_-\|_{L^2(\Sigma_\tau)} \|\mathcal L_Z^{\le N-1}h_{A\underline L}\|_{L^2(\Sigma_\tau)}\,d\tau \\
    & \qquad + \int_{\tau_i}^{\tau_{i+1}} r^{\frac12-\delta} \|\mathscr L_Z^{\le N-9}\mathscr T_{L\underline L}[\psi] \|_{L^2(\Sigma_\tau)} \|\mathscr L_Z^{\le N-1}\psi_-\|_{L^2(\Sigma_\tau)} \|\mathcal L_Z^{\le N-1}h_{A\underline L}\|_{L^2(\Sigma_\tau)}\,d\tau \\
    & \qquad + \int_{\tau_i}^{\tau_{i+1}} r^{\frac12-\delta} \|\mathscr L_Z^{\le N-9}\mathscr T_{AB}[\psi] \|_{L^2(\Sigma_\tau)} \|\mathscr L_Z^{\le N-1}\psi_-\|_{L^2(\Sigma_\tau)} \|\mathcal L_Z^{\le N-1}h_{A\underline L}\|_{L^2(\Sigma_\tau)}\,d\tau.
\end{align*}
The second integral is the most delicate; the remaining integrals are estimated as above. We consider the specific configuration of the vector fields of $\mathscr L_Z^{\le N-1}\psi_-$. If they contain at least one $\underline L$, then we have
\begin{align*}
    & \int_{\tau_i}^{\tau_{i+1}} r^{-\frac12-\delta} \|\nabla_{\underline L}\mathscr L_Z^{\le N-8}\psi_+\|_{L^2(\Sigma_\tau)} \|\mathscr L_Z^{\le N-1}\psi_-\|_{L^2(\Sigma_\tau)} \|\mathcal L_Z^{\le N-1}h_{A\underline L}\|_{L^2(\Sigma_\tau)}\,d\tau \\
    & \lesssim  \int_{\tau_i}^{\tau_{i+1}} r^{-\frac12-\delta} \|\nabla_{\underline L}\mathscr L_Z^{\le N-8}\psi_+\|_{L^2(\Sigma_\tau)} \|\slashed\nabla\mathscr L_Z^{\le N-2}\psi_+\|_{L^2(\Sigma_\tau)} \|\mathcal L_Z^{\le N-1}h_{A\underline L}\|_{L^2(\Sigma_\tau)}\,d\tau \\
    & \qquad +  \int_{\tau_i}^{\tau_{i+1}} r^{-\frac12-\delta} \|\nabla_{\underline L}\mathscr L_Z^{\le N-8}\psi_+\|_{L^2(\Sigma_\tau)} \|\mathscr L_Z^{\le N-2}(\nabla_{\underline L}\mathrm{Tr}h \cdot \psi_-)\|_{L^2(\Sigma_\tau)} \|\mathcal L_Z^{\le N-1}h_{A\underline L}\|_{L^2(\Sigma_\tau)}\,d\tau \\
    & \qquad + \int_{\tau_i}^{\tau_{i+1}} r^{-\frac12-\delta} \|\nabla_{\underline L}\mathscr L_Z^{\le N-8}\psi_+\|_{L^2(\Sigma_\tau)} \|\mathscr L_Z^{\le N-2}(\slashed\nabla\mathrm{Tr}h \cdot \psi_-)\|_{L^2(\Sigma_\tau)} \|\mathcal L_Z^{\le N-1}h_{A\underline L}\|_{L^2(\Sigma_\tau)}\,d\tau,
\end{align*}
and hence the use of the Hardy inequality gives
\begin{align*}
    & \int_{\tau_i}^{\tau_{i+1}} r^{-\frac12-\delta} \|\nabla_{\underline L}\mathscr L_Z^{\le N-8}\psi_+\|_{L^2(\Sigma_\tau)} \|\slashed\nabla\mathscr L_Z^{\le N-2}\psi_+\|_{L^2(\Sigma_\tau)} \|\mathcal L_Z^{\le N-1}h_{A\underline L}\|_{L^2(\Sigma_\tau)}\,d\tau \\
    & \lesssim \int_{\tau_i}^{\tau_{i+1}} r^{\frac12-\delta} \mathscr E^D_{\le N-7}[\psi]  \mathcal E^D_{\le N-2}[\psi]  \mathcal E^W_{\le N-1}[h_{A\underline L}] \,d\tau \\
    & \lesssim \sup_{\tau\in[\tau_i,\tau_{i+1}]} \mathscr E^D_{\le N-7}[\psi](\tau) \left( \int_{\tau_i}^{\tau_{i+1}} r^{-\delta}\, {}^{(0)}\mathcal E^D_{\le N-2}[\psi]^2(\tau)\,d\tau    \right)^\frac12 \left( \int_{\tau_i}^{\tau_{i+1}}  {}^{(1-\delta)}\mathcal E^W_{\le N-1}[h_{A\underline L}]^2(\tau)\,d\tau    \right)^\frac12,
\end{align*}
and we only consider the low-high and high-low interactions for the second and the third integrals:
\begin{align*}
    & \int_{\tau_i}^{\tau_{i+1}} r^{-\frac12-\delta} \|\nabla_{\underline L}\mathscr L_Z^{\le N-8}\psi_+\|_{L^2(\Sigma_\tau)} \|\mathscr L_Z^{\le N-2}(\nabla_{\underline L}\mathrm{Tr}h \cdot \psi_-)\|_{L^2(\Sigma_\tau)} \|\mathcal L_Z^{\le N-1}h_{A\underline L}\|_{L^2(\Sigma_\tau)}\,d\tau \\
    & \lesssim \int_{\tau_i}^{\tau_{i+1}} r^{-\frac12-\delta} \|\nabla_{\underline L}\mathscr L_Z^{\le N-8}\psi_+\|_{L^2(\Sigma_\tau)} \|\nabla_{\underline L}\mathcal L_Z^{\le N-2}\mathrm{Tr}h \cdot \mathscr L_Z^{\le N-11} \psi_-)\|_{L^2(\Sigma_\tau)} \|\mathcal L_Z^{\le N-1}h_{A\underline L}\|_{L^2(\Sigma_\tau)}\,d\tau \\
    & \qquad + \int_{\tau_i}^{\tau_{i+1}} r^{-\frac12-\delta} \|\nabla_{\underline L}\mathscr L_Z^{\le N-8}\psi_+\|_{L^2(\Sigma_\tau)} \|\nabla_{\underline L}\mathcal L_Z^{\le N-11}\mathrm{Tr}h \cdot \mathscr L_Z^{\le N-2} \psi_-)\|_{L^2(\Sigma_\tau)} \|\mathcal L_Z^{\le N-1}h_{A\underline L}\|_{L^2(\Sigma_\tau)}\,d\tau.
\end{align*}
Applying the Hardy inequality gives
\begin{align*}
    & \int_{\tau_i}^{\tau_{i+1}} r^{-\frac12-\delta} \|\nabla_{\underline L}\mathscr L_Z^{\le N-8}\psi_+\|_{L^2(\Sigma_\tau)} \|\nabla_{\underline L}\mathcal L_Z^{\le N-2}\mathrm{Tr}h \cdot \mathscr L_Z^{\le N-11} \psi_-)\|_{L^2(\Sigma_\tau)} \|\mathcal L_Z^{\le N-1}h_{A\underline L}\|_{L^2(\Sigma_\tau)}\,d\tau \\
    & \lesssim \int_{\tau_i}^{\tau_{i+1}} r^{1-\delta} \mathscr E^D_{\le N-7}[\psi] \mathcal E^D_{\le N-9}[\psi] \mathcal E^W_{\le N-1}[h_{A\underline L}]^2\,d\tau \\
    & \lesssim \sup_{\tau\in[\tau_i,\tau_{i+1}]} \mathscr E^D_{\le N-7}[\psi](\tau)  {}^{(0)} \mathcal E^D_{\le N-9}[\psi](\tau) \int_{\tau_i}^{\tau_{i+1}}{}^{(1-\delta)} \mathcal E^W_{\le N-1}[h_{A\underline L}]^2(\tau)\,d\tau,
\end{align*}
The low-high interaction is estimated similarly using the Hardy inequality; we omit the details. We now exclude commutation fields equal to $\underline L$ from the configuration of $\mathscr L_Z^{\le N-1}\psi_-$. If at least one field is $L$, then
\begin{align*}
    & \int_{\tau_i}^{\tau_{i+1}} r^{-\frac12-\delta} \|\nabla_{\underline L}\mathscr L_Z^{\le N-8}\psi_+\|_{L^2(\Sigma_\tau)} \|\mathscr L_Z^{\le N-1}\psi_-\|_{L^2(\Sigma_\tau)} \|\mathcal L_Z^{\le N-1}h_{A\underline L}\|_{L^2(\Sigma_\tau)}\,d\tau \\
    & \lesssim  \int_{\tau_i}^{\tau_{i+1}} r^{-\frac12-\delta} \mathscr E^D_{\le N-7}[\psi] \mathcal E^D_{\le N-2}[\psi] \mathcal E^W_{\le N-2}[h_{A\underline L}]\,d\tau,
\end{align*}
which gives the desired control. On the other hand, if they consist of only rotation fields, then we see that
\begin{align*}
    & \int_{\tau_i}^{\tau_{i+1}} r^{-\frac12-\delta} \|\nabla_{\underline L}\mathscr L_Z^{\le N-8}\psi_+\|_{L^2(\Sigma_\tau)} \|\mathscr L_Z^{\le N-1}\psi_-\|_{L^2(\Sigma_\tau)} \|\mathcal L_Z^{\le N-1}h_{A\underline L}\|_{L^2(\Sigma_\tau)}\,d\tau \\
    & \lesssim \int_{\tau_i}^{\tau_{i+1}} r^{\frac32-\delta} \mathscr E^D_{\le N-7}[\psi] \mathcal E^D_{\le N-2}[\psi] \mathcal E^W_{\le N-2}[h_{A\underline L}]\,d\tau \\
    & \lesssim  (\tau_{i+1})^\frac12 \left( \int_{\tau_i}^{\tau_{i+1}} \mathscr E^D_{\le N-7}[\psi]^2(\tau) {}^{(1-\delta)}\mathcal E^D_{\le N-2}[\psi]^2(\tau)\,d\tau \right)^\frac12 \left( \int_{\tau_i}^{\tau_{i+1}} {}^{(1-\delta)}\mathcal E^W_{\le N-2}[h_{A\underline L}]^2(\tau)\,d\tau\right)^\frac12.
\end{align*}
\begin{rem}\label{rem-energy-hierarchy1}
We emphasise that the Dirac current energy plays an important role in closing the nonlinear estimates even for relatively low-order terms. Indeed, if the spinor field were treated solely through its associated wave equation, the $L^2$-norm of $\nabla_{\underline L}\psi$ on outgoing null hypersurfaces could not be directly controlled by the corresponding energy flux. For the $\psi_+$ component, however, the required transversal derivative can be controlled by applying the Dirac current estimate to the appropriately commuted spinor field, without any additional loss of derivatives or radial weights.

Alternatively, one could invoke the Hardy inequality and estimate the term through the wave energy, at the expense of an additional derivative and a loss in the radial weight. Such losses may result in unnecessarily large growth in the energy estimates. The use of the Dirac current energy therefore allows us to preserve a more favorable hierarchy of nonlinear estimates.
\end{rem}
Repeating the argument for lower-order energies with $k\le N-2$ gives
\begin{align}
    \begin{aligned}
        &  \int_{\tau_i}^{\tau_{i+1}} r^{-\frac12-\delta} \|\nabla_{\underline L}\mathscr L_Z^{\le N-9}\psi_+\|_{L^2(\Sigma_\tau)} \|\mathscr L_Z^{\le N-2}\psi_-\|_{L^2(\Sigma_\tau)} \|\mathcal L_Z^{\le N-2}h_{A\underline L}\|_{L^2(\Sigma_\tau)}\,d\tau \\
        & \lesssim (\tau_{i+1})^{\frac12} \sup_{\tau\in[\tau_i,\tau_{i+1}]} \mathscr E^D_{\le N-8}[\psi](\tau) \left( \int_{\tau_i}^{\tau_{i+1}} {}^{(1-\delta)}\mathcal E^D_{\le N-3}[\psi]^2(\tau)\,d\tau \right)^\frac12 \left( \int_{\tau_i}^{\tau_{i+1}} {}^{(1-\delta)}\mathcal E^W_{\le N-3}[h_{A\underline L}]^2(\tau)\,d\tau\right)^\frac12.
    \end{aligned}
\end{align}
For $k\le N-3$, it suffices to consider the spacetime integral of the form:
\begin{align}
    \int_{\mathcal D^{\tau_{i+1}}_{\tau_i}} r^{1-\delta} \nabla_{\underline L}\mathscr L_Z^{\le N-11}\psi_+ \cdot \mathscr L_Z^{\le N-3}\psi_- \cdot \mathcal L_Z^{\le N-3}h_{A\underline L}\,dV.
\end{align}
We simply use H\"older's inequality to obtain the $L^6\times L^6\times L^6$ product and the weighted Sobolev inequality and then we obtain
\begin{align}
    \begin{aligned}
        & \left| \int_{\mathcal D^{\tau_{i+1}}_{\tau_i}} r^{1-\delta} \nabla_{\underline L}\mathscr L_Z^{\le N-11}\psi_+ \cdot \mathscr L_Z^{\le N-3}\psi_- \cdot \mathcal L_Z^{\le N-3}h_{A\underline L}\,dV \right| \\
        & \lesssim \int_{\tau_i}^{\tau_{i+1}} r^{1-\delta} \mathcal E^D_{\le N-8}[\psi] \mathcal E^D_{\le N-2}[\psi]\mathcal E^W_{\le N-2}[h_{A\underline L}]\,d\tau \\
        & \lesssim \left( \int_{\tau_i}^{\tau_{i+1}}{}^{(0)}\mathcal E^D_{\le N-8}[\psi]^2(\tau) {}^{(1-\delta)}\mathcal E^D_{\le N-2}[\psi]^2(\tau)\,d\tau  \right)^\frac12 \left( \int_{\tau_i}^{\tau_{i+1}} {}^{(1-\delta)}\mathcal E^W_{\le N-2}[h_{A\underline L}]^2(\tau)\,d\tau \right)^\frac12.
    \end{aligned}
\end{align}
On the other hand, we also have to consider the spacetime integral, where the derivative $\nabla_{\underline L}$ acts on the higher-order term $\mathscr L_Z^{\le N-1}\psi_-$:
\begin{align}
    \int_{\mathcal D^{\tau_{i+1}}_{\tau_i}} r^{1-\delta} \mathscr L_Z^{\le N-11}\psi_+\cdot \nabla_{\underline L}\mathscr L_Z^{\le N-1}\psi_-\cdot \mathcal L_Z^{\le N-1}h_{A\underline L}\,dV.
\end{align}
Using the null decomposition for the spinor fields, we write
\begin{align}
\begin{aligned}
    & \left|  \int_{\mathcal D^{\tau_{i+1}}_{\tau_i}} r^{1-\delta} \mathscr L_Z^{\le N-11}\psi_+\cdot \nabla_{\underline L}\mathscr L_Z^{\le N-1}\psi_-\cdot \mathcal L_Z^{\le N-1}h_{A\underline L}\,dV \right| \\
    & \lesssim  \int_{\mathcal D^{\tau_{i+1}}_{\tau_i}} r^{1-\delta}  | \mathscr L_Z^{\le N-11}\psi_+|  |\slashed\nabla_{}\mathscr L_Z^{\le N-1}\psi_+| | \mathcal L_Z^{\le N-1}h_{A\underline L}|\,dV \\
    & \qquad +  \int_{\mathcal D^{\tau_{i+1}}_{\tau_i}} r^{1-\delta} |\mathscr L_Z^{\le N-11}\psi_+ |  | \mathscr L_Z^{\le N-1}(\nabla_{\underline L}\mathrm{Tr}h \cdot \psi_-)|  | \mathcal L_Z^{\le N-1}h_{A\underline L}| \,dV \\
    & \qquad + \int_{\mathcal D^{\tau_{i+1}}_{\tau_i}} r^{1-\delta} |\mathscr L_Z^{\le N-11}\psi_+ |  | \mathscr L_Z^{\le N-1}(\slashed\nabla_{}\mathrm{Tr}h \cdot \psi_-)|  | \mathcal L_Z^{\le N-1}h_{A\underline L}| \,dV
    \end{aligned}
\end{align}
For the first integral, we simply use the weighted Sobolev inequality on outgoing null hypersurfaces and then the Hardy inequality with the Cauchy-Schwarz inequality. For the second integral, we only consider the high-low and the low-high interactions:
\begin{align}
    \begin{aligned}
       & \int_{\mathcal D^{\tau_{i+1}}_{\tau_i}} r^{1-\delta} |\mathscr L_Z^{\le N-11}\psi_+ |  | \nabla_{\underline L}\mathscr L_Z^{\le N-1} \mathrm{Tr}h \cdot \mathscr L_Z^{\le N-11} \psi_-|  | \mathcal L_Z^{\le N-1}h_{A\underline L}| \,dV \\
       & \qquad +  \int_{\mathcal D^{\tau_{i+1}}_{\tau_i}} r^{1-\delta} |\mathscr L_Z^{\le N-11}\psi_+ |  | \nabla_{\underline L}\mathscr L_Z^{\le N-11} \mathrm{Tr}h \cdot \mathscr L_Z^{\le N-1} \psi_-|  | \mathcal L_Z^{\le N-1}h_{A\underline L}| \,dV.
    \end{aligned}
\end{align}
If the commutation fields acting on the higher-order term include at least one $L$ or one of the rotation fields, H\"older's inequality gives an $L^\infty\times L^2\times L^\infty\times L^2$ product, and the weighted Sobolev inequality on outgoing null hypersurfaces yields
\begin{align*}
    & \int_{\mathcal D^{\tau_{i+1}}_{\tau_i}} r^{1-\delta} |\mathscr L_Z^{\le N-11}\psi_+ |  | \nabla_{\underline L}\mathscr L_Z^{\le N-1} \mathrm{Tr}h \cdot \mathscr L_Z^{\le N-11} \psi_-)|  | \mathcal L_Z^{\le N-1}h_{A\underline L}| \,dV \\
    & \lesssim \int_{\tau_i}^{\tau_{i+1}} r^{2-\delta} \mathcal E^D_{\le N-9}[\psi]^2(\tau) \mathcal E^W_{\le N-1}[h_{L\underline L}](\tau)\mathcal E^W_{\le N-1}[h_{A\underline L}]\,d\tau ,
\end{align*}
where the remaining estimate follows directly. If the vector fields of the higher-order term consist of only $\underline L$, we use the foliation by the ingoing cones $\underline{\mathcal C}_v$ and use the weighted Sobolev embedding to get
\begin{align*}
    & \int_{\mathcal D^{\tau_{i+1}}_{\tau_i}} r^{1-\delta} |\mathscr L_Z^{\le N-11}\psi_+ |  | \nabla_{\underline L}\mathscr L_Z^{\le N-1} \mathrm{Tr}h \cdot \mathscr L_Z^{\le N-11} \psi_-)|  | \mathcal L_Z^{\le N-1}h_{A\underline L}| \,dV \\
    & \lesssim \sup_{v:\tau\le\tau(v)}{}^{(0)}\mathcal F^D_{\le N-9}[\psi]{}^{(0)}\mathcal F^W_{\le N-1}[h_{L\underline L}] \int r^{-\delta} \|\slashed\nabla \mathscr L_Z^{\le N-9}\psi_+\|_{L^2(\underline{\mathcal C}_v)} \|\mathcal L_Z^{\le N-1}h_{A\underline L}\|_{L^2(\underline{\mathcal C}_v)}\,dv  \\
    & \qquad + \sup_{v:\tau\le\tau(v)}{}^{(0)}\mathcal F^D_{\le N-9}[\psi]{}^{(0)}\mathcal F^W_{\le N-1}[h_{L\underline L}] \int r^{-\delta} \| \mathscr L_Z^{\le N-9}(\nabla_{\underline L}\mathrm{Tr}h\cdot \psi_-)\|_{L^2(\underline{\mathcal C}_v)} \|\mathcal L_Z^{\le N-1}h_{A\underline L}\|_{L^2(\underline{\mathcal C}_v)}\,dv  \\
    & \qquad + \sup_{v:\tau\le\tau(v)}{}^{(0)}\mathcal F^D_{\le N-9}[\psi]{}^{(0)}\mathcal F^W_{\le N-1}[h_{L\underline L}] \int r^{-\delta} \| \mathscr L_Z^{\le N-9}(\slashed\nabla_{}\mathrm{Tr}h\cdot \psi_-)\|_{L^2(\underline{\mathcal C}_v)} \|\mathcal L_Z^{\le N-1}h_{A\underline L}\|_{L^2(\underline{\mathcal C}_v)}\,dv  ,
\end{align*}
The remaining estimate follows directly. The low-high interaction is handled similarly.

\subsection{Control of $h_{\underline L\underline L}$ }
It remains to estimate the energy of the component $h_{\underline L\underline L}$. We begin with the integral
\begin{align}
\int_{\mathcal D^{\tau_{i+1}}_{\tau_i}} r^{2-\delta} \langle \mathscr L_Z^{\le N-11} \psi, \gamma^L \nabla_{\underline L} \mathscr L_Z^{\le N-1}\psi\rangle \nabla_{L}\mathcal L_Z^{\le N-1}h_{\underline L\underline L}\,dV	.
\end{align}
We first use the integration by parts in the direction $\nabla_{\underline L}$. As in the preceding estimate, the delicate case occurs when the derivative $\nabla_{\underline L}$ acts on the term $\nabla_L\mathcal L_Z^{\le N-1}h_{\underline L\underline L} $, in which case we encounter a derivative loss. Using the wave equation for the field $h_{\underline L\underline L}$, the most serious integral is the following:
\begin{align}
    \int_{\mathcal D^{\tau_{i+1}}_{\tau_i}} r^{1-\delta} \langle \mathscr L_Z^{\le N-11}\psi, \gamma^L \mathscr L_Z^{\le N-1}\psi\rangle \nabla_{\underline L}\mathcal L_Z^{\le N-1}h_{\underline L\underline L}\,dV.
\end{align}
To deal with the above integral, we apply the integration by parts along the $\underline L$-direction once again, and hence we encounter the following spacetime integrals:
\begin{align}
     \int_{\mathcal D^{\tau_{i+1}}_{\tau_i}} r^{1-\delta} |\nabla_{\underline L} \mathscr L_Z^{\le N-11}\psi_+|  | \mathscr L_Z^{\le N-1}\psi_+|  |\mathcal L_Z^{\le N-1}h_{\underline L\underline L}|\,dV ,\\
      \int_{\mathcal D^{\tau_{i+1}}_{\tau_i}} r^{1-\delta} | \mathscr L_Z^{\le N-11}\psi_+|  |\nabla_{\underline L} \mathscr L_Z^{\le N-1}\psi_+|  |\mathcal L_Z^{\le N-1}h_{\underline L\underline L}|\,dV.
\end{align}
For the second integral, we use the weighted Sobolev inequality together with the Hardy inequality to get
\begin{align*}
    & \int_{\mathcal D^{\tau_{i+1}}_{\tau_i}} r^{1-\delta} | \mathscr L_Z^{\le N-11}\psi_+|  |\nabla_{\underline L} \mathscr L_Z^{\le N-1}\psi_+|  |\mathcal L_Z^{\le N-1}h_{\underline L\underline L}|\,dV \\
    & \lesssim \int_{\tau_i}^{\tau_{i+1}} r^{\frac32-\delta} \mathcal E^D_{\le N-9}[\psi] \mathscr E^D_{\le N}[\psi] \mathcal E^W_{\le N-1}[h_{\underline L\underline L}]\,d\tau \\
    & \lesssim (\tau_{i+1})^\frac12 \sup_{\tau\in[\tau_i,\tau_{i+1}]} \mathscr E^D_{\le N}[\psi](\tau) \left( \int_{\tau_i}^{\tau_{i+1}} {}^{(1-\delta)}  \mathcal E^D_{\le N-9}[\psi]^2(\tau)\,d\tau   \right)^\frac12 \left( \int_{\tau_i}^{\tau_{i+1}} {}^{(1-\delta)}\mathcal E^W_{\le N-1}[h_{\underline L\underline L}]^2(\tau)\,d\tau \right)^\frac12 .
\end{align*}
This estimate does not provide sufficient control to close the top-order weighted energy bound for $h_{\underline L\underline L}$.
For the unfavorable component $h_{\underline L\underline L}$, we only establish the weighted energy estimates ${}^{(2-\delta)}\mathcal E^W_{\le k}[h_{\underline L\underline L}]$ for $k\le N-3$, without extending them to the top order. Indeed, such higher-order weighted estimates are never required at any stage of the nonlinear closure argument.

In what follows, we control the weighted energy of $h_{\underline L\underline L}$ as follows:
\begin{align}
    {}^{(2-\delta)}\mathcal E^W_{\le N-3}[h_{\underline L\underline L}] \Rightarrow {}^{(1)}\mathcal E^W_{\le N-2}[h_{\underline L\underline L}] \Rightarrow {}^{(0)}\mathcal E^W_{\le N-1}[h_{\underline L\underline L}].
\end{align}

\subsubsection{Control of ${}^{(2-\delta)}\mathcal E^W_{\le N-3}[h_{\underline L\underline L}]$ }
We consider the control of the weighted energy ${}^{(2-\delta)}\mathcal E^W_{\le N-3}[h_{\underline L\underline L}]$. To do this, we need to control the integral
\begin{align}
\int_{\mathcal D^{\tau_{i+1}}_{\tau_i}} r^{2-\delta} \langle \mathscr L_Z^{\le N-11} \psi, \gamma^L \nabla_{\underline L} \mathscr L_Z^{\le N-3}\psi\rangle \nabla_{L}\mathcal L_Z^{\le N-3}h_{\underline L\underline L}\,dV	.
\end{align}
By the integration by parts along the $\underline L$-direction, we have
\begin{align*}
	& \int_{\mathcal D^{\tau_{i+1}}_{\tau_i}} r^{2-\delta} \langle \mathscr L_Z^{\le N-11}\psi, \gamma^L \nabla_{\underline L} \mathscr L_Z^{\le N-3}\psi\rangle \nabla_{L}\mathcal L_Z^{\le N-3}h_{\underline L\underline L}\,dV \\
	& = \int_{\substack{ u=\tau_{i+1}-R \\ v\ge \tau_{i+1}+R }} 	r^{2-\delta} \langle \mathscr L_Z^{\le N-11}\psi, \gamma^L  \mathscr L_Z^{\le N-3}\psi\rangle \nabla_{L}\mathcal L_Z^{\le N-3}h_{\underline L\underline L}\,d\sigma -  \int_{\substack{ u=\tau_{i}-R \\ v\ge \tau_{i}+R }} 	r^{2-\delta} \langle \mathscr L_Z^{\le N-11}\psi, \gamma^L  \mathscr L_Z^{\le N-3}\psi\rangle \nabla_{L}\mathcal L_Z^{\le N-3}h_{\underline L\underline L}\,d\sigma \\
	& + \int_{r=R}   	r^{2-\delta} \langle \mathscr L_Z^{\le N-11}\psi, \gamma^L  \mathscr L_Z^{\le N-3}\psi\rangle \nabla_{L}\mathcal L_Z^{\le N-3}h_{\underline L\underline L}\,d\sigma \\
	& + \int_{\mathcal D^{\tau_{i+1}}_{\tau_i}} r^{2-\delta} \langle \nabla_{\underline L}\mathscr L_Z^{\le N-11} \psi, \gamma^L \mathscr L_Z^{\le N-3}\psi\rangle \nabla_{L}\mathcal L_Z^{\le N-3}h_{\underline L\underline L}\,dV + \int_{\mathcal D^{\tau_{i+1}}_{\tau_i}} r^{2-\delta} \langle \mathscr L_Z^{\le N-11} \psi, \gamma^L \mathscr L_Z^{\le N-3}\psi\rangle \nabla_{\underline L} \nabla_{L}\mathcal L_Z^{\le N-3}h_{\underline L\underline L}\,dV.
\end{align*}
Applying H\"older's inequality with exponents $4,4,2$, followed by the Sobolev inequality, yields
\begin{align*}
	& \left| \int_{\substack{ u=\tau_{i+1}-R \\ v\ge \tau_{i+1}+R }} 	r^{2-\delta} \langle \mathscr L_Z^{\le N-11}\psi, \gamma^L  \mathscr L_Z^{\le N-3}\psi\rangle \nabla_{L}\mathcal L_Z^{\le N-3}h_{\underline L\underline L}\,d\sigma  \right| \\
	& \lesssim {}^{(0)}\mathcal E^D_{\le N-10}[\psi](\tau_{i+1}) {}^{(2-\delta)}\mathcal E^D_{\le N-2}[\psi] (\tau_{i+1}) {}^{(2-\delta)}\mathcal E^W_{\le N-3}[h](\tau_{i+1}),
\end{align*}
The boundary integral along the timelike surface $\{r=R\}$ is controlled directly.

For the spacetime integral in which $\nabla_{\underline L}$ acts on the lower-order term $\mathscr L_Z^{\le N-11}\psi_+$, we first apply the weighted Sobolev inequality and then use the null decomposition to obtain
\begin{align*}
	& \left|  \int_{\mathcal D^{\tau_{i+1}}_{\tau_i}} r^{2-\delta} \langle \nabla_{\underline L} \mathscr L_Z^{\le N-11} \psi, \gamma^L \mathscr L_Z^{\le N-3}\psi\rangle \nabla_{L}\mathcal L_Z^{\le N-3}h_{\underline L\underline L}\,dV  \right| \\
	& \lesssim  \int_{\mathcal D^{\tau_{i+1}}_{\tau_i}} r^{2-\delta} | \nabla_{\underline L}\mathscr L_Z^{\le N-11} \psi_+| | \mathscr L_Z^{\le N-3}\psi_+||\nabla_{L}\mathcal L_Z^{\le N-3}h_{\underline L\underline L}|\,dV \\
	& \lesssim \int_{\tau_i}^{\tau_{i+1}} r^{2-\delta} \| \nabla_{L}\nabla_{\underline L}\mathscr L_Z^{\le  N-10}\psi_+\|_{L^2(\Sigma_\tau)} \mathcal E^D_{\le N-2}[\psi] \mathcal E^W_{\le N-3}[h]\,d\tau \\
	& \lesssim  \int_{\tau_i}^{\tau_{i+1}} r^{1-\delta} \mathcal E^D_{\le N-9}[\psi] \mathcal E^D_{\le N-2}[\psi] \mathcal E^W_{\le N-3}[h_{\underline L\underline L}]\,d\tau  + \int_{\tau_i}^{\tau_{i+1}} r^{1-\delta} \|\nabla_{\underline L}\mathscr L_Z^{\le 1}\psi_+\|_{L^2(\Sigma_\tau)} \mathcal E^D_{\le N-2}[\psi] \mathcal E^W_{\le N-3}[h_{\underline L\underline L}]\,d\tau \\
	& \qquad + \int_{\tau_i}^{\tau_{i+1}} r^{1-\delta} \| \nabla_{L}\mathscr L_Z^{\le N-10}\psi_+\|_{L^2(\Sigma_\tau)} \mathcal E^D_{\le N-2}[\psi] \mathcal E^W_{\le N-3}[h_{\underline L\underline L}]\,d\tau \\
    & \qquad + \int_{\tau_i}^{\tau_{i+1}} r^{2-\delta} \|\mathscr L_Z^{\le N-10} (\mathscr T_{L\underline L}[\psi]\psi) \|_{L^2(\Sigma_\tau)} \mathcal E^D_{\le N-2}[\psi] \mathcal E^W_{\le N-3}[h_{\underline L\underline L}]\,d\tau \\
    &\qquad + \int_{\tau_i}^{\tau_{i+1}} r^{2-\delta} \|\mathscr L_Z^{\le N-10} (\mathscr T_{AB}[\psi] \psi) \|_{L^2(\Sigma_\tau)} \mathcal E^D_{\le N-2}[\psi] \mathcal E^W_{\le N-3}[h_{\underline L\underline L}]\,d\tau    \\
    & \qquad + \int_{\tau_i}^{\tau_{i+1}} r^{2-\delta} \|\mathscr L_Z^{\le N-10} ( \nabla_{\nu}\mathrm{Tr}h \gamma^\lambda\gamma^\nu \nabla_\lambda \psi)  \|_{L^2(\Sigma_\tau)} \mathcal E^D_{\le N-2}[\psi] \mathcal E^W_{\le N-3}[h_{\underline L\underline L}]\,d\tau 
\end{align*}
For the first integral, we have
\begin{align*}
    & \int_{\tau_i}^{\tau_{i+1}} r^{1-\delta} \mathcal E^D_{\le N-9}[\psi] \mathcal E^D_{\le N-2}[\psi] \mathcal E^W_{\le N-3}[h_{\underline L\underline L}]\,d\tau   \\
    & \lesssim \left( \int_{\tau_i}^{\tau_{i+1}} {}^{(0)}\mathcal E^D_{\le N-9}[\psi]^2(\tau){}^{(1-\delta)}\mathcal E^D_{\le N-2}[\psi]^2(\tau)\,d\tau \right)^\frac12 \left( \int_{\tau_i}^{\tau_{i+1}} {}^{(1-\delta)}\mathcal E^W_{\le N-3}[h_{\underline L\underline L}]^2(\tau)\,d\tau \right)^\frac12,
\end{align*}
and for the second integral,
\begin{align*}
    & \int_{\tau_i}^{\tau_{i+1}} r^{1-\delta} \|\nabla_{\underline L}\mathscr L_Z^{\le N-10}\psi_+\|_{L^2(\Sigma_\tau)} \mathcal E^D_{\le N-2}[\psi] \mathcal E^W_{\le N-3}[h_{\underline L\underline L}]\,d\tau  \\
    & \lesssim \left( \int_{\tau_i}^{\tau_{i+1}} \mathscr E^D_{\le N-9}[\psi]^2(\tau)  {}^{(1-\delta)}\mathcal E^D_{\le N-2}[\psi]^2(\tau)\,d\tau \right)^\frac12 \left( \int_{\tau_i}^{\tau_{i+1}} {}^{(1-\delta)}\mathcal E^W_{\le N-3}[h_{\underline L\underline L}]^2(\tau)\,d\tau \right)^\frac12 \\
    & \lesssim \sup_{\tau\in [\tau_i,\tau_{i+1}] } \mathscr E^D_{\le N-9}[\psi](\tau) \left( \int_{\tau_i}^{\tau_{i+1}}   {}^{(1-\delta)}\mathcal E^D_{\le N-2}[\psi]^2(\tau)\,d\tau \right)^\frac12 \left( \int_{\tau_i}^{\tau_{i+1}} {}^{(1-\delta)}\mathcal E^W_{\le N-3}[h_{\underline L\underline L}]^2(\tau)\,d\tau \right)^\frac12,
\end{align*}
and for the integral including an additional nonlinearity of the energy-momentum tensor for the spinor fields, we see that
\begin{align*}
    & \int_{\tau_i}^{\tau_{i+1}} r^{2-\delta} \|\mathscr L_Z^{\le N-10}( \mathscr T_{L\underline L}[\psi]\psi) \|_{L^2(\Sigma_\tau)} \mathcal E^D_{\le N-2}[\psi] \mathcal E^W_{\le N-3}[h_{\underline L\underline L}]\,d\tau \\
    & \qquad + \int_{\tau_i}^{\tau_{i+1}} r^{2-\delta} \|\mathscr L_Z^{\le N-10} \mathscr T_{AB}[\psi] \|_{L^2(\Sigma_\tau)} \mathcal E^D_{\le N-2}[\psi] \mathcal E^W_{\le N-3}[h_{\underline L\underline L}]\,d\tau \\
    & \lesssim \int_{\tau_i}^{\tau_{i+1}} r^{\frac32-\delta} \mathcal E^D_{\le N-9}[\psi]^3 \mathcal E^D_{\le N-2}[\psi] \mathcal E^W_{\le N-3}[h_{\underline L\underline L}]\,d\tau \\
    & \lesssim \left( \int_{\tau_i}^{\tau_{i+1}} {}^{(0)}\mathcal E^D_{\le N-9}[\psi]^2(\tau) {}^{(1)} \mathcal E^D_{\le N-9}[\psi]^4(\tau)  {}^{(1-\delta)}\mathcal E^D_{\le N-2}[\psi]^2(\tau)\,d\tau \right)^\frac12 \left( \int_{\tau_i}^{\tau_{i+1}} {}^{(1-\delta)}\mathcal E^W_{\le N-3}[h_{\underline L\underline L}]^2(\tau)\,d\tau \right)^\frac12 \\
    & \lesssim \sup_{\tau\in[\tau_i,\tau_{i+1}] }{}^{(0)}\mathcal E^D_{\le N-9}[\psi](\tau) {}^{(1)} \mathcal E^D_{\le N-9}[\psi]^2(\tau) \left( \int_{\tau_i}^{\tau_{i+1}}  {}^{(1-\delta)}\mathcal E^D_{\le N-2}[\psi]^2(\tau)\,d\tau \right)^\frac12 \left( \int_{\tau_i}^{\tau_{i+1}} {}^{(1-\delta)}\mathcal E^W_{\le N-3}[h_{\underline L\underline L}]^2(\tau)\,d\tau \right)^\frac12 .
\end{align*}
The last integral including an additional nonlinearity $\nabla h\cdot\nabla\psi$ is also controlled similarly. Indeed, we have
\begin{align*}
    & \int_{\tau_i}^{\tau_{i+1}} r^{2-\delta} \|\mathscr L_Z^{\le N-10} ( \nabla_{\nu}\mathrm{Tr}h \gamma^\lambda\gamma^\nu \nabla_\lambda \psi)  \|_{L^2(\Sigma_\tau)} \mathcal E^D_{\le N-2}[\psi] \mathcal E^W_{\le N-3}[h_{\underline L\underline L}]\,d\tau \\
    & \lesssim \int_{\tau_i}^{\tau_{i+1}} r^{2-\delta} \mathcal E^D_{\le N-9}[\psi] \mathcal E^W_{\le N-9}[h_{\underline L\underline L}]  \mathcal E^D_{\le N-2}[\psi] \mathcal E^W_{\le N-3}[h_{\underline L\underline L}]\,d\tau \\
    & \lesssim \left( \int_{\tau_i}^{\tau_{i+1}} {}^{(1)}\mathcal E^D_{\le N-9}[\psi]^2(\tau)\mathcal E^W_{\le N-9}[h_{\underline L\underline L}]^2(\tau) {}^{(1-\delta)}\mathcal E^D_{\le N-2}[\psi]^2(\tau)\,d\tau \right)^\frac12 \left( \int_{\tau_i}^{\tau_{i+1}} {}^{(1-\delta)}\mathcal E^W_{\le N-3}[h_{\underline L\underline L}]^2(\tau)\,d\tau \right)^\frac12 \\
    & \lesssim \sup_{\tau\in[\tau_i,\tau_{i+1}] }{}^{(1)}\mathcal E^D_{\le N-9}[\psi](\tau)\mathcal E^W_{\le N-9}[h_{\underline L\underline L}](\tau) \left( \int_{\tau_i}^{\tau_{i+1}}  {}^{(1-\delta)}\mathcal E^D_{\le N-2}[\psi]^2(\tau)\,d\tau \right)^\frac12 \left( \int_{\tau_i}^{\tau_{i+1}} {}^{(1-\delta)}\mathcal E^W_{\le N-3}[h_{\underline L\underline L}]^2(\tau)\,d\tau \right)^\frac12
\end{align*}
As the derivative $\nabla_{\underline L}$ acts on the higher-order term $h_{\underline L\underline L}$, we see that
\begin{align*}
	& \left| \int_{\mathcal D^{\tau_{i+1}}_{\tau_i}} r^{2-\delta} \langle \mathscr L_Z^{\le N-11} \psi, \gamma^L \mathscr L_Z^{\le N-3}\psi\rangle \nabla_{\underline L} \nabla_{L}\mathcal L_Z^{\le N-3}h_{\underline L\underline L}\,dV \right| \\
	& \lesssim \int_{\mathcal D^{\tau_{i+1}}_{\tau_i}} r^{2-\delta} \langle \mathscr L_Z^{\le N-11}\psi, \gamma^L \mathscr L_Z^{\le N-3}\psi\rangle \slashed\Delta\mathcal L_Z^{\le N-3}h_{\underline L\underline L}\,dV + \int_{\mathcal D^{\tau_{i+1}}_{\tau_i}} r^{1-\delta} \langle \mathscr L_Z^{\le N-11}\psi, \gamma^L \mathscr L_Z^{\le N-3}\psi\rangle \nabla_{\underline L} \mathcal L_Z^{\le N-3}h_{\underline L\underline L}\,dV \\
	& \qquad + \int_{\mathcal D^{\tau_{i+1}}_{\tau_i}} r^{1-\delta} \langle \mathscr L_Z^{\le N-11} \psi, \gamma^L \mathscr L_Z^{\le N-3}\psi\rangle  \nabla_{L}\mathcal L_Z^{\le N-3}h_{\underline L\underline L}\,dV + \int_{\mathcal D^{\tau_{i+1}}_{\tau_i}} r^{2-\delta} \langle \mathscr L_Z^{\le N-11}\psi, \gamma^L \mathscr L_Z^{\le N-3}\psi\rangle \mathcal L_Z^{\le N-3} \mathscr T_{\underline L\underline L} \,dV 
\end{align*}
For the first integral, we use the integration by parts with respect to the angular variables. In fact, we have
\begin{align*}
    & \left| \int_{\mathcal D^{\tau_{i+1}}_{\tau_i}} r^{2-\delta} \langle \mathscr L_Z^{\le N-11}\psi, \gamma^L \mathscr L_Z^{\le N-3}\psi\rangle \slashed\Delta\mathcal L_Z^{\le N-3}h_{\underline L\underline L}\,dV \right| \\
    & \lesssim  \int_{\mathcal D^{\tau_{i+1}}_{\tau_i}} r^{2-\delta} |\slashed\nabla( \mathscr L_Z^{\le N-11}\psi_+ \cdot \mathscr L_Z^{\le N-3}\psi_+) | |\slashed\nabla\mathcal L_Z^{\le N-3}h_{\underline L\underline L}|\,dV \\
    & \lesssim \int_{\tau_i}^{\tau_{i+1}}r^{1-\delta} \|\mathscr L_Z^{\le N-11}\psi_+\|_{L^4(\Sigma_\tau)}\|\mathscr L_Z^{\le N-2}\psi_+\|_{L^4(\Sigma_\tau)} \mathcal E^W_{\le N-3}[h_{\underline L\underline L}]\,d\tau \\
    & \lesssim \int_{\tau_i}^{\tau_{i+1}}r^{1-\delta}  \mathcal E^D_{\le N-9}[\psi] \mathcal E^D_{\le N-1}[\psi] \mathcal E^W_{\le N-3}[h_{\underline L\underline L}]\,d\tau \\
    & \lesssim \left( \int_{\tau_i}^{\tau_{i+1}} {}^{(0)}\mathcal E^D_{\le N-9}[\psi]^2(\tau) {}^{(1-\delta)}\mathcal E^D_{\le N-1}[\psi]^2(\tau)\,d\tau  \right)^\frac12 \left( \int_{\tau_i}^{\tau_{i+1}} {}^{(1-\delta)} \mathcal E^W_{\le N-3}[h_{\underline L\underline L}]^2(\tau)\,d\tau   \right)^\frac12 \\
    & \lesssim \sup_{\tau\in[\tau_i,\tau_{i+1}]} {}^{(0)}\mathcal E^D_{\le N-9}[\psi](\tau) \left( \int_{\tau_i}^{\tau_{i+1}}  {}^{(1-\delta)}\mathcal E^D_{\le N-1}[\psi]^2(\tau)\,d\tau  \right)^\frac12 \left( \int_{\tau_i}^{\tau_{i+1}} {}^{(1-\delta)} \mathcal E^W_{\le N-3}[h_{\underline L\underline L}]^2(\tau)\,d\tau   \right)^\frac12 .
\end{align*}
The third integral is estimated as the first, using H\"older's inequality with exponents $4,4,2$; we omit the details. For the fourth integral, which contains an additional nonlinear term, we have
\begin{align*}
    & \int_{\mathcal D^{\tau_{i+1}}_{\tau_i}} r^{2-\delta} \langle \mathscr L-Z^{\le N-11}\psi, \gamma^L \mathscr L_Z^{\le N-3}\psi\rangle \mathcal L_Z^{\le N-3} \mathscr T_{\underline L\underline L} \,dV  \\
    & \lesssim \int_{\tau_i}^{\tau_{i+1}} r^{2-\delta}\| \mathscr L_Z^{\le N-11} \psi_+\|_{L^\infty(\Sigma_\tau)}  \| \mathscr L_Z^{\le N-3}\psi_+\|_{L^4(\Sigma_\tau)} \|\mathscr L_Z^{\le N-11}\psi_+\|_{L^\infty(\Sigma_\tau)} \|\mathscr L_Z^{\le N-2}\psi_+\|_{L^4(\Sigma_\tau)}  \,d\tau  \\
    & \lesssim \int_{\tau_i}^{\tau_{i+1}} r^{1-\delta} \mathcal E^D_{\le N-9}[\psi]^2 \mathcal E^D_{\le N-1}[\psi]^2\,d\tau \\
    & \lesssim \sup_{\tau\in[\tau_{i},\tau_{i+1}]}{}^{(0)} \mathcal E^D_{\le N-9}[\psi]^2(\tau) \int_{\tau_i}^{\tau_{i+1}} {}^{(1-\delta)}  \mathcal E^D_{\le N-1}[\psi]^2(\tau)\,d\tau.
\end{align*}
It remains to estimate the second integral, which is the most delicate.
For this integral, we again integrate by parts in the $\underline L$ direction to obtain
\begin{align*}
	& \int_{\mathcal D^{\tau_{i+1}}_{\tau_i}} r^{1-\delta} \langle \mathscr L_Z^{\le N-11} \psi, \gamma^L \mathscr L_Z^{\le N-3}\psi\rangle \nabla_{\underline L} \mathcal L_Z^{\le N-3}h_{\underline L\underline L}\,dV \\
	& = \int_{\substack{u=\tau_{i+1}-R \\ v\ge \tau_{i+1}+ R}} r^{1-\delta} \langle \mathscr L_Z^{\le N-11} \psi, \gamma^L \mathscr L_Z^{\le N-3}\psi\rangle  \mathcal L_Z^{\le N-3}h_{\underline L\underline L}\,d\sigma -  \int_{\substack{u=\tau_{i}-R \\ v\ge \tau_{i}+ R}} r^{1-\delta} \langle \mathscr L_Z^{\le N-11}\psi, \gamma^L \mathscr L_Z^{\le N-3}\psi\rangle  \mathcal L_Z^{\le N-3}h_{\underline L\underline L}\,d\sigma \\
    & + \int_{r=R} r^{1-\delta} \langle \mathscr L_Z^{\le N-11} \psi, \gamma^L \mathscr L_Z^{\le N-3}\psi\rangle  \mathcal L_Z^{\le N-3}h_{\underline L\underline L}\,d\sigma \\
	& \qquad + \int_{\mathcal D^{\tau_{i+1}}_{\tau_i}} r^{1-\delta} \langle \nabla_{\underline L} \mathscr L_Z^{\le N-11} \psi, \gamma^L \mathscr L_Z^{\le N-3}\psi\rangle \mathcal L_Z^{\le N-3}h_{\underline L\underline L}\,dV \\
    & \qquad +  \int_{\mathcal D^{\tau_{i+1}}_{\tau_i}} r^{1-\delta} \langle  \mathscr L_Z^{\le N-11}  \psi, \gamma^L \nabla_{\underline L} \mathscr L_Z^{\le N-3}\psi\rangle \mathcal L_Z^{\le N-3}h_{\underline L\underline L}\,dV 
\end{align*}
We now turn to the spacetime integrals. The energy identity controls the boundary terms on outgoing null hypersurfaces and timelike surfaces.

For the integral in which $\nabla_{\underline L}$ falls on the higher-order term $\mathscr L_Z^{\le N-3}\psi$, we apply H\"older's inequality with exponents $6,6,6$, followed by the weighted Sobolev inequality, to obtain
\begin{align*}
    & \left| \int_{\mathcal D^{\tau_{i+1}}_{\tau_i}} r^{1-\delta} \langle  \mathscr L_Z^{\le N-11} \psi, \gamma^L \nabla_{\underline L} \mathscr L_Z^{\le N-3}\psi\rangle \mathcal L_Z^{\le N-3}h_{\underline L\underline L}\,dV \right| \\
    & \lesssim \int_{\tau_i}^{\tau_{i+1}}r^{1-\delta} \mathcal E^D_{\le N-10}[\psi](\tau) \| \nabla_L\nabla_{\underline L}\mathscr L_Z^{\le N-2}\psi_+\|_{L^2(\Sigma_\tau)} \mathcal E^W_{\le N-2}[h_{\underline L\underline L}](\tau)\,d\tau \\
    & \lesssim \int_{\tau_i}^{\tau_{i+1}}r^{-\delta} \mathcal E^D_{\le N-10}[\psi](\tau) \| \slashed\nabla\mathscr L_Z^{\le N-1}\psi_+\|_{L^2(\Sigma_\tau)} \mathcal E^W_{\le N-2}[h_{\underline L\underline L}](\tau)\,d\tau \\
    & \qquad + \int_{\tau_i}^{\tau_{i+1}}r^{-\delta} \mathcal E^D_{\le N-10}[\psi](\tau) \| \nabla_{\underline L}\mathscr L_Z^{\le N-2}\psi_+\|_{L^2(\Sigma_\tau)} \mathcal E^W_{\le N-2}[h_{\underline L\underline L}](\tau)\,d\tau \\
    & \qquad + \int_{\tau_i}^{\tau_{i+1}}r^{-\delta} \mathcal E^D_{\le N-10}[\psi](\tau) \| \nabla_L\mathscr L_Z^{\le N-2}\psi_+\|_{L^2(\Sigma_\tau)} \mathcal E^W_{\le N-2}[h_{\underline L\underline L}](\tau)\,d\tau \\
    & \qquad + \int_{\tau_i}^{\tau_{i+1}}r^{1-\delta} \mathcal E^D_{\le N-10}[\psi](\tau) \|  \mathscr L_Z^{\le N-2}( \nabla_\nu \mathrm{Tr}h \gamma^\lambda \gamma^\nu\nabla_\lambda\psi )\|_{L^2(\Sigma_\tau)} \mathcal E^W_{\le N-2}[h_{\underline L\underline L}](\tau)\,d\tau \\
    &  + \int_{\tau_i}^{\tau_{i+1}}r^{1-\delta} \mathcal E^D_{\le N-10}[\psi](\tau) \|  \mathscr L_Z^{\le N-2}( \mathscr T_{L\underline L}[\psi] \psi )\|_{L^2(\Sigma_\tau)} \mathcal E^W_{\le N-2}[h_{\underline L\underline L}](\tau)\,d\tau \\
    & \qquad + \int_{\tau_i}^{\tau_{i+1}}r^{1-\delta} \mathcal E^D_{\le N-10}[\psi](\tau) \|  \mathscr L_Z^{\le N-2}( \mathscr T_{AB}[\psi] \psi )\|_{L^2(\Sigma_\tau)} \mathcal E^W_{\le N-2}[h_{\underline L\underline L}](\tau)\,d\tau.
\end{align*}
For the first integral, we have
\begin{align*}
    & \int_{\tau_i}^{\tau_{i+1}}r^{-\delta} \mathcal E^D_{\le N-10}[\psi](\tau) \| \slashed\nabla\mathscr L_Z^{\le N-1}\psi_+\|_{L^2(\Sigma_\tau)} \mathcal E^W_{\le N-2}[h_{\underline L\underline L}](\tau)\,d\tau \\
    & \lesssim \left( \int_{\tau_i}^{\tau_{i+1}} {}^{(0)}\mathcal E^D_{\le N-10}[\psi]^2 \, {}^{(0)}\mathcal E^D_{\le N-1}[\psi]^2\,d\tau \right)^\frac12 \left( \int_{\tau_i}^{\tau_{i+1}} {}^{(0)}\mathcal E^W_{\le N-2}[h_{\underline L\underline L}]^2(\tau)\,d\tau \right)^\frac12 \\
    & \lesssim  \sup_{\tau\in[\tau_i,\tau_{i+1}]}{}^{(0)}\mathcal E^D_{\le N-10}[\psi](\tau) \left( \int_{\tau_i}^{\tau_{i+1}}  \, {}^{(0)}\mathcal E^D_{\le N-1}[\psi]^2\,d\tau \right)^\frac12 \left( \int_{\tau_i}^{\tau_{i+1}} {}^{(0)}\mathcal E^W_{\le N-2}[h_{\underline L\underline L}]^2(\tau)\,d\tau \right)^\frac12.
\end{align*}
For the second integral,
\begin{align*}
    & \int_{\tau_i}^{\tau_{i+1}}r^{-\delta} \mathcal E^D_{\le N-10}[\psi](\tau) \| \nabla_{\underline L}\mathscr L_Z^{\le N-2}\psi_+\|_{L^2(\Sigma_\tau)} \mathcal E^W_{\le N-2}[h_{\underline L\underline L}](\tau)\,d\tau \\
    & \lesssim \sup_{\tau\in[\tau_i,\tau_{i+1}]}\mathscr E^D_{\le N-1}[\psi](\tau) \int_{\tau_i}^{\tau_{i+1}}r^{-\delta} \mathcal E^D_{\le N-10}[\psi](\tau) \mathcal E^W_{\le N-2}[h_{\underline L\underline L}](\tau)\,d\tau \\
    & \lesssim \sup_{\tau\in[\tau_i,\tau_{i+1}]}\mathscr E^D_{\le N-1}[\psi](\tau) \left( \int_{\tau_i}^{\tau_{i+1}}r^{-\delta}\,{}^{(0)} \mathcal E^D_{\le N-10}[\psi]^2(\tau)\,d\tau \right)^\frac12 \left(\int_{\tau_i}^{\tau_{i+1}} {}^{(0} \mathcal E^W_{\le N-2}[h_{\underline L\underline L}]^2(\tau)\,d\tau  \right)^\frac12.
\end{align*}
For the fourth integral, using the Clifford algebra, we write
\begin{align*}
    & \int_{\tau_i}^{\tau_{i+1}}r^{1-\delta} \mathcal E^D_{\le N-10}[\psi](\tau) \|  \mathscr L_Z^{\le N-2}( \nabla_\nu \mathrm{Tr}h \gamma^\lambda \gamma^\nu\nabla_\lambda\psi )\|_{L^2(\Sigma_\tau)} \mathcal E^W_{\le N-2}[h_{\underline L\underline L}](\tau)\,d\tau \\
    & \lesssim \int_{\tau_i}^{\tau_{i+1}}r^{1-\delta} \mathcal E^D_{\le N-10}[\psi](\tau) \|  \mathscr L_Z^{\le N-2}( \nabla_L \mathrm{Tr}h \nabla_{\underline L}\psi_+ )\|_{L^2(\Sigma_\tau)} \mathcal E^W_{\le N-2}[h_{\underline L\underline L}](\tau)\,d\tau \\
    & \qquad + \int_{\tau_i}^{\tau_{i+1}}r^{1-\delta} \mathcal E^D_{\le N-10}[\psi](\tau) \|  \mathscr L_Z^{\le N-2}( \nabla_L \mathrm{Tr}h \slashed\nabla\psi_+ )\|_{L^2(\Sigma_\tau)} \mathcal E^W_{\le N-2}[h_{\underline L\underline L}](\tau)\,d\tau \\
    & + \int_{\tau_i}^{\tau_{i+1}}r^{1-\delta} \mathcal E^D_{\le N-10}[\psi](\tau) \|  \mathscr L_Z^{\le N-2}( \nabla_{\underline L} \mathrm{Tr}h \nabla_{ L}\psi_- )\|_{L^2(\Sigma_\tau)} \mathcal E^W_{\le N-2}[h_{\underline L\underline L}](\tau)\,d\tau \\
    & \qquad + \int_{\tau_i}^{\tau_{i+1}}r^{1-\delta} \mathcal E^D_{\le N-10}[\psi](\tau) \|  \mathscr L_Z^{\le N-2}( \nabla_L \mathrm{Tr}h \slashed\nabla\psi_- )\|_{L^2(\Sigma_\tau)} \mathcal E^W_{\le N-2}[h_{\underline L\underline L}](\tau)\,d\tau \\
    & + \int_{\tau_i}^{\tau_{i+1}}r^{1-\delta} \mathcal E^D_{\le N-10}[\psi](\tau) \|  \mathscr L_Z^{\le N-2}( \slashed\nabla\mathrm{Tr}h \nabla_{ L}\psi_+ )\|_{L^2(\Sigma_\tau)} \mathcal E^W_{\le N-2}[h_{\underline L\underline L}](\tau)\,d\tau \\
    & \qquad + \int_{\tau_i}^{\tau_{i+1}}r^{1-\delta} \mathcal E^D_{\le N-10}[\psi](\tau) \|  \mathscr L_Z^{\le N-2}( \slashed\nabla \mathrm{Tr}h \nabla_{\underline L}\psi_- )\|_{L^2(\Sigma_\tau)} \mathcal E^W_{\le N-2}[h_{\underline L\underline L}](\tau)\,d\tau \\
    & \qquad +\int_{\tau_i}^{\tau_{i+1}}r^{1-\delta} \mathcal E^D_{\le N-10}[\psi](\tau) \|  \mathscr L_Z^{\le N-2}( \slashed\nabla \mathrm{Tr}h \slashed\nabla\psi_\pm )\|_{L^2(\Sigma_\tau)} \mathcal E^W_{\le N-2}[h_{\underline L\underline L}](\tau)\,d\tau.
\end{align*}
For each integrals, it is enough to consider the low-high and high-low interaction between additional nonlinearities $\nabla\mathrm{Tr}h\cdot \nabla\psi$. In fact, we write
\begin{align*}
    & \int_{\tau_i}^{\tau_{i+1}}r^{1-\delta} \mathcal E^D_{\le N-10}[\psi](\tau) \|  \mathscr L_Z^{\le N-2}( \nabla_L \mathrm{Tr}h \nabla_{\underline L}\psi_+ )\|_{L^2(\Sigma_\tau)} \mathcal E^W_{\le N-2}[h_{\underline L\underline L}](\tau)\,d\tau \\
    & \lesssim \int_{\tau_i}^{\tau_{i+1}}r^{1-\delta} \mathcal E^D_{\le N-10}[\psi](\tau) \|   \nabla_L \mathcal L_Z^{\le N-2} \mathrm{Tr}h \nabla_{\underline L}\mathscr L_Z^{\le N-11}\psi_+ )\|_{L^2(\Sigma_\tau)} \mathcal E^W_{\le N-2}[h_{\underline L\underline L}](\tau)\,d\tau \\
    & \qquad + \int_{\tau_i}^{\tau_{i+1}}r^{1-\delta} \mathcal E^D_{\le N-10}[\psi](\tau) \|   \nabla_L \mathcal L_Z^{\le N-11} \mathrm{Tr}h \nabla_{\underline L}\mathscr L_Z^{\le N-2}\psi_+ )\|_{L^2(\Sigma_\tau)} \mathcal E^W_{\le N-2}[h_{\underline L\underline L}](\tau)\,d\tau \\
    & \lesssim \int_{\tau_i}^{\tau_{i+1}}r^{\frac12-\delta} \mathcal E^D_{\le N-10}[\psi](\tau) \mathcal E^W_{\le N-2}[h_{\underline L\underline L}](\tau) \mathcal E^D_{\le N-9}[\psi](\tau)  \mathcal E^W_{\le N-2}[h_{\underline L\underline L}](\tau)\,d\tau \\ 
    & \qquad + \int_{\tau_i}^{\tau_{i+1}}r^{\frac12-\delta} \mathcal E^D_{\le N-10}[\psi](\tau) \mathcal E^W_{\le N-8}[h_{\underline L\underline L}](\tau) \mathscr E^D_{\le N-1}[\psi](\tau)  \mathcal E^W_{\le N-2}[h](\tau)\,d\tau \\
    & \lesssim \sup_{\tau\in[\tau_i,\tau_{i+1}]}{}^{(0)} \mathcal E^D_{\le N-9}[\psi](\tau) {}^{(1)} \mathcal E^D_{\le N-9}[\psi](\tau)\int_{\tau_i}^{\tau_{i+1}}r^{-\delta} \, {}^{(0)}  \mathcal E^W_{\le N-2}[h]^2(\tau)\,d\tau \\
    & \qquad + \sup_{\tau\in[\tau_i,\tau_{i+1}]}\mathscr E^D_{\le N-1}[\psi](\tau) \left(\int_{\tau_i}^{\tau_{i+1}} {}^{(0)}  \mathcal E^W_{\le N-2}[h_{\underline L\underline L}]^2(\tau)\,d\tau \right)^\frac12 \left(\int_{\tau_i}^{\tau_{i+1}} {}^{(1-\delta)}\mathcal E^D_{\le N-10}[\psi]^2(\tau) {}^{(0)}\mathcal E^W_{\le N-8}[h_{\underline L\underline L}]^2(\tau)  \,d\tau \right)^\frac12.
\end{align*}
The remaining integrals are controlled similarly; we omit the details.
The preceding estimates show that controlling the weighted energy ${}^{(2-\delta)}\mathcal E^W_{\le N-3}[h_{\underline L\underline L}]$ requires a bound for ${}^{(1)}\mathcal E^W_{\le N-2}[h_{\underline L\underline L}]$.

\subsubsection{Control of ${}^{(1)}\mathcal E^W_{\le N-2}[h_{\underline L\underline L}]$ }
To do this, we consider the integral
\begin{align}
    \int_{\mathcal D^{\tau_{i+1}}_{\tau_i}} r \langle \mathscr L_Z^{\le N-11} \psi, \gamma^L \nabla_{\underline L}\mathscr L_Z^{\le N-2}\psi\rangle \nabla_L \mathcal L_Z^{\le N-2}h_{\underline L\underline L}\,dV.
\end{align}
As in the preceding estimate, we integrate by parts in the $\nabla_{\underline L}$ direction:
\begin{align*}
	& \int_{\mathcal D^{\tau_{i+1}}_{\tau_i}} r^{} \langle \mathscr L_Z^{\le N-11}\psi, \gamma^L \nabla_{\underline L} \mathscr L_Z^{\le N-2}\psi\rangle \nabla_{L}\mathcal L_Z^{\le N-2}h_{\underline L\underline L}\,dV \\
	& = \int_{\substack{ u=\tau_{i+1}-R \\ v\ge \tau_{i+1}+R }} 	r^{} \langle \mathscr L_Z^{\le N-11} \psi, \gamma^L  \mathscr L_Z^{\le N-2}\psi\rangle \nabla_{L}\mathcal L_Z^{\le N-2}h_{\underline L\underline L}\,d\sigma -  \int_{\substack{ u=\tau_{i}-R \\ v\ge \tau_{i}+R }} 	r^{} \langle  \mathscr L_Z^{\le N-11}\psi, \gamma^L  \mathscr L_Z^{\le N-2}\psi\rangle \nabla_{L}\mathcal L_Z^{\le N-2}h_{\underline L\underline L}\,d\sigma \\
	& + \int_{r=R}   	r^{} \langle \mathscr L_Z^{\le N-11}\psi, \gamma^L  \mathscr L_Z^{\le N-2}\psi\rangle \nabla_{L}\mathcal L_Z^{\le N-2}h_{\underline L\underline L}\,d\sigma \\
	& + \int_{\mathcal D^{\tau_{i+1}}_{\tau_i}} r^{} \langle \nabla_{\underline L} \mathscr L_Z^{\le N-11}\psi, \gamma^L \mathscr L_Z^{\le N-2}\psi\rangle \nabla_{L}\mathcal L_Z^{\le N-2}h_{\underline L\underline L}\,dV + \int_{\mathcal D^{\tau_{i+1}}_{\tau_i}} r^{} \langle \mathscr L_Z^{\le N-11}\psi, \gamma^L \mathscr L_Z^{\le N-2}\psi\rangle \nabla_{\underline L} \nabla_{L}\mathcal L_Z^{\le N-2}h_{\underline L\underline L}\,dV.
\end{align*}
For the boundary integral along the null hypersurface, we apply H\"older's inequality to obtain the product $L^4\times L^4\times L^2$ and apply the weighted Sobolev inequality and then we obtain
\begin{align*}
    & \left| \int_{\substack{ u=\tau_{i+1}-R \\ v\ge \tau_{i+1}+R }} 	r^{} \langle \mathscr L_Z^{\le N-11}\psi, \gamma^L  \mathscr L_Z^{\le N-2}\psi\rangle \nabla_{L}\mathcal L_Z^{\le N-2}h_{\underline L\underline L}\,d\sigma \right| \\
    & \lesssim (\tau_{i+1})^{\frac\delta2}\, {}^{(0)}\mathcal E^D_{\le N-10}[\psi](\tau_{i+1}) {}^{(2-\delta)}\mathcal E^D_{\le N-1}[\psi](\tau_{i+1}) {}^{(0)}\mathcal E^W_{\le N-2}[h](\tau_{i+1}).
\end{align*}

For the spacetime integral in which $\nabla_{\underline L}$ acts on the higher-order term $\nabla_L\mathcal L_Z^{\le N-2}h_{\underline L\underline L}$, we use the wave equation for $h_{\underline L\underline L}$ to obtain
\begin{align*}
    & \int_{\mathcal D^{\tau_{i+1}}_{\tau_i}} r^{} \langle \mathscr L_Z^{\le N-11} \psi, \gamma^L \mathscr L_Z^{\le N-2}\psi\rangle \nabla_{\underline L} \nabla_{L}\mathcal L_Z^{\le N-2}h_{\underline L\underline L}\,dV \\
    & = \int_{\mathcal D^{\tau_{i+1}}_{\tau_i}} r^{} \langle \mathscr L_Z^{\le N-11} \psi, \gamma^L \mathscr L_Z^{\le N-2}\psi\rangle \slashed\Delta\mathcal L_Z^{\le N-2}h_{\underline L\underline L}\,dV +  \int_{\mathcal D^{\tau_{i+1}}_{\tau_i}}  \langle\mathscr L_Z^{\le N-11} \psi, \gamma^L \mathscr L_Z^{\le N-2}\psi\rangle \nabla_{\underline L}\mathcal L_Z^{\le N-2}h_{\underline L\underline L}\,dV \\
    & \qquad + \int_{\mathcal D^{\tau_{i+1}}_{\tau_i}}  \langle \mathscr L_Z^{\le N-11}\psi, \gamma^L \mathscr L_Z^{\le N-2}\psi\rangle \nabla_{ L}\mathcal L_Z^{\le N-2}h_{\underline L\underline L}\,dV + \int_{\mathcal D^{\tau_{i+1}}_{\tau_i}}r  \langle \mathscr L_Z^{\le N-11}\psi, \gamma^L \mathscr L_Z^{\le N-2}\psi\rangle \mathcal L_Z^{\le N-2}\mathscr T_{\underline L\underline L}[\psi] \,dV.
\end{align*}
For the first integral, applying H\"older's inequality with exponents $4,4,2$ gives
\begin{align*}
& \left|  \int_{\mathcal D^{\tau_{i+1}}_{\tau_i}} r^{} \langle \mathscr L_Z^{\le N-11} \psi, \gamma^L \mathscr L_Z^{\le N-2}\psi\rangle \slashed\Delta\mathcal L_Z^{\le N-2}h_{\underline L\underline L}\,dV \right| \\
& \lesssim  \int_{\mathcal D^{\tau_{i+1}}_{\tau_i}}  | \mathscr L_Z^{\le N-11} \psi_+ | | \mathscr L_Z^{\le N-2}\psi_+|  |\slashed\nabla\mathcal L_Z^{\le N-1}h_{\underline L\underline L}|\,dV \\
& \lesssim \int_{\tau_i}^{\tau_{i+1}} \mathcal E^D_{\le N-10}[\psi]\mathcal E^D_{\le N-1}[\psi]\mathcal E^W_{\le N-1}[h_{\underline L\underline L}]\,d\tau \\
& \lesssim \sup_{\tau\in[\tau_i,\tau_{i+1}]} {}^{(0)}\mathcal E^W_{\le N-1}[h_{\underline L\underline L}](\tau) \left( \int_{\tau_i}^{\tau_{i+1}} {}^{(0)}\mathcal E^D_{\le N-10}[\psi]^2(\tau)\,d\tau  \right)^\frac12 \left( \int_{\tau_i}^{\tau_{i+1}} {}^{(0)}\mathcal E^D_{\le N-1}[\psi]^2(\tau)\,d\tau  \right)^\frac12.
\end{align*}
The control of the third integral is also very similar. For the fourth integral including the additional nonlinearity $\mathscr T_{\underline L\underline L}$, we simply apply H\"older's inequality and see that
\begin{align*}
    & \left| \int_{\mathcal D^{\tau_{i+1}}_{\tau_i}}r  \langle \mathscr L_Z^{\le N-11}\psi, \gamma^L \mathscr L_Z^{\le N-2}\psi\rangle \mathcal L_Z^{\le N-2}\mathscr T_{\underline L\underline L}[\psi] \,dV \right| \\
    & \lesssim \int_{\mathcal D^{\tau_{i+1}}_{\tau_i}}r  | \mathscr L_Z^{\le N-11}\psi_+| | \mathscr L_Z^{\le N-2}\psi| |\mathscr L_Z^{\le N-11}\psi_+| |\nabla_{\underline L}\mathscr L_Z^{\le N-2}\psi_+| \,dV \\
    & \lesssim \int_{\tau_i}^{\tau_{i+1}} \mathcal E^D_{\le N-9}[\psi]^2 \mathscr E^D_{\le N-2}[\psi] \mathscr E^D_{\le N-1}[\psi]\,d\tau \\
    & \lesssim \sup_{\tau\in[\tau_i,\tau_{i+1}]} \mathscr E^D_{\le N-1}[\psi]^2(\tau) \int_{\tau_i}^{\tau_{i+1}} {}^{(0)} \mathcal E^D_{\le N-9}[\psi]^2(\tau) \,d\tau.
\end{align*}
It remains to estimate the second integral, which is the most delicate.
Following the previous argument, we use the integration by parts in the direction $\nabla_{\underline L}$ and we write
\begin{align*}
    & \int_{\mathcal D^{\tau_{i+1}}_{\tau_i}}  \langle  \mathscr L_Z^{\le N-11}\psi, \gamma^L \mathscr L_Z^{\le N-2}\psi\rangle \nabla_{\underline L}\mathcal L_Z^{\le N-2}h_{\underline L\underline L}\,dV \\
    & = \int_{\substack{u=\tau_{i+1}-R \\ v\ge \tau_{i+1}+R } } \langle  \mathscr L_Z^{\le N-11}\psi, \gamma^L \mathscr L_Z^{\le N-2}\psi\rangle \mathcal L_Z^{\le N-2}h_{\underline L\underline L}\,d\sigma - \int_{\substack{u=\tau_{i}-R \\ v\ge \tau_{i}+R } } \langle  \mathscr L_Z^{\le N-11}\psi, \gamma^L \mathscr L_Z^{\le N-2}\psi\rangle \mathcal L_Z^{\le N-2}h_{\underline L\underline L}\,d\sigma \\
    & + \int_{ r=R } \langle  \mathscr L_Z^{\le N-11}\psi, \gamma^L \mathscr L_Z^{\le N-2}\psi\rangle \mathcal L_Z^{\le N-2}h_{\underline L\underline L}\,d\sigma \\
    & + \int_{\mathcal D^{\tau_{i+1}}_{\tau_i}}  \langle   \nabla_{\underline L}\mathscr L_Z^{\le N-11} \psi, \gamma^L \mathscr L_Z^{\le N-2}\psi\rangle \mathcal L_Z^{\le N-2}h_{\underline L\underline L}\,dV +  \int_{\mathcal D^{\tau_{i+1}}_{\tau_i}}  \langle  \mathscr L_Z^{\le N-11} \psi, \gamma^L \nabla_{\underline L}\mathscr L_Z^{\le N-2}\psi\rangle \mathcal L_Z^{\le N-2}h_{\underline L\underline L}\,dV .
\end{align*}
The energy identity controls the boundary integrals; we omit the details. When $\nabla_{\underline L}$ acts on the lower-order term, applying the $L^4\times L^2\times L^4$ product estimate and the weighted Sobolev inequality gives
\begin{align*}
    & \left| \int_{\mathcal D^{\tau_{i+1}}_{\tau_i}}  \langle  \nabla_{\underline L} \mathscr L_Z^{\le N-11}\psi, \gamma^L \mathscr L_Z^{\le N-2}\psi\rangle \mathcal L_Z^{\le N-2}h_{\underline L\underline L}\,dV  \right| \\
    & \lesssim \int_{\tau_i}^{\tau_{i+1}} \|\slashed\Delta \mathscr L_Z^{\le N-10}\psi_+\|_{L^2(\Sigma_\tau)} \|\mathscr L_Z^{\le N-2}\psi_+\|_{L^2(\Sigma_\tau)} \|\nabla_L\mathcal L_Z^{\le N-1}h_{\underline L\underline L}\|_{L^2(\Sigma_\tau)} \\
    & \qquad + \int_{\tau_i}^{\tau_{i+1}}r^{-1}  \|\nabla_{\underline L} \mathscr L_Z^{\le N-10}\psi_+\|_{L^2(\Sigma_\tau)} \|\mathscr L_Z^{\le N-2}\psi_+\|_{L^2(\Sigma_\tau)} \|\nabla_L\mathcal L_Z^{\le N-1}h_{\underline L\underline L}\|_{L^2(\Sigma_\tau)} \\
    & \qquad + \int_{\tau_i}^{\tau_{i+1}}r^{-1}  \|\nabla_L \mathscr L_Z^{\le N-10}\psi_+\|_{L^2(\Sigma_\tau)} \|\mathscr L_Z^{\le N-2}\psi_+\|_{L^2(\Sigma_\tau)} \|\nabla_L\mathcal L_Z^{\le N-1}h_{\underline L\underline L}\|_{L^2(\Sigma_\tau)} \\
    & \qquad + \int_{\tau_i}^{\tau_{i+1}} \|\mathscr L_Z^{\le N-10}(\mathscr T_{L\underline L}[\psi]\psi) \|_{L^2(\Sigma_\tau)} \|\mathscr L_Z^{\le N-2}\psi_+\|_{L^2(\Sigma_\tau)} \|\nabla_L\mathcal L_Z^{\le N-1}h_{\underline L\underline L}\|_{L^2(\Sigma_\tau)} \\
    & \qquad + \int_{\tau_i}^{\tau_{i+1}} \|\mathscr L_Z^{\le N-10}(\mathscr T_{AB}[\psi]\psi) \|_{L^2(\Sigma_\tau)} \|\mathscr L_Z^{\le N-2}\psi_+\|_{L^2(\Sigma_\tau)} \|\nabla_L\mathcal L_Z^{\le N-1}h_{\underline L\underline L}\|_{L^2(\Sigma_\tau)} \\
    & \qquad + \int_{\tau_i}^{\tau_{i+1}} \|\mathscr L_Z^{\le N-10}(\nabla_\nu \mathrm{Tr}h \gamma^\lambda\gamma^\nu \nabla_\lambda\psi ) \|_{L^2(\Sigma_\tau)} \|\mathscr L_Z^{\le N-2}\psi_+\|_{L^2(\Sigma_\tau)} \|\nabla_L\mathcal L_Z^{\le N-1}h_{\underline L\underline L}\|_{L^2(\Sigma_\tau)}.
\end{align*}
The remaining estimate then yields the required decay.

Finally, we have to deal with the integral, where the derivative $\nabla_{\underline L}$ acts on the higher-order term $\mathscr L_Z^{\le N-2}\psi_+$. We see that
\begin{align*}
    & \left| \int_{\mathcal D^{\tau_{i+1}}_{\tau_i}}  \langle \mathscr L_Z^{\le N-11}\psi, \gamma^L \nabla_{\underline L}\mathscr L_Z^{\le N-2}\psi\rangle \mathcal L_Z^{\le N-2}h_{\underline L\underline L}\,dV  \right| \\
    & \lesssim \int_{\mathcal D^{\tau_{i+1}}_{\tau_i}}  | \mathscr L_Z^{\le N-11}\psi_+| | \mathscr L_Z^{\le N-1}\psi_+| | \mathcal L_Z^{\le N-2}h_{\underline L\underline L}|\,dV \\
    & \lesssim \int_{\tau_i}^{\tau_{i+1}} r^{-\frac12} \mathcal E^D_{\le N-9}[\psi] \mathscr E^D_{\le N-1}[\psi] \|\mathcal L_Z^{\le N-2}h_{\underline L\underline L}\|_{L^2(\Sigma_\tau)} \,d\tau \\
    & \lesssim \int_{\tau_i}^{\tau_{i+1}} r^{\frac12} \mathcal E^D_{\le N-9}[\psi] \mathscr E^D_{\le N-1}[\psi] \mathcal E^W_{\le N-2}[h_{\underline L\underline L}] \,d\tau \\
    & \lesssim (\tau_{i+1})^{\frac\delta2} \sup_{\tau\in[\tau_i,\tau_{i+1}]}\mathscr E^D_{\le N-1}[\psi](\tau) \left( \int_{\tau_i}^{\tau_{i+1}} {}^{(1-\delta)}\mathcal E^D_{\le N-9}[\psi]^2(\tau)\,d\tau  \right)^\frac12 \left( \int_{\tau_i}^{\tau_{i+1}} {}^{(0)}\mathcal E^W_{\le N-2}[h_{\underline L\underline L}]^2(\tau)\,d\tau  \right)^\frac12.
\end{align*}
\subsubsection{Control of ${}^{(0)}\mathcal E^W_{\le N-1}[h_{\underline L\underline L}]$ }
We arrive at the step of the control of the energy
${}^{(0)}\mathcal E^W_{\le N-1}[h_{\underline L\underline L}]$. To do this, we consider the integral:
\begin{align}
    \int_{\mathcal D^{\tau_{i+1}}_{\tau_i}}  \langle \mathscr L_Z^{\le N-11} \psi, \gamma^L \nabla_{\underline L}\mathscr L_Z^{\le N-1}\psi\rangle \nabla_L \mathcal L_Z^{\le N-1}h_{\underline L\underline L}\,dV.
\end{align}
By repeating the previous argument, we need to control the following integral
\begin{align}
     \int_{\mathcal D^{\tau_{i+1}}_{\tau_i}} r^{-1} |\mathscr L_Z^{\le N-11}\psi_+| |\mathscr L_Z^{\le N-1}\psi_+| |\nabla_{\underline L}\mathcal L_Z^{\le N-1}h_{\underline L\underline L}|  \,dV.
\end{align}
To do this, we use the Sobolev inequality along ingoing null and the Hardy inequality to get
\begin{align*}
    &   \int_{\mathcal D^{\tau_{i+1}}_{\tau_i}} r^{-1} |\mathscr L_Z^{\le N-11}\psi_+| |\mathscr L_Z^{\le N-1}\psi_+| |\nabla_{\underline L}\mathcal L_Z^{\le N-1}h_{\underline L\underline L}|  \,dV \\
    & \lesssim \int r^{-\frac32} \|\nabla_{\underline L}\mathscr L_Z^{\le N-9}\psi_+\|_{L^2(\underline{\mathcal C}_v)} \|\mathscr L_Z^{\le N-1}\psi_+\|_{L^2(\underline{\mathcal C}_v)} \|\nabla_{\underline L}\mathcal L_Z^{\le N-1}h_{\underline L\underline L}\|_{L^2(\underline{\mathcal C}_v)} \,dv \\
    & \lesssim \sup_{v} {}^{(0)}\mathcal F^W_{\le N-1}[h_{\underline L\underline L}] \int r^{-\frac32} \|\nabla_{\underline L}\mathscr L_Z^{\le N-9}\psi_+\|_{L^2(\underline{\mathcal C}_v)} \|\mathscr L_Z^{\le N-1}\psi_+\|_{L^2(\underline{\mathcal C}_v)}\,dv \\
    & \lesssim (\tau_{i+1})^{\frac\delta2} \sup_{v} {}^{(0)}\mathcal F^W_{\le N-1}[h_{\underline L\underline L}]  \left( \int_{\tau_i}^{\tau_{i+1}} {}^{(0)}\mathcal E^D_{\le N-8}[\psi]^2(\tau)\,d\tau  \right)^\frac12  \left( \int_{\tau_i}^{\tau_{i+1}} {}^{(1-\delta)}\mathcal E^D_{\le N-1}[\psi]^2(\tau)\,d\tau  \right)^\frac12 .
\end{align*}
\begin{rem}\label{rem-energy-hierarchy2}
We conclude this section by emphasizing the role of the Dirac current energy in closing the weighted energy estimates. This role becomes particularly apparent in the closure of the weighted energy estimates for the tensorial wave components, where Dirac current estimates at both relatively high and low commutation orders are employed. Although some of these terms could alternatively be estimated through the wave energy by invoking a Hardy inequality, this would introduce additional losses of derivatives and radial weights, potentially leading to unnecessarily large growth in the energy bounds. The Dirac current energy allows us to avoid such losses and maintain a more favorable energy hierarchy.

Importantly, even when the Dirac current energy is employed at relatively high commutation orders, it is only required up to order $N-1$. In particular, the closure of the weighted energy estimates does not require the top-order Dirac current energy.

More importantly, the Dirac current energy is required only up to order \(N-1\), whereas the wave energies are controlled up to order \(N\). This asymmetric energy hierarchy allows us to avoid the potential loss of derivatives arising in the top-order Dirac current estimates, without requiring any additional cancellation mechanism at that order.
\end{rem}

\section{Improved weighted-energy estimates for lower-order $k\le N-9$ }\label{sec:imp-weight}

\subsection{Control of the spinor fields for $k\le N-9$ }
We consider the integral
\begin{align}
\int_{\mathcal D^{\tau_{i+1}}_{\tau_i}} r^2 \langle \nabla_\nu \mathcal L_Z^{\le k_1} \mathrm{Tr}h \gamma^\lambda\gamma^\nu \nabla_\lambda\mathscr L_Z^{\le k_2} \psi, \gamma^T \nabla_{ L}\mathscr L_Z^{\le N-9}\psi\rangle\,dV,
\end{align}
with $k_1+k_2\le N-9$.
\begin{prop}
    For $k\le N-9$, we have
    \begin{align}
        \begin{aligned}
            \left| \int_{\mathcal D^{\tau_{i+1}}_{\tau_i}} r^2 \langle \nabla_\nu \mathcal L_Z^{\le k_1} \mathrm{Tr}h \gamma^\lambda\gamma^\nu \nabla_\lambda\mathscr L_Z^{\le k_2} \psi, \gamma^T \nabla_{ L}\mathscr L_Z^{\le N-9}\psi\rangle\,dV \right| \lesssim C^3\varepsilon^3.
        \end{aligned}
    \end{align}
\end{prop}
It suffices to consider the case $(\nu,\lambda)=(L,\underline L)$.  
We have
\begin{align*}
    & \left| \int_{\mathcal D^{\tau_{i+1}}_{\tau_i}} r^2 \langle \nabla_L \mathcal L_Z^{\le N-11} \mathrm{Tr}h \gamma^{\underline L}\gamma^L \nabla_{\underline L}\mathscr L_Z^{\le N-9} \psi, \gamma^T \nabla_{L}\mathscr L_Z^{\le N-9}\psi\rangle\,dV \right| \\
    & \lesssim \int_{\tau_i}^{\tau_{i+1}} r^{\frac32} \mathcal E^W_{\le N-11}[h_{L\underline L}] \|\nabla_L\nabla_{\underline L}\mathscr L_Z^{\le N-7}\psi \|_{L^2(\Sigma_\tau)} \mathcal E^D_{\le N-9}[\psi]\,d\tau \\
    & \lesssim \int_{\tau_i}^{\tau_{i+1}} r^{\frac12} \mathcal E^W_{\le N-11}[h_{L\underline L}] \mathcal E^D_{\le N-6}[\psi]  \mathcal E^D_{\le N-9}[\psi]\,d\tau \\
    & \qquad + \int_{\tau_i}^{\tau_{i+1}} r^{\frac12} \mathcal E^W_{\le N-11}[h_{L\underline L}] \|\nabla_{\underline L}\mathscr L_Z^{\le N-7}\psi \|_{L^2(\Sigma_\tau)} \mathcal E^D_{\le N-9}[\psi]\,d\tau \\
    & \qquad + \int_{\tau_i}^{\tau_{i+1}} r^{\frac32} \mathcal E^W_{\le N-11}[h_{L\underline L}] \|\mathscr L_Z^{\le N-7}( \mathscr T_{L\underline L}[\psi]\psi)     \|_{L^2(\Sigma_\tau)} \mathcal E^D_{\le N-9}[\psi]\,d\tau \\
    & \qquad + \int_{\tau_i}^{\tau_{i+1}} r^{\frac32} \mathcal E^W_{\le N-11}[h_{L\underline L}] \|\mathscr L_Z^{\le N-7}( \mathscr T_{AB}[\psi]\psi)     \|_{L^2(\Sigma_\tau)} \mathcal E^D_{\le N-9}[\psi]\,d\tau \\
    & \qquad + \int_{\tau_i}^{\tau_{i+1}} r^{\frac32} \mathcal E^W_{\le N-11}[h_{L\underline L}] \|\mathscr L_Z^{\le N-7}(\nabla_\nu \mathrm{Tr}h \gamma^\lambda\gamma^\nu \nabla_\lambda\psi  )     \|_{L^2(\Sigma_\tau)} \mathcal E^D_{\le N-9}[\psi]\,d\tau.
\end{align*}
For the second integral of the right-hand side of the last inequality, we see that
\begin{align*}
    & \int_{\tau_i}^{\tau_{i+1}} r^{\frac12} \mathcal E^W_{\le N-11}[h_{L\underline L}] \|\nabla_{\underline L}\mathscr L_Z^{\le N-7}\psi \|_{L^2(\Sigma_\tau)} \mathcal E^D_{\le N-9}[\psi]\,d\tau  \\
    & \lesssim (\tau_{i+1})^{\frac\delta2} \sup_{\tau\in[\tau_i,\tau_{i+1}]} \mathscr E^D_{\le N-6}[\psi](\tau) \left(  \int_{\tau_i}^{\tau_{i+1}}{}^{(0)}\mathcal E^W_{\le N-11}[h_{L\underline L}]^2(\tau)\,d\tau \right)^\frac12 \left( \int_{\tau_i}^{\tau_{i+1}}{}^{(1-\delta)}\mathcal E^D_{\le N-9}[\psi]^2(\tau) \right)^\frac12.
\end{align*}
The remaining integrals are controlled directly; we omit the details.
\subsection{Control of the tensor fields for $k\le N-9$}
\begin{prop}
    For $k\le N-9$, we have
    \begin{align}
        \begin{aligned}
            \left|  \int_{\mathcal D^{\tau_{i+1}}_{\tau_i}} r^2 \mathcal L_Z^{\le k} \mathscr T_{\mu\nu}[\psi] \nabla_L\mathcal L_Z^{\le k}h_{\mu\nu}\,dV \right| \lesssim C^3\varepsilon^3.
        \end{aligned}
    \end{align}
\end{prop}
To estimate the weighted energy of $h$, it suffices to consider the component $h_{\underline L\underline L}$. We therefore estimate the following spacetime integral:
\begin{align}
    \int_{\mathcal D^{\tau_{i+1}}_{\tau_i}} r^2 \langle \mathscr L_Z^{\le N-11}\psi, \gamma^L \nabla_{\underline L}\mathscr L_Z^{\le N-9}\psi\rangle \nabla_L\mathcal L_Z^{\le N-9}h_{\underline L\underline L}\,dV.
\end{align}
Applying the $L^4\times L^4\times L^2$ product estimate and the weighted Sobolev inequality gives
\begin{align*}
    & \left| \int_{\mathcal D^{\tau_{i+1}}_{\tau_i}} r^2 \langle \mathscr L_Z^{\le N-11}\psi, \gamma^L \nabla_{\underline L}\mathscr L_Z^{\le N-9}\psi\rangle \nabla_L\mathcal L_Z^{\le N-9}h_{\underline L\underline L}\,dV \right| \\
    & \lesssim \int_{\tau_i}^{\tau_{i+1}} r \mathcal E^D_{\le N-10}[\psi]\mathcal E^D_{\le N-8}[\psi] \mathcal E^W_{\le N-9}[h_{\underline L\underline L}]\,d\tau \\
    & \qquad + \int_{\tau_i}^{\tau_{i+1}} r \mathcal E^D_{\le N-10}[\psi] \|\nabla_{\underline L}\mathscr L_Z^{\le N-8}\psi\|_{L^2(\Sigma_\tau)} \mathcal E^W_{\le N-9}[h_{\underline L\underline L}]\,d\tau \\
    & \qquad + \int_{\tau_i}^{\tau_{i+1}} r^2 \mathcal E^D_{\le N-10}[\psi] \|\mathscr L_Z^{\le N-8}(\Box_g \mathrm{Tr}h \cdot \psi)\|_{L^2(\Sigma_\tau)} \mathcal E^W_{\le N-9}[h_{\underline L\underline L}]\,d\tau \\
    & \qquad + \int_{\tau_i}^{\tau_{i+1}} r^2 \mathcal E^D_{\le N-10}[\psi] \|\mathscr L_Z^{\le N-8}(\nabla_{\nu}\mathrm{Tr}h \gamma^\lambda\gamma^\nu \nabla_\lambda \psi)\|_{L^2(\Sigma_\tau)} \mathcal E^W_{\le N-9}[h_{\underline L\underline L}]\,d\tau.
\end{align*}
Then we have
\begin{align*}
    & \int_{\tau_i}^{\tau_{i+1}} r \mathcal E^D_{\le N-10}[\psi] \|\nabla_{\underline L}\mathscr L_Z^{\le N-8}\psi\|_{L^2(\Sigma_\tau)} \mathcal E^W_{\le N-9}[h_{\underline L\underline L}]\,d\tau  \\
    & \lesssim  \sup_{\tau\in[\tau_i,\tau_{i+1}]} \mathscr E^D_{\le N-7}[\psi](\tau)  \left( \int_{\tau_i}^{\tau_{i+1}} {}^{(1)}\mathcal E^D_{\le N-10}[\psi]^2(\tau)\,d\tau   \right)^\frac12 \left( \int_{\tau_i}^{\tau_{i+1}}{}^{(1)}\mathcal E^W_{\le N-9}[h_{\underline L\underline L}]^2(\tau)\,d\tau   \right)^\frac12.
\end{align*}

\section{Improved decay estimates }\label{sec:decay}
\subsection{Improved decay for the spinor fields}
We now establish decay estimates for the energies ${}^{(1)}\mathcal E^D_{\le k}[\psi]$ and ${}^{(0)}\mathcal E^D_{\le k}[\psi]$.

\begin{prop}
    For $k_1,k_2,k$ with $k_1+k_2\le k$ and $k\le N-1$, we have
    \begin{align}
        \begin{aligned}
          &  \int_{\mathcal D^{\tau_{i+1}}_{\tau_i}} r \langle \nabla_\nu \mathcal L_Z^{\le k_1} \mathrm{Tr}h \gamma^{\lambda}\gamma^{\nu} \nabla_{\lambda}\mathscr L_Z^{\le k_2} \psi, \gamma^T \nabla_{L}\mathscr L_Z^{\le k}\psi\rangle\,dV \\
          & \lesssim (\tau_{i+1})^{\frac\delta2} \sup_{v:\tau\le \tau(v)}{}^{(0)}\mathcal F^D_{\le k}[\psi](v,\tau_i,\tau_{i+1})\left( \int_{\tau_i}^{\tau_{i+1}} {}^{(0)}\mathcal E^W_{\ll k}[h_{L\underline L}]^2(\tau)\,d\tau \right)^\frac12 \left( \int_{\tau_i}^{\tau_{i+1}} {}^{(1-\delta)}\mathcal E^D_{\le k}[\psi]^2(\tau)\,d\tau \right)^\frac12,
        \end{aligned}
    \end{align}
    and
    \begin{align}
        \begin{aligned}
             &  \int_{\mathcal D^{\tau_{i+1}}_{\tau_i}}  \langle \nabla_\nu \mathcal L_Z^{\le k_1} \mathrm{Tr}h \gamma^{\lambda}\gamma^{\nu} \nabla_{\lambda}\mathscr L_Z^{\le k_2} \psi, \gamma^T \nabla_{L}\mathscr L_Z^{\le k}\psi\rangle\,dV \\
             & \lesssim \sup_{v:\tau\le\tau(v)} {}^{(0)}\mathcal F^D_{\le k}[\psi](v,\tau_i,\tau_{i+1}) \left( \int_{\tau_i}^{\tau_{i+1}} {}^{(0)}\mathcal E^W_{\ll k}[h_{L\underline L}]^2(\tau)\,d\tau \right)^\frac12 \left( \int_{\tau_i}^{\tau_{i+1}} {}^{(0)}\mathcal E^D_{\le k}[\psi]^2(\tau)\,d\tau \right)^\frac12 
        \end{aligned}
    \end{align}
    For $k\le N-4$, we have
    \begin{align}
        \begin{aligned}
          &  \int_{\mathcal D^{\tau_{i+1}}_{\tau_i}} r \langle \nabla_\nu \mathcal L_Z^{\le k_1} \mathrm{Tr}h \gamma^{\lambda}\gamma^{\nu} \nabla_{\lambda}\mathscr L_Z^{\le k_2} \psi, \gamma^T \nabla_{L}\mathscr L_Z^{\le k}\psi\rangle\,dV \\
          & \lesssim \sup_{\tau\in[\tau_i,\tau_{i+1}]} \mathscr E^D_{\le k+3}[\psi](\tau) \left( \int_{\tau_i}^{\tau_{i+1}} {}^{(0)}\mathcal E^W_{\ll k}[h_{L\underline L}]^2(\tau)\,d\tau  \right)^\frac12  \left( \int_{\tau_i}^{\tau_{i+1}} {}^{(0)}\mathcal E^D_{\le k}[\psi]^2(\tau)\,d\tau  \right)^\frac12,
        \end{aligned}
    \end{align}
    and
    \begin{align}
        \begin{aligned}
             &  \int_{\mathcal D^{\tau_{i+1}}_{\tau_i}} \langle \nabla_\nu \mathcal L_Z^{\le k_1} \mathrm{Tr}h \gamma^{\lambda}\gamma^{\nu} \nabla_{\lambda}\mathscr L_Z^{\le k_2} \psi, \gamma^T \nabla_{L}\mathscr L_Z^{\le k}\psi\rangle\,dV \\
             & \lesssim \left( \int_{\tau_i}^{\tau_{i+1}} {}^{(0)}\mathcal E^W_{\ll k}[h_{L\underline L}]^2(\tau) {}^{(0)}\mathcal E^D_{\le k}[\psi]^2(\tau)\,d\tau \right)^\frac12 \left( \int_{\tau_i}^{\tau_{i+1}} {}^{(0)}\mathcal E^D_{\le k+3}[\psi]^2(\tau)\,d\tau \right)^\frac12 
        \end{aligned}
    \end{align}
\end{prop}
A repetition of the iteration allows us to establish the following decay estimates:
\begin{cor}
    After a suitable iterative procedure, we obtain the following improved decay estimates for the weighted energy of the spinor fields:
    \begin{align}
        \begin{aligned}
            {}^{(2-\delta)}\mathcal E^D_{\le N-1}[\psi]^2(\tau) \le C^2\varepsilon^2, \quad {}^{(1)}\mathcal E^D_{\le N-1}[\psi]^2(\tau) \le C^2\varepsilon^2 \tau^{-1+\delta}, \quad {}^{(0)}\mathcal E^D_{\le N-4}[\psi]^2(\tau) \le C^2\varepsilon^2\tau^{-2+\delta}.
        \end{aligned}
    \end{align}
\end{cor}
Thanks to the improved weighted energy estimates for the spinor fields of lower-order $k\le N-9$, we also have the following:
\begin{cor}
    For the lower-order cases: $k\le N-9$, 
    \begin{align}
        \begin{aligned}
            {}^{(2)}\mathcal E^D_{\le k}[\psi]^2(\tau) \le C^2\varepsilon^2, \quad {}^{(1)}\mathcal E^D_{\le k}[\psi]^2(\tau) \le C^2\varepsilon^2\tau^{-1}, \quad {}^{(0)}\mathcal E^D_{\le k}[\psi]^2(\tau) \le C^2\varepsilon^2\tau^{-2}.
        \end{aligned}
    \end{align}
\end{cor}

\subsubsection{Decay of ${}^{(1)}\mathcal E^D_{\le N-1}[\psi]$}
In view of the previous section, it is enough to consider the integrals of the form:
\begin{align}
     \int_{\mathcal D^{\tau_{i+1}}_{\tau_i}} r \langle \nabla_L \mathcal L_Z^{\le N-11} \mathrm{Tr}h \gamma^{\underline L}\gamma^L \nabla_{\underline L}\mathscr L_Z^{\le N-1} \psi, \gamma^T \nabla_{L}\mathscr L_Z^{\le N-1}\psi\rangle\,dV.
\end{align}
We proceed as before. We first apply H\"older's inequality to obtain an $L^\infty\times L^2\times L^2$ product, followed by the weighted Sobolev inequality on the ingoing null hypersurfaces $\underline{\mathcal C}_v$. We then use the wave equation for $h$ to decompose $\nabla_{\underline L}\nabla_L\mathcal L_Z^{\le N-9}\mathrm{Tr}h$, reducing the estimate to the following integral:
\begin{align}
    \int r^{-\frac12} \|\nabla_{\underline L}\mathcal L_Z^{\le N-9}\mathrm{Tr} h\|_{L^2(\underline{\mathcal C}_v)}  \mathcal F^D_{\le N-1}[\psi] \|\nabla_L\mathscr L_Z^{\le N-1}\psi\|_{L^2(\underline{\mathcal C}_v)}\,dv .
\end{align}
Applying the Hardy and Cauchy--Schwarz inequalities gives
\begin{align*}
    &  \int r^{-\frac12} \|\nabla_{\underline L}\mathcal L_Z^{\le N-9}\mathrm{Tr} h\|_{L^2(\underline{\mathcal C}_v)}  \mathcal F^D_{\le N-1}[\psi] \|\nabla_L\mathscr L_Z^{\le N-1}\psi\|_{L^2(\underline{\mathcal C}_v)}\,dv \\
    & \lesssim (\tau_{i+1})^{\frac\delta2} \sup_{v:\tau\le\tau(v)}{}^{(0)}\mathcal F^D_{\le N-1}[\psi](v,\tau_i,\tau_{i+1}) \left( \int_{\tau_i}^{\tau_{i+1}} {}^{(0)}\mathcal E^W_{\le N-8}[h_{L\underline L}]^2(\tau)\,d\tau \right)^\frac12 \left( \int_{\tau_i}^{\tau_{i+1}} {}^{(1-\delta)}\mathcal E^D_{\le N-1}[\psi]^2(\tau)\,d\tau \right)^\frac12. 
\end{align*}
We recall the weighted energy inequality:
\begin{align}
    \int_{\tau_i}^{\tau_{i+1}}{}^{(1-\delta)}\mathcal E^D_{\le N-1}[\psi]^2(\tau)\,d\tau & \le {}^{(2-\delta)}\mathcal E^D_{\le N-1}[\psi]^2(\tau_{i})+ C^3\varepsilon^3(\tau_i)^{\frac34+\frac{\delta-\epsilon}{2}},
\end{align}
which gives that there exists a $\tau_i^*\in[\tau_i,\tau_{i+1}] $ such that
\begin{align}
\begin{aligned}
    {}^{(1-\delta)}\mathcal E^D_{\le N-1}[\psi]^2(\tau_i^*) &\le \tau_i^{-1}\,{}^{(2-\delta)}\mathcal E^D_{\le N-1}[\psi]^2(\tau_i)+C^3\varepsilon^3 \tau_i^{-1+\frac34+\frac{\delta-\epsilon}{2}} \\
    & \le  C^2\varepsilon^2(\tau_i)^{-\epsilon}.
    \end{aligned}
\end{align}
To obtain the decay at $\tau=\tau_{i+1}$, we apply the weighted-energy inequality on the domain $\mathcal D^{\tau_{i+1}}_{\tau_i^*}$:
\begin{align}
\begin{aligned}
    {}^{(1)}\mathcal E^D_{\le N-1}[\psi]^2(\tau_{i+1}) &\le {}^{(1)}\mathcal E^D_{\le N-1}[\psi]^2(\tau_i^*)+C^3\varepsilon^3 (\tau_i)^{-\frac14+\frac32\delta-\frac{\epsilon}{2}} \\
    & \le (\tau_i^*)^\delta \,  {}^{(1-\delta)}\mathcal E^D_{\le N-1}[\psi]^2(\tau_i^*) +C^3\varepsilon^3 (\tau_i)^{-\frac14+\frac32\delta-\frac{\epsilon}{2}} \\
    & \lesssim C^2\varepsilon^2 (\tau_i)^{\delta-\epsilon}.
    \end{aligned}
\end{align}
Repeating the above argument, we obtain
\begin{align}
    \begin{aligned}
        {}^{(1)}\mathcal E^D_{\le k}[\psi]^2(\tau) \lesssim \begin{cases}
            \tau^{-\frac38+\frac74\delta-\frac\epsilon4}, \quad k = N-2, \\
            \tau^{-\frac{7}{16}+\frac{15}{8}\delta-\frac\epsilon8}, \quad k= N-3.
        \end{cases}
    \end{aligned}
\end{align}
\subsubsection{Decay of ${}^{(0)}\mathcal E^D_{\le N-1}[\psi]$}
Following the previous argument, we consider the integral of the form:
\begin{align}
    \int_{\mathcal D^{\tau_{i+1}}_{\tau_i}}  \langle \nabla_L \mathcal L_Z^{\le N-11} \mathrm{Tr}h \gamma^{\underline L}\gamma^L \nabla_{\underline L}\mathscr L_Z^{\le N-1} \psi, \gamma^T \nabla_{L}\mathscr L_Z^{\le N-1}\psi\rangle\,dV.
\end{align}
As in the preceding argument, we use the foliation by the ingoing null cones $\underline{\mathcal C}_v$ and apply the weighted Sobolev inequality to the lower-order term $\nabla_L\mathcal L_Z^{\le N-11}\mathrm{Tr}h$, thereby recovering a derivative $\nabla_{\underline L}$. The wave equation for $h$ then reduces the estimate to the following integral:
\begin{align}
    \int r^{-\frac32} \|\nabla_{\underline L}\mathcal L_Z^{\le N-9}\mathrm{Tr}h\|_{L^2(\underline{\mathcal C_v})}\mathcal F^D_{\le N-1}[\psi] \|\nabla_L\mathscr L_Z^{\le N-1}\psi\|_{L^2(\underline{\mathcal C_v})}\,dv .
\end{align}
Then we have
\begin{align*}
    &  \int r^{-\frac32} \|\nabla_{\underline L}\mathcal L_Z^{\le N-9}\mathrm{Tr}h\|_{L^2(\underline{\mathcal C_v})}\mathcal F^D_{\le N-1}[\psi] \|\nabla_L\mathscr L_Z^{\le N-1}\psi\|_{L^2(\underline{\mathcal C_v})}\,dv \\
    & \lesssim \sup_{v:\tau\le\tau(v)} {}^{(0)}\mathcal F^D_{\le N-1}[\psi](v,\tau_i,\tau_{i+1}) \left( \int_{\tau_i}^{\tau_{i+1}} {}^{(0)}\mathcal E^W_{\le N-8}[h_{L\underline L}]^2(\tau)\,d\tau \right)^\frac12 \left( \int_{\tau_i}^{\tau_{i+1}} {}^{(0)}\mathcal E^D_{\le N-1}[\psi]^2(\tau)\,d\tau \right)^\frac12 .
\end{align*}
Then we have
\begin{align}
    \begin{aligned}
        {}^{(0)}\mathcal E^D_{\le k}[\psi]^2(\tau) \lesssim \begin{cases}
            \tau^{-\frac34+\frac32\delta}, \quad k = N-1, \\
            \tau^{-\frac{17}{16}+\frac{21}{8}\delta-\frac\epsilon8}, \quad k = N-2, \\
            \tau^{-\frac{23}{16}+\frac{21}{8}\delta-\frac\epsilon8}, \quad k= N-3.
        \end{cases}
    \end{aligned}
\end{align}

Now we can improve the growth of the weighted energy ${}^{(2-\delta)}\mathcal E^D_{\le k}[\psi]$. Indeed, schematically, we have
\begin{align}
    {}^{(2-\delta)}\mathcal E^D_{\le k}[\psi]^2(\tau) \lesssim {}^{(\delta)}\mathcal E^D_{\le k}[\psi] {}^{(2-\delta)}\mathcal E^W_{\ll k}[h_{L\underline L}](\tau) {}^{(2-\delta)}\mathcal E^D_{\le k}[\psi](\tau).
\end{align}
We now iterate this estimate.
A direct application of the above improved decay of the weighted energy gives
\begin{align}
    \begin{aligned}
         {}^{(2-\delta)}\mathcal E^D_{\le k}[\psi]^2(\tau) & \lesssim \begin{cases}
             \tau^{\frac12+\delta-\frac\epsilon8}, \quad k =N-1, \\
             \tau^{\frac14+3\delta}, \quad k = N-2, \\
             \tau^{\frac18+3\delta}, \quad k = N-3, \\
             1, \quad k \le N-4.
         \end{cases} \Rightarrow {}^{(1)}\mathcal E^D_{\le k}[\psi]^2(\tau) \lesssim \begin{cases}
             \tau^{-\frac12+2\delta}, \quad k = N-1, \\
             \tau^{-\frac34+4\delta}, \quad k = N-2, \\
             \tau^{-\frac78+4\delta}, \quad k = N-3, \\
             \tau^{-1+\delta}, \quad k\le N-4, 
         \end{cases} \\
         \Rightarrow  {}^{(0)}\mathcal E^D_{\le k}[\psi]^2(\tau) &\lesssim \begin{cases}
             \tau^{-1+2\delta}, \quad k = N-1, \\
             \tau^{-\frac54+3\delta}, \quad k = N-2, \\
             \tau^{-\frac32+5\delta}, \quad k = N-3, \\
             \tau^{-\frac74+4\delta}, \quad k \le N-4.
         \end{cases}
    \end{aligned}
\end{align}
Repeating the argument above shows that the top-order weighted energy is uniformly bounded. More precisely,
    \begin{align}
        \begin{aligned}
            {}^{(2-\delta)}\mathcal E^D_{\le N-1}[\psi]^2(\tau) \le C^2\varepsilon^2, \quad {}^{(1)}\mathcal E^D_{\le N-1}[\psi]^2(\tau) \le C^2\varepsilon^2 \tau^{-1+\delta}, \quad {}^{(0)}\mathcal E^D_{\le N-4}[\psi]^2(\tau) \le C^2\varepsilon^2\tau^{-2+\delta}.
        \end{aligned}
    \end{align}

\subsubsection{Decay of ${}^{(1)}\mathcal E^D_{\le N-4}[\psi]$}
In particular, when $k\le N-4$, Sobolev embedding applied to the higher-order term $\nabla_{\underline L}\mathscr L_Z^{\le N-4}\psi$ yields
\begin{align*}
    & \left|  \int_{\mathcal D^{\tau_{i+1}}_{\tau_i}} r \langle \nabla_L \mathcal L_Z^{\le N-11} \mathrm{Tr}h \gamma^{\underline L}\gamma^L \nabla_{\underline L}\mathscr L_Z^{\le N-4} \psi, \gamma^T \nabla_{L}\mathscr L_Z^{\le N-4}\psi\rangle\,dV \right| \\
    & \lesssim \int_{\tau_i}^{\tau_{i+1}} r^{-\frac12} \mathcal E^W_{\le N-11}[h_{L\underline L}] \mathcal E^D_{\le N-4}[\psi] \mathcal E^D_{\le N-1}[\psi] \,d\tau \\
    & \qquad + \int_{\tau_i}^{\tau_{i+1}} r^{-\frac12} \mathcal E^W_{\le N-11}[h_{L\underline L}] \|\nabla_{\underline L}\mathscr L_Z^{\le N-2}\psi_+\|_{L^2(\Sigma_\tau)}  \mathcal E^D_{\le N-4}[\psi]\,d\tau \\
    & \qquad + \int_{\tau_i}^{\tau_{i+1}} r^{\frac12} \mathcal E^W_{\le N-11}[h_{L\underline L}] \|\mathscr L_Z^{\le N-2}(\Box_g\mathrm{Tr}h \cdot \psi)\|_{L^2(\Sigma_\tau)}  \mathcal E^D_{\le N-4}[\psi]\,d\tau \\
    & \qquad + \int_{\tau_i}^{\tau_{i+1}} r^{\frac12} \mathcal E^W_{\le N-11}[h_{L\underline L}] \|\mathscr L_Z^{\le N-2}(\mathscr T_{L\underline L} \psi)\|_{L^2(\Sigma_\tau)}  \mathcal E^D_{\le N-4}[\psi]\,d\tau \\
    & \qquad + \int_{\tau_i}^{\tau_{i+1}} r^{\frac12} \mathcal E^W_{\le N-11}[h_{L\underline L}] \|\mathscr L_Z^{\le N-2}(\mathscr T_{AB} \psi)\|_{L^2(\Sigma_\tau)}  \mathcal E^D_{\le N-4}[\psi]\,d\tau.
\end{align*}
Then we see that
\begin{align*}
    & \int_{\tau_i}^{\tau_{i+1}} r^{-\frac12} \mathcal E^W_{\le N-11}[h_{L\underline L}] \|\nabla_{\underline L}\mathscr L_Z^{\le N-2}\psi_+\|_{L^2(\Sigma_\tau)}  \mathcal E^D_{\le N-4}[\psi]\,d\tau \\
    & \lesssim \sup_{\tau\in[\tau_i,\tau_{i+1}]} \mathscr E^D_{\le N-1}[\psi](\tau) \left( \int_{\tau_i}^{\tau_{i+1}} {}^{(0)}\mathcal E^W_{\le N-11}[h_{L\underline L}]^2(\tau)\,d\tau  \right)^\frac12  \left( \int_{\tau_i}^{\tau_{i+1}} {}^{(0)}\mathcal E^D_{\le N-4}[\psi]^2(\tau)\,d\tau  \right)^\frac12  .
\end{align*}
Applying the Hardy inequality gives
\begin{align*}
     & \int_{\tau_i}^{\tau_{i+1}} r^{-\frac12} \mathcal E^W_{\le N-11}[h_{L\underline L}] \|\nabla_{\underline L}\mathscr L_Z^{\le N-2}\psi_+\|_{L^2(\Sigma_\tau)}  \mathcal E^D_{\le N-4}[\psi]\,d\tau  \\
     & \lesssim   \int_{\tau_i}^{\tau_{i+1}} r^{\frac12} \mathcal E^W_{\le N-11}[h_{L\underline L}](\tau)   \mathcal E^D_{\le N-4}[\psi](\tau) \mathcal E^D_{\le N-1}[\psi] \,d\tau \\
     & \lesssim  \sup_{\tau\in[\tau_i,\tau_{i+1}]}{}^{(0)} \mathcal E^W_{\le N-11}[h_{L\underline L}](\tau) \int_{\tau_i}^{\tau_{i+1}}  {}^{(1-\delta)}   \mathcal E^D_{\le N-4}[\psi] (\tau)\,d\tau ,
\end{align*}
which yields a slightly weaker bound.


\subsubsection{Decay of ${}^{(0)}\mathcal E^D_{\le N-4}[\psi]$}
Applying the weighted Sobolev inequality to the higher-order term $\nabla_{\underline L}\mathscr L_Z^{\le N-4}\psi$ reduces the estimate to the following integral:
\begin{align}
    \int_{\tau_i}^{\tau_{i+1}} r^{-\frac32} \mathcal E^W_{\le N-11}[h_{L\underline L}] \|\nabla_{\underline L}\mathscr L_Z^{\le N-2}\psi_+\|_{L^2(\Sigma_\tau)} \mathcal E^D_{\le N-4}[\psi]\,d\tau .
\end{align}
Then we use the Hardy inequality and the Cauchy-Schwarz inequality to get the following bound:
\begin{align*}
   \left( \int_{\tau_i}^{\tau_{i+1}} {}^{(0)}\mathcal E^W_{\le N-11}[h_{L\underline L}]^2(\tau) {}^{(0)}\mathcal E^D_{\le N-4}[\psi]^2(\tau)\,d\tau \right)^\frac12 \left( \int_{\tau_i}^{\tau_{i+1}} {}^{(0)}\mathcal E^D_{\le N-1}[\psi]^2(\tau)\,d\tau \right)^\frac12 .
\end{align*}

\subsection{Improved decay for the tensor fields}

In this section, we establish improved decay estimates for the tensor fields.

Rather than treating every null component of $h$, we focus on $h_{A\underline L}$ and $h_{\underline L\underline L}$. For the components $h_{LL}$, $h_{AB}$, $h_{L\underline L}$, and $h_{AL}$, the argument from the previous section gives the following:
\begin{prop}
    By a suitable iterative procedure, we have
    \begin{align}
        \begin{aligned}
            {}^{(2-\delta)}\mathcal E^W_{\le N-1}[h_{LL}]^2(\tau)+  {}^{(2-\delta)}\mathcal E^W_{\le N-1}[h_{AB}]^2(\tau)+  {}^{(2-\delta)}\mathcal E^W_{\le N-1}[h_{L\underline L}]^2(\tau)+  {}^{(2-\delta)}\mathcal E^W_{\le N-1}[h_{AL}]^2(\tau) \le C^2\varepsilon^2, \\
             {}^{(1)}\mathcal E^W_{\le N-1}[h_{LL}]^2(\tau)+  {}^{(1)}\mathcal E^W_{\le N-1}[h_{AB}]^2(\tau)+  {}^{(1)}\mathcal E^W_{\le N-1}[h_{L\underline L}]^2(\tau)+  {}^{(1)}\mathcal E^W_{\le N-1}[h_{AL}]^2(\tau) \le C^2\varepsilon^2 \tau^{-1+\delta} , \\
              {}^{(0)}\mathcal E^W_{\le N-4}[h_{LL}]^2(\tau)+  {}^{(0)}\mathcal E^W_{\le N-4}[h_{AB}]^2(\tau)+  {}^{(0)}\mathcal E^W_{\le N-4}[h_{L\underline L}]^2(\tau)+  {}^{(0)}\mathcal E^W_{\le N-4}[h_{AL}]^2(\tau) \le C^2\varepsilon^2 \tau^{-2+\delta} .
        \end{aligned}
    \end{align}
\end{prop}

Thanks to the improved decay for the spinor fields obtained in the previous subsection, we can improve the growth of the higher-order weighted energy for the components $h_{A\underline L}$ and $h_{\underline L\underline L}$:

\begin{prop}
    We have for $k\le N-1$,
    \begin{align}
        \begin{aligned}
         &   \left| \int_{\mathcal D^{\tau_{i+1}}_{\tau_i}} r^{} \mathcal L_Z^{\le k}\mathscr T_{\underline L\underline L}[\psi] \nabla_L\mathcal L_Z^{\le k}h_{\underline L\underline L}\,dV  \right| \\
            & \lesssim  \sup_{\tau\in[\tau_i,\tau_{i+1}]}\mathscr E^D_{\le k}[\psi](\tau) \left( \int_{\tau_i}^{\tau_{i+1}} {}^{(1)}\mathcal E^D_{\ll k}[\psi]^2(\tau)\,d\tau \right)^\frac12  \left( \int_{\tau_i}^{\tau_{i+1}} {}^{(0)}\mathcal E^W_{\le k}[h_{\underline L\underline L}]^2(\tau)\,d\tau \right)^\frac12 ,
        \end{aligned}
    \end{align}
    and for $k\le N-2$,
    \begin{align}
        \begin{aligned}
             &   \left| \int_{\mathcal D^{\tau_{i+1}}_{\tau_i}} r^{2-\delta} \mathcal L_Z^{\le k}\mathscr T_{\underline L\underline L}[\psi] \nabla_L\mathcal L_Z^{\le k}h_{\underline L\underline L}\,dV  \right| \\
             & \lesssim (\tau_{i+1})^{1-\delta} \sup_{\tau\in[\tau_i,\tau_{i+1}]} \mathscr E^D_{\le k}[\psi](\tau) \left(\int_{\tau_i}^{\tau_{i+1}}  {}^{(0)}\mathcal E^D_{\ll k}[\psi]^2(\tau)\,d\tau \right)^\frac12 \left(\int_{\tau_i}^{\tau_{i+1}}  {}^{(0)}\mathcal E^W_{\le k}[h_{\underline L\underline L}]^2(\tau)\,d\tau \right)^\frac12 
        \end{aligned}
    \end{align}
\end{prop}
Applying the preceding proposition and repeating the argument used to control ${}^{(2-\delta)}\mathcal E^W_{\le N-1}[h_{A\underline L}]$ yields
\begin{prop}
    For the components $h_{A\underline L}$ and $h_{\underline L\underline L}$, we have
    \begin{align}
        \begin{aligned}
          {}^{(2-\delta)}\mathcal E^W_{\le N-2}[h_{A\underline L}]^2(\tau) \le C^2\varepsilon^2 ,\quad {}^{(2-\delta)}\mathcal E^W_{\le N-1}[h_{A\underline L}]^2(\tau) \le C^2\varepsilon^2 \tau^{\frac12}, \\
          {}^{(0)}\mathcal E^W_{\le N-1}[h_{\underline L\underline L}]^2(\tau) \le C^2\varepsilon^2\tau^{-1+2\delta} , \quad  {}^{(1)}\mathcal E^W_{\le N-1}[h_{\underline L\underline L}]^2(\tau) \le C^2\varepsilon^2 \tau^{2\delta}, \quad  {}^{(2-\delta)}\mathcal E^W_{\le N-2}[h_{\underline L\underline L}]^2(\tau) \le C^2\varepsilon^2 \tau^{\frac12+\delta}. 
        \end{aligned}
    \end{align}
\end{prop}

Furthermore, the improved weighted-energy estimates for the lower-order cases $k\le N-9$ immediately imply the following:
\begin{cor}
    For $k\le N-9$, all null components of the tensor field $h_{\mu\nu}$ satisfy the following decay estimates:
    \begin{align}
        \begin{aligned}
            {}^{(2)}\mathcal E^W_{\le k}[h_{\mu\nu}]^2(\tau) \le C^2\varepsilon^2,\quad  {}^{(1)}\mathcal E^W_{\le k}[h_{\mu\nu}]^2(\tau) \le C^2\varepsilon^2 \tau^{-1}, \quad  {}^{(0)}\mathcal E^W_{\le k}[h_{\mu\nu}]^2(\tau) \le C^2\varepsilon^2\tau^{-2}.  
        \end{aligned}
    \end{align}
\end{cor}

The improved growth bound for ${}^{(2-\delta)}\mathcal E^W_{\le N-1}[h_{A\underline L}]$ follows by the same argument as in the previous section; we omit the details.

We consider the integral
\begin{align}
    \int_{\mathcal D^{\tau_{i+1}}_{\tau_i}} r \langle \mathscr L_Z^{\le N-11}\psi, \gamma^L\nabla_{\underline L}\mathscr L_Z^{\le N-1}\psi\rangle \nabla_L\mathcal L_Z^{\le N-1}h_{\underline L\underline L}\,dV,
\end{align}
If $\nabla_{\underline L}$ acts on the lower-order term $\mathscr L_Z^{\le N-11}\psi_+$, the estimate is better. Integrating by parts in the $\underline L$ direction reduces the remaining case to the following integral:
\begin{align}
 \int_{\mathcal D^{\tau_{i+1}}_{\tau_i}} \langle \mathscr L_Z^{\le N-11}\psi, \gamma^L \mathscr L_Z^{\le N-1}\psi\rangle \nabla_{\underline L} \mathcal L_Z^{\le N-1} h_{\underline L\underline L}\,dV.   
\end{align}
Applying the Sobolev inequality along the ingoing null hypersurfaces yields
\begin{align}
    \begin{aligned}
    & \left|  \int_{\mathcal D^{\tau_{i+1}}_{\tau_i}} \langle \mathscr L_Z^{\le N-11}\psi, \gamma^L \mathscr L_Z^{\le N-1}\psi\rangle \nabla_{\underline L} \mathcal L_Z^{\le N-1} h_{\underline L\underline L}\,dV \right| \\
    & \lesssim \sup_{v}{}^{(0)}\mathcal F^W_{\le N-1}[h_{\underline L\underline L}]  \int r^{-\frac12} \|\nabla_{\underline L}\mathscr L_Z^{\le N-9}\psi_+\|_{L^2(\underline{\mathcal C}_v)} \|\mathscr L_Z^{\le N-1}\psi_+\|_{L^2(\underline{\mathcal C}_v)}\,dv \\
    & \lesssim (\tau_{i+1})^{\frac12+\delta} \sup_{v}{}^{(0)}\mathcal F^W_{\le N-1}[h_{\underline L\underline L}] \left( \int_{\tau_i}^{\tau_{i+1}} {}^{(1-\delta)}\mathcal E^D_{\le N-8}[\psi]^2(\tau)\,d\tau \right)^\frac12 \left( \int_{\tau_i}^{\tau_{i+1}} {}^{(1-\delta)}\mathcal E^D_{\le N-1}[\psi]^2(\tau)\,d\tau \right)^\frac12, 
    \end{aligned}
\end{align}
where we used the Hardy inequality.

We now estimate ${}^{(2-\delta)}\mathcal E^W_{\le N-2}[h_{\underline L\underline L}]$. It suffices to control the following spacetime integral:
\begin{align}
    \int_{\mathcal D^{\tau_{i+1}}_{\tau_i}} r^{2-\delta} \mathcal L_Z^{\le N-2}\mathscr T_{\underline L\underline L}[\psi] \nabla_L\mathcal L_Z^{\le N-2}h_{\underline L\underline L}\,dV 
\end{align}
leads us to the control of the following spacetime integral:
\begin{align}
    \int_{\mathcal D^{\tau_{i+1}}_{\tau_i}} r^{1-\delta} \mathcal L_Z^{\le N-2}\mathscr T_{\underline L\underline L}[\psi] \nabla_{\underline L}\mathcal L_Z^{\le N-2}h_{\underline L\underline L}\,dV .
\end{align}
Integrating by parts in the $\underline L$ direction then gives
\begin{align}
      \int_{\mathcal D^{\tau_{i+1}}_{\tau_i}} r^{1-\delta}  \langle \mathscr L_Z^{\le N-11}\psi , \gamma^L \nabla_{\underline L} \mathscr L_Z^{\le N-2}\psi\rangle \mathcal L_Z^{\le N-2}h_{\underline L\underline L}\,dV,
\end{align}
while the control of the spacetime integral where the derivative $\nabla_{\underline L}$ acts on the lower-order term is much simpler. Now we apply H\"older's inequality to obtain the $L^4\times L^2\times L^4$ product and then the weighted Sobolev inequality. Then we obtain
\begin{align}
    \begin{aligned}
        & \int_{\tau_i}^{\tau_{i+1}} r^{1-\delta} \mathcal E^D_{\le N-9}[\psi]\mathscr E^D_{\le N-1}[\psi] \mathcal E^W_{\le N-1}[h_{\underline L\underline L}]\,d\tau \\
        & \lesssim (\tau_{i+1})^{\frac12-\delta} \sup_{\tau\in[\tau_i,\tau_{i+1}]} \mathscr E^D_{\le N-1}[\psi](\tau) \left(\int_{\tau_i}^{\tau_{i+1}}  {}^{(1)}\mathcal E^D_{\le N-9}[\psi]^2(\tau)\,d\tau \right)^\frac12 \left(\int_{\tau_i}^{\tau_{i+1}}  {}^{(0)}\mathcal E^W_{\le N-1}[h_{\underline L\underline L}]^2(\tau)\,d\tau \right)^\frac12 .
    \end{aligned}
\end{align}
\section{Interior domain: $r\le R$}\label{sec:interior}

This section controls the spinor and tensor fields in the interior domain $\{r\le R\}$. We prove the following:
\begin{prop}
    We have
    \begin{align}
        \begin{aligned}
            {}_{\rm int}\mathcal E^D_{\le N-1}[\psi]^2(\tau)+{}_{\rm int}\mathcal E^W_{\le N-1}[h_{\mu\nu}]^2(\tau) \le C^2\varepsilon^2 \tau^{\epsilon}, \\
             {}_{\rm int}\mathcal E^D_{\le N-1}[\psi]^2(\tau;|r-3M|>1)+{}_{\rm int}\mathcal E^W_{\le N-1}[h_{\mu\nu}]^2(\tau;|r-3M|>1) \le C^2\varepsilon^2 \tau^{-1+\delta}.
        \end{aligned}
    \end{align}
\end{prop}
We recall that near the event horizon $\mathcal H^+$, the vector fields \(T\), \(Y\), and the rotational vector fields \(\Omega\) form a nondegenerate frame for the spacetime derivatives, where \(Y\) is chosen to be regular and transversal to the horizon. Away from the event horizon, including in a neighborhood of the trapped set, the vector fields \(T\), \(\partial_{r^*}\), and \(\Omega\) span the spacetime tangent directions.

We may continue to use the null vector fields $L$ and $\underline L$. Indeed, with a suitable choice of this null pair, we have
\begin{align*}
    T = \frac12(L+\underline L), \ \partial_{r^*} = \frac12(L-\underline L),
\end{align*}
which implies that $\nabla_L\psi$ and $\nabla_{\underline L}\psi$ are linear combinations of $\nabla_T\psi$ and $\nabla_{\partial_{r^*}}\psi$. When one encounters the commutator with $L$ or $\underline L$, one can compute it as a linear combination of $T$ and $\partial_{r^*}$ instead of null vectors directly.

However, this decomposition raises a top-order derivative issue: $T$ is a tangential Killing field, whereas $\partial_{r^*}$ is transversal and is not Killing. Indeed, we have
\begin{align}
    {}^{(\partial_{r^*})} \pi_{tt} = - (1-\frac{2M}{r})\frac{M}{r^2} , \   {}^{(\partial_{r^*})} \pi_{r^* r^* } =  (1-\frac{2M}{r})\frac{M}{r^2}, \  {}^{(\partial_{r^*})} \pi_{AB} = \frac1r(1-\frac{2M}{r})\slashed{g}_{AB}, 
\end{align}
where
\begin{align*}
    g = -(1-\frac{2M}r)dt^2+(1-\frac{2M}r)(dr^*)^2+ \slashed{g}_{AB}d\omega^A d\omega^B.
\end{align*}
Thus the commutator with $\partial_{r^*}$ might be problematic near the trapping region especially at the top-order, due to the derivative loss from the non-degenerate integrated local energy estimates.

Nevertheless, the intrinsic structures of the Dirac and wave equations allow compositions involving these transversal vector fields to be rewritten, up to lower-order terms, using only the tangential Killing fields $T$ and $\Omega$, at the cost of differentiated nonlinearities. More precisely, we have the following:
\begin{lem}
    Away from the event horizon, including near the trapping region $\{r=3M\}$, one has
    \begin{align}
        \begin{aligned}
            {}_{\rm int} \mathscr E^D[ \mathscr L_{\partial_{r^*}}^{\le k}\psi]^2(\tau) & \lesssim \sum_{k_1+k_2\le k} {}_{\rm int}\mathscr E^D [\mathscr L_T^{\le k_1}\mathscr L_{\Omega}^{\le k_2}\psi]^2(\tau) + \sum_{Z\in \{T,\Omega, \partial_{r^*} \} } \| \mathscr L_Z^{\le k-1}( \nabla_\nu \mathrm{Tr}h \gamma^\nu \psi)\|_{L^2(\Sigma_\tau \cap \{  r\le R \} ) \cap {}^{(o)} (\mathcal H^+)^c } , \\
            & \qquad + {}_{\rm int}\mathcal E^D_{\le k-1}[\psi]^2(\tau), \\
           {}_{\rm int}\mathcal E^W[ \mathcal L_{\partial_{r^*}}^{\le k}h_{\mu\nu} ]^2(\tau) &\lesssim \sum_{k_1+k_2\le k} {}_{\rm int}\mathcal E^W[ \mathcal L_T^{\le k_1}\mathcal L_{\Omega}^{\le k_2}h_{\mu\nu} ] + \sum_{Z\in \{T,\Omega,\partial_{r^*} \} } \| \mathcal L_Z^{\le k-1} \mathscr T_{\mu\nu}[\psi]\|_{L^2(\Sigma_\tau \cap \{ r \le R \} \cap {}^{(o)}(\mathcal H^+)^c  ) } , \\
           & \qquad + {}_{\rm int}\mathcal E^W_{\le k-1}[ h_{\mu\nu} ]^2(\tau),
        \end{aligned}
    \end{align}
    where the notation ${}^{(o)}(\mathcal H^+)^c$ means an open neighborhood of the complement of the horizon $\mathcal H^+$.
\end{lem}
\begin{proof}
From the above metric $g$ in terms of the coordinates $(t,r^*,\omega)$, the Clifford algebra gives
\begin{align}
   ( \gamma^{\partial_{r^*}})^2 = 1-\frac{2M}{r}, \ (\gamma^t)^2 = - \left( 1-\frac{2M}{r} \right).
\end{align}

 From the Dirac equation $\gamma^{t}\nabla_t\psi+\gamma^{\partial_{r^*}}\nabla_{\partial_{r^*}}\psi + \gamma^{e_A}\nabla_{e_A}\psi = -i\nabla_\nu\mathrm{Tr}h \gamma^\nu \psi $, we immediately have
 \begin{align*}
   {}_{\rm int}  \mathscr E^D_{}[\nabla_{\partial_{r^*}}\psi]^2(\tau) \lesssim  {}_{\rm int}\mathscr E^D[\nabla_T\psi]^2(\tau)+ {}_{\rm int}\mathscr E^D[ \slashed\nabla\psi ]^2(\tau) + \| \nabla_\nu\mathrm{Tr}h \gamma^\nu \psi \|_{L^2(\Sigma_\tau \cap \{  r \le R \} \cap {}^{(o)}(\mathcal H^+)^c ) }^2
 \end{align*}
 To control the energy of the term $\mathscr L_{\partial_{r^*}}^{\le2}\psi$, it suffices to consider the energy $\mathscr E^D_{}[ \nabla_{\partial_{r^*}}^2\psi ]^2(\tau)$. To do this, we write using the Dirac equation,
 \begin{align*}
     \nabla_{\partial_{r^*}} ( \nabla_{\partial_{r^*}}\psi) & = \nabla_{\partial_{r^*}}( \mu^{-1} \gamma^{\partial_r^*}\gamma^{\partial_{r^*}}\nabla_{\partial_{r^*}}\psi) \\
     & =  \nabla_{\partial_{r^*}}( \mu^{-1} \gamma^{\partial_r^*}( \gamma^t\nabla_{t}\psi +\gamma^{e_A}\nabla_{e_A}\psi + \nabla_\nu \mathrm{Tr}h\gamma^\nu \psi  )  ) \\
     & = \mu^{-1} \gamma^{\partial_{r^*}}\left(\gamma^t \nabla_{\partial_{r^*}}\nabla_t\psi + \gamma^{e_A}\nabla_{\partial_{r^*}}\nabla_{e_A}\psi + \nabla_{\partial_{r^*}} ( \nabla_\nu \mathrm{Tr}h \gamma^\nu \psi)  \right) + (\textrm{lower-order terms}).
 \end{align*}
 Then we see that
 \begin{align*}
      \gamma^{\partial_{r^*}} \gamma^t \nabla_{\partial_{r^*}}\nabla_t\psi & =   \gamma^{\partial_{r^*}} \gamma^t\nabla_t \nabla_{\partial_{r^*}}\psi +  \gamma^{\partial_{r^*}} \gamma^t\mathscr R_{r^* t} \psi \\
      & = \gamma^t \nabla_t( \gamma^t\nabla_t\psi + \gamma^{e_A}\nabla_{e_A}\psi + \nabla_\nu \mathrm{Tr}h \gamma^\nu \psi) + (\textrm{lower-order terms}). 
 \end{align*}
 Thus, any composition of vector fields that includes the transversal vector $\partial_{r^*}$ can be rewritten, up to lower-order terms, in terms of the Killing fields $T$ and $\Omega$ and differentiated nonlinearities. The degeneracy at $r=2M$ is not problematic, since we are only concerned with the above argument away from the event horizon, especially near the trapping region $\{r=3M\}$, where the commutator with the transversal vector makes an issue of the derivative loss.

 Concerning the tensorial fields, we use the wave equation. In fact, for each components $h_{\mu\nu}$, we have
 \begin{align}
     \nabla_{\partial_{r^*}}^2 h_{\mu\nu} = \nabla_{T}^2h_{\mu\nu}- \mu \slashed\Delta h_{\mu\nu} + \mu \mathscr T_{\mu\nu}[\psi] + (\textrm{lower-order terms}).
 \end{align}
 In order to control the energy ${}_{\rm int}\mathcal E^W[ \nabla_{\partial_{r^*}}h_{\mu\nu} ]$, we first write
 \begin{align*}
     {}_{\rm int}\mathcal E^W[ \nabla_{\partial_{r^*}}h_{\mu\nu} ]^2(\tau) = \int_{ \Sigma_\tau \cap \{r\le R\} } |\nabla_{T}\nabla_{\partial_{r^*}}h_{\mu\nu}|^2+|\nabla_{\partial_{r^*}}\nabla_{\partial_{r^*}}h_{\mu\nu}|^2+|\slashed\nabla\nabla_{\partial_{r^*}} h_{\mu\nu}|^2\,d\sigma .
 \end{align*}
 Then we see that
 \begin{align*}
     & \int_{ \Sigma_\tau \cap \{r\le R\} } |\nabla_{\partial_{r^*}}\nabla_{\partial_{r^*}}h_{\mu\nu}|^2\,d\sigma \\
     & = \int_{ \Sigma_\tau \cap \{r\le R\} } |\nabla_T^2 h_{\mu\nu}|^2+|\slashed\nabla^2 h_{\mu\nu}|^2 + |\mathscr T_{\mu\nu}[\psi]|^2  \,d\sigma \\
     & \lesssim {}_{\rm int}\mathcal E^W[ \mathcal L_T^{\le1} h_{\mu\nu} ]^2(\tau)+ {}_{\rm int}\mathcal E^W [  \mathcal L_\Omega^{\le1} h_{\mu\nu} ]^2(\tau) +\| \mathscr T_{\mu\nu}[\psi] \|_{L^2(\Sigma_\tau \cap \{r\le R\}\cap {}^{(o)}(\mathcal H^+)^c  }^2
 \end{align*}
 and
 \begin{align*}
     & \int_{ \Sigma_\tau \cap \{r\le R\} } |\nabla_{T}\nabla_{\partial_{r^*}}h_{\mu\nu}|^2\,d\sigma \\
     & \lesssim \int_{ \Sigma_\tau \cap \{r\le R\} } |\nabla_{\partial_{r^*}} \nabla_{T}h_{\mu\nu}|^2 + |[ \nabla_T, \nabla_{\partial_{r^*}} ] h_{\mu\nu}|^2 \,d\sigma \\
     & \lesssim {}_{\rm int}\mathcal E^W_{}[ \nabla_T h_{\mu\nu} ]^2(\tau) + {}_{\rm int}\mathcal E^W [ h_{\mu\nu}]^2(\tau),
 \end{align*}
 and similarly we have
 \begin{align*}
     & \int_{ \Sigma_\tau \cap \{r\le R\} } |\slashed\nabla\nabla_{\partial_{r^*}} h_{\mu\nu}|^2\,d\sigma  \lesssim {}_{\rm int}\mathcal E^W[ \mathcal L_\Omega^{} h_{\mu\nu} ]^2(\tau) +  {}_{\rm int}\mathcal E^W[  h_{\mu\nu} ]^2(\tau)  .
 \end{align*}
To control the energy ${}_{\rm int}\mathcal E^W[ \nabla_{\partial_{r^*}}^2h_{\mu\nu} ] $, we directly apply the above wave equation. For higher-order cases, we simply follow an inductive argument, which completes the proof. 
\end{proof}
The preceding lemma shows that the derivative loss in the non-degenerate integrated local energy estimates does not affect the commutator bounds. We therefore focus on estimating the nonlinearity, particularly at top order.
\subsection{Control of the energy near the trapping region}
We focus on controlling the solution near the trapping region, $|r-3M|<c$, for a fixed small $c>0$.
\subsubsection{Control of the spinor fields}
We first recall the fundamental energy inequality for the Dirac equation in terms of the wave-type on the interior domain $\{r\le R\}$ with $3M< R$:
\begin{align}
    \begin{aligned}
      & {}_{\rm int} \mathcal E^D_{\le k}[\psi]^2(\tau_{i+1})+{}_{}\mathcal E^{D,\mathcal H^+ }_{\le k}[\psi]^2(\tau_i,\tau_{i+1}) + \int_{\tau_i}^{\tau_{i+1}} \chi  {}_{\rm int}\mathcal E^D_{\le k}[\psi]^2(\tau)\,d\tau  \\
      & \qquad \lesssim {}_{\rm int}\mathcal E^D_{\le k}[\psi]^2(\tau_i) + \mathcal B^{R }_{\le k}[\psi]^2(\tau_i,\tau_{i+1})+\left| \int_{\tau_i}^{\tau_{i+1}}\int_{\Sigma_{\tau}\cap \{r\le R\} }  \langle \mathscr L_Z^{\le k}(\nabla_{\nu}\mathrm{Tr}h \gamma^\lambda\gamma^\nu \nabla_\lambda\psi), \gamma^T \nabla_{\mathcal N}\mathscr L_Z^{\le k}\psi\rangle \,d\sigma_{\Sigma_\tau} d\tau\right|,  
    \end{aligned}
\end{align}
where the function $\chi$ degenerate near the trapping region $\{r=3M\}$.
Here we note that the vector field $\mathcal N$ has been chosen so that it is future-directed time-like, and is equal to the vector field $T$ away from an open neighborhood of the event horizon. Furthermore, the bulk term associated with $\mathcal N$ has a good sign near the event horizon.

We need to consider the integral
\begin{align}
    \int_{\tau_i}^{\tau_{i+1}}\int_{r\le R} \langle \nabla_\nu \mathcal L_Z^{\le k_1} \mathrm{Tr}h \gamma^\lambda\gamma^\nu \nabla_\lambda\mathscr L_Z^{\le k_2} \psi, \gamma^T \nabla_{\mathcal N}\mathscr L_Z^{\le N-1}\psi\rangle\,dV.
\end{align}
where $k_1+k_2\le N-1$.
We recall the integrated local energy decay estimates in the interior domain:
\begin{align}
    \begin{aligned}
\int_{\tau_i}^{\tau_{i+1}}\int_{r\le R} \mathcal J^{\mathcal N}[ \mathscr L_Z^{\le k}\psi ] \,d\sigma_{\Sigma_\tau} d\tau & \lesssim {}_{\rm int}\mathcal E_{\le k+1}^D[\psi]^2(\tau_{i+1})+{}_{\rm int}\mathcal E_{\le k+1}^D[\psi]^2(\tau_{i}) \\
    & \qquad + \mathcal E^{D,\mathcal H^+}_{\le k+1}[\psi]^2(\tau_i,\tau_{i+1}) + \mathcal B^{R}_{\le k+1}[\psi]^2(\tau_i,\tau_{i+1}) \\
    & + \left| \int_{\tau_i}^{\tau_{i+1}}\int_{\Sigma_{\tau}\cap \{r\le R\} }  \langle \mathscr L_Z^{\le k+1}(\nabla_{\nu}\mathrm{Tr}h \gamma^\lambda\gamma^\nu \nabla_\lambda\psi), \gamma^T \nabla_{\mathcal N}\mathscr L_Z^{\le k+1}\psi\rangle \,d\sigma_{\Sigma_\tau} d\tau\right|.
    \end{aligned}
\end{align}
We consider the cases $k_1=N-1$ and $k_2=N-1$.
\begin{align*}
    & \int_{\tau_i}^{\tau_{i+1}}\int_{r\le R} \langle \nabla_\nu \mathcal L_Z^{\le N-1} \mathrm{Tr}h \gamma^\lambda\gamma^\nu \nabla_\lambda \psi, \gamma^T \nabla_{\mathcal N}\mathscr L_Z^{\le N-1}\psi\rangle\,dV \\
    & \lesssim \sup_{\tau\in[\tau_i,\tau_{i+1}]} \left( {}_{\rm int}\mathcal E^W_{\le N-1}[h_{L\underline L}](\tau)+  {}_{\rm int}\mathcal E^W_{\le N-1}[h_{AB}](\tau)  \right) {}_{\rm int}\mathcal E^D_{\le N-1}[\psi](\tau)  \int_{\tau_i}^{\tau_{i+1}}{}_{\rm int}  \mathcal E^D_{\le N-11}[\psi](\tau)\,d\tau  \\
    & \lesssim (\tau_{i+1}-\tau_i)^\frac12 \sup_{\tau\in[\tau_i,\tau_{i+1}]} \left( {}_{\rm int}\mathcal E^W_{\le N-1}[h_{L\underline L}](\tau)+  {}_{\rm int}\mathcal E^W_{\le N-1}[h_{AB}](\tau)  \right) {}_{\rm int}\mathcal E^D_{\le N-1}[\psi](\tau)  \left( \int_{\tau_i}^{\tau_{i+1}}{}_{\rm int}  \mathcal E^D_{\le N-11}[\psi]^2(\tau)\,d\tau  \right)^\frac12
\end{align*}
and if $k_2=N-1$, then
\begin{align*}
    & \int_{\tau_i}^{\tau_{i+1}}\int_{r\le R} \langle \nabla_\nu \mathrm{Tr}h \gamma^\lambda\gamma^\nu \nabla_\lambda \mathscr L_Z^{\le N-1} \psi, \gamma^T \nabla_{\mathcal N}\mathscr L_Z^{\le N-1}\psi\rangle\,dV \\
    & \lesssim \sup_{\tau\in[\tau_i,\tau_{i+1}]} {}_{\rm int}\mathcal E^D_{\le N-1}[\psi]^2(\tau) \int_{\tau_i}^{\tau_{i+1}}{}_{\rm int}  \mathcal E^W_{\le N-11}[h_{L\underline L}](\tau)\,d\tau  \\
    & \lesssim (\tau_{i+1}-\tau_i)^\frac12 \sup_{\tau\in[\tau_i,\tau_{i+1}]} {}_{\rm int}\mathcal E^D_{\le N-1}[\psi]^2(\tau) \left( \int_{\tau_i}^{\tau_{i+1}}{}_{\rm int}  \mathcal E^W_{\le N-11}[h_{L\underline L}]^2 (\tau)\,d\tau \right)^\frac12.
\end{align*}
On the other hand, if $k_1+k_2\le N-2$, 
\begin{align*}
    & \left| \int_{\tau_i}^{\tau_{i+1}}\int_{\Sigma_\tau\cap\{ r\le R \} } \langle \nabla_\nu \mathcal L_Z^{\le N-11}\mathrm{Tr}h \gamma^\lambda\gamma^\nu \nabla_\lambda \mathscr L_Z^{\le N-2} \psi, \gamma^T\nabla_{\mathcal N}\mathscr L_Z^{\le N-2}\psi \rangle \,dV \right| \\
    & \lesssim \int_{\tau_i}^{\tau_{i+1}} {}_{\rm int}\mathcal E^W_{\le N-9}[h_{L\underline L}](\tau) {}_{\rm int}\mathcal E^D_{\le N-2}[\psi]^2(\tau)\,d\tau \\
    & \lesssim \sup_{\tau\in[\tau_i,\tau_{i+1}]} {}_{\rm int}\mathcal E^W_{\le N-9}[h_{L\underline L}](\tau) \int_{\tau_i}^{\tau_{i+1}}  {}_{\rm int}\mathcal E^D_{\le N-2}[\psi]^2(\tau)\,d\tau \\
    & \lesssim \sup_{\tau\in[\tau_i,\tau_{i+1}]} {}_{\rm int}\mathcal E^W_{\le N-8}[h_{L\underline L}](\tau) {}_{\rm int}\mathcal E^D_{\le N-1}[\psi]^2(\tau),
\end{align*}
and
\begin{align*}
    & \left| \int_{\tau_i}^{\tau_{i+1}}\int_{\Sigma_\tau\cap\{ r\le R \} } \langle \nabla_\nu \mathcal L_Z^{\le N-2}\mathrm{Tr}h \gamma^\lambda\gamma^\nu \nabla_\lambda \mathscr L_Z^{\le N-11} \psi, \gamma^T\nabla_{\mathcal N}\mathscr L_Z^{\le N-2}\psi \rangle \,dV \right| \\
    & \lesssim \int_{\tau_i}^{\tau_{i+1}} {}_{\rm int}\mathcal E^W_{\le N-2}[h_{L\underline L}](\tau) {}_{\rm int}\mathcal E^D_{\le N-11}[\psi](\tau){}_{\rm int}\mathcal E^D_{\le N-2}[\psi](\tau)\,d\tau \\
    & \lesssim \sup_{\tau\in[\tau_i,\tau_{i+1}]} {}_{\rm int}\mathcal E^D_{\le N-8}[\psi](\tau) \left(  \int_{\tau_i}^{\tau_{i+1}}  {}_{\rm int}\mathcal E^D_{\le N-2}[\psi]^2(\tau)\right)^\frac12 \left(  \int_{\tau_i}^{\tau_{i+1}}  {}_{\rm int}\mathcal E^W_{\le N-2}[h_{L\underline L}]^2(\tau)\right)^\frac12 \\
    & \lesssim \sup_{\tau\in[\tau_i,\tau_{i+1}]} {}_{\rm int}\mathcal E^D_{\le N-8}[\psi](\tau) {}_{\rm int}\mathcal E^D_{\le N-1}[\psi](\tau){}_{\rm int}\mathcal E^W_{\le N-1}[h_{L\underline L}](\tau),
\end{align*}
where we used the integrated local energy decay estimates without the degeneracy at $r=3M$ at the expense of one derivative loss. However, when using the integrated local energy decay estimates, we also have to consider the spacetime integral of the nonlinearity. More precisely, the Morawetz estimates for the spinor fields on the interior domain say
\begin{align}
    \begin{aligned}
        \int_{\tau_i}^{\tau_{i+1}} {}_{\rm int}\mathcal E^D_{\le N-2}[\psi]^2(\tau)\,d\tau & \lesssim {}_{\rm int}\mathcal E^D_{\le N-1}[\psi]^2(\tau_{i+1}) + {}_{\rm int}\mathcal E^D_{\le N-1}[\psi]^2(\tau_{i}) \\
        & \qquad + {}_{\rm int}\mathcal E^{D,\mathcal H^+}_{\le N-1}[\psi]^2(\tau_i,\tau_{i+1}) + \mathcal B^R_{\le N-1}[\psi]^2(\tau_i,\tau_{i+1}) \\
        & \qquad + \left| \int_{\tau_i}^{\tau_{i+1}}\int_{\Sigma_\tau\cap\{r\le R\} }  \langle \mathscr L_Z^{\le N-1}(\nabla_{\nu}\mathrm{Tr}h \gamma^\lambda\gamma^\nu \nabla_\lambda\psi), \gamma^T \nabla_{\mathcal N}\mathscr L_Z^{\le N-1}\psi\rangle \,d\sigma_{\Sigma_\tau} d\tau \right|.
    \end{aligned}
\end{align}
For the spacetime integral of the nonlinearity, we repeat the previous argument and we obtain
\begin{align*}
   &  \left| \int_{\tau_i}^{\tau_{i+1}} \int_{\Sigma_\tau\cap\{r\le R\} } \langle \mathscr L_Z^{\le N-1}(\nabla_{\nu}\mathrm{Tr}h \gamma^\lambda\gamma^\nu \nabla_\lambda\psi), \gamma^T \nabla_{\mathcal N}\mathscr L_Z^{\le N-1}\psi\rangle \,d\sigma_{\Sigma_\tau} d\tau \right| \\
   & \lesssim (\tau_{i+1}-\tau_i)^\frac12 \sup_{\tau\in[\tau_i,\tau_{i+1}]} \left( {}_{\rm int}\mathcal E^W_{\le N-1}[h_{L\underline L}](\tau)+  {}_{\rm int}\mathcal E^W_{\le N-1}[h_{AB}](\tau)  \right) {}_{\rm int}\mathcal E^D_{\le N-1}[\psi](\tau)  \left( \int_{\tau_i}^{\tau_{i+1}}{}_{\rm int}  \mathcal E^D_{\le N-11}[\psi]^2(\tau)\,d\tau  \right)^\frac12 \\
   & \qquad +  (\tau_{i+1}-\tau_i)^\frac12 \sup_{\tau\in[\tau_i,\tau_{i+1}]} {}_{\rm int}\mathcal E^D_{\le N-1}[\psi]^2(\tau) \left( \int_{\tau_i}^{\tau_{i+1}}{}_{\rm int}  \mathcal E^W_{\le N-11}[h_{L\underline L}]^2 (\tau)\,d\tau \right)^\frac12,
\end{align*}
and hence this can be absorbed into the energy ${}_{\rm int}\mathcal E^D_{\le N-1}[\psi]^2(\tau_i)$. The energy flux ${}_{\rm int}\mathcal E^{D,\mathcal H^+}_{\le N-1}[\psi]^2(\tau_i,\tau_{i+1})$ on the event horizon $\mathcal H^+$ is also immediately controlled by the energy ${}_{\rm int}\mathcal E^D_{\le N-1}[\psi]^2(\tau_i)$ in view of the energy inequality, while the energy flux $\mathcal B^R_{\le N-1}[\psi]^2(\tau_i,\tau_{i+1})$ along the time-like surface $\{r=R\}$ is controlled by the far-region energy using an averaging argument. Indeed, we have $\mathcal B^R_{\le N-1}[\psi]^2(\tau_i,\tau_{i+1}) \lesssim \sup_{\tau\in[\tau_i,\tau_{i+1}]} {}^{(0)} \mathcal E^D_{\le N-1}[\psi]^2(\tau) $.

The integrated local energy decay estimates for the tensor fields in the interior domain are as follows:
\begin{align}
    \begin{aligned}
        \int_{\tau_i}^{\tau_{i+1}} {}_{\rm int}\mathcal E^W_{\le N-2}[h_{\mu\nu}]^2(\tau)\,d\tau & \lesssim {}_{\rm int}\mathcal E^W_{\le N-1}[h_{\mu\nu}]^2(\tau_{i+1})+ {}_{\rm int}\mathcal E^W_{\le N-1}[h_{\mu\nu}]^2(\tau_{i}) \\
        & \qquad +{}_{\rm int}\mathcal E^{W,\mathcal H^+}_{\le N-1}[h_{\mu\nu}]^2(\tau_i,\tau_{i+1}) + \mathcal B^R_{\le N-1}[h_{\mu\nu}]^2(\tau_i,\tau_{i+1}) \\
        & \qquad + \left| \int_{\tau_i}^{\tau_{i+1}}\int_{\Sigma_\tau\cap\{r\le R\} } \mathcal L_Z^{\le N-1}\langle \psi,\gamma_{\mu}\nabla_\nu\psi\rangle \nabla_{\mathcal N}\mathcal L_Z^{\le N-1}h_{\mu\nu}\,d\sigma_{\Sigma_\tau}d\tau  \right|.
    \end{aligned}
\end{align}
Then the spacetime integral of the nonlinearity can be controlled as follows:
\begin{align*}
    &  \left| \int_{\tau_i}^{\tau_{i+1}}\int_{\Sigma_\tau\cap\{r\le R\} } \mathcal L_Z^{\le N-1}\langle \psi,\gamma_{\mu}\nabla_\nu\psi\rangle \nabla_{\mathcal N}\mathcal L_Z^{\le N-1}h_{\mu\nu}\,d\sigma_{\Sigma_\tau}d\tau  \right| \\
    & \lesssim (\tau_{i+1}-\tau_i)^{\frac12} \sup_{\tau\in[\tau_i,\tau_{i+1}]} {}_{\rm int}\mathcal E^D_{\le N-1}[\psi](\tau){}_{\rm int}\mathcal E^W_{\le N-1}[h_{\mu\nu}](\tau) \left( \int_{\tau_i}^{\tau_{i+1}} {}_{\rm int}\mathcal E^W_{\le N-11}[\psi]^2(\tau)\,d\tau \right)^\frac12,
\end{align*}
which can be absorbed into the energy ${}_{\rm int}\mathcal E^W_{\le N-1}[h_{\mu\nu}]^2(\tau_i)+{}_{\rm int}\mathcal E^D_{\le N-1}[\psi]^2(\tau_i) $. The energy inequality also controls the fluxes through the event horizon and the timelike surface.

\subsubsection{Control of the tensor fields}
As in the previous subsection, we apply the energy inequality to the tensor field $h_{\mu\nu}$:
\begin{align}
    \begin{aligned}
      & {}_{\rm int} \mathcal E^W_{\le k}[h]^2(\tau_{i+1})+{}_{}\mathcal E^{W,\mathcal H^+ }_{\le k}[h]^2(\tau_i,\tau_{i+1}) + \int_{\tau_i}^{\tau_{i+1}} \chi  {}_{\rm int}\mathcal E^W_{\le k}[h]^2(\tau)\,d\tau  \\
      & \qquad \lesssim {}_{\rm int}\mathcal E^W_{\le k}[h]^2(\tau_i) + \mathcal B^{R }_{\le k}[h]^2(\tau_i,\tau_{i+1})+\left|\int_{\tau_i}^{\tau_{i+1}}\int_{\Sigma_\tau\cap \{r\le R\} }  \mathcal L_Z^{\le k} \langle \psi, \gamma_{\mu}\nabla_\nu \psi\rangle \nabla_{\mathcal N}\mathcal L_Z^{\le k}h_{\mu\nu}  \,d\sigma_{\Sigma_\tau} d\tau\right|.  
    \end{aligned}
\end{align}
At the highest derivative order, $k=N-1$, it suffices to estimate the following integral:
\begin{align}
    \int_{\tau_i}^{\tau_{i+1}}\int_{\Sigma_\tau\cap \{r\le R\} }  \langle \psi, \gamma_{\mu}\nabla_\nu\mathscr L_Z^{\le N-1}  \psi\rangle \nabla_{\mathcal N}\mathcal L_Z^{\le N-1}h_{\mu\nu}  \,d\sigma_{\Sigma_\tau} d\tau. 
\end{align}
Using the Sobolev inequality, we have
\begin{align*}
    & \left| \int_{\tau_i}^{\tau_{i+1}}\int_{\Sigma_\tau\cap \{r\le R\} }  \langle \psi, \gamma_{\mu}\nabla_\nu\mathscr L_Z^{\le N-1}  \psi\rangle \nabla_{\mathcal N}\mathcal L_Z^{\le N-1}h_{\mu\nu}  \,d\sigma_{\Sigma_\tau} d\tau \right| \\
    & \lesssim \int_{\tau_i}^{\tau_{i+1}} {}_{\rm int}\mathcal E^D_{\le N-9}[\psi](\tau) {}_{\rm int}\mathcal E^D_{\le N-1}[\psi](\tau) {}_{\rm int}\mathcal E^W_{\le N-1}[h](\tau)\,d\tau \\
    & \lesssim \sup_{\tau\in[\tau_i,\tau_{i+1}]}  {}_{\rm int}\mathcal E^D_{\le N-1}[\psi](\tau) {}_{\rm int}\mathcal E^W_{\le N-1}[h](\tau)\int_{\tau_i}^{\tau_{i+1}} {}_{\rm int}\mathcal E^D_{\le N-9}[\psi](\tau) \,d\tau \\
    & \lesssim (\tau_{i+1}-\tau_i)^{\frac12} \sup_{\tau\in[\tau_i,\tau_{i+1}]}  {}_{\rm int}\mathcal E^D_{\le N-1}[\psi](\tau) {}_{\rm int}\mathcal E^W_{\le N-1}[h](\tau)  \left( \int_{\tau_i}^{\tau_{i+1}} {}_{\rm int}\mathcal E^D_{\le N-9}[\psi]^2(\tau) \,d\tau \right)^\frac12,
\end{align*}
in which case a logarithmic growth occurs.

If $k_1+k_2\le N-2$, we see that
\begin{align*}
   &  \left| \int_{\tau_i}^{\tau_{i+1}}\int_{\Sigma_\tau\cap \{r\le R\} }  \langle \psi, \gamma_{\mu}\nabla_\nu\mathscr L_Z^{\le N-2}  \psi\rangle \nabla_{\mathcal N}\mathcal L_Z^{\le N-2}h_{\mu\nu}  \,d\sigma_{\Sigma_\tau} d\tau \right|  \\
   & \lesssim \int_{\tau_i}^{\tau_{i+1}} {}_{\rm int}\mathcal E^D_{\le N-9}[\psi](\tau) {}_{\rm int}\mathcal E^D_{\le N-2}[\psi](\tau){}_{\rm int}\mathcal E^W_{\le N-2}[h](\tau)\,d\tau \\
   & \lesssim \sup_{\tau\in[\tau_i,\tau_{i+1}]} {}_{\rm int}\mathcal E^D_{\le N-9}[\psi](\tau) \int_{\tau_i}^{\tau_{i+1}} {}_{\rm int}\mathcal E^D_{\le N-2}[\psi]^2(\tau)+ {}_{\rm int}\mathcal E^W_{\le N-2}[h]^2(\tau)\,d\tau \\
   & \lesssim \sup_{\tau\in[\tau_i,\tau_{i+1}]} {}_{\rm int}\mathcal E^D_{\le N-9}[\psi](\tau)\left( {}_{\rm int}\mathcal E^D_{\le N-1}[\psi]^2(\tau)+ {}_{\rm int}\mathcal E^W_{\le N-1}[h]^2(\tau) \right),
\end{align*}
where we used the integrated local energy decay estimates without the degeneracy near $\{r=3M\}$. 

\subsection{Control of the energy away from the trapping region }
Here we control the energy of the spinor fields on the interior domain away from the trapping region $\{r=3M\}$, i.e., 
\begin{align}
  {}_{\rm int}^{}\mathcal E^D_{\le k}[\psi]^2(\tau; |r-3M|>1 )  :=  \int_{\Sigma_{\tau}\cap\{r\le R\}\cap \{ |r-3M|>1 \} } \mathcal J^{\mathcal N}[ \mathscr L_Z^{\le k}\psi ]\,d\sigma_{\Sigma_\tau}.
\end{align}
To do this, we write
\begin{align*}
    & \left| \int_{\tau_i}^{\tau_{i+1}}\int_{r\le R \cap \{|r-3M|>1\} } \langle \nabla_\nu \mathcal L_Z^{\le N-1} \mathrm{Tr}h \gamma^\lambda\gamma^\nu \nabla_\lambda \psi, \gamma^T \nabla_{\mathcal N}\mathscr L_Z^{\le N-1}\psi\rangle\,dV \right| \\
    & \lesssim \sup_{\tau\in[\tau_i,\tau_{i+1}]} {}_{\rm int}\mathcal E^D_{\le N-13}[\psi](\tau) \int_{\tau_i}^{\tau_{i+1}} \chi \left( {}_{\rm int}\mathcal E^W_{\le N-1}[h_{L\underline L}]^2(\tau)+ {}_{\rm int}\mathcal E^W_{\le N-1}[h_{AB}]^2(\tau)+ {}_{\rm int}\mathcal E^D_{\le N-1}[\psi]^2(\tau) \right)\,d\tau \\
    & \lesssim \sup_{\tau\in[\tau_i,\tau_{i+1}]} {}_{\rm int}\mathcal E^D_{\le N-13}[\psi](\tau) & \\
    & \qquad \times \left( {}_{\rm int}\mathcal E^W_{\le N-1}[h_{L\underline L}]^2(\tau; |r-3M|>1)+ {}_{\rm int}\mathcal E^W_{\le N-1}[h_{AB}]^2(\tau; |r-3M|>1)+ {}_{\rm int}\mathcal E^D_{\le N-1}[\psi]^2(\tau; |r-3M|>1) \right),
\end{align*}
Here we used the integrated local energy decay estimate that degenerates at the trapped set. Since the domain is away from trapping, we may insert $\chi$ freely. The boundary flux along $r=R$ is therefore controlled solely by the exterior energy ${}^{(0)}\mathcal E^D_{\le N-1}[\psi]$. The tensor-field estimate follows similarly; we omit the details.

Thus the spacetime integral of the nonlinearity satisfies the following inequality:
\begin{align}
    \begin{aligned}
        \left| \int_{\tau_i}^{\tau_{i+1}}\int_{r\le R \cap \{|r-3M|>1\} } \langle \nabla_\nu \mathcal L_Z^{\le N-1} \mathrm{Tr}h \gamma^\lambda\gamma^\nu \nabla_\lambda \psi, \gamma^T \nabla_{\mathcal N}\mathscr L_Z^{\le N-1}\psi\rangle\,dV \right|  & \lesssim \sup_{\tau\in[\tau_i,\tau_{i+1}]} {}_{\rm int}\mathcal E^D_{\le N-13}[\psi](\tau) ,
    \end{aligned}
\end{align}
and using the $r^p$-hierarchy, we have
\begin{align}
    &  {}_{\rm int}\mathcal E^D_{\le N-13}[\psi]^2(\tau) \lesssim \tau^{-2} {}_{\rm int}\mathcal E^D_{\le N-11}[\psi]^2(\tau).
\end{align}
\section{Improved energy estimates for the spinor fields}\label{sec:dirac-current}
This section estimates the spinor fields using the Dirac current.

As discussed in Remarks \ref{rem-energy-hierarchy}, \ref{rem-energy-hierarchy1}, and \ref{rem-energy-hierarchy2}, the spinorial wave energy and the Dirac-current energy play complementary roles. The nonlinear estimates cannot be closed using the wave energy alone. The Dirac-current energy is essential at both lower and higher derivative orders.

Moreover, the top-order Dirac-current energy has not been used in the preceding nonlinear estimates. It therefore suffices to control the Dirac-current energy for $k\le N-1$.

We recall the Dirac equation:
\begin{align}
    i\gamma^\mu\nabla_\mu\psi = \nabla_\nu \mathrm{Tr}h \gamma^\nu \psi.
\end{align}
Applying the vector fields $\mathscr L_Z^{\le N-1}$, we have
\begin{align}
    i\gamma^\mu\nabla_\mu( \mathscr L_Z^{\le N-1}\psi) = [ i\gamma^\mu\nabla_\mu, \mathscr L_Z^{\le N-1}  ]\psi + \sum_{k\ge0}^{N-1} \nabla_\nu \mathcal L_Z^{\le N-k}\mathrm{Tr}h \gamma^\nu \mathscr L_Z^{\le k}\psi .
\end{align}

\subsection{Control on the far-region}
Now we apply the divergence theorem on the far region $\mathcal D^{\tau_{i+1}}_{\tau_i} $:
\begin{align}
    \begin{aligned}
        & \int_{\Sigma_{\tau_{i+1}}} \langle \mathscr L_Z^{\le N-1}\psi , \gamma^L \mathscr L_Z^{\le N-1}\psi\rangle \,d\sigma + \int_{\mathcal I^{\tau_{i+1}-R}_{\tau_i-R}}  \langle \mathscr L_Z^{\le N-1}\psi, \gamma^{\underline L} \mathscr L_Z^{\le N-1}\psi\rangle\,d\sigma  \\
        & = \int_{\mathcal D^{\tau_{i+1}}_{\tau_i}} \langle [ i\gamma^\mu\nabla_\mu, \mathscr L_Z^{\le N-1} ]\psi, \mathscr L_Z^{\le N-1}\psi\rangle\,dV + \sum_{k\ge0}^{N-1} \int_{\mathcal D^{\tau_{i+1}}_{\tau_i}} \langle  \nabla_\nu \mathcal L_Z^{\le N-k}\mathrm{Tr}h \gamma^\nu \mathscr L_Z^{\le k}\psi , \mathscr L_Z^{\le N-1}\psi\rangle\,dV.
    \end{aligned}
\end{align}
The commutator terms were estimated in the preceding section, so we focus here on the spacetime integral of the nonlinearity.

\subsubsection{$\nu = L $ }
It is enough to consider the low-high and high-low interactions:
\begin{align}
     \int_{ \mathcal D^{\tau_{i+1}}_{\tau_i}} \langle  \nabla_{L} \mathcal L_Z^{\le N-13}\mathrm{Tr}h \gamma^{L} \mathscr L_Z^{\le N-1}\psi , \mathscr L_Z^{\le N-1}\psi\rangle\,dV, \\
      \int_{ \mathcal D^{\tau_{i+1}}_{\tau_i}} \langle  \nabla_{L} \mathcal L_Z^{\le N-1}\mathrm{Tr}h \gamma^{L} \mathscr L_Z^{\le N-13}\psi , \mathscr L_Z^{\le N-1}\psi\rangle\,dV.
\end{align}
For the low-high interaction, we examine the commutation fields in $\mathscr L_Z^{\le N}\psi$. If it contains at least one $L$, then we immediately have $\|\mathscr L_Z^{\le N}\psi_+\|_{L^2(\Sigma_\tau)}\lesssim \mathcal E^D_{\le N-1}[\psi](\tau) $. If it consists of at least one rotation $\Omega$, then $\|\mathscr L_Z^{\le N}\psi_+\|_{L^2(\Sigma_\tau)}\lesssim r\mathcal E^D_{\le N-1}[\psi](\tau) $. On the other hand, if the vector fields consist of only $\underline L$, we apply the Hardy inequality and get
\begin{align*}
    & \left|  \int_{ \mathcal D^{\tau_{i+1}}_{\tau_i}} \langle  \nabla_{L} \mathcal L_Z^{\le N-13}\mathrm{Tr}h \gamma^{L} \mathscr L_Z^{\le N-1}\psi , \mathscr L_Z^{\le N-1}\psi\rangle\,dV \right| \\
    & \lesssim \int_{\tau_i}^{\tau_{i+1}} r^{\frac12} \mathcal E^D_{\le N-10}[\psi]\mathcal E^D_{\le N-1}[\psi]\mathscr E^D_{\le N-1}[\psi]\,d\tau \\
    & \lesssim (\tau_{i+1})^{\frac\delta2} \sup_{\tau\in[\tau_i,\tau_{i+1}]} \mathscr E^D_{\le N-1}[\psi](\tau) \left( \int_{\tau_i}^{\tau_{i+1}} {}^{(0)}\mathcal E^D_{\le N-10}[\psi]^2(\tau)\,d\tau \right)^\frac12 \left( \int_{\tau_i}^{\tau_{i+1}} {}^{(1-\delta)}\mathcal E^D_{\le N-1}[\psi]^2(\tau)\,d\tau \right)^\frac12  .
\end{align*}

For the high-low interaction, we simply apply the weighted Sobolev inequality to the lower-order term and obtain
\begin{align*}
    & \left| \int_{ \mathcal D^{\tau_{i+1}}_{\tau_i}} \langle  \nabla_{L} \mathcal L_Z^{\le N-1}\mathrm{Tr}h \gamma^{L} \mathscr L_Z^{\le N-13}\psi , \mathscr L_Z^{\le N-1}\psi\rangle\,dV \right| \\
    & \lesssim \int_{\tau_i}^{\tau_{i+1}} r^{-\frac12} \mathcal E^W_{\le N-1}[h_{L\underline L}] \mathcal E^D_{\le N-10}[\psi]\mathscr E^D_{\le N-1}[\psi]\,d\tau \\
    & \lesssim \sup_{\tau\in[\tau_i,\tau_{i+1}]} \mathscr E^D_{\le N-1}[\psi](\tau) \left( \int_{\tau_i}^{\tau_{i+1}}{}^{(0)}\mathcal E^D_{\le N-10}[\psi]^2(\tau)\,d\tau \right)^\frac12 \left( \int_{\tau_i}^{\tau_{i+1}}{}^{(0)}\mathcal E^W_{\le N-1}[h_{L\underline L}]^2(\tau)\,d\tau \right)^\frac12. 
\end{align*}
\subsubsection{$\nu=\underline L$}
We now consider the case $\nu=\underline L$. As in the previous case, we need only estimate the low-high and high-low interactions:
\begin{align}
     \int_{ \mathcal D^{\tau_{i+1}}_{\tau_i}} \langle  \nabla_{\underline L} \mathcal L_Z^{\le N-13}\mathrm{Tr}h \gamma^{\underline L} \mathscr L_Z^{\le N-1}\psi , \mathscr L_Z^{\le N-1}\psi\rangle\,dV, \\
      \int_{ \mathcal D^{\tau_{i+1}}_{\tau_i}} \langle  \nabla_{\underline L} \mathcal L_Z^{\le N-1}\mathrm{Tr}h \gamma^{\underline L} \mathscr L_Z^{\le N-13}\psi , \mathscr L_Z^{\le N-1}\psi\rangle\,dV.
\end{align}
For the low-high interaction, we integrate by parts in the $\nabla_{\underline L}$ direction, reducing the estimate to the following spacetime integral:
\begin{align}
   \int_{ \mathcal D^{\tau_{i+1}}_{\tau_i}} \langle  \mathcal L_Z^{\le N-13}\mathrm{Tr}h \gamma^{\underline L}  \nabla_{\underline L} \mathscr L_Z^{\le N-1}\psi , \mathscr L_Z^{\le N-1}\psi\rangle\,dV.
\end{align}
Then we use the null decomposition for the spinor fields and the weighted Sobolev inequality along outgoing null and obtain
\begin{align}
    \int_{\tau_i}^{\tau_{i+1}} r^{-\frac12} \mathcal E^W_{\le N-11}[h] \mathcal E^D_{\le N-1}[\psi] \|\mathscr L_Z^{\le N-1}\psi_-\|_{L^2(\Sigma_\tau)}\,d\tau .
\end{align}
Note that the term $\psi_-$ cannot be controlled by the energy $\mathscr E^D_{\le N-1}[\psi]$ in a direct way. To remedy this, we consider the configuration of the vector fields of the term $\mathscr L_Z^{\le N-1}\psi_-$. If it contains at least one $L$, then we immediately have $\mathcal E^D_{\le N-2}[\psi]$. If the vector fields consist of at least one $\underline L$, then the null decomposition of the spinor fields implies that $\mathscr L_Z^{\le N-1}\psi_-$ turns out to be $\slashed\nabla\mathscr L_Z^{\le N-2}\psi_+ $ up to additional nonlinearities. On the other hand, if it consists of only rotation fields, then we have $\|\mathscr L_Z^{\le N-1}\psi_-\|_{L^2(\Sigma_\tau)} \lesssim r\mathcal E^D_{\le N-2}[\psi](\tau) $. Therefore, combining all the observations, we obtain
\begin{align}
    \begin{aligned}
        & \int_{\tau_i}^{\tau_{i+1}} r^{-\frac12} \mathcal E^W_{\le N-11}[h] \mathcal E^D_{\le N-1}[\psi] \|\mathscr L_Z^{\le N-1}\psi_-\|_{L^2(\Sigma_\tau)}\,d\tau \\
        & \lesssim \int_{\tau_i}^{\tau_{i+1}} r^{\frac12} \mathcal E^W_{\le N-11}[h] \mathcal E^D_{\le N-2}[\psi]  \mathcal E^D_{\le N-1}[\psi]  \,d\tau \\
        & \lesssim (\tau_{i+1})^{\frac\delta2} \int_{\tau_i}^{\tau_{i+1}} {}^{(0)}\mathcal E^W_{\le N-10}[h_{L\underline L}](\tau) {}^{(1-\delta)}\mathcal E^D_{\le N-1}[\psi]^2(\tau)\,d\tau 
    \end{aligned}
\end{align}
On the other hand, for the high-low interaction, we foliate the far region by the ingoing null cones $\underline{\mathcal C}_v$ and use the weighted Sobolev inequality for the lower-order term to get
\begin{align}
\begin{aligned}
   & \int  r^{-\frac12} \mathcal F^W_{\le N-1}[h_{L\underline L}] \|\nabla_{\underline L}\mathscr L_Z^{\le N-11}\psi_-\|_{L^2(\underline{\mathcal C}_v) } \|\mathscr L_Z^{\le N-1}\psi_-\|_{L^2(\underline{\mathcal C}_v) } \\
    & \lesssim \sup_{v:\tau\le\tau(v)}{}^{(0)}  \mathcal F^W_{\le N-1}[h_{L\underline L}] \int r^{-\frac12} \|\slashed\nabla \mathscr L_Z^{\le N-11}\psi_+\|_{L^2(\underline{\mathcal C}_v) }  \|\mathscr L_Z^{\le N-1}\psi_-\|_{L^2(\underline{\mathcal C}_v) } , 
    \end{aligned}
\end{align}
up to additional nonlinearities. If the vector fields of the term $\mathscr L_Z^{\le N-1}\psi_-$ contain at least one $L$, then we obtain the desired bound after the Cauchy-Schwarz inequality. If they contain at least one $\underline L$, then it becomes $\slashed\nabla\mathscr L_Z^{\le N-2}\psi_+$ up to additional nonlinearities and hence we can obtain the desired bound after the Cauchy-Schwarz inequality. On the other hand, if the vector fields $\mathscr L_Z^{\le N-1}$ consist of only rotation fields, then we see that
\begin{align}
    \begin{aligned}
        &  \int r^{-\frac12} \|\slashed\nabla \mathscr L_Z^{\le N-11}\psi_+\|_{L^2(\underline{\mathcal C}_v) }  \|\mathscr L_Z^{\le N-1}\psi_-\|_{L^2(\underline{\mathcal C}_v) }  \\
        & \lesssim  \int r^{\frac12} \|\slashed\nabla \mathscr L_Z^{\le N-11}\psi_+\|_{L^2(\underline{\mathcal C}_v) }  \|\slashed\nabla\mathscr L_Z^{\le N-2}\psi_-\|_{L^2(\underline{\mathcal C}_v) } \\
        & \lesssim (\tau_{i+1})^{\frac\delta2} \left( \int_{\tau_i}^{\tau_{i+1}} {}^{(0)}\mathcal E^D_{\le N-11}[\psi]^2(\tau)\,d\tau  \right)^\frac12  \left( \int_{\tau_i}^{\tau_{i+1}} {}^{(1-\delta)}\mathcal E^D_{\le N-2}[\psi]^2(\tau)\,d\tau  \right)^\frac12.
    \end{aligned}
\end{align}
The case $\nu=A$ follows by the same argument; we omit the details.

\subsection{Control on the interior region}
The interior Dirac-current energy is controlled by the spinorial wave energy. Recall its definition:
\begin{align}
    {}_{\rm int}\mathscr E^D_{\le k}[\psi]^2(\tau) := \int_{\Sigma_\tau \cap \{r\le R\} } \langle \mathscr L_Z^{\le k}\psi, \gamma({n}) \mathscr L_Z^{\le k}\psi\rangle d\sigma_{\Sigma_\tau}.
\end{align}
Here, the future-directed unit normal vector $n$ is chosen to be $T$ away from the horizon $\mathcal H^+$ and $n=\mathcal N$ near the horizon. 

For $k\ge1$, we have
\begin{align}
     {}_{\rm int}\mathscr E^D_{\le k}[\psi]^2(\tau) \lesssim {}_{\rm int}\mathcal E^D_{\le k-1}[\psi]^2(\tau).
\end{align}
If $k=0$, since $r\le R$, we see that
\begin{align}
\begin{aligned}
     {}_{\rm int}\mathscr E^D_{\le 0}[\psi]^2(\tau) = \int_{\Sigma_\tau \cap \{r\le R\} } \langle \psi, \gamma(n)\psi\rangle \,d\sigma_{\Sigma_\tau} \lesssim_R  \int_{\Sigma_\tau \cap \{r\le R\} }\frac1{r^2} \langle \psi, \gamma(n)\psi\rangle \,d\sigma_{\Sigma_\tau} \lesssim {}_{\rm int}\mathcal E^D_{\le0}[\psi]^2(\tau). 
     \end{aligned}
\end{align}



\appendix

\section{Global nonlinear stability of the Maxwell-Dirac system in the Lorenz gauge on the Schwarzschild black hole spacetime}\label{appendix:maxwell-dirac}
We consider the Maxwell--Dirac system on the Schwarzschild exterior:
\begin{align}
    \begin{aligned}
        i\gamma^\mu\mathbf D_\mu\psi = 0, \\
        \nabla^\mu F_{\mu\nu} = J_\nu,
    \end{aligned}
\end{align}
where $\mathbf D_\mu = \nabla_\mu+iA_\mu$ is the gauge-covariant derivative and $F=dA$ is the curvature $2$-form and the $1$-form $A$ is the gauge potential. We impose the Lorentz gauge condition $\nabla^\mu A_\mu=0$.

We obtain the following tensorial wave-Dirac system:
\begin{align}
    \begin{aligned}
        i\gamma^\mu\nabla_\mu\psi = A_\mu \gamma^\mu \psi, \\
        \Box_g A_\nu = {R^\mu}_{\! \nu}A_\mu +J_\nu.
    \end{aligned}
\end{align}
Since the Schwarzschild spacetime is Ricci-flat, the term ${R^\mu}_\nu$ is identically zero.

\begin{thm}\label{thm-md}
    Consider the Maxwell--Dirac system in the Lorenz gauge on the Schwarzschild manifold $(\mathcal M,g)$, foliated by the hypersurfaces $\Sigma_\tau$. Fix $N\ge17$, $0<\epsilon_0<1$, and $0<\delta<\frac1{10}$. Let the initial data be $(\psi,A_\mu)|_{\Sigma_{\tau=0}}:=(\psi_0,a_\mu)$ and assume that they satisfy the smallness condition:
    \begin{align}
        \mathcal X^D_{N}[\psi]^2(0)+\mathcal X^W_{N}[A]^2(0) \le (\epsilon_0)^2.
    \end{align}
    Then there exist $0\le\varepsilon<\epsilon_0$ sufficiently small and $C\ge1$ sufficiently large such that the following energy estimates hold:
    \begin{align}
    \mathcal X^D_{N}[\psi]^2(\tau)+\mathcal X^W_{N}[A]^2(\tau) \le C^2\varepsilon^2 \tau^{\delta}.
    \end{align}

    The $r^p$-weighted energies of the spinor and gauge fields are bounded, with mild growth at top order:
    \begin{align}
      {}^{(1)}\mathcal E^D_{\le N-1} [ \psi ]^2(\tau) \le  C^2\varepsilon^2 \tau^{\delta}, \ {}^{(1)}\mathcal E^W_{\le N-1}[(A_L,A_{\underline L},\slashed{A}) ] \le C^2\varepsilon^2 \tau^{2\delta} ,
    \end{align}
    and
    \begin{align}
    	 {}^{(2-\delta)}\mathcal E^D_{\le N-2} [ \psi ]^2(\tau) + {}^{(2-\delta)}\mathcal E^W_{\le N-3}[(A_L,A_{\underline L},\slashed{A}) ] \le C^2\varepsilon^2 . 
    \end{align}

\end{thm}

\subsection{Bootstrap assumptions}

For $k\le N-1$, we assume that
\begin{align}
\begin{aligned}
    {}^{(1)}\mathcal E^D_{\le N-1}[\psi]^2(\tau) \le C^2\varepsilon^2 \tau^{\delta},  \quad {}^{(1)}\mathcal E^W_{\le N-1}[(A,A_{\underline L},\slashed{A})]^2(\tau) \le C^2\varepsilon^2 \tau^{2\delta}, \\
    {}^{(2-\delta)}\mathcal E^D_{\le N-2}[\psi]^2(\tau) \le C^2\varepsilon^2\tau^{\delta}, \quad {}^{(2-\delta)}\mathcal E^W_{\le N-3}[(A,A_{\underline L},\slashed{A})]^2(\tau) \le C^2\varepsilon^2 \tau^{\delta}.
    \end{aligned}
\end{align}

For lower-order $k\le N-9$, we assume the following improved weighted energy inequality:
\begin{align}
    {}^{(2)}\mathcal E^D_{\le N-9}[\psi]^2(\tau) \le C^2\varepsilon^2, \quad {}^{(2)}\mathcal E^W_{\le N-10}[(A,A_{\underline L},\slashed{A})]^2(\tau) \le C^2\varepsilon^2,
\end{align}
and the following decay:
\begin{align}
    \begin{aligned}
      {}^{(1)}  \mathcal E^D_{\le N-9}[\psi]^2(\tau) \le C^2\varepsilon^2\tau^{-1}, \quad {}^{(1)}\mathcal E^W_{\le N-10}[(A,A_{\underline L},\slashed{A})]^2(\tau) \le C^2\varepsilon^2\tau^{-1}, \\
      {}^{(0)}  \mathcal E^D_{\le N-9}[\psi]^2(\tau) \le C^2\varepsilon^2\tau^{-2}, \quad {}^{(0)}\mathcal E^W_{\le N-10}[(A,A_{\underline L},\slashed{A})]^2(\tau) \le C^2\varepsilon^2\tau^{-2}.
    \end{aligned}
\end{align}

For the Dirac equation, we estimate the following spacetime integrals:
\begin{align}
    \int_{\mathcal D^{\tau_{i+1}}_{\tau_i}} r^{p} \langle \mathscr L_Z^{\le k} ( \nabla_\lambda A_\mu \psi) \gamma^\lambda \gamma^\mu, \nabla_L\mathscr L_Z^{\le k}\psi\rangle\, dV+  \int_{\mathcal D^{\tau_{i+1}}_{\tau_i}} r^{p} \langle \mathscr L_Z^{\le k} (  A_\mu \nabla_\lambda \psi) \gamma^\lambda \gamma^\mu, \nabla_L\mathscr L_Z^{\le k}\psi\rangle\, dV,
\end{align}
where $p=2-\delta$, $k\le N-2$, and $p=1$, $k\le N-1$.

For the gauge field, it remains to estimate the following spacetime integral:
\begin{align}
     \int_{\mathcal D^{\tau_{i+1}}_{\tau_i}} r^{p} \mathcal L_Z^{\le k} \langle \psi, \gamma_\nu \psi\rangle \nabla_L \mathcal L_Z^{\le k}A_\nu\,dV,
\end{align}
where $p=2-\delta$, $k\le N-3$, and $p=1$, $k\le N-1$.

\subsection{Control of the weighted energy for the spinor fields }

\subsubsection{Control of ${}^{(1)}\mathcal E^D_{\le N-1}[\psi] $ }
We consider the integral
\begin{align}
    \int_{\mathcal D^{\tau_{i+1}}_{\tau_i}} r^{} \langle \mathscr L_Z^{\le N-1} (  A_L \nabla_{\underline L} \psi) \gamma^{\underline L} \gamma^{L} , \gamma^T \nabla_L\mathscr L_Z^{\le N-1}\psi\rangle\, dV.
\end{align}
For the high-low interaction, we integrate by parts in the $\nabla_{\underline L}$ direction, reducing the estimate to the following integral:
\begin{align*}
   &  \int_{\mathcal D^{\tau_{i+1}}_{\tau_i}}  \mathcal L_Z^{\le N-1}A_L \cdot \mathscr L_Z^{\le N-13}\psi_+ \cdot \nabla_{\underline L}\mathscr L_Z^{\le N-1}\psi_+\,dV. 
\end{align*}
Integrating by parts in the $\underline L$ direction again gives
\begin{align*}
   &  \int_{\mathcal D^{\tau_{i+1}}_{\tau_i}}  \nabla_{\underline L}\mathcal L_Z^{\le N-1}A_L \cdot \mathscr L_Z^{\le N-13}\psi_+ \cdot \mathscr L_Z^{\le N-1}\psi_+\,dV \\
   & \qquad + \int_{\mathcal D^{\tau_{i+1}}_{\tau_i}}  \mathcal L_Z^{\le N-1}A_L \cdot \nabla_{\underline L}\mathscr L_Z^{\le N-13}\psi_+ \cdot \mathscr L_Z^{\le N-1}\psi_+\,dV.
\end{align*}
For the first integral, we use the null condition for the gauge potential and write
\begin{align*}
    & \left|  \int_{\mathcal D^{\tau_{i+1}}_{\tau_i}}  \nabla_{\underline L}\mathcal L_Z^{\le N-1}A_L \cdot \mathscr L_Z^{\le N-13}\psi_+ \cdot \mathscr L_Z^{\le N-1}\psi_+\,dV  \right| \\
    & \lesssim  \int_{\mathcal D^{\tau_{i+1}}_{\tau_i}}  | \nabla_{ L}\mathcal L_Z^{\le N-1}A_{\underline L}| | \mathscr L_Z^{\le N-13}\psi_+| | \mathscr L_Z^{\le N-1}\psi_+|\,dV \\
    & \qquad + \int_{\mathcal D^{\tau_{i+1}}_{\tau_i}}  | \slashed\nabla_{ }\mathcal L_Z^{\le N-1}\slashed{A} | | \mathscr L_Z^{\le N-13}\psi_+| | \mathscr L_Z^{\le N-1}\psi_+|\,dV.
\end{align*}
The weighted Sobolev inequality on outgoing null hypersurfaces and the Hardy inequality yield the following bounds for these two integrals:
\begin{align*}
    \sup_{\tau\in[\tau_i,\tau_{i+1}]} {}^{(0)}\mathcal E^W_{\le N-1}[(A_{\underline L},\slashed{A}) ](\tau)\left( \int_{\tau_i}^{\tau_{i+1}} {}^{(1-\delta)}\mathcal E^D_{\le N-11}[\psi]^2(\tau)\,d\tau  \right)^\frac12 \left( \int_{\tau_i}^{\tau_{i+1}} {}^{(0)}\mathcal E^D_{\le N-1}[\psi]^2(\tau)\,d\tau  \right)^\frac12 .
\end{align*}
It remains to estimate the integral
\begin{align}
    \int_{\mathcal D^{\tau_{i+1}}_{\tau_i}}  \mathcal L_Z^{\le N-1}A_L \cdot \nabla_{\underline L}\mathscr L_Z^{\le N-13}\psi_+ \cdot \mathscr L_Z^{\le N-1}\psi_+\,dV.
\end{align}
We now consider the vector-field configuration in the higher-order terms. If at least one vector field is $L$, the estimate follows directly. If at least one is $\underline L$, we use the null condition $\nabla^\mu A_\mu=0$ and repeat the preceding argument. It remains to consider the case in which all vector fields are rotations; then
\begin{align*}
    & \left| \int_{\mathcal D^{\tau_{i+1}}_{\tau_i}}  \mathcal L_Z^{\le N-1}A_L \cdot \nabla_{\underline L}\mathscr L_Z^{\le N-13}\psi_+ \cdot \mathscr L_Z^{\le N-1}\psi_+\,dV \right| \\
    & \lesssim \int_{\tau_i}^{\tau_{i+1}} r^{\frac12} \mathcal E^W_{\le N-2}[A_L] \mathcal E^D_{\le N-10}[\psi] \mathscr E^D_{\le N-1}[\psi]\,d\tau \\
    & \lesssim \sup_{\tau\in[\tau_i,\tau_{i+1}]} \mathscr E^D_{\le N-1}[\psi](\tau) \left( \int_{\tau_i}^{\tau_{i+1}} {}^{(1)}\mathcal E^D_{\le N-10}[\psi]^2(\tau)\,d\tau \right)^\frac12 \left( \int_{\tau_i}^{\tau_{i+1}} {}^{(0)}\mathcal E^W_{\le N-2}[A_L]^2(\tau)\,d\tau \right)^\frac12
\end{align*}

The low--high interaction is easier: applying the $L^\infty$ Sobolev inequality to the lower-order term $\mathcal L_Z^{\le N-13}A_L$, followed by the Cauchy--Schwarz inequality, gives the desired bound. We omit the details.

We now consider the high-low interaction involving the gauge potential $A_{\underline L}$:
\begin{align}
    \int_{\mathcal D^{\tau_{i+1}}_{\tau_i}}  r \langle \mathcal L_Z^{\le N-1}A_{\underline L} \gamma^{L}\gamma^{\underline L}\nabla_L \mathscr L_Z^{\le N-13}\psi, \gamma^T\nabla_L\mathscr L_Z^{\le N-1}\psi\rangle\,dV.
\end{align}
Applying the weighted Sobolev inequality on ingoing null hypersurfaces and then using the null decomposition for the spinor fields gives
\begin{align}
    \begin{aligned}
       &  \int r^{-\frac12} \|\mathcal L_Z^{\le N-1}A_{\underline L}\|_{L^2(\underline{\mathcal C}_v)} \|\nabla_L \mathscr L_Z^{\le N-10}\psi_+\|_{L^2(\underline{\mathcal C}_v)} \|\nabla_L\mathscr L_Z^{\le N-1}\psi_-\|_{L^2(\underline{\mathcal C}_v)}\,dv \\
       & \qquad +  \int r^{\frac12} \|\mathcal L_Z^{\le N-1}A_{\underline L}\|_{L^2(\underline{\mathcal C}_v)} \| \mathscr L_Z^{\le N-10}(A_{\underline L}\cdot \psi_-)\|_{L^2(\underline{\mathcal C}_v)} \|\nabla_L\mathscr L_Z^{\le N-1}\psi_-\|_{L^2(\underline{\mathcal C}_v)}\,dv \\
       & \qquad + \int r^{\frac12} \|\mathcal L_Z^{\le N-1}A_{\underline L}\|_{L^2(\underline{\mathcal C}_v)} \| \mathscr L_Z^{\le N-10}(\slashed{A_{}}\cdot \psi_+)\|_{L^2(\underline{\mathcal C}_v)} \|\nabla_L\mathscr L_Z^{\le N-1}\psi_-\|_{L^2(\underline{\mathcal C}_v)}\,dv.
    \end{aligned}
\end{align}
It remains to consider the case in which all higher-order commutation fields are rotational; otherwise, we use either the null decomposition for spinors or the Lorenz condition $\nabla^\mu A_\mu=0$. Moreover, $\|\mathcal L_Z^{\le N-1}A_{\underline L}\|_{L^2(\underline{\mathcal C}_v)}\lesssim r \|\slashed\nabla\mathcal L_Z^{\le N-2}A_{\underline L}\|_{L^2(\underline{\mathcal C}_v)} \lesssim r^{\frac12} {}^{(1)}\mathcal F^W_{\le N-2}[A_{\underline L}]$. Hence, the three integrals are bounded by
\begin{align*}
    \sup_{v:\tau\le\tau(v)}{}^{(1)}\mathcal F^W_{\le N-2}[A_{\underline L}] \left( \int_{\tau_i}^{\tau_{i+1}} {}^{(0)}\mathcal E^D_{\le N-10}[\psi]^2(\tau)\,d\tau   \right)^\frac12 \left( \int_{\tau_i}^{\tau_{i+1}} {}^{(0)}\mathcal E^D_{\le N-1}[\psi]^2(\tau)\,d\tau   \right)^\frac12.
\end{align*}

The low--high interaction is even easier. Applying the Sobolev inequality on outgoing null hypersurfaces gives
\begin{align*}
    & \left| \int_{\mathcal D^{\tau_{i+1}}_{\tau_i}}  r \langle \mathcal L_Z^{\le N-13}A_{\underline L} \gamma^{L}\gamma^{\underline L}\nabla_L \mathscr L_Z^{\le N-1}\psi, \gamma^T\nabla_L\mathscr L_Z^{\le N-1}\psi\rangle\,dV \right| \\
    & \lesssim \int_{\tau_i}^{\tau_{i+1}} r^{\frac12} \mathcal E^W_{\le N-11}[A_{\underline L}] \mathcal E^D_{\le N-1}[\psi]^2\,d\tau \\
    & \lesssim \sup_{\tau\in[\tau_i,\tau_{i+1}]} {}^{(1)}\mathcal E^W_{\le N-11}[A_{\underline L}]^2(\tau) \int_{\tau_i}^{\tau_{i+1}}{}^{(0)} \mathcal E^D_{\le N-1}[\psi]^2(\tau)\,d\tau. 
\end{align*}

We now consider the following integrals:
\begin{align}
    \int_{\mathcal D^{\tau_{i+1}}_{\tau_i}} r^{}  \langle \nabla_L \mathcal L_Z^{\le N-1}\slashed{A} \gamma^L \mathscr L_Z^{\le N-13}\psi, \gamma^T \nabla_L\mathscr L_Z^{\le N-1}\psi\rangle\,dV, \\
     \int_{\mathcal D^{\tau_{i+1}}_{\tau_i}} r^{}  \langle \mathcal L_Z^{\le N-1}\slashed{A} \gamma^L \nabla_L \mathscr L_Z^{\le N-13}\psi, \gamma^T \nabla_L\mathscr L_Z^{\le N-1}\psi\rangle\,dV , \\
      \int_{\mathcal D^{\tau_{i+1}}_{\tau_i}} r^{}  \langle \nabla_L \mathcal L_Z^{\le N-13}\slashed{A} \gamma^L \mathscr L_Z^{\le N-1}\psi, \gamma^T \nabla_L\mathscr L_Z^{\le N-1}\psi\rangle\,dV, \\
       \int_{\mathcal D^{\tau_{i+1}}_{\tau_i}} r^{}  \langle  \mathcal L_Z^{\le N-13}\slashed{A} \gamma^L \nabla_L \mathscr L_Z^{\le N-1}\psi, \gamma^T \nabla_L\mathscr L_Z^{\le N-1}\psi\rangle\,dV, \\
        \int_{\mathcal D^{\tau_{i+1}}_{\tau_i}} r^{}  \langle \nabla_{\underline L} \mathcal L_Z^{\le N-1}\slashed{A} \gamma^{\underline L} \mathscr L_Z^{\le N-13}\psi, \gamma^T \nabla_L\mathscr L_Z^{\le N-1}\psi\rangle\,dV, \\
         \int_{\mathcal D^{\tau_{i+1}}_{\tau_i}} r^{}  \langle  \mathcal L_Z^{\le N-1}\slashed{A} \gamma^{\underline L} \nabla_{\underline L} \mathscr L_Z^{\le N-13}\psi, \gamma^T \nabla_L\mathscr L_Z^{\le N-1}\psi\rangle\,dV, \\
          \int_{\mathcal D^{\tau_{i+1}}_{\tau_i}} r^{}  \langle \nabla_{\underline L} \mathcal L_Z^{\le N-13}\slashed{A} \gamma^{\underline L} \mathscr L_Z^{\le N-1}\psi, \gamma^T \nabla_L\mathscr L_Z^{\le N-1}\psi\rangle\,dV , \\
           \int_{\mathcal D^{\tau_{i+1}}_{\tau_i}} r^{}  \langle  \mathcal L_Z^{\le N-13}\slashed{A} \gamma^{\underline L}\nabla_{\underline L} \mathscr L_Z^{\le N-1}\psi, \gamma^T \nabla_L\mathscr L_Z^{\le N-1}\psi\rangle\,dV.
\end{align}
Here we deal with the following two integrals:
\begin{align}
    \int_{\mathcal D^{\tau_{i+1}}_{\tau_i}} r^{}  \langle \nabla_L \mathcal L_Z^{\le N-13}\slashed{A} \gamma^L \mathscr L_Z^{\le N-1}\psi, \gamma^T \nabla_L\mathscr L_Z^{\le N-1}\psi\rangle\,dV,  \\
     \int_{\mathcal D^{\tau_{i+1}}_{\tau_i}} r^{}  \langle \nabla_{\underline L} \mathcal L_Z^{\le N-1}\slashed{A} \gamma^{\underline L} \mathscr L_Z^{\le N-13}\psi, \gamma^T \nabla_L\mathscr L_Z^{\le N-1}\psi\rangle\,dV,
\end{align}
since the remaining integrals are controlled directly.
Using the null decomposition for the spinor fields and integrating by parts in the $\underline L$ direction, we obtain
\begin{align*}
    &\int_{\mathcal D^{\tau_{i+1}}_{\tau_i}} r^{}  \langle \nabla_L \mathcal L_Z^{\le N-13}\slashed{A} \gamma^L \mathscr L_Z^{\le N-1}\psi, \gamma^T \nabla_L\mathscr L_Z^{\le N-1}\psi\rangle\,dV \\
    & =\int_{\mathcal D^{\tau_{i+1}}_{\tau_i}} r^{2}  \langle \nabla_L \mathcal L_Z^{\le N-13}\slashed{A} \gamma^L \gamma^{\underline L}\nabla_{\underline L} \mathscr L_Z^{\le N-2}\psi, \gamma^T \nabla_L\mathscr L_Z^{\le N-1}\psi\rangle\,dV \\
    & = \int_{\substack{u=\tau_{i+1}-R \\ v \ge \tau_{i+1}+R } } r^{2}  \langle \nabla_L \mathcal L_Z^{\le N-13}\slashed{A} \gamma^L \gamma^{\underline L}  \mathscr L_Z^{\le N-2}\psi, \gamma^T \nabla_L\mathscr L_Z^{\le N-1}\psi\rangle\,d\sigma \\
    & \qquad - \int_{\substack{u=\tau_{i}-R \\ v \ge \tau_{i}+R } } r^{2}  \langle \nabla_L \mathcal L_Z^{\le N-13}\slashed{A} \gamma^L \gamma^{\underline L}  \mathscr L_Z^{\le N-2}\psi, \gamma^T \nabla_L\mathscr L_Z^{\le N-1}\psi\rangle\,d\sigma  \\
    & + \int_{r=R } r^{2}  \langle \nabla_L \mathcal L_Z^{\le N-13}\slashed{A} \gamma^L \gamma^{\underline L}  \mathscr L_Z^{\le N-2}\psi, \gamma^T \nabla_L\mathscr L_Z^{\le N-1}\psi\rangle\,d\sigma  \\
    & + \int_{\mathcal D^{\tau_{i+1}}_{\tau_i}} r^{2}  \langle \nabla_{\underline L} \nabla_L \mathcal L_Z^{\le N-13}\slashed{A} \gamma^L \gamma^{\underline L} \mathscr L_Z^{\le N-2}\psi, \gamma^T \nabla_L\mathscr L_Z^{\le N-1}\psi\rangle\,dV \\
    & \qquad + \int_{\mathcal D^{\tau_{i+1}}_{\tau_i}} r^{2}  \langle  \nabla_L \mathcal L_Z^{\le N-13}\slashed{A} \gamma^L \gamma^{\underline L} \mathscr L_Z^{\le N-2}\psi, \gamma^T \nabla_{\underline L} \nabla_L\mathscr L_Z^{\le N-1}\psi\rangle\,dV.
\end{align*}
The boundary integral along the null hypersurface is controlled by the energy identity:
\begin{align*}
    & \left| \int_{\substack{u=\tau_{i+1}-R \\ v \ge \tau_{i+1}+R } } r^{2}  \langle \nabla_L \mathcal L_Z^{\le N-13}\slashed{A} \gamma^L \gamma^{\underline L}  \mathscr L_Z^{\le N-2}\psi, \gamma^T \nabla_L\mathscr L_Z^{\le N-1}\psi\rangle\,d\sigma \right| \\
    & \lesssim {}^{(2)}\mathcal E^W_{\le N-12}[\slashed{A}](\tau_{i+1}) {}^{(1)}\mathcal E^D_{\le N-1}[\psi]^2(\tau_{i+1}) .
\end{align*}
When $\nabla_{\underline L}$ acts on the lower-order term, we use the wave equation for the gauge potential $\slashed{A}$ to write
\begin{align*}
    & \int_{\mathcal D^{\tau_{i+1}}_{\tau_i}} r^{2}  \langle \nabla_{\underline L} \nabla_L \mathcal L_Z^{\le N-13}\slashed{A} \gamma^L \gamma^{\underline L} \mathscr L_Z^{\le N-2}\psi, \gamma^T \nabla_L\mathscr L_Z^{\le N-1}\psi\rangle\,dV \\
    & = \int_{\mathcal D^{\tau_{i+1}}_{\tau_i}}   \langle \mathcal L_Z^{\le N-11}\slashed{A} \gamma^L \gamma^{\underline L} \mathscr L_Z^{\le N-2}\psi, \gamma^T \nabla_L\mathscr L_Z^{\le N-1}\psi\rangle\,dV \\
    & \qquad + \int_{\mathcal D^{\tau_{i+1}}_{\tau_i}} r^{}  \langle \nabla_L \mathcal L_Z^{\le N-13}\slashed{A} \gamma^L \gamma^{\underline L} \mathscr L_Z^{\le N-2}\psi, \gamma^T \nabla_L\mathscr L_Z^{\le N-1}\psi\rangle\,dV \\
    & \qquad + \int_{\mathcal D^{\tau_{i+1}}_{\tau_i}} r^{}  \langle \nabla_{\underline L} \mathcal L_Z^{\le N-13}\slashed{A} \gamma^L \gamma^{\underline L} \mathscr L_Z^{\le N-2}\psi, \gamma^T \nabla_L\mathscr L_Z^{\le N-1}\psi\rangle\,dV \\
    & \qquad + \int_{\mathcal D^{\tau_{i+1}}_{\tau_i}} r^{2}  \langle  \mathcal L_Z^{\le N-13}(\psi_+\cdot\psi_-)\gamma^L \gamma^{\underline L} \mathscr L_Z^{\le N-2}\psi, \gamma^T \nabla_L\mathscr L_Z^{\le N-1}\psi\rangle\,dV.
\end{align*}
For the second integral above, the $L^4\times L^4\times L^2$ weighted Sobolev inequality on the ingoing null hypersurfaces gives
\begin{align*}
   & \left|  \int_{\mathcal D^{\tau_{i+1}}_{\tau_i}} r^{}  \langle \nabla_L \mathcal L_Z^{\le N-13}\slashed{A} \gamma^L \gamma^{\underline L} \mathscr L_Z^{\le N-2}\psi, \gamma^T \nabla_L\mathscr L_Z^{\le N-1}\psi\rangle\,dV \right| \\
   & \lesssim \int r \|\nabla_{\underline L}\nabla_L \mathcal L_Z^{\le N-12}\slashed{A}\|_{L^2(\underline{\mathcal C}_v)} \|\nabla_{\underline L}\mathscr L_Z^{\le N-1}\psi_-\|_{L^2(\underline{\mathcal C}_v)} \|\nabla_L\mathscr L_Z^{\le N-1}\psi_-\|_{L^2(\underline{\mathcal C}_v)} \\
   & \lesssim \sup_{v:\tau\le\tau(v)} {}^{(0)}\mathcal F^W_{\le N-10}[\slashed{A}] \int_{\tau_i}^{\tau_{i+1}} {}^{(0)}\mathcal E^D_{\le N-1}[\psi]^2(\tau)\,d\tau .
\end{align*}
For the third integral, which contains $\nabla_{\underline L}$, we integrate by parts again to obtain
\begin{align*}
    & \int_{\mathcal D^{\tau_{i+1}}_{\tau_i}} r^{}  \langle \nabla_{\underline L} \mathcal L_Z^{\le N-13}\slashed{A} \gamma^L \gamma^{\underline L} \mathscr L_Z^{\le N-2}\psi, \gamma^T \nabla_L\mathscr L_Z^{\le N-1}\psi\rangle\,dV \\
    & = \int_{\substack{ u=\tau_{i+1}-R \\ v\ge \tau_{i+1}+R }  } r^{}  \langle  \mathcal L_Z^{\le N-13}\slashed{A} \gamma^L \gamma^{\underline L} \mathscr L_Z^{\le N-2}\psi, \gamma^T \nabla_L\mathscr L_Z^{\le N-1}\psi\rangle\,d\sigma \\
    & \qquad - \int_{\substack{ u=\tau_{i+1}-R \\ v\ge \tau_{i+1}+R }  } r^{}  \langle  \mathcal L_Z^{\le N-13}\slashed{A} \gamma^L \gamma^{\underline L} \mathscr L_Z^{\le N-2}\psi, \gamma^T \nabla_L\mathscr L_Z^{\le N-1}\psi\rangle\,d\sigma \\
    &  \int_{r=R } r^{}  \langle  \mathcal L_Z^{\le N-13}\slashed{A} \gamma^L \gamma^{\underline L} \mathscr L_Z^{\le N-2}\psi, \gamma^T \nabla_L\mathscr L_Z^{\le N-1}\psi\rangle\,d\sigma \\
    & + \int_{\mathcal D^{\tau_{i+1}}_{\tau_i}} r^{}  \langle  \mathcal L_Z^{\le N-13}\slashed{A} \gamma^L \gamma^{\underline L} \nabla_{\underline L}\mathscr L_Z^{\le N-2}\psi, \gamma^T \nabla_L\mathscr L_Z^{\le N-1}\psi\rangle\,dV \\
    & \qquad + \int_{\mathcal D^{\tau_{i+1}}_{\tau_i}} r^{}  \langle  \mathcal L_Z^{\le N-13}\slashed{A} \gamma^L \gamma^{\underline L} \mathscr L_Z^{\le N-2}\psi, \gamma^T \nabla_{\underline L}\nabla_L\mathscr L_Z^{\le N-1}\psi\rangle\,dV.
\end{align*}
The boundary integrals are controlled by the corresponding energy identity; we omit the details.
After integrating by parts in the angular variables when necessary, we obtain
\begin{align*}
    & \left| \int_{\mathcal D^{\tau_{i+1}}_{\tau_i}} r^{}  \langle  \mathcal L_Z^{\le N-13}\slashed{A} \gamma^L \gamma^{\underline L} \nabla_{\underline L}\mathscr L_Z^{\le N-2}\psi, \gamma^T \nabla_L\mathscr L_Z^{\le N-1}\psi\rangle\,dV \right| \\
    & \qquad + \left| \int_{\mathcal D^{\tau_{i+1}}_{\tau_i}} r^{}  \langle  \mathcal L_Z^{\le N-13}\slashed{A} \gamma^L \gamma^{\underline L} \mathscr L_Z^{\le N-2}\psi, \gamma^T \nabla_{\underline L}\nabla_L\mathscr L_Z^{\le N-1}\psi\rangle\,dV \right| \\
    & \lesssim \int_{\tau_i}^{\tau_{i+1}} r^{\frac12} \mathcal E^W_{\le N-11}[\slashed{A}] \mathcal E^D_{\le N-2}[\psi] \mathcal E^D_{\le N-1}[\psi] \,d\tau \\
    & \lesssim (\tau_{i+1})^{\frac\delta2} \sup_{\tau\in[\tau_i,\tau_{i+1}]} {}^{(0)}\mathcal E^W_{\le N-11}[\slashed{A}](\tau) \left( \int_{\tau_i}^{\tau_{i+1}} {}^{(1-\delta)}\mathcal E^D_{\le N-2}[\psi]^2(\tau)  \right)^\frac12 \left( \int_{\tau_i}^{\tau_{i+1}} {}^{(0)}\mathcal E^D_{\le N-1}[\psi]^2(\tau)  \right)^\frac12.
\end{align*}
We now consider the integral:
\begin{align}
     \int_{\mathcal D^{\tau_{i+1}}_{\tau_i}} r^{}  \langle \nabla_{\underline L} \mathcal L_Z^{\le N-1}\slashed{A} \gamma^{\underline L} \mathscr L_Z^{\le N-13}\psi, \gamma^T \nabla_L\mathscr L_Z^{\le N-1}\psi\rangle\,dV.
\end{align}
Integrating by parts in the $\nabla_{\underline L}$ direction gives
\begin{align*}
    &  \int_{\mathcal D^{\tau_{i+1}}_{\tau_i}} r^{}  \langle \nabla_{\underline L} \mathcal L_Z^{\le N-1}\slashed{A} \gamma^{\underline L} \mathscr L_Z^{\le N-13}\psi, \gamma^T \nabla_L\mathscr L_Z^{\le N-1}\psi\rangle\,dV \\
    & =  \int_{\substack{ u=\tau_{i+1}-R \\ v \ge \tau_{i+1}+R } } r^{}  \langle \mathcal L_Z^{\le N-1}\slashed{A} \gamma^{\underline L} \mathscr L_Z^{\le N-13}\psi, \gamma^T \nabla_L\mathscr L_Z^{\le N-1}\psi\rangle\,d\sigma \\
    & \qquad - \int_{\substack{ u=\tau_{i}-R \\ v \ge \tau_{i}+R } } r^{}  \langle \mathcal L_Z^{\le N-1}\slashed{A} \gamma^{\underline L} \mathscr L_Z^{\le N-13}\psi, \gamma^T \nabla_L\mathscr L_Z^{\le N-1}\psi\rangle\,d\sigma \\
    & + \int_{ r=R } r^{}  \langle \mathcal L_Z^{\le N-1}\slashed{A} \gamma^{\underline L} \mathscr L_Z^{\le N-13}\psi, \gamma^T \nabla_L\mathscr L_Z^{\le N-1}\psi\rangle\,d\sigma  \\
    & +  \int_{\mathcal D^{\tau_{i+1}}_{\tau_i}} r^{}  \langle  \mathcal L_Z^{\le N-1}\slashed{A} \gamma^{\underline L} \nabla_{\underline L}\mathscr L_Z^{\le N-13}\psi, \gamma^T \nabla_L\mathscr L_Z^{\le N-1}\psi\rangle\,dV \\
& \qquad +  \int_{\mathcal D^{\tau_{i+1}}_{\tau_i}} r^{}  \langle \mathcal L_Z^{\le N-1}\slashed{A} \gamma^{\underline L} \mathscr L_Z^{\le N-13}\psi, \gamma^T \nabla_{\underline L} \nabla_L\mathscr L_Z^{\le N-1}\psi\rangle\,dV.
\end{align*}
To control the boundary integral along the null hypersurface, we use
\begin{align*}
    & \left|  \int_{\substack{ u=\tau_{i+1}-R \\ v \ge \tau_{i+1}+R } } r^{}  \langle \mathcal L_Z^{\le N-1}\slashed{A} \gamma^{\underline L} \mathscr L_Z^{\le N-13}\psi, \gamma^T \nabla_L\mathscr L_Z^{\le N-1}\psi\rangle\,d\sigma \right| \\
    & \lesssim {}^{(1)}\mathcal E^D_{\le N-12}[\psi](\tau_{i+1}) {}^{(1)}\mathcal E^W_{\le N-1}[\slashed{A}](\tau_{i+1}) {}^{(1)}\mathcal E^D_{\le N-1}[\psi](\tau_{i+1}).
\end{align*}
For the spacetime integral in which $\nabla_{\underline L}$ acts on the lower-order term, the weighted Sobolev inequality on outgoing null hypersurfaces gives
\begin{align*}
    & \left|  \int_{\mathcal D^{\tau_{i+1}}_{\tau_i}} r^{}  \langle  \mathcal L_Z^{\le N-1}\slashed{A} \gamma^{\underline L} \nabla_{\underline L}\mathscr L_Z^{\le N-13}\psi, \gamma^T \nabla_L\mathscr L_Z^{\le N-1}\psi\rangle\,dV \right| \\
    & \lesssim \int_{\tau_i}^{\tau_{i+1}} r^{\frac12} \mathcal E^W_{\le N-1}[\slashed{A}] \mathcal E^D_{\le N-10}[\psi] \mathcal E^D_{\le N-1}[\psi]\,d\tau \\
    & \lesssim \sup_{\tau\in[\tau_i,\tau_{i+1}]} {}^{(1)}\mathcal E^D_{\le N-10}[\psi](\tau) \left( \int_{\tau_i}^{\tau_{i+1}} {}^{(0)}\mathcal E^W_{\le N-1}[\slashed{A}]^2(\tau)  \right)^\frac12 \left( \int_{\tau_i}^{\tau_{i+1}} {}^{(0)}\mathcal E^D_{\le N-1}[\psi]^2(\tau)  \right)^\frac12.
\end{align*}
For the spacetime integral with a derivative loss, we have
\begin{align*}
   &  \left| \int_{\mathcal D^{\tau_{i+1}}_{\tau_i}} r^{}  \langle \mathcal L_Z^{\le N-1}\slashed{A} \gamma^{\underline L} \mathscr L_Z^{\le N-13}\psi, \gamma^T \nabla_{\underline L} \nabla_L\mathscr L_Z^{\le N-1}\psi_+\rangle\,dV \right| \\
   & \lesssim  \int_{\mathcal D^{\tau_{i+1}}_{\tau_i}} r^{}  | \mathcal L_Z^{\le N-1}\slashed{A} | | \mathscr L_Z^{\le N-13}\psi_-| | \slashed\Delta\mathscr L_Z^{\le N-1}\psi_+|\,dV  +  \int_{\mathcal D^{\tau_{i+1}}_{\tau_i}}   | \mathcal L_Z^{\le N-1}\slashed{A} | | \mathscr L_Z^{\le N-13}\psi_-| | \nabla_L\mathscr L_Z^{\le N-1}\psi_+|\,dV \\
   & \qquad  +  \int_{\mathcal D^{\tau_{i+1}}_{\tau_i}}   | \mathcal L_Z^{\le N-1}\slashed{A} | | \mathscr L_Z^{\le N-13}\psi_-| | \nabla_{\underline L} \mathscr L_Z^{\le N-1}\psi_+|\,dV + \int_{\mathcal D^{\tau_{i+1}}_{\tau_i}} r   | \mathcal L_Z^{\le N-1}\slashed{A} | | \mathscr L_Z^{\le N-13}\psi_-| |  \mathscr L_Z^{\le N-1} \gamma^\lambda\nabla_\lambda( A_\mu\gamma^\mu\psi  ) |\,dV
\end{align*}
For the first integral, we integrate by parts in the angular variables to obtain
\begin{align*}
   & \left| \int_{\mathcal D^{\tau_{i+1}}_{\tau_i}} r^{}  \langle \mathcal L_Z^{\le N-1}\slashed{A} \gamma^{\underline L} \mathscr L_Z^{\le N-13}\psi, \gamma^T \nabla_{\underline L} \nabla_L\mathscr L_Z^{\le N-1}\psi\rangle\,dV \right| \\
    & \lesssim \sup_{\tau\in[\tau_i,\tau_{i+1}]} {}^{(1)}\mathcal E^D_{\le N-10}[\psi](\tau) \left( \int_{\tau_i}^{\tau_{i+1}} {}^{(0)}\mathcal E^W_{\le N-1}[\slashed{A}]^2(\tau)  \right)^\frac12 \left( \int_{\tau_i}^{\tau_{i+1}} {}^{(0)}\mathcal E^D_{\le N-1}[\psi]^2(\tau)  \right)^\frac12.
\end{align*}
The second integral is controlled similarly, while the fourth is controlled directly. For the third integral, we use the integration by parts along the $\underline L$-direction once again and then it suffices to consider the integral:
\begin{align}
    \int_{\mathcal D^{\tau_{i+1}}_{\tau_i}}   | \nabla_{\underline L}\mathcal L_Z^{\le N-1}\slashed{A} | | \mathscr L_Z^{\le N-13}\psi_-| |  \mathscr L_Z^{\le N-1}\psi_+|\,dV,
\end{align}
The integral in which the derivative acts on the lower-order term is controlled directly. Now we use the foliation by ingoing cones $\underline{\mathcal C}_v\cap\mathcal D^{\tau_{i+1}}_{\tau_i} $ and apply the weighted Sobolev inequality to get
\begin{align*}
    &  \int_{\mathcal D^{\tau_{i+1}}_{\tau_i}}   | \nabla_{\underline L}\mathcal L_Z^{\le N-1}\slashed{A} | | \mathscr L_Z^{\le N-13}\psi_-| |  \mathscr L_Z^{\le N-1}\psi_+|\,dV \\
    & \lesssim \sup_{v:\tau\le\tau(v)}\mathcal F^W_{\le N-1}[\slashed{A}] \int r^{-\frac12} \| \nabla_{\underline L}\mathscr L_Z^{\le N-11}\psi_-\|_{L^2(\underline{\mathcal C}_v)} \|\mathscr L_Z^{\le N-1}\psi_+\|_{L^2(\underline{\mathcal C}_v)}\,dv \\
    & \lesssim (\tau_{i+1})^{\frac\delta2}\sup_{v:\tau\le\tau(v)}\mathcal F^W_{\le N-1}[\slashed{A}] \left( \int_{\tau_i}^{\tau_{i+1}} {}^{(0)}\mathcal E^D_{\le N-11}[\psi]^2(\tau)\,d\tau   \right)^{\frac12} \left( \int_{\tau_i}^{\tau_{i+1}} {}^{(1-\delta)}\mathcal E^D_{\le N-1}[\psi]^2(\tau)\,d\tau   \right)^{\frac12}  ,
\end{align*}
where we used the Hardy inequality.

By a repetition of the above argument, we also have
\begin{align}
    {}^{(1)}\mathcal E^D_{\le N-2}[\psi]^2(\tau) \lesssim 1,
\end{align}
via an iteration.

\subsubsection{Control of ${}^{(2-\delta)}\mathcal E^D_{\le N-2}[\psi]$ }
We first consider the spacetime integral including the gauge potential $\slashed{A}$, since this is the most delicate case:
\begin{align}
    \int_{\mathcal D^{\tau_{i+1}}_{\tau_i}} r^{2-\delta} \langle \nabla_{\lambda}\mathscr L_Z^{\le N-2}( A_{e_B}\gamma^{\lambda}\gamma^{e_B}\psi), \gamma^T \nabla_L\mathscr L_Z^{\le N-2}\psi\rangle\,dV.
\end{align}
We focus on the low--high and high--low interactions. It remains to estimate the following integrals:
\begin{align}
    \int_{\mathcal D^{\tau_{i+1}}_{\tau_i}} r^{2-\delta}  \langle \nabla_L \mathcal L_Z^{\le N-2}\slashed{A} \gamma^L \mathscr L_Z^{\le N-13}\psi, \gamma^T \nabla_L\mathscr L_Z^{\le N-2}\psi\rangle\,dV, \\
     \int_{\mathcal D^{\tau_{i+1}}_{\tau_i}} r^{2-\delta}  \langle \mathcal L_Z^{\le N-2}\slashed{A} \gamma^L \nabla_L \mathscr L_Z^{\le N-13}\psi, \gamma^T \nabla_L\mathscr L_Z^{\le N-2}\psi\rangle\,dV , \\
      \int_{\mathcal D^{\tau_{i+1}}_{\tau_i}} r^{2-\delta}  \langle \nabla_L \mathcal L_Z^{\le N-13}\slashed{A} \gamma^L \mathscr L_Z^{\le N-2}\psi, \gamma^T \nabla_L\mathscr L_Z^{\le N-2}\psi\rangle\,dV, \\
       \int_{\mathcal D^{\tau_{i+1}}_{\tau_i}} r^{2-\delta}  \langle  \mathcal L_Z^{\le N-13}\slashed{A} \gamma^L \nabla_L \mathscr L_Z^{\le N-2}\psi, \gamma^T \nabla_L\mathscr L_Z^{\le N-2}\psi\rangle\,dV, \\
        \int_{\mathcal D^{\tau_{i+1}}_{\tau_i}} r^{2-\delta}  \langle \nabla_{\underline L} \mathcal L_Z^{\le N-2}\slashed{A} \gamma^{\underline L} \mathscr L_Z^{\le N-13}\psi, \gamma^T \nabla_L\mathscr L_Z^{\le N-2}\psi\rangle\,dV, \\
         \int_{\mathcal D^{\tau_{i+1}}_{\tau_i}} r^{2-\delta}  \langle  \mathcal L_Z^{\le N-2}\slashed{A} \gamma^{\underline L} \nabla_{\underline L} \mathscr L_Z^{\le N-13}\psi, \gamma^T \nabla_L\mathscr L_Z^{\le N-2}\psi\rangle\,dV, \\
          \int_{\mathcal D^{\tau_{i+1}}_{\tau_i}} r^{2-\delta}  \langle \nabla_{\underline L} \mathcal L_Z^{\le N-13}\slashed{A} \gamma^{\underline L} \mathscr L_Z^{\le N-2}\psi, \gamma^T \nabla_L\mathscr L_Z^{\le N-2}\psi\rangle\,dV , \\
           \int_{\mathcal D^{\tau_{i+1}}_{\tau_i}} r^{2-\delta}  \langle  \mathcal L_Z^{\le N-13}\slashed{A} \gamma^{\underline L}\nabla_{\underline L} \mathscr L_Z^{\le N-2}\psi, \gamma^T \nabla_L\mathscr L_Z^{\le N-2}\psi\rangle\,dV.
\end{align}

The first integral is straightforward to estimate: applying the weighted Sobolev and Cauchy--Schwarz inequalities gives
\begin{align*}
    & \left|  \int_{\mathcal D^{\tau_{i+1}}_{\tau_i}} r^{2-\delta}  \langle \nabla_L \mathcal L_Z^{\le N-2}\slashed{A} \gamma^L \mathscr L_Z^{\le N-13}\psi, \gamma^T \nabla_L\mathscr L_Z^{\le N-2}\psi\rangle\,dV \right| \\
    & \lesssim \int_{\tau_i}^{\tau_{i+1}} r^{\frac32-\delta} \mathcal E^W_{\le N-2}[\slashed{A}] \mathcal E^D_{\le N-11}[\psi] \mathcal E^D_{\le N-2}[\psi]\,d\tau \\
    & \lesssim \sup_{\tau\in[\tau_i,\tau_{i+1}]} {}^{(1)}\mathcal E^W_{\le N-2}[\slashed{A}](\tau) \left( \int_{\tau_i}^{\tau_{i+1}}{}^{(1-\delta)}\mathcal E^D_{\le N-11}[\psi]^2(\tau)\,d\tau  \right)^\frac12 \left( \int_{\tau_i}^{\tau_{i+1}}{}^{(1-\delta)}\mathcal E^D_{\le N-2}[\psi]^2(\tau)\,d\tau  \right)^\frac12 .
\end{align*}
The second integral is controlled directly. Indeed, we have
\begin{align*}
    & \left|  \int_{\mathcal D^{\tau_{i+1}}_{\tau_i}} r^{2-\delta}  \langle \mathcal L_Z^{\le N-2}\slashed{A} \gamma^L \nabla_L \mathscr L_Z^{\le N-13}\psi, \gamma^T \nabla_L\mathscr L_Z^{\le N-2}\psi\rangle\,dV  \right| \\
    & \lesssim  \int_{\mathcal D^{\tau_{i+1}}_{\tau_i}} r^{2-\delta} |\mathcal L_Z^{\le N-2}\slashed{A}| |\nabla_L\mathscr L_Z^{\le N-13}\psi_+| |\nabla_L\mathscr L_Z^{\le N-2}\psi_-| \,dV \\
    & \lesssim  \int_{\mathcal D^{\tau_{i+1}}_{\tau_i}} r^{1-\delta} |\mathcal L_Z^{\le N-2}\slashed{A}| |\mathscr L_Z^{\le N-12}\psi_+| |\nabla_L\mathscr L_Z^{\le N-2}\psi_-| \,dV \\
    & \lesssim \int_{\tau_i}^{\tau_{i+1}} \mathcal E^W_{\le N-1}[\slashed{A}] \mathcal E^D_{\le N-11}[\psi] \mathcal E^D_{\le N-2}[\psi]\,d\tau \\
    & \lesssim \sup_{\tau\in[\tau_i,\tau_{i+1}]} {}^{(1)}\mathcal E^D_{\le N-11}[\psi](\tau) \left( \int_{\tau_i}^{\tau_{i+1}}{}^{(1-\delta)}\mathcal E^D_{\le N-2}[\psi]^2(\tau)\,d\tau  \right)^\frac12 \left( \int_{\tau_i}^{\tau_{i+1}}{}^{(0)}\mathcal E^W_{\le N-1}[\slashed{A}]^2(\tau)\,d\tau  \right)^\frac12. 
\end{align*}
The third integral is subtle. Using the null decomposition for the spinor fields and the integration by parts along the $\underline L$-direction, we have
\begin{align*}
    & \int_{\mathcal D^{\tau_{i+1}}_{\tau_i}} r^{2-\delta}  \langle  \nabla_L\mathcal L_Z^{\le N-13}\slashed{A} \gamma^L  \mathscr L_Z^{\le N-2}\psi, \gamma^T \nabla_L\mathscr L_Z^{\le N-2}\psi\rangle\,dV \\
    & = \int_{\mathcal D^{\tau_{i+1}}_{\tau_i}} r^{3-\delta}  \langle  \nabla_L\mathcal L_Z^{\le N-13}\slashed{A} \gamma^L \gamma^{\underline L} \nabla_{\underline L} \mathscr L_Z^{\le N-3}\psi, \gamma^T \nabla_L\mathscr L_Z^{\le N-2}\psi\rangle\,dV \\
    & = \int_{\substack{u=\tau_{i+1}-R \\ v \ge \tau_{i+1}+R }} r^{3-\delta}  \langle  \nabla_L\mathcal L_Z^{\le N-13}\slashed{A} \gamma^L \gamma^{\underline L}  \mathscr L_Z^{\le N-3}\psi, \gamma^T \nabla_L\mathscr L_Z^{\le N-2}\psi\rangle\,d\sigma \\
    & \qquad- \int_{\substack{u=\tau_{i}-R \\ v \ge \tau_{i}+R }} r^{3-\delta}  \langle  \nabla_L\mathcal L_Z^{\le N-13}\slashed{A} \gamma^L \gamma^{\underline L}  \mathscr L_Z^{\le N-3}\psi, \gamma^T \nabla_L\mathscr L_Z^{\le N-2}\psi\rangle\,d\sigma \\
    & + \int_{r=R } r^{3-\delta}  \langle  \nabla_L\mathcal L_Z^{\le N-13}\slashed{A} \gamma^L \gamma^{\underline L}  \mathscr L_Z^{\le N-3}\psi, \gamma^T \nabla_L\mathscr L_Z^{\le N-2}\psi\rangle\,d\sigma  \\
    & +\int_{\mathcal D^{\tau_{i+1}}_{\tau_i}} r^{3-\delta}  \langle \nabla_{\underline L} \nabla_L\mathcal L_Z^{\le N-13}\slashed{A} \gamma^L \gamma^{\underline L}  \mathscr L_Z^{\le N-3}\psi, \gamma^T \nabla_L\mathscr L_Z^{\le N-2}\psi\rangle\,dV \\
    & \qquad + \int_{\mathcal D^{\tau_{i+1}}_{\tau_i}} r^{3-\delta}  \langle  \nabla_L\mathcal L_Z^{\le N-13}\slashed{A} \gamma^L \gamma^{\underline L}  \mathscr L_Z^{\le N-3}\psi, \gamma^T \nabla_{\underline L} \nabla_L\mathscr L_Z^{\le N-2}\psi\rangle\,dV.
\end{align*}
The boundary term on the outgoing null hypersurfaces is controlled by the weighted Sobolev inequality:
\begin{align*}
    & \left| \int_{\substack{u=\tau_{i+1}-R \\ v \ge \tau_{i+1}+R }} r^{3-\delta}  \langle  \nabla_L\mathcal L_Z^{\le N-13}\slashed{A} \gamma^L \gamma^{\underline L}  \mathscr L_Z^{\le N-3}\psi, \gamma^T \nabla_L\mathscr L_Z^{\le N-2}\psi\rangle\,d\sigma \right| \\
    & \lesssim {}^{(2-\delta)}\mathcal E^W_{\le N-13}[\slashed{A}](\tau_{i+1}) {}^{(2-\delta)}\mathcal E^D_{\le N-2}[\psi](\tau_{i+1}) {}^{(1)}\mathcal E^D_{\le N-1}[\psi](\tau_{i+1}).
\end{align*}
We note that the above bound improves the bootstrap assumption. Indeed, the lower-order term satisfies the inequality ${}^{(2-\delta)}\mathcal E^W_{\le N-13}[A]^2(\tau) \lesssim {}^{(1)}\mathcal E^W_{\le N-13}[A]^2(\tau) \tau^{1-\delta}\lesssim \tau^{-\delta} $.

When $\nabla_{\underline L}$ acts on the lower-order term, we use the wave equation for the gauge potential $\slashed{A}$ to write
\begin{align*}
    & \left| \int_{\mathcal D^{\tau_{i+1}}_{\tau_i}} r^{3-\delta}  \langle \nabla_{\underline L} \nabla_L\mathcal L_Z^{\le N-13}\slashed{A} \gamma^L \gamma^{\underline L}  \mathscr L_Z^{\le N-3}\psi, \gamma^T \nabla_L\mathscr L_Z^{\le N-2}\psi\rangle\,dV \right| \\
    & \lesssim  \int_{\mathcal D^{\tau_{i+1}}_{\tau_i}} r^{1-\delta} |\mathcal L_Z^{\le N-11}\slashed{A}| |\mathscr L_Z^{\le N-3}\psi_-| |\nabla_L \mathscr L_Z^{\le N-2} \psi_-| \,dV \\
    & \qquad +  \int_{\mathcal D^{\tau_{i+1}}_{\tau_i}} r^{2-\delta} |\nabla_L \mathcal L_Z^{\le N-12}\slashed{A}| |\mathscr L_Z^{\le N-3}\psi_-| |\nabla_L \mathscr L_Z^{\le N-2} \psi_-| \,dV \\
    & \qquad +  \int_{\mathcal D^{\tau_{i+1}}_{\tau_i}} r^{2-\delta} |\nabla_{\underline L} \mathcal L_Z^{\le N-12}\slashed{A}| |\mathscr L_Z^{\le N-3}\psi_-| |\nabla_L \mathscr L_Z^{\le N-2} \psi_-| \,dV \\
    & \qquad +  \int_{\mathcal D^{\tau_{i+1}}_{\tau_i}} r^{2-\delta} |\mathcal L_Z^{\le N-12}(\psi_+\cdot\psi_-)| |\mathscr L_Z^{\le N-3}\psi_-| |\nabla_L \mathscr L_Z^{\le N-2} \psi_-| \,dV .
\end{align*}
Then
\begin{align*}
    &  \int_{\mathcal D^{\tau_{i+1}}_{\tau_i}} r^{1-\delta} |\mathcal L_Z^{\le N-11}\slashed{A}| |\mathscr L_Z^{\le N-3}\psi_-| |\nabla_L \mathscr L_Z^{\le N-2} \psi_-| \,dV \\
    & \lesssim \int_{\tau_i}^{\tau_{i+1}} r^{1-\delta} \mathcal E^W_{\le N-10}[\slashed{A}] \mathcal E^D_{\le N-2}[\psi]^2\,d\tau \\
    & \lesssim \int_{\tau_i}^{\tau_{i+1}} {}^{(0)} \mathcal E^W_{\le N-10}[\slashed{A}](\tau) {}^{(1-\delta)} \mathcal E^D_{\le N-2}[\psi]^2(\tau) \,d\tau ,
\end{align*}
For the second integral, the weighted Sobolev inequality on ingoing null hypersurfaces gives
\begin{align*}
    &  \int_{\mathcal D^{\tau_{i+1}}_{\tau_i}} r^{2-\delta} |\nabla_L \mathcal L_Z^{\le N-12}\slashed{A}| |\mathscr L_Z^{\le N-3}\psi_-| |\nabla_L \mathscr L_Z^{\le N-2} \psi_-| \,dV \\
    & \lesssim \int r^{2-\delta} \|\nabla_{\underline L}\nabla_L \mathcal L_Z^{N-11}\slashed{A}\|_{L^2(\underline{\mathcal C}_v)} \|\nabla_{\underline L}\mathscr L_Z^{\le N-2}\psi_-\|_{L^2(\underline{\mathcal C}_v)} \|\nabla_L\mathscr L_Z^{\le N-2}\psi_-\|_{L^2(\underline{\mathcal C}_v)} \,dv  \\
    & \lesssim \int r^{1-\delta} \|\slashed\nabla \mathcal L_Z^{\le N-10}\slashed{A}\|_{L^2(\underline{\mathcal C}_v)} \|\slashed\nabla \mathscr L_Z^{\le N-2}\psi_+\|_{L^2(\underline{\mathcal C}_v)} \|\nabla_L\mathscr L_Z^{\le N-2}\psi_-\|_{L^2(\underline{\mathcal C}_v)}\,dv \\
    & \qquad + \int r^{1-\delta}\|\nabla_L \mathcal L_Z^{\le N-11}\slashed{A}\|_{L^2(\underline{\mathcal C}_v)} \|\slashed\nabla \mathscr L_Z^{\le N-2}\psi_+\|_{L^2(\underline{\mathcal C}_v)} \|\nabla_L\mathscr L_Z^{\le N-2}\psi_-\|_{L^2(\underline{\mathcal C}_v)}\,dv \\
    & \qquad +  \int r^{1-\delta}\|\nabla_{\underline L} \mathcal L_Z^{\le N-11}\slashed{A}\|_{L^2(\underline{\mathcal C}_v)} \|\slashed\nabla \mathscr L_Z^{\le N-2}\psi_+\|_{L^2(\underline{\mathcal C}_v)} \|\nabla_L\mathscr L_Z^{\le N-2}\psi_-\|_{L^2(\underline{\mathcal C}_v)}\,dv.
\end{align*}
Then we have
\begin{align*}
    & \int r^{1-\delta} \|\slashed\nabla \mathcal L_Z^{\le N-10}\slashed{A}\|_{L^2(\underline{\mathcal C}_v)} \|\slashed\nabla \mathscr L_Z^{\le N-2}\psi_+\|_{L^2(\underline{\mathcal C}_v)} \|\nabla_L\mathscr L_Z^{\le N-2}\psi_-\|_{L^2(\underline{\mathcal C}_v)}\,dv \\
    & \qquad + \int r^{1-\delta}\|\nabla_L \mathcal L_Z^{\le N-11}\slashed{A}\|_{L^2(\underline{\mathcal C}_v)} \|\slashed\nabla \mathscr L_Z^{\le N-2}\psi_+\|_{L^2(\underline{\mathcal C}_v)} \|\nabla_L\mathscr L_Z^{\le N-2}\psi_-\|_{L^2(\underline{\mathcal C}_v)}\,dv \\
    & \lesssim \sup_{v:\tau\le\tau(v) } {}^{(2-\delta)}\mathcal F^D_{\le N-2}[\psi] \left( \int_{\tau_i}^{\tau_{i+1}}{}^{(0)}\mathcal E^W_{\le N-10}[\slashed{A}]^2(\tau)\,d\tau  \right)^\frac12  \left( \int_{\tau_i}^{\tau_{i+1}}{}^{(0)}\mathcal E^D_{\le N-2}[\psi]^2(\tau)\,d\tau  \right)^\frac12,
\end{align*}
and 
\begin{align*}
    &  \int r^{1-\delta}\|\nabla_{\underline L} \mathcal L_Z^{\le N-11}\slashed{A}\|_{L^2(\underline{\mathcal C}_v)} \|\slashed\nabla \mathscr L_Z^{\le N-2}\psi_+\|_{L^2(\underline{\mathcal C}_v)} \|\nabla_L\mathscr L_Z^{\le N-2}\psi_-\|_{L^2(\underline{\mathcal C}_v)}\,dv \\
    & \lesssim (\tau_{i+1})^{1-\delta} \sup_{v:\tau\le \tau(v)} {}^{(0)}\mathcal F^W_{\le N-11}[\slashed{A}] \int_{\tau_i}^{\tau_{i+1}} {}^{(0)}\mathcal E^D_{\le N-2}[\psi]^2(\tau)\,d\tau .
\end{align*}
The factor $\tau^{1-\delta}$ is acceptable because the lower-order term ${}^{(0)}\mathcal F^W_{\le N-11}[\slashed{A}]$ decays sufficiently fast.

For the integral including the $\nabla_{\underline L}$ derivative, we use the integration by parts and get
\begin{align*}
    &  \int_{\mathcal D^{\tau_{i+1}}_{\tau_i}} r^{2-\delta} |\nabla_{\underline L} \mathcal L_Z^{\le N-12}\slashed{A}| |\mathscr L_Z^{\le N-3}\psi_-| |\nabla_L \mathscr L_Z^{\le N-2} \psi_-| \,dV \\
    & \lesssim \int_{\mathcal D^{\tau_{i+1}}_{\tau_i}} r^{2-\delta} | \mathcal L_Z^{\le N-12}\slashed{A}| |\nabla_{\underline L}\mathscr L_Z^{\le N-3}\psi_-| |\nabla_L \mathscr L_Z^{\le N-2} \psi_-| \,dV \\
    & \qquad +  \int_{\mathcal D^{\tau_{i+1}}_{\tau_i}} r^{2-\delta} | \mathcal L_Z^{\le N-12}\slashed{A}| |\mathscr L_Z^{\le N-3}\psi_-| |\nabla_{\underline L}\nabla_L \mathscr L_Z^{\le N-2} \psi_-| \,dV \\
    & \lesssim \int_{\mathcal D^{\tau_{i+1}}_{\tau_i}} r^{2-\delta} | \mathcal L_Z^{\le N-12}\slashed{A}| |\slashed\nabla\mathscr L_Z^{\le N-3}\psi_+| |\nabla_L \mathscr L_Z^{\le N-2} \psi_-| \,dV  \\
    & \qquad +  \int_{\mathcal D^{\tau_{i+1}}_{\tau_i}} r^{2-\delta} | \mathcal L_Z^{\le N-12}\slashed{A}| |\slashed\nabla\mathscr L_Z^{\le N-3}\psi_-| |\nabla_L \mathscr L_Z^{\le N-2} \psi_-| \,dV,
\end{align*}
where we integrated by parts in the angular variables to obtain the last inequality. The remaining estimate follows directly, so we omit the details.

For the spacetime integral in which $\nabla_{\underline L}$ acts on the higher-order term, we apply the weighted Sobolev inequality on ingoing null hypersurfaces to obtain
\begin{align*}
    & \left| \int_{\mathcal D^{\tau_{i+1}}_{\tau_i}} r^{3-\delta}  \langle  \nabla_L\mathcal L_Z^{\le N-13}\slashed{A} \gamma^L \gamma^{\underline L}  \mathscr L_Z^{\le N-3}\psi, \gamma^T \nabla_{\underline L} \nabla_L\mathscr L_Z^{\le N-2}\psi\rangle\,dV \right| \\
    & \lesssim \int_{\mathcal D^{\tau_{i+1}}_{\tau_i}} r^{2-\delta} |  \nabla_L\mathcal L_Z^{\le N-13}\slashed{A} | |  \mathscr L_Z^{\le N-3}\psi_-| | \nabla_L\mathscr L_Z^{\le N-1}\psi_+ |\,dV \\
    & \lesssim \int r^{2-\delta}  \|\nabla_{\underline L}  \nabla_L\mathcal L_Z^{\le N-12}\slashed{A} \|_{L^2(\underline{\mathcal C}_v)} \| \nabla_{\underline L} \mathscr L_Z^{\le N-2}\psi_-\|_{L^2(\underline{\mathcal C}_v)} \| \nabla_L\mathscr L_Z^{\le N-1}\psi_+ \|_{L^2(\underline{\mathcal C}_v)}\,dv \\
    & \lesssim  \int r^{1-\delta}  \|\slashed\nabla \mathcal L_Z^{\le N-12}\slashed{A} \|_{L^2(\underline{\mathcal C}_v)} \| \slashed\nabla \mathscr L_Z^{\le N-2}\psi_+\|_{L^2(\underline{\mathcal C}_v)} \| \nabla_L\mathscr L_Z^{\le N-1}\psi_+ \|_{L^2(\underline{\mathcal C}_v)}\,dv \\
    & \qquad + \int r^{1-\delta}  \|\nabla_L \mathcal L_Z^{\le N-12}\slashed{A} \|_{L^2(\underline{\mathcal C}_v)} \| \slashed\nabla \mathscr L_Z^{\le N-2}\psi_+\|_{L^2(\underline{\mathcal C}_v)} \| \nabla_L\mathscr L_Z^{\le N-1}\psi_+ \|_{L^2(\underline{\mathcal C}_v)}\,dv \\
    & \qquad + \int r^{1-\delta}  \|\nabla_{\underline L} \mathcal L_Z^{\le N-12}\slashed{A} \|_{L^2(\underline{\mathcal C}_v)} \| \slashed\nabla \mathscr L_Z^{\le N-2}\psi_+\|_{L^2(\underline{\mathcal C}_v)} \| \nabla_L\mathscr L_Z^{\le N-1}\psi_+ \|_{L^2(\underline{\mathcal C}_v)}\,dv \\
    & \qquad + \int r^{1-\delta}  \| \mathcal L_Z^{\le N-12}(\psi_+\cdot\psi_-)\|_{L^2(\underline{\mathcal C}_v)} \| \slashed\nabla \mathscr L_Z^{\le N-2}\psi_+\|_{L^2(\underline{\mathcal C}_v)} \| \nabla_L\mathscr L_Z^{\le N-1}\psi_+ \|_{L^2(\underline{\mathcal C}_v)}\,dv.
\end{align*}
Then 
\begin{align*}
    & \int r^{1-\delta}  \|\slashed\nabla \mathcal L_Z^{\le N-12}\slashed{A} \|_{L^2(\underline{\mathcal C}_v)} \| \slashed\nabla \mathscr L_Z^{\le N-2}\psi_+\|_{L^2(\underline{\mathcal C}_v)} \| \nabla_L\mathscr L_Z^{\le N-1}\psi_+ \|_{L^2(\underline{\mathcal C}_v)}\,dv \\
    & \qquad + \int r^{1-\delta}  \|\nabla_L \mathcal L_Z^{\le N-12}\slashed{A} \|_{L^2(\underline{\mathcal C}_v)} \| \slashed\nabla \mathscr L_Z^{\le N-2}\psi_+\|_{L^2(\underline{\mathcal C}_v)} \| \nabla_L\mathscr L_Z^{\le N-1}\psi_+ \|_{L^2(\underline{\mathcal C}_v)}\,dv \\
    & \lesssim \sup_{v:\tau\le\tau(v)} {}^{(2-\delta)}\mathcal F^D_{\le N-2}[\psi] \left( \int_{\tau_i}^{\tau_{i+1}} {}^{(0)}\mathcal E^W_{\le N-11}[\slashed{A}]\,d\tau \right)^\frac12 \left( \int_{\tau_i}^{\tau_{i+1}} {}^{(0)}\mathcal E^D_{\le N-1}[\psi]\,d\tau \right)^\frac12 ,
\end{align*}
and
\begin{align*}
    & \int r^{1-\delta}  \|\nabla_{\underline L} \mathcal L_Z^{\le N-12}\slashed{A} \|_{L^2(\underline{\mathcal C}_v)} \| \slashed\nabla \mathscr L_Z^{\le N-2}\psi_+\|_{L^2(\underline{\mathcal C}_v)} \| \nabla_L\mathscr L_Z^{\le N-1}\psi_+ \|_{L^2(\underline{\mathcal C}_v)}\,dv \\
    & \lesssim (\tau_{i+1})^{\frac12-\frac\delta2} \sup_{v:\tau\le\tau(v)} {}^{(0)}\mathcal F^W_{\le N-11}[\slashed{A}] \left( \int_{\tau_i}^{\tau_{i+1}} {}^{(1-\delta)}\mathcal E^D_{\le N-2}[\psi]\,d\tau \right)^\frac12 \left( \int_{\tau_i}^{\tau_{i+1}} {}^{(0)}\mathcal E^D_{\le N-1}[\psi]\,d\tau \right)^\frac12.
\end{align*}
The additional factor $\tau^{\frac12-\frac\delta2}$ is acceptable because ${}^{(0)}\mathcal F^W_{\le N-11}[\slashed{A}]$ decays sufficiently fast.

The fourth integral is controlled directly.

To estimate the fifth integral, we integrate by parts in the $\underline L$ direction to obtain
\begin{align*}
    &  \int_{\mathcal D^{\tau_{i+1}}_{\tau_i}} r^{2-\delta}  \langle \nabla_{\underline L} \mathcal L_Z^{\le N-2}\slashed{A} \gamma^{\underline L} \mathscr L_Z^{\le N-13}\psi, \gamma^T \nabla_L\mathscr L_Z^{\le N-2}\psi\rangle\,dV \\
    & = \int_{\substack{u=\tau_{i+1}-R \\ v\ge \tau_{i+1}+R } } r^{2-\delta}  \langle  \mathcal L_Z^{\le N-2}\slashed{A} \gamma^{\underline L} \mathscr L_Z^{\le N-13}\psi, \gamma^T \nabla_L\mathscr L_Z^{\le N-2}\psi\rangle\,d\sigma \\
    & \qquad - \int_{\substack{u=\tau_{i}-R \\ v\ge \tau_{i}+R } } r^{2-\delta}  \langle  \mathcal L_Z^{\le N-2}\slashed{A} \gamma^{\underline L} \mathscr L_Z^{\le N-13}\psi, \gamma^T \nabla_L\mathscr L_Z^{\le N-2}\psi\rangle\,d\sigma \\
    & + \int_{r=R } r^{2-\delta}  \langle  \mathcal L_Z^{\le N-2}\slashed{A} \gamma^{\underline L} \mathscr L_Z^{\le N-13}\psi, \gamma^T \nabla_L\mathscr L_Z^{\le N-2}\psi\rangle\,d\sigma \\
    & + \int_{\mathcal D^{\tau_{i+1}}_{\tau_i}} r^{2-\delta}  \langle  \mathcal L_Z^{\le N-2}\slashed{A} \gamma^{\underline L} \nabla_{\underline L} \mathscr L_Z^{\le N-13}\psi, \gamma^T \nabla_L\mathscr L_Z^{\le N-2}\psi\rangle\,dV \\
    & \qquad +  \int_{\mathcal D^{\tau_{i+1}}_{\tau_i}} r^{2-\delta}  \langle  \mathcal L_Z^{\le N-2}\slashed{A} \gamma^{\underline L}  \mathscr L_Z^{\le N-13}\psi, \gamma^T \nabla_{\underline L}\nabla_L\mathscr L_Z^{\le N-2}\psi\rangle\,dV.
\end{align*}
The boundary terms on outgoing null hypersurfaces are controlled by the energy identity:
\begin{align*}
    & \left| \int_{\substack{u=\tau_{i+1}-R \\ v\ge \tau_{i+1}+R } } r^{2-\delta}  \langle  \mathcal L_Z^{\le N-2}\slashed{A} \gamma^{\underline L} \mathscr L_Z^{\le N-13}\psi, \gamma^T \nabla_L\mathscr L_Z^{\le N-2}\psi\rangle\,d\sigma  \right| \\
    & \lesssim {}^{(0)} \mathcal E^W_{\le N-1}[\slashed{A}](\tau_{i+1}) {}^{(2-\delta)} \mathcal E^D_{\le N-12}[\psi](\tau_{i+1}) {}^{(2-\delta)} \mathcal E^D_{\le N-2}[\psi](\tau_{i+1}) .
\end{align*}
The term in which $\nabla_{\underline L}$ acts on the lower-order factor is controlled directly. When it acts on the higher-order factor, we use the null decomposition for the spinor fields and integrate by parts in the angular variables. This reduces the estimate to the following integral:
\begin{align}
    \int_{\mathcal D^{\tau_{i+1}}_{\tau_i}} r^{2-\delta} \langle \slashed\nabla \mathscr L_Z^{\le N-2} \slashed{A} \gamma^{\underline L}\mathscr L_Z^{\le N-13}\psi, \gamma^T \nabla_L \mathscr L_Z^{\le N-2}\psi\rangle\,dV.
\end{align}
Using the weighted Sobolev inequality, we have
\begin{align*}
    (\tau_{i+1})^{1-\delta} \sup_{\tau\in[\tau_i,\tau_{i+1}]} {}^{(0)}\mathcal E^D_{\le N-11}[\psi](\tau) \left( \int_{\tau_i}^{\tau_{i+1}} {}^{(1-\delta)}\mathcal E^D_{\le N-2}[\psi]^2(\tau)\,d\tau\right)^\frac12  \left( \int_{\tau_i}^{\tau_{i+1}} {}^{(0)}\mathcal E^W_{\le N-2}[\slashed{A}]^2(\tau)\,d\tau\right)^\frac12 .
\end{align*}
The factor $\tau^{1-\delta}$ is acceptable because ${}^{(0)}\mathcal E^D_{\le N-11}[\psi]$ decays sufficiently fast.

Consequently, iteration improves the bound for ${}^{(2-\delta)}\mathcal E^D_{\le N-2}[\psi]$ and yields a uniform estimate.


\subsubsection{Control of ${}^{(2)}\mathcal E^D_{\le N-9}[\psi]$}
We consider the spacetime integral
\begin{align}
    \int_{\mathcal D^{\tau_{i+1}}_{\tau_i}} r^2 \langle \nabla_L \mathcal L_Z^{\le N-13} \slashed{A} \gamma^L \mathscr L_Z^{\le N-9}\psi, \gamma^T \nabla_L\mathscr L_Z^{\le N-9}\psi\rangle\,dV.
\end{align}
Then
\begin{align*}
    & \left|  \int_{\mathcal D^{\tau_{i+1}}_{\tau_i}} r^2 \langle \nabla_L \mathcal L_Z^{\le N-13} \slashed{A} \gamma^L \mathscr L_Z^{\le N-9}\psi, \gamma^T \nabla_L\mathscr L_Z^{\le N-9}\psi\rangle\,dV \right| \\
    & \lesssim \int_{\tau_i}^{\tau_{i+1}} r^{\frac32} \mathcal E^W_{\le N-13}[\slashed{A}] \mathcal E^D_{\le N-7}[\psi] \mathcal E^D_{\le N-9}[\psi]\,d\tau \\
    & \lesssim \sup_{\tau\in[\tau_i,\tau_{i+1}]} {}^{(2-\delta)}\mathcal E^W_{\le N-13}[\slashed{A}](\tau) \left( \int_{\tau_i}^{\tau_{i+1}} {}^{(0)}\mathcal E^D_{\le N-7}[\psi]^2(\tau)\,d\tau \right)^\frac12 \left( \int_{\tau_i}^{\tau_{i+1}} {}^{(1-\delta)}\mathcal E^D_{\le N-9}[\psi]^2(\tau)\,d\tau \right)^\frac12.
\end{align*}

\subsubsection{Control of ${}^{(2)}\mathcal E^W_{\le N-10}[A]$}
We consider the spacetime integral
\begin{align}
    \int_{\mathcal D^{\tau_{i+1}}_{\tau_i}} r^2 \langle \mathscr L_Z^{\le N-13}\psi, \mathscr L_Z^{\le N-10}\rangle \nabla_L \mathcal L_Z^{\le N-10}A\,dV. 
\end{align}
Then
\begin{align*}
    & \left|  \int_{\mathcal D^{\tau_{i+1}}_{\tau_i}} r^2 \langle \mathscr L_Z^{\le N-13}\psi, \mathscr L_Z^{\le N-10}\psi \rangle \nabla_L \mathcal L_Z^{\le N-10}A\,dV \right| \\
    & \lesssim \int_{\tau_i}^{\tau_{i+1}}r^2  \mathcal E^D_{\le N-12}[\psi]\mathcal E^D_{\le N-9}[\psi] \mathcal E^W_{\le N-10}[A]\,d\tau  \\
    & \lesssim \sup_{\tau\in[\tau_i,\tau_{i+1}]} {}^{(2)}\mathcal E^W_{\le N-10}[A]   \left( \int_{\tau_i}^{\tau_{i+1}} {}^{(1)}\mathcal E^D_{\le N-12}[\psi]^2(\tau)\,d\tau \right)^\frac12 \left( \int_{\tau_i}^{\tau_{i+1}} {}^{(1)}\mathcal E^D_{\le N-9}[\psi]^2(\tau)\,d\tau \right)^\frac12.
\end{align*}


\subsection{Control of the weighted energy for the gauge potential}

\subsubsection{Control of ${}^{(1)}\mathcal E^W_{\le N-1}[A_L] $ }
We consider the integral of the form:
\begin{align}
    \int_{\mathcal D^{\tau_{i+1}}_{\tau_i}} r^{} \mathscr L_Z^{\le N-13}\psi_- \cdot \mathscr L_Z^{\le N-1}\psi_- \cdot \nabla_L \mathcal L_Z^{\le N-1}A_L\,dV.
\end{align}
It suffices to consider the case in which only rotation fields act on the higher-order term.
Using the null decomposition for the spinor fields, we have
\begin{align*}
    & \left|  \int_{\mathcal D^{\tau_{i+1}}_{\tau_i}} r^{} \mathscr L_Z^{\le N-13}\psi_- \cdot \mathscr L_Z^{\le N-1}\psi_- \cdot \nabla_L \mathcal L_Z^{\le N-1}A_L\,dV \right| \\
    & \lesssim  \int_{\mathcal D^{\tau_{i+1}}_{\tau_i}} r^{2} |\mathscr L_Z^{\le N-13}\psi_-| |\nabla_L \mathscr L_Z^{\le N-2}\psi_+| | \nabla_L \mathcal L_Z^{\le N-1}A_L|\,dV \\
    & \lesssim \int_{\tau_i}^{\tau_{i+1}} r^{\frac32} \mathcal E^D_{\le N-11}[\psi] \mathcal E^D_{\le N-2}[\psi] \mathcal E^W_{\le N-1}[A_L]\,d\tau \\
    & \lesssim (\tau_{i+1})^{\frac\delta2}\sup_{\tau\in[\tau_i,\tau_{i+1}]} {}^{(1)}\mathcal E^W_{\le N-1}[A_L](\tau)  \left( \int_{\tau_i}^{\tau_{i+1}} {}^{(1)}\mathcal E^D_{\le N-11}[\psi]^2(\tau)\,d\tau  \right)^\frac12 \left( \int_{\tau_i}^{\tau_{i+1}} {}^{(1-\delta)}\mathcal E^D_{\le N-2}[\psi]^2(\tau)\,d\tau  \right)^\frac12.
\end{align*}
\subsubsection{Control of ${}^{(1)}\mathcal E^W_{\le N-1}[A_{\underline L}] $ }
We consider the integral of the form:
\begin{align}
    \int_{\mathcal D^{\tau_{i+1}}_{\tau_i}} r^{} \mathscr L_Z^{\le N-13}\psi_+ \cdot \mathscr L_Z^{\le N-1}\psi_+ \cdot \nabla_L \mathcal L_Z^{\le N-1}A_{\underline L}\,dV.
\end{align}
We simply use the weighted Sobolev inequality and get
\begin{align*}
    & \left|  \int_{\mathcal D^{\tau_{i+1}}_{\tau_i}} r^{} \mathscr L_Z^{\le N-13}\psi_+ \cdot \mathscr L_Z^{\le N-1}\psi_+ \cdot \nabla_L \mathcal L_Z^{\le N-1}A_{\underline L}\,dV \right| \\
    & \lesssim \int_{\tau_i}^{\tau_{i+1}} r^{\frac12} \mathcal E^D_{\le N-11}[\psi] \mathscr E^D_{\le N-1}[\psi] \mathcal E^D_{\le N-1}[A_{\underline L}] \\
    & \lesssim \sup_{\tau\in[\tau_i,\tau_{i+1}]}\mathscr E^D_{\le N-1}[\psi](\tau)  \left( \int_{\tau_i}^{\tau_{i+1}} {}^{(1)}\mathcal E^D_{\le N-11}[\psi]^2(\tau)\,d\tau  \right)^\frac12 \left( \int_{\tau_i}^{\tau_{i+1}} {}^{(0)}\mathcal E^W_{\le N-1}[A_{\underline L}]^2(\tau)\,d\tau  \right)^\frac12
\end{align*}

\subsubsection{Control of ${}^{(1)}\mathcal E^W_{\le N-1}[\slashed{A}] $ }
We need to deal with the following two integrals:
\begin{align}
    &  \int_{\mathcal D^{\tau_{i+1}}_{\tau_i}} r^{} \mathscr L_Z^{\le N-13}\psi_- \cdot \mathscr L_Z^{\le N-1}\psi_+ \cdot \nabla_L \mathcal L_Z^{\le N-1}\slashed{A} \,dV, \\
    & \int_{\mathcal D^{\tau_{i+1}}_{\tau_i}} r^{} \mathscr L_Z^{\le N-13}\psi_+ \cdot \mathscr L_Z^{\le N-1}\psi_- \cdot \nabla_L \mathcal L_Z^{\le N-1}\slashed{A} \,dV.
\end{align}
The same argument yields the same bound; we omit the details.

In particular, repeating the above argument, one can show that
\begin{align}
  {}^{(1)}\mathcal E^W_{\le N-3}[A]^2(\tau) \lesssim 1 , \quad  {}^{(1)}\mathcal E^W_{\le N-2}[A]^2(\tau) \lesssim \tau^{\delta}.
\end{align}
\subsubsection{Control of ${}^{(2-\delta)}\mathcal E^W_{\le N-3}[A] $ }
We consider the integral of the form
\begin{align}
    \int_{\mathcal D^{\tau_{i+1}}_{\tau_i}} r^{2-\delta} \mathscr L_Z^{\le N-13}\psi \cdot \mathscr L_Z^{\le N-3}\psi \cdot \nabla_L \mathcal L_Z^{\le N-3}A\,dV.
\end{align}
Applying the weighted Sobolev inequality on outgoing null hypersurfaces to the $L^4\times L^4\times L^2$ product yields
\begin{align*}
    & \left|     \int_{\mathcal D^{\tau_{i+1}}_{\tau_i}} r^{2-\delta} \mathscr L_Z^{\le N-13}\psi \cdot \mathscr L_Z^{\le N-3}\psi \cdot \nabla_L \mathcal L_Z^{\le N-3}A\,dV \right| \\
    & \lesssim \int_{\tau_i}^{\tau_{i+1}} r^{2-\delta} \mathcal E^D_{\le N-12}[\psi]\mathcal E^D_{\le N-2}[\psi] \mathcal E^W_{\le N-3}[A]\,d\tau \\
    & \lesssim \sup_{\tau\in[\tau_i,\tau_{i+1}]} {}^{(2-\delta)}\mathcal E^W_{\le N-3}[A](\tau) \left( \int_{\tau_i}^{\tau_{i+1}} {}^{(1-\delta)}\mathcal E^D_{\le N-12}[\psi]^2(\tau)\,d\tau   \right)^\frac12 \left( \int_{\tau_i}^{\tau_{i+1}} {}^{(1-\delta)}\mathcal E^D_{\le N-2}[\psi]^2(\tau)\,d\tau   \right)^\frac12.
\end{align*}


\subsubsection{Control of ${}^{(0)}\mathcal E^W_{\le N-1}[A] $ }
We see that
\begin{align}
    \begin{aligned}
        & \left| \int_{\mathcal D^{\tau_{i+1}}_{\tau_i}}  \mathscr L_Z^{\le N-13}\psi \cdot \mathscr L_Z^{\le N-1}\psi \cdot \nabla_L \mathcal L_Z^{\le N-1}A\,dV \right| \\ 
        & \lesssim \int_{\tau_i}^{\tau_{i+1}} r^\frac12 \mathcal E^D_{\le N-11}[\psi]\mathcal E^D_{\le N-1}[\psi]\mathcal E^W_{\le N-1}[A]\,d\tau \\
        & \lesssim (\tau_{i+1})^{\frac\delta2} \sup_{\tau\in[\tau_i,\tau_{i+1}]} {}^{(1)}\mathcal E^W_{\le N-1}[A](\tau) \left( \int_{\tau_i}^{\tau_{i+1}} {}^{(0)}\mathcal E^D_{\le N-11}[\psi]^2(\tau)\,d\tau   \right)^\frac12 \left( \int_{\tau_i}^{\tau_{i+1}} {}^{(0)}\mathcal E^D_{\le N-1}[\psi]^2(\tau)\,d\tau   \right)^\frac12.
    \end{aligned}
\end{align}

\subsection{Interior region}

\subsubsection{Control of the energy for the spinor fields}
To estimate the spinor-field energy in the interior, we consider two energies: the Dirac-current energy ${}_{\rm int} \mathscr E^D_{\le k}[\psi]$ and the spinorial wave energy ${}_{\rm int}\mathcal E^D_{\le k}[\psi]$. As shown in Section~\ref{sec:interior}, the former is controlled by the latter. It therefore suffices to estimate the following integral:
\begin{align}
    \begin{aligned}
        & \sum_{k_1+k_2\le k}\int_{\tau_i}^{\tau_{i+1}} \int_{\Sigma_\tau \cap \{r\le R\} } \langle \nabla_\lambda\mathcal L_Z^{\le k_1}   A_\mu \gamma^\lambda \gamma^\mu \mathscr L_Z^{\le k_2} \psi, \gamma^T \nabla_{\mathcal N} \mathscr L_Z^{\le k}\psi\rangle \,d\sigma_{\Sigma_\tau}d\tau , \\
        & \sum_{k_1+k_2\le k}\int_{\tau_i}^{\tau_{i+1}} \int_{\Sigma_\tau \cap \{r\le R\} } \langle \mathcal L_Z^{\le k_1}   A_\mu \gamma^\lambda \gamma^\mu \nabla_\lambda \mathscr L_Z^{\le k_2} \psi, \gamma^T \nabla_{\mathcal N} \mathscr L_Z^{\le k}\psi\rangle \,d\sigma_{\Sigma_\tau}d\tau .
    \end{aligned}
\end{align}
We need only consider the top-order case $k=N-1$, for the low--high and high--low interactions.

We first deal with the integral
\begin{align}
    & \int_{\tau_i}^{\tau_{i+1}} \int_{\Sigma_\tau \cap \{r\le R\} } \langle \nabla_\lambda\mathcal L_Z^{\le N-13}   A_\mu \gamma^\lambda \gamma^\mu \mathscr L_Z^{\le N-1} \psi, \gamma^T \nabla_{\mathcal N} \mathscr L_Z^{\le N-1}\psi\rangle \,d\sigma_{\Sigma_\tau}d\tau .
\end{align}
The Sobolev embedding gives
\begin{align*}
    & \left|  \int_{\tau_i}^{\tau_{i+1}} \int_{\Sigma_\tau \cap \{r\le R\} } \langle \nabla_\lambda\mathcal L_Z^{\le N-13}   A_\mu \gamma^\lambda \gamma^\mu \mathscr L_Z^{\le N-1} \psi, \gamma^T \nabla_{\mathcal N} \mathscr L_Z^{\le N-1}\psi\rangle \,d\sigma_{\Sigma_\tau}d\tau  \right| \\
    & \lesssim \int_{\tau_i}^{\tau_{i+1}} {}_{\rm int}\mathcal E^W_{\le N-9}[A_{\mu}] {}_{\rm int}\mathcal E^D_{\le N-2}[\psi] {}_{\rm int}\mathcal E^D_{\le N-1}[\psi]\,d\tau \\
    & \lesssim \sup_{\tau\in[\tau_i,\tau_{i+1}]} {}_{\rm int}\mathcal E^D_{\le N-1}[\psi](\tau) \left( \int_{\tau_i}^{\tau_{i+1}} {}_{\rm int}\mathcal E^W_{\le N-9}[A_{\mu}]^2(\tau)\,d\tau \right)^\frac12 \left( \int_{\tau_i}^{\tau_{i+1}} {}_{\rm int}\mathcal E^D_{\le N-2}[\psi]^2(\tau)\,d\tau \right)^\frac12. 
\end{align*}
As in Section~\ref{sec:interior}, we apply the non-degenerate integrated local energy decay estimate, which incurs a derivative loss. This gives
\begin{align*}
    & \left| \int_{\tau_i}^{\tau_{i+1}} \int_{\Sigma_\tau \cap \{r\le R\} } \langle \mathcal L_Z^{\le N-13}   A_\mu \gamma^\lambda \gamma^\mu \nabla_\lambda\mathscr L_Z^{\le N-1} \psi, \gamma^T \nabla_{\mathcal N} \mathscr L_Z^{\le N-1}\psi\rangle \,d\sigma_{\Sigma_\tau}d\tau \right| \\
    & \lesssim (\tau_{i+1}-\tau_i)^{\frac12} \sup_{\tau\in[\tau_i,\tau_{i+1}]} {}_{\rm int}\mathcal E^D_{\le N-1}[\psi]^2(\tau) \left( \int_{\tau_i}^{\tau_{i+1}} {}_{\rm int}\mathcal E^W_{\le N-11}[A_\mu]^2(\tau)\,d\tau  \right)^{\frac12} 
\end{align*}
which yields a logarithmic growth after an application of the non-degenerate integrated local energy decay estimates. On the other hand, for the high-low interaction, we see that
\begin{align*}
    & \left| \int_{\tau_i}^{\tau_{i+1}} \int_{\Sigma_\tau \cap \{r\le R\} } \langle \nabla_\lambda\mathcal L_Z^{\le N-1}   A_\mu \gamma^\lambda \gamma^\mu \mathscr L_Z^{\le N-13} \psi, \gamma^T \nabla_{\mathcal N} \mathscr L_Z^{\le N-1}\psi\rangle \,d\sigma_{\Sigma_\tau}d\tau \right| \\
    & \lesssim \int_{\tau_i}^{\tau_{i+1}} {}_{\rm int}\mathcal E^W_{\le N-1}[A_\mu] {}_{\rm int}\mathcal E^D_{\le N-11}[\psi] {}_{\rm int}\mathcal E^D_{\le N-1}[\psi]\,d\tau \\
    & \lesssim \sup_{\tau\in[\tau_i,\tau_{i+1}]} {}_{\rm int}\mathcal E^W_{\le N-1}[A_\mu](\tau) {}_{\rm int}\mathcal E^D_{\le N-1}[\psi](\tau)\left( \int_{\tau_i}^{\tau_{i+1}} {}_{\rm int}\mathcal E^D_{\le N-11}[\psi]^2(\tau)\,d\tau  \right)^{\frac12} ,
\end{align*}
and
\begin{align*}
    &  \left| \int_{\tau_i}^{\tau_{i+1}} \int_{\Sigma_\tau \cap \{r\le R\} } \langle \mathcal L_Z^{\le N-1}   A_\mu \gamma^\lambda \gamma^\mu \nabla_\lambda\mathscr L_Z^{\le N-13} \psi, \gamma^T \nabla_{\mathcal N} \mathscr L_Z^{\le N-1}\psi\rangle \,d\sigma_{\Sigma_\tau}d\tau \right| \\
    & \lesssim \int_{\tau_i}^{\tau_{i+1}}{}_{\rm int}\mathcal E^W_{\le N-2}[A_\mu] {}_{\rm int}\mathcal E^D_{\le N-10}[\psi] {}_{\rm int}\mathcal E^D_{\le N-1}[\psi]\,d\tau \\
    & \lesssim \sup_{\tau\in[\tau_i,\tau_{i+1}]} {}_{\rm int}\mathcal E^D_{\le N-1}[\psi](\tau) \left( \int_{\tau_i}^{\tau_{i+1}} {}_{\rm int}\mathcal E^D_{\le N-10}[\psi]^2(\tau)\,d\tau  \right)^{\frac12}  \left( \int_{\tau_i}^{\tau_{i+1}} {}_{\rm int}\mathcal E^W_{\le N-2}[A_\mu]^2(\tau)\,d\tau  \right)^{\frac12} ,
\end{align*}
where the remaining estimate follows directly.
\subsubsection{Control of the energy for the gauge fields}
We have to consider the integral
\begin{align}
    \begin{aligned}
        \int_{\tau_i}^{\tau_{i+1}} \int_{\Sigma_\tau\cap\{r\le R\} } \langle \mathscr L_Z^{\le N-13}\psi, \gamma^\mu \mathscr L_Z^{\le N-1}\psi\rangle \nabla_{\mathcal N}\mathcal L_Z^{\le N-1}A_\mu\,d\sigma_{\Sigma_\tau}d\tau .
    \end{aligned}
\end{align}
The preceding argument gives
\begin{align*}
    & \left| \int_{\tau_i}^{\tau_{i+1}} \int_{\Sigma_\tau\cap\{r\le R\} } \langle \mathscr L_Z^{\le N-13}\psi, \gamma^\mu \mathscr L_Z^{\le N-1}\psi\rangle \nabla_{\mathcal N}\mathcal L_Z^{\le N-1}A_\mu\,d\sigma_{\Sigma_\tau}d\tau  \right| \\
    & \lesssim \int_{\tau_i}^{\tau_{i+1}} {}_{\rm int}\mathcal E^D_{\le N-11}[\psi] {}_{\rm int}\mathcal E^D_{\le N-2}[\psi] {}_{\rm int}\mathcal E^W_{\le N-1}[A_\mu]\,d\tau \\
    & \lesssim \sup_{\tau\in[\tau_i,\tau_{i+1}]} {}_{\rm int}\mathcal E^W_{\le N-1}[A_\mu](\tau) \left( \int_{\tau_i}^{\tau_{i+1}} {}_{\rm int}\mathcal E^D_{\le N-10}[\psi]^2(\tau)\,d\tau  \right)^{\frac12}  \left( \int_{\tau_i}^{\tau_{i+1}} {}_{\rm int}\mathcal E^D_{\le N-2}[\psi]^2(\tau)\,d\tau , \right)^{\frac12}.
\end{align*}
Therefore, we deduce that the interior energy of the gauge field is uniformly bounded at the top-order.

It remains to estimate the interior energy away from the trapping region, where decay is available. The proof follows the argument in Section~\ref{sec:interior}; we omit the details.

\subsection{Energy estimates for the spinor fields}
This subsection is devoted to the energy estimates for the spinor fields.

It suffices to consider the far region. As noted above, the top-order Dirac-current energy is not needed to close the nonlinear estimates. We therefore consider only $k\le N-1$.

We consider the integral
\begin{align}
  \sum_{k_1+k_2\le k}  \int_{\mathcal D^{\tau_{i+1}}_{\tau_i}} \langle  \mathcal L_Z^{\le k_1} A_{\mu}\gamma^\mu \mathscr L_Z^{\le k_2} \psi , \mathscr L_Z^{\le k}\psi\rangle\,dV .
\end{align}

If $\mu=L$, it suffices to consider
\begin{align}
\int_{\mathcal D^{\tau_{i+1}}_{\tau_i}} \langle \mathcal L_Z^{\le N-13}A_L \gamma^L \mathscr L_Z^{\le N-1}\psi , \mathscr L_Z^{\le N-1}\psi \rangle\,dV 	, \\
\int_{\mathcal D^{\tau_{i+1}}_{\tau_i}} \langle \mathcal L_Z^{\le N-1}A_L \gamma^L \mathscr L_Z^{\le N-13}\psi , \mathscr L_Z^{\le N-1}\psi \rangle\,dV .
\end{align}


We focus on the case in which the vector fields in the higher-order term are all rotations. Schematically,
\begin{align*}
	\int_{\mathcal D^{\tau_{i+1}}_{\tau_i}}  A_L \cdot \mathscr L_{\Omega}^{\le N-1}\psi_+ \cdot  \mathscr L_{\Omega}^{\le N-1}\psi_+ \,dV = \int_{\mathcal D^{\tau_{i+1}}_{\tau_i}} r  A_L \cdot \mathscr L_{\Omega}^{\le N-1}\psi_+ \cdot \slashed\nabla \mathscr L_{\Omega}^{\le N-2}\psi_+ \,dV .
\end{align*}
The weighted Sobolev inequality gives
\begin{align*}
	& \left| \int_{\mathcal D^{\tau_{i+1}}_{\tau_i}} r  A_L \cdot \mathscr L_{\Omega}^{\le N-1}\psi_+ \cdot \slashed\nabla \mathscr L_{\Omega}^{\le N-2}\psi_+ \,dV \right| \\
	& \lesssim \int_{\tau_i}^{\tau_{i+1}} r^{\frac12}\mathcal E^W_{\le N-11}[A_L] \mathscr E^D_{\le N-1}[\psi] \mathcal E^D_{\le N-2}[\psi]\,d\tau \\
	& \lesssim (\tau_{i+1})^{\frac\delta2} \sup_{\tau\in[\tau_i,\tau_{i+1}]} \mathscr E^D_{\le N-1}[\psi](\tau) \left( \int_{\tau_i}^{\tau_{i+1}}{}^{(0)}\mathcal E^W_{\le N-11}[A_L]^2(\tau)\,d\tau  \right)^\frac12 \left( \int_{\tau_i}^{\tau_{i+1}}{}^{(1-\delta)}\mathcal E^D_{\le N-2}[\psi]^2(\tau)\,d\tau  \right)^\frac12,
\end{align*}
which is uniformly bounded. For the high-low interaction, we consider the integral of the form:
\begin{align}
	\int_{\mathcal D^{\tau_{i+1}}_{\tau_i}} \mathcal L_{\Omega}^{\le N-1} A_L \cdot\psi_+ \cdot  \mathscr L_{\Omega}^{\le N-1}\psi_+ \,dV .
\end{align}
Then we see that
\begin{align*}
	& \left| \int_{\mathcal D^{\tau_{i+1}}_{\tau_i}} \mathcal L_{\Omega}^{\le N-1} A_L \cdot\psi_+ \cdot  \mathscr L_{\Omega}^{\le N-1}\psi_+ \,dV \right| \\
	& \lesssim \int_{\mathcal D^{\tau_{i+1}}_{\tau_i}} r^2 | \slashed\nabla \mathcal L_{\Omega}^{\le N-2} A_L| |\psi_+|  |  \slashed\nabla\mathscr L_{\Omega}^{\le N-2}\psi_+| \,dV \\
	& \lesssim \int_{\tau_i}^{\tau_{i+1}} r^{\frac32} \mathcal E^W_{\le N-2}[A_L] \mathcal E^D_{\le N-11}[\psi] \mathcal E^D_{\le N-2}[\psi]\,d\tau \\
	& \lesssim (\tau_{i+1})^{\frac\delta2} \sup_{\tau\in[\tau_i,\tau_{i+1}]} {}^{(1)}\mathcal E^W_{\le N-2}[A_L](\tau) \left(  \int_{\tau_i}^{\tau_{i+1}} {}^{(1)}\mathcal E^D_{\le N-11}[\psi]^2(\tau)\,d\tau  \right)^{\frac12} \left(  \int_{\tau_i}^{\tau_{i+1}} {}^{(1-\delta)}\mathcal E^D_{\le N-2}[\psi]^2(\tau)\,d\tau  \right)^{\frac12}.
\end{align*}
We now consider the case $\mu=\underline L$, beginning with the low--high interaction:
\begin{align}
	\int_{\mathcal D^{\tau_{i+1}}_{\tau_i}}  \mathcal L_Z^{\le N-13}A_{\underline L} \cdot \mathscr L_Z^{\le N-1}\psi_- \cdot \mathscr L_Z^{\le N-1}\psi \,dV. 
\end{align}
If $k\le N-1$, it suffices to consider the following integral:
\begin{align}
	\int_{\mathcal D^{\tau_{i+1}}_{\tau_i}}  A_{\underline L}\cdot\mathscr L_{\Omega}^{\le N-1} \psi_- \cdot  \mathscr L_{\Omega}^{\le N-1}\psi_- \,dV.
\end{align}
Using the null decomposition for the spinor fields, we simply have
\begin{align*}
	& \left| \int_{\mathcal D^{\tau_{i+1}}_{\tau_i}}  A_{\underline L}\cdot\mathscr L_{\Omega}^{\le N-1} \psi_- \cdot  \mathscr L_{\Omega}^{\le N-1}\psi_- \,dV \right| \\
	& \lesssim \int_{\mathcal D^{\tau_{i+1}}_{\tau_i}} r^2  |A_{\underline L}| |\nabla_L\mathscr L_{\Omega}^{\le N-2} \psi_+|^2 \,dV \\
	& \lesssim \int_{\tau_i}^{\tau_{i+1}} r^{\frac32}\mathcal E^W_{\le N-11}[A_{\underline L}] \mathcal E^D_{\le N-2}[\psi]^2\,d\tau \\
	& \lesssim (\tau_{i+1})^{\frac\delta2} \sup_{\tau\in[\tau_i,\tau_{i+1}]} {}^{(1)}\mathcal E^W_{\le N-11}[A_{\underline L}](\tau) \int_{\tau_i}^{\tau_{i+1}}{}^{(1-\delta)} \mathcal E^D_{\le N-2}[\psi]^2\,d\tau.
\end{align*}
On the other hand, for the high-low interaction, we see that
\begin{align*}
	&  \left| \int_{\mathcal D^{\tau_{i+1}}_{\tau_i}} \mathcal L_{\Omega}^{\le N-1} A_{\underline L}\cdot\psi_- \cdot  \mathscr L_{\Omega}^{\le N-1}\psi_- \,dV \right| \\
	& \lesssim \int_{\mathcal D^{\tau_{i+1}}_{\tau_i}}r^2 |\slashed\nabla\mathcal L_{\Omega}^{\le N-2} A_{\underline L}| |\psi_-|  |\nabla_L \mathscr L_{\Omega}^{\le N-2}\psi_+| \,dV \\
	& \lesssim \int_{\tau_i}^{\tau_{i+1}} r^{\frac32} \mathcal E^W_{\le N-2}[A_{\underline L}] \mathcal E^D_{\le N-11}[\psi] \mathcal E^D_{\le N-2}[\psi] \,d\tau \\
	& \lesssim  (\tau_{i+1})^{\frac\delta2} \sup_{\tau\in[\tau_i,\tau_{i+1}]} {}^{(1)}\mathcal E^W_{\le N-2}[A_{\underline L}](\tau) \left(  \int_{\tau_i}^{\tau_{i+1}} {}^{(1)}\mathcal E^D_{\le N-11}[\psi]^2(\tau)\,d\tau  \right)^{\frac12} \left(  \int_{\tau_i}^{\tau_{i+1}} {}^{(1-\delta)}\mathcal E^D_{\le N-2}[\psi]^2(\tau)\,d\tau  \right)^{\frac12}.
\end{align*}
Finally, it remains to estimate the integral
\begin{align}
	\int_{\mathcal D^{\tau_{i+1}}_{\tau_i}} \langle \mathcal L_Z^{\le k_1}\slashed{A} \mathscr L_Z^{\le k_2}\psi, \mathscr L_Z^{\le k}\psi\rangle \,dV,
\end{align}
with $k_1+k_2\le k\le N-1$. 

For $k\le N-1$, we consider the integral of the form:
\begin{align}
\int_{\mathcal D^{\tau_{i+1}}_{\tau_i}} \slashed{A}\cdot \mathscr L_{\Omega}^{\le N-1}\psi_+ \cdot \mathscr L_{\Omega}^{\le N-1}\psi_-\,dV.	
\end{align}
This estimate is also straightforward. Applying the weighted Sobolev inequality on outgoing null hypersurfaces gives
\begin{align*}
	& \left| \int_{\mathcal D^{\tau_{i+1}}_{\tau_i}} \slashed{A}\cdot \mathscr L_{\Omega}^{\le N-1}\psi_+ \cdot \mathscr L_{\Omega}^{\le N-1}\psi_-\,dV \right| \\
	& \lesssim \int_{\mathcal D^{\tau_{i+1}}_{\tau_i}} r |\slashed{A}| | \mathscr L_{\Omega}^{\le N-1}\psi_+| |\nabla_L \mathscr L_{\Omega}^{\le N-2}\psi_+|\,dV \\
	& \lesssim \int_{\tau_i}^{\tau_{i+1}} r^{\frac12} \mathcal E^W_{\le N-11}[\slashed{A}] \mathscr E^D_{\le N-1}[\psi] \mathcal E^D_{\le N-2}[\psi] \,d\tau \\
	& \lesssim (\tau_{i+1})^{\frac\delta2} \sup_{\tau\in[\tau_{i},\tau_{i+1}]}\mathscr E^D_{\le N-1}[\psi](\tau) \left( \int_{\tau_i}^{\tau_{i+1}} {}^{(0)}\mathcal E^W_{\le N-11}[\slashed{A}]^2(\tau)\,d\tau \right)^{\frac12} \left( \int_{\tau_i}^{\tau_{i+1}} {}^{(1-\delta)}\mathcal E^D_{\le N-2}[\psi]^2(\tau)\,d\tau \right)^{\frac12} .
\end{align*}
Finally, it remains to treat the high-low interaction. We must distinguish whether the higher-order spinor field is the $+$ or the $-$ component. 
When the higher-order spinor field is the $-$ component, the weighted Sobolev inequality on outgoing null hypersurfaces gives
\begin{align*}
	& \left|  \int_{\mathcal D^{\tau_{i+1}}_{\tau_i}} \mathcal L_{\Omega}^{\le N-1}\slashed{A} \cdot  \psi_+ \cdot \mathscr L_{\Omega}^{\le N-1}\psi_-\,dV \right| \\
	& \lesssim  \int_{\mathcal D^{\tau_{i+1}}_{\tau_i}} r^2 | \slashed\nabla \mathcal L_{\Omega}^{\le N-2}\slashed{A} | |  \psi_+ | | \nabla_L\mathscr L_{\Omega}^{\le N-2}\psi_+|\,dV \\
	& \lesssim \int_{\tau_i}^{\tau_{i+1}} r^{\frac32} \mathcal E^W_{\le N-2}[\slashed{A}] \|\nabla_L \mathscr L_{Z}^{\le N-11}\psi_+\|_{L^2(\Sigma_{\tau})} \mathcal E^D_{\le N-2}[\psi]\,d\tau \\
	& \lesssim (\tau_{i+1})^{\frac\delta2} \sup_{\tau\in[\tau_i,\tau_{i+1}]}{}^{(1)}\mathcal E^W_{\le N-2}[\slashed{A}](\tau)  \left( \int_{\tau_i}^{\tau_{i+1}} {}^{(1)}\mathcal E^D_{\le N-11}[\psi]^2(\tau)\,d\tau \right)^{\frac12} \left( \int_{\tau_i}^{\tau_{i+1}} {}^{(1-\delta)}\mathcal E^D_{\le N-2}[\psi]^2(\tau)\,d\tau \right)^{\frac12} .
\end{align*}

We now consider the case in which the higher-order spinor field is the $+$ component:

\begin{align}
\int_{\mathcal D^{\tau_{i+1}}_{\tau_i}} \mathcal L_{\Omega}^{\le N-1}\slashed{A} \cdot \mathscr L_{\Omega}^{\le N-1}\psi_+ \cdot \psi_-\,dV.	
\end{align}
To do this, we write
\begin{align*}
	& \left| \int_{\mathcal D^{\tau_{i+1}}_{\tau_i}} \mathcal L_{\Omega}^{\le N-1}\slashed{A} \cdot \mathscr L_{\Omega}^{\le N-1}\psi_+ \cdot \psi_-\,dV \right| \\
	& \lesssim \int_{\tau_{i}}^{\tau_{i+1}} r^{\frac12} \mathcal E^W_{\le N-2}[\slashed{A}] \mathscr E^D_{\le N-1}[\psi] \mathcal E^D_{\le N-11}[\psi]\,d\tau \\
	& \lesssim \sup_{\tau\in[\tau_i,\tau_{i+1]}} \mathscr E^D_{\le N-1}[\psi](\tau) \left( \int_{\tau_i}^{\tau_{i+1}}{}^{(1)} \mathcal E^D_{\le N-11}[\psi]^2(\tau)\,d\tau   \right)^{\frac12}  \left( \int_{\tau_i}^{\tau_{i+1}}{}^{(0)} \mathcal E^W_{\le N-2}[\slashed{A}]^2(\tau)\,d\tau   \right)^{\frac12}.
\end{align*}
\section*{Acknowledgements}
The author would like to express his sincere gratitude to Gustav Holzegel for his guidance and many valuable discussions, particularly for his insightful comments and critical feedback on the formulation of the tensorial wave–Dirac model.
He acknowledges support from the Alexander
von Humboldt Foundation in the framework of the Alexander von Humboldt Professorship endowed by the
Federal Ministry of Education and Research.

This work was also supported by the Deutsche Forschungsgemeinschaft (DFG, German Research Foundation) under Germany's Excellence Strategy EXC 2044/2 -390685587, Mathematics M\"unster: Dynamics-Geometry-Structure.

\section*{Declaration of AI use}
The author used ChatGPT (OpenAI) to assist with English language editing and to improve the clarity and presentation of the manuscript. All AI-assisted suggestions were reviewed and revised by the author, who takes full responsibility for the mathematical content and the final manuscript.


\end{document}